\documentclass[10pt]{amsart}
\usepackage{amscd,amsfonts,amssymb,amsmath,amsthm,latexsym}
\usepackage{url}          
\usepackage{mathrsfs}
\usepackage{geometry}
\usepackage{float}
\usepackage[dvips]{graphicx}
\usepackage{color}
\usepackage{graphics}
\usepackage[all]{xy}

\usepackage{epstopdf}
\usepackage{mathtools}
\usepackage{tikz}
\usetikzlibrary{calc}
\usetikzlibrary{decorations.markings}

\usepackage{amscd,amsfonts,amssymb,amsmath,amsthm,latexsym}

\usepackage{graphics}
\usepackage[all]{xy}
\xyoption{curve}
\xyoption{import}
\xyoption{arc}
\xyoption{ps}

\numberwithin{equation}{section}

\theoremstyle{plain}
    \newtheorem{thm}{Theorem}[section]
    \newtheorem{lemma}[thm]{Lemma}
    \newtheorem{coro}[thm]{Corollary}
    \newtheorem{prop}[thm]{Proposition}

    \newtheorem{question}[thm]{Question}
\theoremstyle{definition}
    \newtheorem{defi}[thm]{Definition}
    \newtheorem{ex}[thm]{Example}
    \newtheorem{remark}[thm]{Remark}
\theoremstyle{remark}
    \newtheorem{case}{Case}

\newcommand{\suchthat}{\ | \ }

\newcommand{\diag}{\operatorname{diag}}

\newcommand{\RA}[1]{R\langle\hspace{-0.05cm}\langle #1\rangle\hspace{-0.05cm}\rangle}

\newcommand{\myid}{1\hspace{-0.15cm}1}
\newcommand{\pathalg}[1]{R\langle #1\rangle}
\newcommand{\maxid}{\mathfrak{m}}

\newcommand{\sqr}{{\raisebox{-1pt}{\rotatebox{45}{$\diamond$}}}}

\newcommand{\image}{\operatorname{Im}}

\newcommand{\dtuple}{\mathbf{d}}
\newcommand{\Z}{\mathbb{Z}}

\newcommand{\C}{\mathbb{C}}
\newcommand{\R}{\mathbb{R}}
\newcommand{\lcm}{\operatorname{lcm}}
\newcommand{\B}{\mathcal{B}}
\newcommand{\Gal}{\operatorname{Gal}}
\newcommand{\Aut}{\operatorname{Aut}}

\newcommand{\marked}{\mathbb{M}}
\newcommand{\orb}{\mathbb{O}}
\newcommand{\punct}{\mathbb{P}}
\newcommand{\surf}{(\Sigma,\marked,\orb)}
\newcommand{\arcsinsurf}{\mathbb{A}^{\circ}\surf}
\newcommand{\taggedinsurf}{\mathbb{A}^{\bowtie}\surf}
\newcommand{\tagfunction}{\mathfrak{t}}
\newcommand{\arc}{k}

\newcommand{\Qtau}{Q(\tau)}
\newcommand{\unredQtau}{\widehat{Q}(\tau)}
\newcommand{\Atau}{A(\tau)}
\newcommand{\unredAtau}{\widehat{A}(\tau)}
\newcommand{\Stau}{S(\tau,\mathbf{x})}
\newcommand{\unredStau}{\widehat{S}(\tau,\mathbf{x})}
\newcommand{\AStau}{(\Atau,\Stau)}
\newcommand{\unredAStau}{(\unredAtau,\unredStau)}
\newcommand{\Qsigma}{Q(\sigma)}

\newcommand{\Asigma}{A(\sigma)}

\newcommand{\Ssigma}{S(\sigma,\mathbf{x})}

\newcommand{\ASsigma}{(\Asigma,\Ssigma)}

\newcommand{\Wtau}{W^{ij}(\tau,\mathbf{x})}
\newcommand{\AWtau}{(A(\tau),\Wtau)}

\newcommand{\Wsigma}{W^{ij}(\sigma,\mathbf{x})}
\newcommand{\AWsigma}{(A(\sigma),\Wsigma)}

\newcommand{\Vtauq}{V^q(\tau,\mathbf{x})}
\newcommand{\AVtauq}{(A(\tau),\Vtauq)}

\newcommand{\Vsigmaq}{V^q(\sigma,\mathbf{x})}
\newcommand{\AVsigmaq}{(A(\sigma),\Vsigmaq)}

\newcommand{\vv}{\operatorname{v}}
\newcommand{\val}{\operatorname{val}}

\newcommand{\short}{\operatorname{short}}

\newcommand{\depth}{\operatorname{depth}}

\newcommand{\BB}{\mathbb{B}}
\newcommand{\CC}{\mathbb{C}}
\newcommand{\tildeB}{\widetilde{\mathbb{B}}}
\newcommand{\tildeC}{\widetilde{\mathbb{C}}}

\begin{document}

\title[SPs arising from surfaces with orbifold points, part I]{Species with potential arising from surfaces with orbifold points of order 2, Part I: one choice of weights}
\author{Jan Geuenich}
\author{Daniel Labardini-Fragoso}
\address{Mathematisches Institut, Universit\"at Bonn, Germany}
\email{jagek@math.uni-bonn.de}
\address{Instituto de Matem\'aticas, Universidad Nacional Aut\'onoma de M\'exico}
\email{labardini@matem.unam.mx}
\date{\today}
\dedicatory{Dedicated to Jerzy Weyman on the occasion of his sixtieth birthday}
\maketitle

\begin{abstract} We present a definition of mutations of species with potential that can be applied to
the species realizations of any skew-symmetrizable matrix~$B$ over cyclic Galois extensions $E/F$ whose base field $F$ has a primitive $[E:F]^{\operatorname{th}}$ root of unity. After providing an example of a globally unfoldable skew-symmetrizable matrix whose species realizations do not admit non-degenerate potentials, we present a construction that associates a species with potential to each tagged triangulation of a surface with marked points and orbifold points of order 2. Then we prove that for any two tagged triangulations related by a flip, the associated species with potential are related by the corresponding mutation (up to a possible change of sign at a cycle), thus showing that these species with potential are non-degenerate. In the absence of orbifold points, the constructions and results specialize to previous work (\cite{Labardini1} and \cite{Labardini-potsnoboundaryrevised}) by the second author.

The species constructed here for each triangulation $\tau$ is a species realization of one of the several matrices that Felikson-Shapiro-Tumarkin have associated to $\tau$ in \cite{FeShTu-orbifolds}, namely, the one that in their setting arises from choosing the number $\frac{1}{2}$ for every orbifold point.
\end{abstract}

\tableofcontents

\section{Introduction}

The quivers with potential associated to triangulations of surfaces with marked points, and the compatibility between QP-mutations and flips of triangulations proved in \cite{Labardini1} and \cite{Labardini-potsnoboundaryrevised}, have proven useful
in different areas of mathematics (cf.\ for example \cite{Bridgeland-Smith}, \cite{DMSS}, \cite{Nagao1}, \cite{Nagao2}, \cite{Smith}) and physics (cf.\ for example \cite{ACCERV1}, \cite{ACCERV2}, \cite{Cecotti}). In this paper we extend the constructions and results from \cite{Labardini1} and \cite{Labardini-potsnoboundaryrevised} to surfaces with marked points and orbifold points of order 2. To be more precise, the main construction of this paper (Definition \ref{def:SP-of-tagged-triangulation}) associates a species and a potential to each triangulation (either ideal or tagged) of a surface with marked points and orbifold points of order 2, and the main result (Theorem \ref{thm:tagged-flips<->SP-mutation--non-empty-boundary}) shows that whenever two triangulations (either ideal or tagged) are related by a flip, the associated species with potential are related by the corresponding SP-mutation. This result in particular establishes the existence of non-degenerate potentials on the species associated here to the (tagged) triangulations of surfaces with marked points and orbifold points of order 2.

Surfaces with marked points and orbifold points arise as one of the combinatorial upshots of hyperbolic geometry: Let $\Gamma$ be a discrete subgroup of $\operatorname{PSL}_2(\mathbb{R})$ for which the orbit space $\mathbb{H}^2/\Gamma$ has finite hyperbolic area (where $\mathbb{H}^2$ is the hyperbolic upper half plane).
If we ignore the hyperbolic metric on $\mathbb{H}^2/\Gamma$ inherited from that on $\mathbb{H}^2$, we can combinatorially describe $\mathbb{H}^2/\Gamma$ by means of four pieces of data:
\begin{itemize}
\item A compact oriented 2-dimensional real surface $\Sigma$ (with empty boundary);
\item a set $\marked$ of \emph{marked points} on $\Sigma$, which correspond to the ($\Gamma$-orbits of) fixed points of parabolic elements of $\Gamma$, and whose removal from $\Sigma$ produces a surface $\Sigma\setminus\marked$ homeomorphic to $\mathbb{H}^2/\Gamma$;
\item a set $\orb$ of \emph{orbifold points} on $\Sigma$, which correspond to the ($\Gamma$-orbits of) fixed points of elliptic elements of $\Gamma$;
\item the \emph{orders} of the orbifold points, which correspond to the orders of the (conjugacy classes of) maximal elliptic subgroups of $\Gamma$ (by definition of elliptic subgroups, these orders are integers greater than 1).
\end{itemize}
If we are interested in all the possible ways to obtain a fixed set of combinatorial data as above via discrete subgroups of $\operatorname{PSL}_2(\R)$, or, informally speaking, in all possible ways of defining a hyperbolic metric on $\Sigma\setminus\marked$ with $\orb$ as the set of orbifold points with the prescribed orders, we arrive at a notion of \emph{Teichm\"uller space}, defined, when $\orb=\varnothing$, as certain subset of the set of conjugacy classes of injective group homomorphisms from the fundamental group $\pi_1(\Sigma\setminus\marked)$ onto discrete subgroups $\Gamma$ of $\operatorname{PSL}_2(\mathbb{R})$ for which the orbit space $\mathbb{H}^2/\Gamma$ has finite hyperbolic area and yields the given combinatorial information (see \cite[Section 2.1]{Penner}).

Penner \cite{Penner} has shown that for any four ideal points of the hyperbolic plane $\mathbb{H}^2$, the \emph{lambda lengths} of the six hyperbolic geodesics determined by these four points satisfy the famous Ptolemy identity $\lambda_{13}\lambda_{24}=\lambda_{12}\lambda_{34}+\lambda_{14}\lambda_{23}$. Building on this idea, Penner \cite{Penner}, Fomin-Shapiro-Thurston \cite{FST} and Fomin-Thurston \cite{FT} have shown that in the absence of orbifold points and the presence of at least one marked point, the subring that the \emph{lambda lengths} of the (tagged) arcs on $(\Sigma,\marked)$ generate in the ring of real-valued functions on the \emph{decorated Teichm\"uller space} of $(\Sigma,\marked)$ carries a natural cluster algebra structure, whose cluster variables are parameterized by the (tagged) arcs on $(\Sigma,\marked)$, and whose clusters are parameterized by the (tagged) triangulations of $(\Sigma,\marked)$, with each (tagged) triangulation yielding a full coordinatization of the decorated Teichm\"uller space via the lambda lengths of its constituent (tagged) arcs, and with the change of coordinates associated to every flip\footnote{Particularly delicate are the flips of \emph{folded sides} of \emph{self-folded triangles}. \emph{Tagged arcs} and \emph{tagged triangulations} are the combinatorial gadgets with which Fomin-Shapiro-Thurston were able to define these flips.}
of (tagged) triangulations being governed by the \emph{Ptolemy relations}. Felikson-Shapiro-Tumarkin \cite{FeShTu-orbifolds} have generalized the results of Penner and Fomin-Thurston so as to encompass the situation where all orbifold points are assumed to have order 2. For the case of orbifold points of arbitrary order, Chekhov-Shapiro \cite{Chekhov-Shapiro} have shown that the natural cluster-like structure provided by lambda lengths is no longer necessarily a cluster algebra, but that it still behaves as a \emph{generalized cluster algebra}.

In this paper we focus on surfaces with marked points and orbifold points of order 2. Let $\surf$ be one such surface. A crucial observation in the works of Fomin-Shapiro-Thurston and Felikson-Shapiro-Tumarkin is that, in this case, to every triangulation of $\surf$ one can associate a \emph{skew-symmetrizable matrix} (equivalently, a \emph{weighted quiver}, that is, a pair $(Q,\dtuple)$ consisting of a loop-free quiver $Q$ and a tuple $\dtuple=(d_i)_{i\in Q_0}$ of positive integers attached to the vertices of $Q$) naturally associated to it, and that, moreover, the combinatorial operation of \emph{flip} on triangulations (also known as \emph{Whitehead move}) is reflected at the matrix level as a \emph{matrix mutation}\footnote{Actually, Felikson-Shapiro-Tumarkin associate $2^{|\orb|}$ different skew-symmetrizable matrices to each single triangulation of $\surf$.}, which is the main ingredient in Fomin-Zelevinsky's definition of cluster algebras. For surfaces without orbifold points, it was shown in \cite{Labardini1} and \cite{Labardini-potsnoboundaryrevised} that this compatibility between flips and mutations is actually stronger: not only does every (tagged) triangulation have a naturally associated matrix (which is skew-symmetric in this case and hence corresponds to a simply-laced quiver --that is, a directed graph with possibly multiple parallel arrows), but actually a naturally associated \emph{quiver with potential}, and, moreover, for (tagged) triangulations related by a flip, not only are their associated matrices (i.e., simply-laced quivers) related by Fomin-Zelevinsky mutation, but their associated quivers with potential are related by a \emph{mutation of quivers with potential}, which is the main ingredient in Derksen-Weyman-Zelevinsky's representation-theoretic approach to cluster algebras. Combined with general results of Amiot \cite{Amiot} and Keller-Yang \cite{Keller-Yang}, this compatibility between flips and QP-mutations yields a categorification, via 2-Calabi-Yau and 3-Calabi-Yau triangulated categories, of the cluster algebras associated to surfaces with marked points (and without orbifold points). The resulting categories have turned out to be useful in symplectic geometry and in the subject of stability conditions.

Felikson-Shapiro-Tumarkin associate several different skew-symmetrizable matrices to each single triangulation $\tau$ of $\surf$. In their article \cite{FeShTu-orbifolds}, these matrices arise from the choice of a number from $\{\frac{1}{2},2\}$ for each orbifold point, see Remark \ref{rem:tau=t_1(tau)}. In this paper we take one of the matrices associated with $\tau$, give a species realization of this matrix, define an explicit potential on this species, and show that whenever two triangulations (either ideal or tagged) are related by a flip, the associated species with potential are related by the corresponding SP-mutation (up to a possible change of sign at a cycle).

The article is organized as follows. Section~\ref{sec:background} contains background material: In
Subsection~\ref{subsec:surfaces-with-orbifold-points-and-triangulations} we recall Felikson-Shapiro-Tumarkin's definition of surfaces with marked points and orbifold points of order 2, as well as of the combinatorial notions of ideal and tagged triangulations of these surfaces, emphasizing the fact that any such triangulation can be obtained from the gluing of a finite set of elementary ``puzzle pieces'' (a fact that is crucial for the proof of our main result). Among the features of ideal and tagged triangulations is the fact that given any such triangulation $\tau$, each orbifold point lies on exactly one arc of $\tau$; an arc that contains an orbifold point is called \emph{pending}, otherwise it is \emph{non-pending}. By the end of Subsection \ref{subsec:surfaces-with-orbifold-points-and-triangulations} we also take the opportunity to illustrate the hyperbolic-geometric origins of our use of the term ``orbifold point of order 2'' and of the notion of ideal triangulation, use and notion that are of an admittedly combinatorial nature in this paper.

Subsection~\ref{subsec:tensor-products-and-bimodules} is devoted to emphasizing a few basic facts concerning bimodules and their tensor products and tensor algebras that will be used throughout the paper.
In particular, we study the direct sum decompositions of the tensor products of the bimodules that will be used as building blocks for species; more precisely, Proposition \ref{prop:tensor-prod-decomp} will tell us that in such decompositions one is forced to consider certain variations of the bimodules $F_i\otimes_{F_{i}\cap F_j}F_j$ that are ``twisted'' by elements of the Galois group $\Gal(F_i\cap F_j/F)$ and are not isomorphic to $F_i\otimes_{F_{i}\cap F_j}F_j$ (here, $F_i$ and $F_j$ are field subextensions of certain cyclic Galois extension $E/F$).

Subsection~\ref{subsec:matrices-vs-quivers} is concerned with the correspondence between skew-symmetrizable matrices and 2-acyclic weighted quivers, and with the translation of the notion of matrix mutation to the language of weighted quivers, while Subsection \ref{subsec:species-realizations} recalls the notion of \emph{species realization} of a skew-symmetrizable matrix and gives a brief account on previous approaches to mutations of species with potential.

Section~\ref{sec:SPs} establishes the algebraic setting used in the rest of the article for the construction and mutations of species and potentials.
Given a weighted quiver $(Q,\dtuple)$, we discuss how the choice of certain cyclic Galois extension $E/F$ allows to attach a field $F_i$ of degree $d_i$ over $F$ to each vertex $i\in Q_0$, and how the choice of a function $g:Q_1\rightarrow\bigcup_{i,j\in Q_0}\Gal(F_i\cap F_j/F)$ attaching to each arrow $a:j\rightarrow i$ an element of the Galois group $\Gal(F_{i}\cap F_{j}/F)$ (we will say that $g$ is a \emph{modulating function} for $(Q,\dtuple)$), allows to associate a ``twisted'' bimodule $F_i^{g_a}\otimes_{F_{i}\cap F_{j}}F_j$ to each such arrow $a:j\rightarrow i$, thus naturally allowing us to associate a \emph{species} to $(Q,\dtuple)$ with respect to the modulating function $g$. This species will be a bimodule over the commutative semisimple ring $R=\bigoplus_{i\in Q_0}F_i$; as such, it will give rise to a (complete) tensor algebra, and this will allow us to generalize some well-known concepts for quivers to the species setting.
In particular, potentials, cyclic derivatives, Jacobian algebras, and mutations of species with potential will be lifted to this setting. We will see that many of the main results from~\cite{DWZ1} regarding mutations of quivers with potential can be generalized to analogous results for species with potential (e.g., the Splitting Theorem, the well-definedness of SP-mutations up to right-equivalence, the involutivity of SP-mutations).
However, an unpleasant behavior of SP-mutations will be pointed out in Example \ref{ex:6263-cycle-species}, which exhibits a skew-symmetrizable matrix $B$ that admits a global unfolding in the sense of Zelevinsky (cf. the unpublished \cite{Zelevinsky-unfolding}, see \cite[Section 4]{FeShTu-unfoldings} or \cite[Definition 14.2]{LZ}) and has the property that, no matter which species realization of $B$ via cyclic Galois extensions and modulating functions we take, the resulting species will not admit non-degenerate potentials at all. This means that the existence of species realizations admitting non-degenerate potentials
should not be expected even for globally unfoldable skew-symmetrizable matrices.
So, even though many of Derksen-Weyman-Zelevinsky's statements are still valid in our species setup, somehow the most important one (namely, the one that asserts the existence of non-degenerate potentials) fails to hold in general.

Section~\ref{sec:species-of-a-triangulation} is devoted to the definition of the weighted quiver $(Q(\tau),\dtuple(\tau))$ and of the species $\Atau$ associated to a tagged triangulation $\tau$ of a surface with marked points and orbifold points of order 2. The weighted quiver $(Q(\tau),\dtuple(\tau))$ really is the weighted quiver corresponding to one of the several skew-symmetrizable matrices associated to $\tau$ by Felikson-Shapiro-Tumarkin in \cite{FeShTu-orbifolds}, where these authors show (in the language of skew-symmetrizable matrices) that if two tagged triangulations are related by the flip of a tagged arc $k$, then the associated weighted quivers are related by the $k^{\operatorname{th}}$ weighted-quiver mutation. The vertices of $Q(\tau)$ are the arcs in $\tau$, the weight tuple $\dtuple(\tau)$ attaches the weight 2 to each pending arc and the weight 1 to each non-pending arc\footnote{Our notion of weight is different from the notion used by Felikson-Shapiro-Tumarkin in \cite{FeShTu-orbifolds}, but it is related to it. The precise relation is given in Remark \ref{rem:tau=t_1(tau)}.}. As for the arrows, they are defined based on the idea of drawing arrows in the clockwise direction inside each triangle, although one has to be careful when drawing the arrows of certain triangulations (e.g., the triangulations that have self-folded triangles). On the species level, $\Atau$ will be defined as the species associated with respect to a specific choice of elements of certain Galois groups. The main feature of this choice is that it will attach non-identity elements to some arrows connecting two pending arcs.

In Section~\ref{sec:potential-of-a-triangulation} we present the construction of the potential $\Stau\in\RA{\Atau}$ associated to a tagged triangulation $\tau$ with respect to a given tuple $\mathbf{x}=(x_p)_{p\in\punct}$ of non-zero elements\footnote{Our reason for allowing arbitrary tuples is the desire to obtain as many non-degenerate potentials on $\Atau$ as possible.} of the ground field $F$.
After giving some examples, we start the road towards the proof of the main theorem of this paper, beginning with Remark \ref{rem:ultimate-goal-and-steps} and ending with Theorem \ref{thm:tagged-flips<->SP-mutation--non-empty-boundary}. The main result of Section \ref{sec:potential-of-a-triangulation} is Theorem \ref{thm:ideal-non-pending-flips<->SP-mutation}, which states that if $\surf$ has empty boundary, $\tau$ is an ideal triangulation of $\surf$, and $k\in\tau$ is an arc which is neither pending nor the folded side of any self-folded triangle of $\tau$, then the SP $\mu_k\AStau$ is right-equivalent to $\ASsigma$, where $\sigma$ is the ideal triangulation obtained from $\tau$ by flipping the arc $k$. The proof of this assertion is quite long as it is achieved through a case-by-case analysis of all possible configurations that the puzzle-piece decompositions of $\tau$ and $\sigma$ can present around $k$. We have deferred this proof to Section \ref{sec:technical-proofs}.

In general, the flips of folded sides of self-folded triangles and the flips of pending arcs are not as obviously compatible with SP-mutation as one would like (some unexpected cycles appear when performing the corresponding SP-mutation, see \cite[Section 4]{Labardini-potsnoboundaryrevised}).
For this reason, in
Section~\ref{sec:pops} we introduce two types of \emph{popped potentials}, associated to ideal triangulations with respect to self-folded triangles and with respect to orbifold points, and establish the relation that these popped potentials have with the potentials $\Stau$ introduced in Section \ref{sec:potential-of-a-triangulation}. To be more precise, given an ideal triangulation $\tau$ and arcs $i,j\in\tau$ forming a self-folded triangle with $i$ as folded side, in Subsection \ref{subsec:self-folded-pop} we define a popped potential $\Wtau$ and show in Theorem \ref{thm:popping-is-right-equiv} that it is always right-equivalent to the potential $\Stau$. Similarly, given an ideal triangulation $\tau$ and an orbifold point $q$, we define a popped potential $\Vtauq$ and show in Theorem \ref{thm:orbi-pop} that it is always right-equivalent to $\Stau$. The proofs of Theorems \ref{thm:popping-is-right-equiv} and \ref{thm:orbi-pop} are quite similar to the proof of \cite[Theorem 6.1]{Labardini-potsnoboundaryrevised}; each of them is reduced to showing the existence of ideal triangulations for which the corresponding popped potentials are right-equivalent to $\Stau$. That such ideal triangulations actually exist is stated in Propositions \ref{prop:pop-existence} and \ref{prop:orb-pop-stronger}. We have deferred the proofs of these propositions to Section \ref{sec:technical-proofs}; each of them consists in exhibiting an ideal triangulation for which a right-equivalence between $\Stau$ and the corresponding popped potentials can be constructed through an \emph{ad hoc} limit process.

With Theorems \ref{thm:ideal-non-pending-flips<->SP-mutation}, \ref{thm:popping-is-right-equiv} and \ref{thm:orbi-pop} at hand, in Section~\ref{sec:flip-and-mutation-compatibility} we show the first main result of this paper, which states that if $\surf$ has empty boundary, and $\tau$ and $\sigma$ are tagged triangulations of $\surf$ related by the flip of a tagged arc $k\in\tau$, then the SP $\mu_k\AStau$ is right-equivalent to $(A(\sigma),S(\sigma,\mathbf{y}))$ for some tuple $\mathbf{y}=(y_p)_{p\in\punct}$ of non-zero elements of $F$. This result, stated as Theorem \ref{thm:tagged-flips<->SP-mutation}, is proved in three steps: First, with the aid of Theorem \ref{thm:popping-is-right-equiv} we show in Theorem \ref{thm:leaving-positive-stratum} that if $\tau$ is an ideal triangulation and $k\in\tau$ is the folded side of a self-folded triangle of $\tau$, then $\mu_k\AStau$ is right-equivalent to $\ASsigma$. Secondly, using Theorem \ref{thm:orbi-pop} we prove in Theorem \ref{thm:flip-pending-arc<->SP-mut} that if $\tau$ is an ideal triangulation and $k\in\tau$ is a pending arc, then $\mu_k\AStau$ is right-equivalent to $(A(\sigma),S(\sigma,\mathbf{y}))$, where $\mathbf{y}=(y_p)_{p\in\punct}$ is a tuple obtained from $\mathbf{x}$ by changing the sign of (exactly) one of the scalars $x_p$ (the puncture $p$ at which the sign change occurs is completely determined by $k$). And thirdly, we deduce Theorem \ref{thm:tagged-flips<->SP-mutation} from Theorems \ref{thm:ideal-non-pending-flips<->SP-mutation}, \ref{thm:leaving-positive-stratum} and Theorem \ref{thm:flip-pending-arc<->SP-mut} through a ``notch deletion'' process first introduced by Fomin-Shapiro-Thurston for surfaces without orbifold points.

In Section \ref{sec:empty=>general} we present our main result, Theorem \ref{thm:tagged-flips<->SP-mutation--non-empty-boundary}, which states that regardless of whether $\surf$ has empty boundary or not, if $\tau$ and $\sigma$ are tagged triangulations of $\surf$ related by the flip of a tagged arc $k\in\tau$, then the SP $\mu_k\AStau$ is right-equivalent to $(A(\sigma),S(\sigma,\mathbf{y}))$ for some tuple $\mathbf{y}=(y_p)_{p\in\punct}$ of non-zero elements of $F$. To prove Theorem \ref{thm:tagged-flips<->SP-mutation--non-empty-boundary} we first show that Derksen-Weyman-Zelevinsky's notion of restriction of quivers with potential can be straightforwardly extended to the setting of SPs, that restriction of SPs commutes with SP-mutations, and that the SP of any tagged triangulation $\tau$ of any surface $\surf$ can be obtained as the restriction of the SP of some tagged triangulation $\widetilde{\tau}$ of a surface with empty boundary ($\widetilde{\tau}$ can be constructed by gluing triangulated punctured polygons along the boundary components of $\Sigma$). Then we combine these facts with Theorem \ref{thm:tagged-flips<->SP-mutation} and produce a proof of Theorem \ref{thm:tagged-flips<->SP-mutation--non-empty-boundary}. As a direct consequence of Theorem \ref{thm:tagged-flips<->SP-mutation--non-empty-boundary} we obtain Corollary \ref{coro:non-degeneracy-of-S(tau)}, which states the non-degeneracy of the SPs $\AStau$. We close Section \ref{sec:empty=>general} by showing that if the boundary of $\surf$ is not empty, then the tuple $\mathbf{y}$ can be taken to be equal to the tuple $\mathbf{x}$.

Section~\ref{sec:examples} provides some illustrative examples.
In particular, in Subsections \ref{subsec:unpunctured-surfaces}, \ref{subsec:type-C} and \ref{subsec:torus-1-orb-point} we consider SPs arising from unpunctured surfaces, which include species realizations for weighted quivers of Dynkin type~$\CC_n$ and Euclidean type~$\tildeC_n$.

Finally, Section \ref{sec:technical-proofs} contains technical considerations and proofs. Theorems \ref{thm:splitting-theorem}, \ref{thm:SP-mut-well-defined-up-to-re} and \ref{thm:ideal-non-pending-flips<->SP-mutation}, as well as Propositions \ref{prop:pop-existence} and  \ref{prop:orb-pop-stronger} are treated therein.

As mentioned before, Felikson-Shapiro-Tumarkin associate not one, but several, skew-symmetrizable matrices to each triangulation of a surface with orbifold points of order 2, see Remark \ref{rem:tau=t_1(tau)}. In this paper we have considered only one of these matrices. In the forthcoming sequel \cite{Geuenich-Labardini-2} to this paper we will present a construction of SPs for triangulations that will encompass other matrices. In particular, we will give constructions of SPs for the skew-symmetrizable matrices that are mutation-equivalent to matrices of types~$\BB$ and~$\tildeB$.

\section*{Acknowledgements}

We thank Raymundo Bautista, Christof Geiss and Jan Schr\"oer for helpful conversations.


Parts of this work were completed during a visit of the first author to Instituto de Matem\'aticas (Mexico City), UNAM, in May--October 2015, and during a visit of the second author to Claire Amiot at Institut Fourier, Universit\'e Joseph Fourier (Grenoble, France) in June 2014. The visit of the first author was made possible by funding provided by the \emph{Bonn International Graduate School in Mathematics}, and the grants \emph{PAPIIT-IA102215} and \emph{CONACyT-238754} of the second author. The visit of the second author was made possible by funding provided by the \emph{CNRS--CONACyT's Solomon Lefschetz International Laboratory (LAISLA)}. We gratefully acknowledge the financial support. We also thank Instituto de Matem\'aticas, Institut Fourier and Claire Amiot for their hospitality. The second author was supported as well by the grant \emph{PAPIIT-IN108114} of Christof Geiss.

Most of the contents of Section \ref{sec:SPs} are directly motivated by \cite{DWZ1}. A few of the results in that section were obtained by Andrei Zelevinsky and the second author a few years ago.

\section{Background}
\label{sec:background}

\subsection{Surfaces with orbifold points of order 2 and their triangulations}
\label{subsec:surfaces-with-orbifold-points-and-triangulations}

\begin{defi}\label{def:surf-with-orb-points} A \emph{surface with marked points and orbifold points of order 2}, or simply a \emph{surface}, is a triple $\surf$, where $\Sigma$ is a compact connected oriented 2-dimensional real surface with (possibly empty) boundary, $\marked$ is a non-empty finite subset of $\Sigma$ containing at least one point from each connected component of the boundary of $\Sigma$, and $\orb$ is a (possibly empty) finite subset of $\Sigma\setminus(\marked\cup\partial\Sigma)$, subject to the condition that if $\Sigma$ is a sphere, then $|\marked|\geq 7$. We refer to the elements of $\marked$ as \emph{marked points} and to the elements of $\orb$ as \emph{orbifold points}. The marked points that lie in the interior of $\Sigma$ are called \emph{punctures}, and the set of punctures of $\surf$ is denoted $\punct$.
\end{defi}

\begin{remark}\begin{enumerate}\item Our use of the term ``orbifold point of order 2'' comes from hyperbolic geometry. See the second paragraph of the introduction.
\item Throughout the paper we will omit the ``order 2'' part of the term ``orbifold point of order 2'', and refer to a triple $\surf$ as a \emph{surface with marked points and orbifold points} or simply as a \emph{surface}.
\item Note that when $\orb=\varnothing$, Definition \ref{def:surf-with-orb-points} recovers the notion of ``surface with marked points'' studied by Fomin-Shapiro-Thurston, cf.\ \cite{FST}.
\end{enumerate}
\end{remark}

\begin{defi}\cite[Section 4]{FeShTu-orbifolds}\label{def:triangulation} Let $\surf$ be a surface with marked points and orbifold points.
\begin{enumerate}\item An \emph{arc} on $\surf$, is a curve $i$ on $\Sigma$ such that:
\begin{itemize}
\item either both of the endpoints of $i$ belong to $\marked$, or $i$ connects a point of $\marked$ with a point of $\orb$;
\item $i$ does not intersect itself, except that its endpoints may coincide;
\item the points in $i$ that are not endpoints do not belong to $\marked\cup\orb\cup\partial\Sigma$;
\item if $i$ cuts out an unpunctured monogon, then such monogon contains at least two orbifold points;
\item if $i$ cuts out an unpunctured digon, then such digon contains at least one orbifold point.
\end{itemize}
\item If $i$ is an arc that connects a point of $\marked$ with a point of $\orb$, we will say that $i$ is a \emph{pending arc}.
\item Two arcs $i_1$ and $i_2$ are \emph{isotopic relative to $\marked\cup\orb$} if there exists a continuous function $H:[0,1]\times\Sigma\rightarrow\Sigma$ such that
    \begin{itemize}
    \item[(a)] $H(0,x)=x$ for all $x\in\Sigma$;
    \item[(b)] $H(1,i_1)=i_2$;
    \item[(c)] $H(t,m)=m$ for all $t\in I$ and all $m\in\marked\cup\orb$;
    \item[(d)] for every $t\in I$, the function $H_t:\Sigma\to\Sigma$ given by $x\mapsto H(t,x)$ is a homeomorphism.
    \end{itemize}
    Arcs will be considered up to isotopy relative to $\marked\cup\orb$, parametrization, and orientation.
\item Two isotopy classes $C_1$ and $C_2$ of arcs are \emph{compatible} if either
 \begin{itemize}
 \item $C_1=C_2$; or
 \item $C_1\neq C_2$ and there are arcs $i_1\in C_1$ and $i_2\in C_2$ such that $i_1$ and $i_2$ do not share an orbifold point as a common endpoint, and, except possibly for their endpoints, $i_1$ and $i_2$ do not intersect.
 \end{itemize}
 If $C_1$ and $C_2$ form a pair of compatible isotopy classes of arcs and we have elements $j_1\in C_1$ and $j_2\in C_2$, we will also say that $j_1$ and $j_2$ are compatible.
\item An \emph{ideal triangulation} of $\surf$ is any maximal collection $\tau$ of pairwise compatible arcs.
\item A \emph{tagged arc} on $\surf$ is an arc $i$ together with a tag that accompanies each of its two ends in such a way that the following five conditions are met:
\begin{itemize}
\item each tag is either a \emph{plain tag} or a \emph{notch};
\item the arc $i$ does not cut out a once-punctured monogon;
\item every end at a marked point that lies on the boundary must be tagged plain;
\item every end at an orbifold point must be tagged plain;
\item both ends of a loop must be tagged in the same way.
\end{itemize}
Note that there are arcs whose ends may be tagged in different ways. Following \cite{FST} and \cite{FeShTu-orbifolds}, in the figures we will omit the plain tags and represent the notched ones by the symbol $\bowtie$. The notion of isotopy relative to $\marked\cup\orb$ between tagged arcs is defined in the obvious way. Tagged arcs will be considered up to isotopy relative to $\marked\cup\orb$, parametrization, and orientation. We denote by $\taggedinsurf$ the set of (isotopy classes of) tagged arcs in $\surf$.
\item Two tagged arcs $i_1$ and $i_2$ are \emph{compatible} if the following conditions are satisfied:
\begin{itemize}	
\item the untagged versions of $i_1$ and $i_2$ are compatible as arcs;
\item if the untagged versions of $i_1$ and $i_2$ are different, then they are tagged in the same way at each end they share.
\item if the untagged versions of $i_1$ and $i_2$ coincide, then
 \begin{itemize}\item if none of the endpoints of $i_1$ and $i_2$ is an orbifold point, then there must be at least one end of the untagged version at which they are tagged in the same way;
 \item if an endpoint of $i_1$ and $i_2$ is an orbifold point, then $i_1$ and $i_2$ are equal as tagged arcs.
 \end{itemize}
\end{itemize}
\item A \emph{tagged triangulation} of $\surf$ is any maximal collection of pairwise compatible tagged arcs.
\end{enumerate}
\end{defi}

The following theorem of Felikson-Shapiro-Tumarkin \cite{FeShTu-orbifolds} states the basic properties of the \emph{flip}, which is a combinatorial move on tagged triangulations.

\begin{thm} Let $\surf$ be a surface with marked points and orbifold points.
\begin{enumerate}
\item If $\tau$ is a tagged triangulation of $\surf$ and $i\in\tau$, then there exists a unique tagged arc $j$ on $\surf$ such that the set $\sigma=(\tau\setminus\{i\})\cup\{j\}$ is a tagged triangulation of $\surf$. We say that $\sigma$ is obtained from $\tau$ by the \emph{flip of $i\in\tau$} and write $\sigma=f_i(\tau)$.
\item If $\surf$ is not a closed surface such that $|\marked|=1$, then any two tagged triangulations of $\surf$ can be obtained from each other by a finite sequence of flips.
\end{enumerate}
\end{thm}

In other words, every tagged arc in a tagged triangulation can be flipped, and any two tagged triangulations are related by a chain of flips provided we are not in the situation where $\partial\Sigma=\varnothing$ and $|\marked|=1$. An example of a flip of a pending arc can be found in the bottom row of Figure \ref{Fig:orb-flip-ideal-hexagon}.

\begin{defi} Let $\surf$ be a surface with marked points and orbifold points, and let $\tau$ be an ideal triangulation of $\surf$.
\begin{enumerate}\item An \emph{ideal triangle} of $\tau$ is the topological closure of a connected component of the complement in $\Sigma$ of the union of the arcs in $\tau$.
\item An ideal triangle $\triangle$ is \emph{interior} if its intersection with the boundary of $\Sigma$ consists only of (possibly none) marked points. Otherwise it will be called \emph{non-interior}.
\item A \emph{self-folded triangle} is an interior ideal triangle that contains exactly two arcs of $\tau$ (see the left side of Figure \ref{Fig:weird_triangles}).
\item An \emph{orbifolded triangle} is an ideal triangle (not necessarily interior) that contains an orbifold point (see the center and right side of Figure \ref{Fig:weird_triangles}).
\end{enumerate}
\end{defi}
        \begin{figure}[!ht]
                \centering
                \includegraphics[scale=.3]{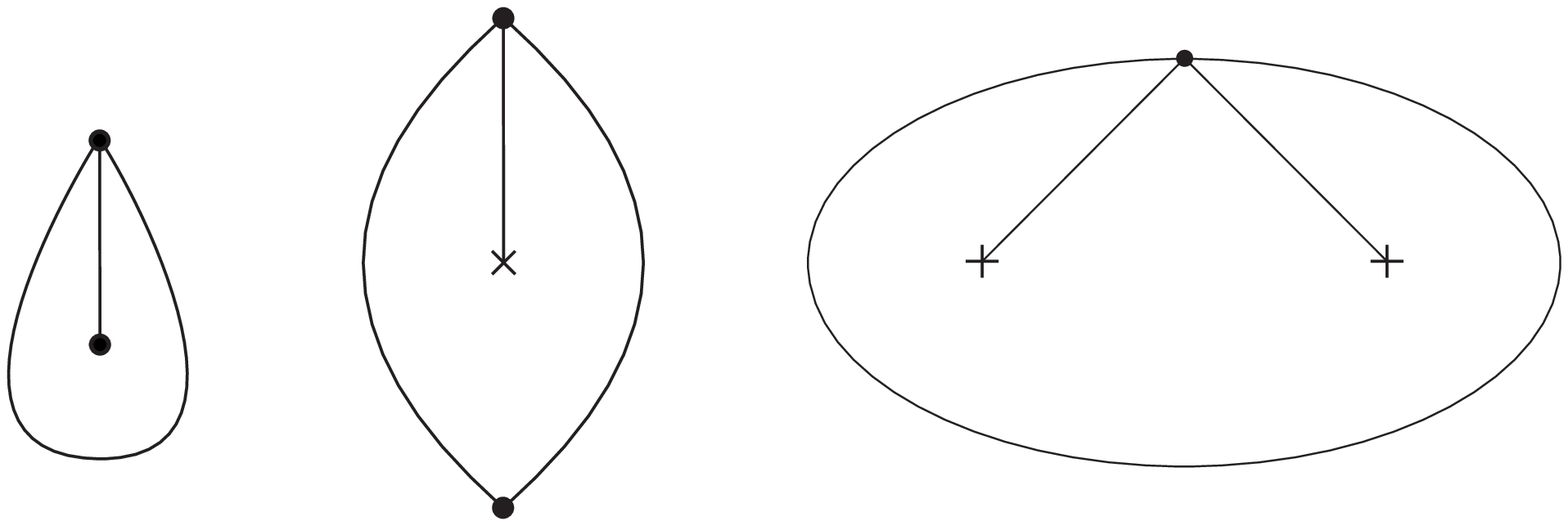}
                \caption{Self-folded triangle (left) and two orbifolded triangles (center and right)}
                \label{Fig:weird_triangles}
        \end{figure}

\begin{remark} Let $\tau$ be an ideal triangulation of $\surf$.
\begin{itemize}
\item An orbifolded triangle can contain one or two orbifold points, but not more than two.
\item An orbifolded triangle that contains exactly one orbifold point is always enclosed by a digon as in Figure \ref{Fig:weird_triangles} (center).
\item An orbifolded triangle that contains exactly two orbifold points is always enclosed by a loop, see Figure \ref{Fig:weird_triangles} (right).
\item If $\triangle$ is a self-folded triangle, then $\triangle$ does not contain orbifold points. Furthermore, if $\triangle$ contains the two (distinct) arcs $i$ and $j$, then one of the arcs $i$ and $j$, say $j$, is a loop that cuts out a once-punctured monogon, and the other one, namely $i$ is an arc that is entirely contained within this monogon and connects the marked point where $j$ is based to the puncture enclosed by $j$; we say that $i$ is the \emph{folded side} of $\triangle$.
\item If $\triangle$ is a self-folded triangle, with $i,j\subseteq\triangle$ as in the previous item, that is, with $j$ a loop enclosing $i$, and if $\triangle'$ is the unique ideal triangle of $\tau$ that contains $j$ and is different from $\triangle$, then $\triangle'$ is not a self-folded triangle and contains at most one orbifold point. See Figure \ref{Fig:adj_sf_orbif}.
\end{itemize}
\end{remark}
        \begin{figure}[!ht]
                \centering
                \includegraphics[scale=.3]{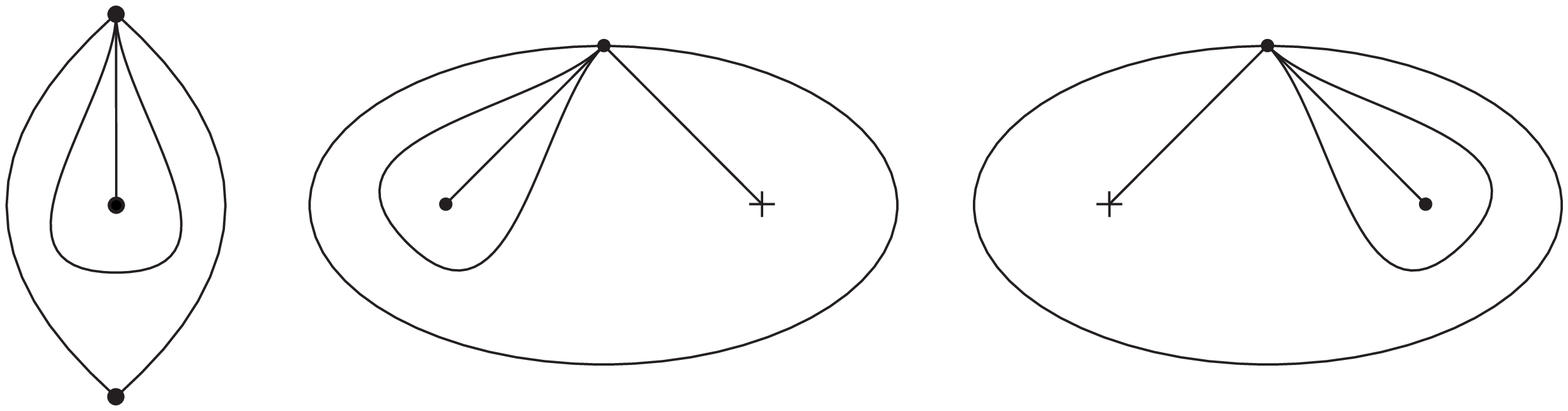}
                \caption{Possibilities for a triangle $\triangle'$ which shares an arc with a self-folded triangle $\triangle$}
                \label{Fig:adj_sf_orbif}
        \end{figure}

We now give a combinatorial description of ideal triangulations in terms of \emph{puzzle-piece decompositions}.
Consider the seven ``puzzle pieces" shown in Figure \ref{Fig:puzzle_pieces_2}.
        \begin{figure}[!ht]
                \centering
                \includegraphics[scale=.3]{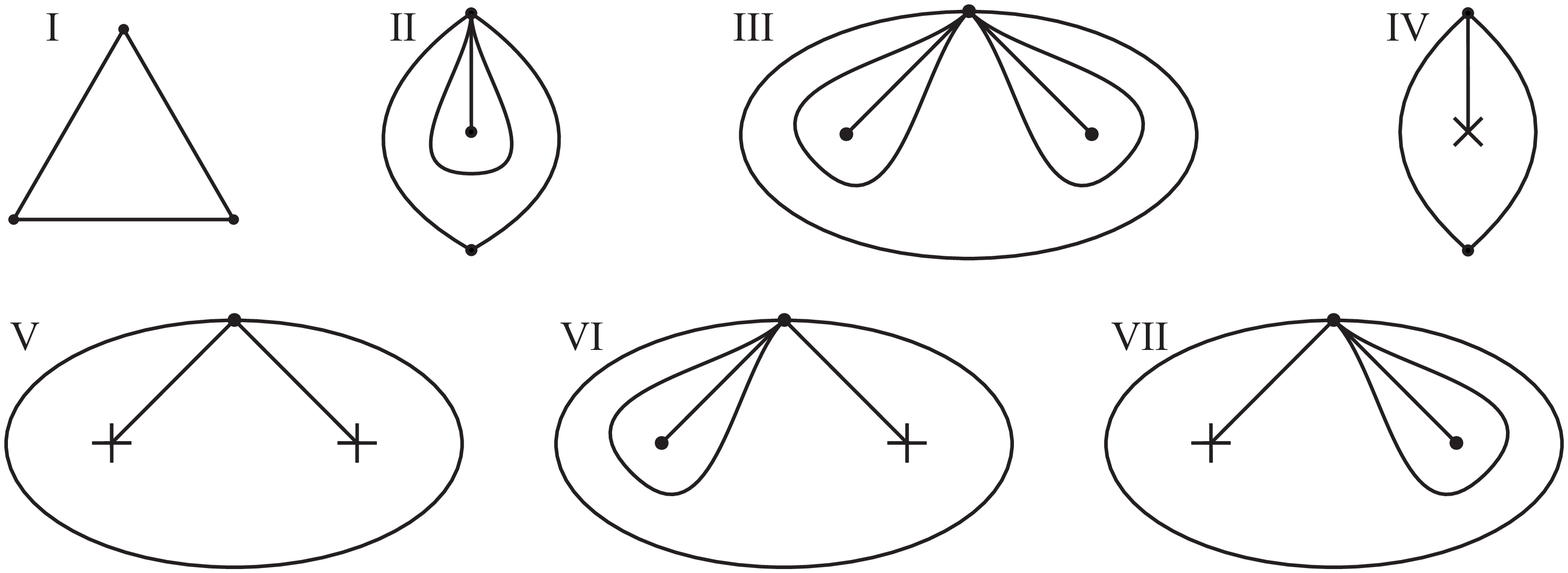}
                \caption{Puzzle pieces}
		\label{Fig:puzzle_pieces_2}
        \end{figure}

Take several copies of these pieces, assign an orientation to each of the outer sides of these copies and fix a partial matching on the set of all outer sides of the copies taken, never matching two sides of the same copy. Then glue the puzzle pieces along the matched sides, making sure the orientations match. Though some partial matchings may not lead to an (ideal triangulation of an) oriented surface, we do have the following.

\begin{thm}\label{thm:existence-of-puzzle-piece-decomps} Any ideal triangulation $\tau$ of an oriented surface $\surf$ can be obtained from a suitable partial matching by means of the procedure just described.
\end{thm}

\begin{defi}\label{def:puzzle-piece-decomps} Any partial matching giving rise to $\tau$ through the procedure just described will be called a \emph{puzzle-piece decomposition} of $\tau$.
\end{defi}

According to Fomin-Shapiro-Thurston \cite{FST}, if $\orb=\varnothing$, then any ideal triangulation of $\surf$ has a puzzle-piece decomposition involving only puzzle pieces of types I, II and III (these types are signaled in Figure \ref{Fig:puzzle_pieces_2}). The existence of puzzle-piece decompositions in the case where $\orb\neq\varnothing$ can be deduced from the existence of puzzle-piece decompositions in the case $\orb=\varnothing$ by treating pending arcs as if they were self-folded triangles.

Although this paper will deal only with the purely combinatorial side of (triangulations of) surfaces with marked points and orbifold points (of order 2),
we believe it is pertinent to illustrate the geometric origin of the concepts introduced in this section, which admittedly have a combinatorial nature.
In the hyperbolic plane $\mathbb{H}^2$, consider the ideal hexagon~$P$ with ideal vertices~$a,b,c,d,e,f$
drawn on the left side of Figure~\ref{Fig:ideal-hexagon}.
\begin{figure}
 ~ \\ [-0.5em]
 \begin{tikzpicture}[scale=0.4,rotate=-60]
  \def \a     {-60}
  \def \r     {5}
  \def \x     {0.57735}

  \draw[very thick] (0,0) circle (\r);

  \draw[thick,draw=green!70!black]
    ([shift=(180:\r*\x)]   \r*\x, \r) arc (180:300:\r*\x);
  \draw[thick,red]
    ([shift=(240:\r*\x)]-  \r*\x, \r) arc (240:360:\r*\x);
  \draw[thick,draw=green!70!black]
    ([shift=(300:\r*\x)]-2*\r*\x,  0) arc (300:420:\r*\x);
  \draw[thick,red]
    ([shift=(  0:\r*\x)]-  \r*\x,-\r) arc (  0:120:\r*\x);
  \draw[thick,magenta]
    ([shift=( 60:\r*\x)]   \r*\x,-\r) arc ( 60:180:\r*\x);
  \draw[thick,blue]
    ([shift=(120:\r*\x)] 2*\r*\x,  0) arc (120:240:\r*\x);

  \node[rotate=330+\a, text=green!70!black]
    at ([shift=(240:\r*\x)]   \r*\x, \r) {\scalebox{0.7}{$>$}};
  \node[rotate= 30+\a, red]
    at ([shift=(300:\r*\x)]-  \r*\x, \r) {\scalebox{0.7}{$>$}};
  \node[rotate= 90+\a, text=green!70!black]
    at ([shift=(  0:\r*\x)]-2*\r*\x,  0) {\scalebox{0.7}{$<$}};
  \node[rotate=150+\a, red]
    at ([shift=( 60:\r*\x)]-  \r*\x,-\r) {\scalebox{0.7}{$<$}};
  \node[rotate=210+\a]
    at ([shift=(120:\r*\x)]   \r*\x,-\r) {$\times$};
  \node[rotate=270+\a]
    at ([shift=(180:\r*\x)] 2*\r*\x,  0) {$\times$};

  \node (c) at ([shift=(180:\r*\x)]   \r*\x, \r) {$$};
  \node (d) at ([shift=(240:\r*\x)]-  \r*\x, \r) {$$};
  \node (e) at ([shift=(300:\r*\x)]-2*\r*\x,  0) {$$};
  \node (f) at ([shift=(  0:\r*\x)]-  \r*\x,-\r) {$$};
  \node (a) at ([shift=( 60:\r*\x)]   \r*\x,-\r) {$$};
  \node (b) at ([shift=(120:\r*\x)] 2*\r*\x,  0) {$$};

  \node[above] at (d) {$f$};
  \node[left]  at (e) {$a$};
  \node[left]  at (f) {$b$};
  \node[below] at (a) {$c$};
  \node[right] at (b) {$d$};
  \node[right] at (c) {$e$};
 \end{tikzpicture}
 \hspace{4em}
 \begin{tikzpicture}[scale=0.65]
  \coordinate (o1) at ( 1, 0);
  \coordinate (o2) at (-1, 0);
  \coordinate  (1) at ( 0, 3);
  \coordinate  (2) at (-3, 0);
  \coordinate  (3) at ( 0,-3);
  \coordinate  (4) at ( 3, 0);

  \draw[thick,blue]    (3) to (o1);
  \draw[thick,magenta] (3) to (o2);

  \draw[thick,draw=green!70!black, text=green!70!black]
    (1) to node[rotate= 45] {\scalebox{0.7}{$<$}} (2);
  \draw[thick,red]
    (2) to node[rotate=-45] {\scalebox{0.7}{$>$}} (3);
  \draw[thick,draw=green!70!black, text=green!70!black]
    (3) to node[rotate= 45] {\scalebox{0.7}{$<$}} (4);
  \draw[thick,red]
    (4) to node[rotate=-45] {\scalebox{0.7}{$>$}} (1);

  \node[rotate=-10] at (o1) {$\times$};
  \node[rotate= 10] at (o2) {$\times$};
  \foreach \i in {1,...,4}
  {
    \node at (\i) {$\bullet$};
  }

  \node at (0,-4.3) {};
 \end{tikzpicture}
  \caption{Left: Ideal hexagon~$P$ in the Poincar\'e disc. Right: The once-punctured torus $\bar{P}/\Gamma$.}
 \label{Fig:ideal-hexagon}
\end{figure}
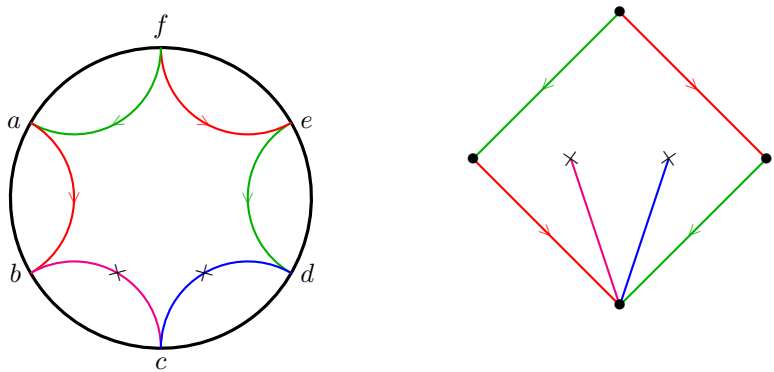

Let $\Gamma \subseteq \operatorname{PSL}_2(\mathbb{R})$ be a subgroup generated by hyperbolic elements~$\gamma_1$, $\gamma_2$ and elliptic elements~$\varepsilon_1$, $\varepsilon_2$ acting on the oriented sides of~$P$ as follows:
\[
  \overrightarrow{ed} \:\:\xmapsto{\:\:\gamma_1\:\:}\:\:      \overrightarrow{fa}
  \,,\hspace{20pt}
  \overrightarrow{ab} \:\:\xmapsto{\:\:\gamma_2\:\:}\:\:      \overrightarrow{fe}
  \,,\hspace{20pt}
  \overrightarrow{cb} \:\:\xmapsto{\:\:\varepsilon_1\:\:}\:\: \overrightarrow{bc}
  \,,\hspace{20pt}
  \overrightarrow{cd} \:\:\xmapsto{\:\:\varepsilon_2\:\:}\:\: \overrightarrow{dc},
\]
and with the property that $\rho=\varepsilon_2 \varepsilon_1 \gamma_2\gamma_1\gamma_2^{-1}\gamma_1^{-1}$ is a parabolic element. It is not hard to actually construct such elements $\gamma_1,\gamma_2,\varepsilon_1,\varepsilon_2\in\operatorname{PSL}_2(\mathbb{R})$ explicitly.

Then a famous theorem of Poincar\'e implies that $\Gamma$ is a Fuchsian group with fundamental domain~$P$.
Since $P$ is a locally finite fundamental domain for $\Gamma$, the quotient~$\mathbb{H}^2/\Gamma$ is homeomorphic to $\overline{P}/\Gamma$,
which is a torus with one point removed (see the picture on the right of Figure~\ref{Fig:ideal-hexagon}).
Note that the unique fixed point of $\varepsilon_1$ and the unique fixed point of $\varepsilon_2$ give rise to two \emph{orbifold points} in $\mathbb{H}^2/\Gamma$. Both of these orbifold points have order 2 since each of $\varepsilon_1$ and $\varepsilon_2$ is an elliptic M\"{o}bius transformation of order 2.

Let us consider the ideal triangulation~$\widetilde{\tau}$ of the hyperbolic plane~$\mathbb{H}^2$ that is given by the $\Gamma$-translates of the sides~$ab$, $bc$, $cd$, $de$, $ef$, $fa$ and the diagonals~$bf$, $cf$, $ce$ of $P$.
This ideal triangulation~$\widetilde{\tau}$ yields a triangulation of~$\mathbb{H}^2/\Gamma$ by geodesics which in turn corresponds combinatorially to an ideal triangulation $\tau$ of the surface $\surf$ which is a torus with one puncture and two orbifold points (of order 2).
The flip of the arc~$bc$ of $\widetilde{\tau}$, whose associated element~$\varepsilon_1 \in \Gamma$ is an elliptic involution mapping $bc$ onto itself, induces the flip of the corresponding pending arc of $\tau$ (see Figure~\ref{Fig:orb-flip-ideal-hexagon}).

\begin{figure}[h!]
 ~ \\ [-0em]
 \includegraphics[width=0.7\textwidth]{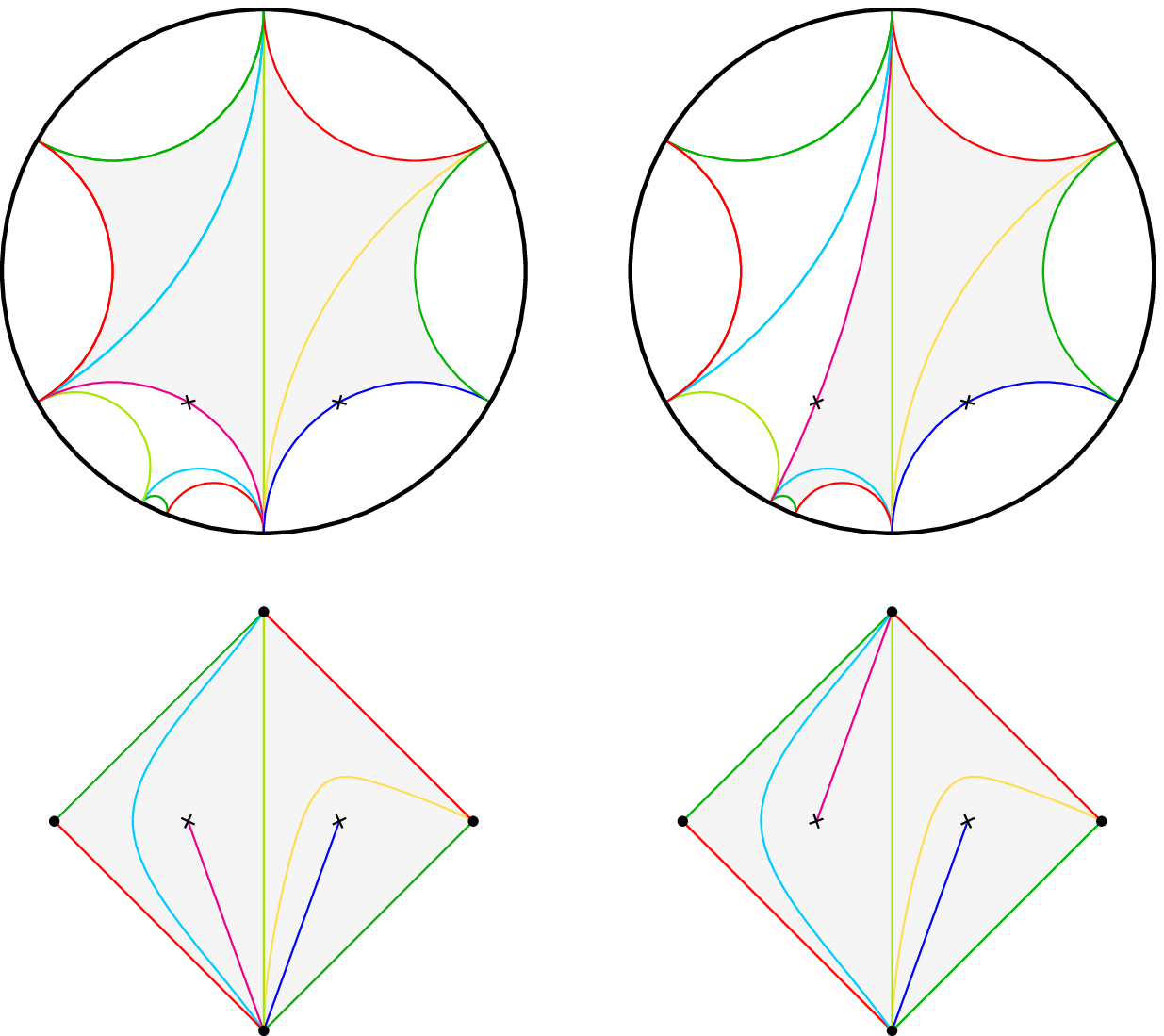}
 \caption{Left: $\widetilde{\tau}$ and $\tau$. Right: Ideal triangulations obtained by flipping the arc~$bc$ in $\widetilde{\tau}$.}
 \label{Fig:orb-flip-ideal-hexagon}
\end{figure}
\subsection{Remarks on tensor products and bimodules}
\label{subsec:tensor-products-and-bimodules}

Let $R$ be a commutative unitary ring, and let $A$ be an $R$-$R$-bimodule.
For every non-negative integer $\ell$ we can form the $\ell$-fold tensor power of $A$ over $R$, that is,
$$
A^\ell=\underset{\text{$\ell$ factors}}{\underbrace{A\otimes_R A\otimes_R\ldots\otimes_RA}},
$$
where, as usual, $A^0=R$ and $A^1=A$. This $\ell$-fold tensor power has a natural $R$-$R$-bimodule structure. The direct sum and the direct product of the $R$-$R$-bimodules $A^\ell$, $\ell\geq 0$, then become rings in the obvious way.

\begin{defi}\label{def:tensor-algebra-&-complete-tensor-algebra} Let $A$ be an $R$-$R$-bimodule. The \emph{tensor algebra of $A$ over $R$} is the direct sum
\begin{equation}\label{eq:def-of-R<A>}
\pathalg{A}=\bigoplus_{\ell\geq 0}A^\ell,
\end{equation}
while the \emph{complete tensor algebra of $A$ over $R$} is the direct product
\begin{equation}\label{eq:def-of-R<<A>>}
\RA{A}=\prod_{\ell\geq 0}A^\ell.
\end{equation}
In both $\pathalg{A}$ and $\RA{A}$, the multiplication is given by the rule
$$
\left(\sum_{\ell_1=0}^\infty x^{(\ell_1)}\right) \left(\sum_{\ell_2}^\infty y^{(\ell_2)}\right)=
\sum_{\ell=0}^\infty \left(\sum_{\ell_1+\ell_2=\ell} x^{(\ell_1)}y^{(\ell_2)}\right),
$$
where $x^{(\ell_1)}\in A^{\ell_1}$, $y^{(\ell_2)}\in A^{\ell_2}$, and $x^{(\ell_1)}y^{(\ell_2)}=x^{(\ell_1)}\otimes y^{(\ell_2)}\in A^{\ell_1}\otimes_R A^{\ell_2}=A^{\ell_1+\ell_2}$.
\end{defi}

Given that $R$ plays an essential role in the definition of $\pathalg{A}$ and $\RA{A}$, that every homogeneous component in the decompositions \eqref{eq:def-of-R<A>} and \eqref{eq:def-of-R<<A>>} is an $R$-$R$-bimodule, and that $R$ appears as the degree-0 component,  we will say that \emph{$\pathalg{A}$ (resp.\ $\RA{A}$) is an $R$-algebra}, despite the fact that $R$ may not be a central subring of $\pathalg{A}$ (resp.\ $\RA{A}$).
Accordingly, we will say that a ring homomorphism (resp. isomorphism) from $\RA{A}$ or $\pathalg{A}$ to $\RA{A'}$ or $\pathalg{A'}$ is an \emph{$R$-algebra homomorphism} (resp. \emph{$R$-algebra isomorphism}) if its restriction to $R$ is the identity. This terminology, although not standard\footnote{Usually, the definition of $R$-algebra requires $R$ to be commutative and to sit as a central subring.}, has the advantage of rendering the following statement quite neat:

\begin{prop}\label{prop:univ-property-of-R<<A>>} Let $A$ and $A'$ be $R$-$R$-bimodules. Every $R$-$R$-bimodule homomorphism $A\rightarrow\RA{A'}$ whose image is contained in $\prod_{\ell\geq 1}A'^{\ell}$ can be uniquely extended to an $R$-algebra homomorphism $\RA{A}\rightarrow\RA{A'}$.
\end{prop}

Often times one can induce ring homomorphisms $\pathalg{A}\rightarrow\RA{A'}$ through group homomorphisms $A\rightarrow\RA{A'}$ that are not $R$-$R$-bimodule homomorphisms under the standard $R$-$R$-bimodule structure of $\RA{A'}$. For instance:

\begin{prop}\label{prop:ring-endomorphisms-of-RA} Let $A$ be an $R$-$R$-bimodule, and suppose that $\Psi^{(0)}$ is a ring automorphism of $R$ and $\Psi^{(1)}$ is a group endomorphism of $A$ such that $\Psi^{(1)}(ra)=\Psi^{(0)}(r)\Psi^{(1)}(a)$ and $\Psi^{(1)}(ar)=\Psi^{1}(a)\Psi^{(0)}(r)$ for every $r\in R$ and every $a\in A$. Then, for every $\ell\geq 2$, the rule
$$
a_1\otimes \ldots \otimes a_\ell\mapsto \Psi^{(1)}(a_1)\otimes\ldots\otimes\Psi^{(1)}(a_\ell)
$$
produces a well-defined group homomorphism $\Psi^{(\ell)}:A^\ell\rightarrow A^\ell$.
Moreover, the assignment
$$
\sum_{\ell\geq 0}u^{(\ell)}\mapsto\sum_{\ell\geq 0}\Psi^{(\ell)}(u^{(\ell)}), \ \ \ \ \ u^{(\ell)}\in A^\ell,
$$
constitutes a well-defined ring endomorphism $\Psi:\RA{A}\rightarrow\RA{A}$.
This ring endomorphism satisfies $\Psi|_R=\Psi^{(0)}$ and $\Psi|_A=\Psi^{(1)}$, and any ring endomorphism of $\RA{A}$ whose restrictions to $R$ and $A$ respectively coincide with $\Psi^{(0)}$ and $\Psi^{(1)}$ has to be equal to $\Psi$.
\end{prop}

The rest of this subsection deals with direct sum decompositions of tensor products of certain bimodules. The results, which are almost certainly well-known, were found in joint work of the second author and A. Zelevinsky a few years ago.

In latter sections, some elements of certain Galois groups will appear as labels of arrows of quivers. One of the aims of this subsection is to provide a justification of their appearance.

Let $E/F$ be a finite-degree cyclic Galois extension, and suppose that $F$ contains a primitive $d^{\operatorname{th}}$ root of unity, where $d=[E:F]$. This implies the existence of an element $v\in E$ such that for every subextension $K/L$ of $E/F$ the set $\B_{K/L}=\{1,v_{K/L},v_{K/L}^2,\ldots,v_{K/L}^{[K:L]-1}\}$ is an eigenbasis of $K/L$, where $v_{K/L}=v^{\frac{d}{[K:L]}}$.

Let $J/F$ and $K/F$ be subextensions of $E/F$. 
For each element $\rho_1$ of the Galois group $\Gal(J\cap K/F)$, consider the left action of $J$ on itself given by the multiplication it has as a field, and the right action of $J\cap K$ on $J$ given by
$$
m\star x=m\rho_1(x) \ \ \ \text{for} \ m\in J \ \text{and} \ x\in J\cap K,
$$
where the product $m\rho_1(x)$ is taken according to the multiplication that $J$ has as a subfield of $E$.
We denote by $J^{\rho_1}$ the $J$-$J\cap K$-bimodule obtained in this way. Note that $J^{\rho_1}$ is not isomorphic to $J=J^{\myid}$ as a $J$-$J\cap K$-bimodule unless $\rho_1=\myid_{J\cap K}$.

Consider the $J$-$K$-bimodule $J^{\rho_1}\otimes_{J\cap K}K$. Given $x\in J^{\rho_1}$, $y\in K$ and $z\in J\cap K$, we have
$$
x\otimes (zy)=(x\star z)\otimes y=(x\rho_1(z))\otimes y \ \ \ \text{and}\ \ \  (xz)\otimes y=(x\star\rho_1^{-1}(z))\otimes y=x\otimes (\rho_1^{-1}(z)y),
$$
where the products $x\rho_1(z)$, $xz$, $zy$ and $\rho_1^{-1}(z)y$ are taken according to the multiplications that $J$ and $K$ have as subfields of $E$. This means that we can bypass the use of the symbol $\star$ when passing the elements of $J\cap K$ through the tensor symbol $\otimes=\otimes_{J\cap K}$, and write this passage entirely in terms of the field multiplications of $J$ and $K$; the passage from right to left then requires applying $\rho$ while the passage from left to right requires applying $\rho_1^{-1}$.
Henceforth we will always bypass the use of the symbol $\star$ without any further apology.

It is not hard to see that $J^{\rho_1}\otimes_{J\cap K}K$ is a simple object in the category of $J$-$L$-bimodules on which $F$ acts centrally, and that $J^{\rho_1}\otimes_{J\cap K}K$ and $J^{\rho'}\otimes_{J\cap K}K$ are not isomorphic as $J$-$L$-bimodules if $\rho_1$ and $\rho'$ are different elements of $\Gal(J\cap K/F)$. Let $L/F$ be a further field subextension of $E/F$ and $\rho_2\in\Gal(K\cap L/F)$. The following result shows how to decompose $(J^{\rho_1}\otimes_{J\cap K}K)\otimes_{K}(K^{\rho_2}\otimes_{K\cap L}L)$ as a direct sum of simples of the category of $J$-$L$-bimodules on which $J\cap L$ acts centrally.

\begin{prop}\label{prop:tensor-prod-decomp} Let $\rho_1\in\Gal(J\cap K/F)$ and $\rho_2\in\Gal(K\cap L/F)$.
There is a $J$-$L$-bimodule isomorphism
$$
(J^{\rho_1}\otimes_{J\cap K}K)\otimes_{K}(K^{\rho_2}\otimes_{K\cap L}L)\cong\left(\bigoplus_{\rho}J^{\rho}\otimes_{J\cap L}L\right)^{[K:(J\cap K)(K\cap L)]},
$$
where $\rho$ runs through the elements of the Galois group $\Gal(J\cap L/F)$ whose restriction to $J\cap K\cap L$ equals $(\rho_1|_{J\cap K\cap L})(\rho_2|_{J\cap K\cap L})$. This provides a decomposition of $(J^{\rho_1}\otimes_{J\cap K}K)\otimes_{K}(K^{\rho_2}\otimes_{K\cap L}L)$ as a direct sum of simples in the category of $J$-$L$-bimodules on which $F$ acts centrally.
\end{prop}

Before proving Proposition \ref{prop:tensor-prod-decomp} we would like to make an observation. Note that if $\rho_1=\myid_{J\cap K}$, then $J^{\rho_1}\otimes_{J\cap K}K$ is the usual tensor product $J\otimes_{J\cap K}K$. The same observation applies to $K^{\rho_2}\otimes_{K\cap L}L$ provided $\rho_2=\myid_{K\cap L}$. So, Proposition \ref{prop:tensor-prod-decomp} tells us that, even if we start with the usual tensor products $J\otimes_{J\cap K}K$ and $K\otimes_{K\cap L}L$, with their usual bimodule structure inherited from the standard bimodule structures of $J$, $K$ and $L$, then as soon as we are interested in decomposing $(J\otimes_{J\cap K}K)\otimes_{K}(K\otimes_{K\cap L}L)$ as a direct sum of simple $J$-$L$-bimodules, 
we are forced to introduce the less frequently used bimodules $J^\rho\otimes_{J\cap L}L$, that are not isomorphic to the usual $J\otimes_{J\cap L}L$ as $J$-$L$-bimodules whenever $\rho\neq\myid_{J\cap L}$. 
This is what will force us to label arrows of quivers with elements of Galois groups later on.


Proposition \ref{prop:tensor-prod-decomp} is an obvious consequence of the next two lemmas.

\begin{lemma}\label{lemma:JotimesKotimesL} Let $\rho_1\in\Gal(J\cap K/F)$ and $\rho_2\in\Gal(K\cap L/F)$.
The $J$-$L$-bimodule $(J^{\rho_1}\otimes_{J\cap K}K)\otimes_{K}(K^{\rho_2}\otimes_{K\cap L}L)$ is isomorphic to the $J$-$L$-bimodule $(J^{\rho_1|\rho_2|}\otimes_{J\cap K\cap L}L)^{[K:(J\cap K)(K\cap L)]}$, where $\rho_1|$ and $\rho_2|$ are the restrictions of $\rho_1$ and $\rho_2$ to $J\cap K\cap L$.
\end{lemma}

\begin{proof} Let $\B=\B_{K/(J\cap K)(K\cap L)}$ be an eigenbasis of $K/(J\cap K)(K\cap L)$. A routinary check shows that, for each $\omega\in\B$, the rule
$
x\otimes y\mapsto x\otimes \omega\otimes 1\otimes y
$
produces a well-defined $J$-$L$-bimodule homomorphism $\psi_\omega:J^{\rho_1|\rho_2|}\otimes_{J\cap K\cap L}L\rightarrow (J^{\rho_1}\otimes_{J\cap K}K)\otimes_{K}(K^{\rho_2}\otimes_{K\cap L}L)$. Set $\psi=[\psi_\omega]_{\omega\in\B}:(J^{\rho_1|\rho_2|}\otimes_{J\cap K\cap L}L)^{[K:(J\cap K)(K\cap L)]}\rightarrow (J^{\rho_1}\otimes_{J\cap K}K)\otimes_{K}(K^{\rho_2}\otimes_{K\cap L}L)$, which is a $J$-$L$-bimodule homomorphism. Since every element of $K$ can be written as a $(J\cap K)(K\cap L)$-linear combination of elements of $\B$, and since any element of $(J\cap K)(K\cap L)$ can be written as a finite sum $\sum_{t}j_tl_t$, with each $j_t\in J\cap K$ and each $l_t\in K\cap L$, we deduce that $\psi$ is surjective.

Since $F$ acts centrally on $(J^{\rho_1|\rho_2|}\otimes_{J\cap K\cap L}L)^{[K:(J\cap K)(K\cap L)]}$ and $(J^{\rho_1}\otimes_{J\cap K}K)\otimes_{K}(K^{\rho_2}\otimes_{K\cap L}L)$, each of these bimodules has a well-defined dimension over $F$. It is fairly easy to show that $\dim_F((J^{\rho_1|\rho_2|}\otimes_{J\cap K\cap L}L)^{[K:(J\cap K)(K\cap L)]})=\dim_F((J^{\rho_1}\otimes_{J\cap K}K)\otimes_{K}(K^{\rho_2}\otimes_{K\cap L}L))$. The surjectivity of $\psi$ then implies that $\psi$ is injective as well.
\end{proof}

\begin{lemma}\label{lemma:JotimesL}
The $J$-$L$-bimodule $J^{\rho_1|\rho_2|}\otimes_{J\cap K\cap L}L$ is isomorphic to the $J$-$L$-bimodule
$$
\bigoplus_{\rho}J^{\rho}\otimes_{J\cap L}L,
$$
where the sum runs over all the elements $\rho$ of $\Gal(J\cap L/F)$ whose restriction to $J\cap K\cap L$ equals $\rho_1|\rho_2|$.
\end{lemma}

\begin{proof} Let us abbreviate $\B=\B_{J\cap L/J\cap K\cap L}$, which is an eigenbasis of the subextension $J\cap L/J\cap K\cap L$. For each $\rho\in\Gal(J\cap L/F)$ such that $\rho|_{J\cap K\cap L}=\rho_1|\rho_2|$, let
$$
\varphi_\rho:J^{\rho}\otimes_{J\cap L}L\rightarrow J^{\rho_1|\rho_2|}\otimes_{J\cap K\cap L}L
$$
be the $J$-$L$-bimodule homomorphism given by
$$
\varphi_\rho(x\otimes y)=\frac{1}{[J\cap L:J\cap K\cap L]}\sum_{\omega\in\B}\rho(\omega^{-1})x\otimes y\omega.
$$
The fact that $\varphi_\rho$ is $J\cap L$-balanced and hence well-defined follows from the fact that the product of any two elements of $\B$ is an $F$-multiple of some other element of $\B$. Assembling all the maps $\varphi_\rho$, with $\rho$ as in the statement of the lemma, we obtain a $J$-$L$-bimodule homomorphism
$$
\varphi:\bigoplus_\rho J^{\rho}\otimes_{J\cap L}L\rightarrow J^{\rho_1|\rho_2|}\otimes_{J\cap K\cap L}L.
$$
To verify that $\varphi$ is surjective, for each $\rho$ define $\pi_\rho:J^{\rho_1|\rho_2|}\otimes_{J\cap K\cap L}L\rightarrow J^{\rho}\otimes_{J\cap L}L$ by $\pi_\rho(m\otimes n)= m\otimes n$.
Then for
$m\otimes n\in J^{\rho_1|\rho_2|}\otimes_{J\cap K\cap L}L$ we have $m\otimes n = \sum_{\rho}\varphi_\rho(\pi_\rho(m\otimes n))$, which shows that $\varphi$ is surjective.

Since $\rho\in\Gal(J\cap L/F)$, $F$ acts centrally on $J^{\rho}\otimes_{J\cap L}L$ and on $J^{\rho_1|\rho_2|}\otimes_{J\cap K\cap L}L$, hence each of these bimodules has a well-defined dimension over $F$. It is fairly easy to show that $\dim_F(\bigoplus_\rho J^{\rho}\otimes_{J\cap L}L)=\dim_F(J^{\rho_1|\rho_2|}\otimes_{J\cap K\cap L}L)$. The surjectivity of $\varphi$ then implies that $\varphi$ is injective as well.
\end{proof}

%

To close this section, we record a result that will be useful later on.

\begin{prop}\label{prop:properties-of-J-K-bimodules} Under the assumptions stated at the beginning of the current subsection, let $\mathcal{C}_{J,K}$ be the category of $J$-$K$-bimodules on which $F$ acts centrally and whose dimension over $F$ is finite.
\begin{enumerate}
\item the bimodules $J^{\rho}\otimes_{J\cap K}K$, for $\rho\in\Gal(J\cap K/F)$ form a complete set of pairwise non-isomorphic simple objects in $\mathcal{C}_{J,K}$;
\item for every $M\in\mathcal{C}_{J,K}$ and every $\rho\in\Gal(J\cap K/F)$, the function $\pi_\rho=\pi^M_\rho:M\rightarrow M$ defined by
    $$
    m\mapsto\frac{1}{[J\cap K:F]}\sum_{\omega\in\B_{J\cap K/F}}\rho(\omega^{-1})m\omega
    $$
    is an idempotent $J$-$K$-bimodule homomorphism;
\item for every $M\in\mathcal{C}_{J,K}$ and every pair $\rho_1,\rho_2\in\Gal(J\cap K/F)$, if $\rho_1\neq\rho_2$, then $\pi_{\rho_1}\pi_{\rho_2}=0$;
\item for every $M\in\mathcal{C}_{J,K}$ we have an internal direct sum decomposition
$$
M=\bigoplus_{\rho\in\Gal(J\cap K/F)}\image\pi_\rho;
$$
\item for every $M\in\mathcal{C}_{J,K}$, every $m\in\image\pi_\rho$ and every $x\in J\cap K$ we have $mx=\rho(x)m$;
\item for every $M\in\mathcal{C}_{J,K}$ and every $m\in\image\pi_\rho$ there exists a unique $J$-$K$-bimodule homomorphism $\varphi:J^{\rho}\otimes_{J\cap K}K\rightarrow M$ such that $\varphi(1\otimes 1)=m$.
\end{enumerate}
\end{prop}

\subsection{Skew-symmetrizable matrices and weighted quivers}\label{subsec:matrices-vs-quivers}

In this subsection we recall the notion of weighted quiver and the correspondence between skew-symmetrizable matrices and 2-acyclic weighted quivers.

Let $n$ be a positive integer. Recall that a square matrix $B\in\mathbb{Z}^{n\times n}$ is said to be \emph{skew-symmetrizable} if there exists a diagonal matrix $D=\diag(d_1,\ldots,d_n)$, with $d_1,\ldots,d_n\in\mathbb{Z}_{>0}$, such that $DB$ is skew-symmetric. We say that $D$ is a \emph{skew-symmetrizer of $B$}.

\begin{defi}\label{def:weighted-quiver} A \emph{weighted quiver} is a pair $(Q,\dtuple)$ constituted by a loop-free quiver $Q=(Q_0,Q_1,t,h)$ and a tuple $\dtuple=(d_i)_{i\in Q_0}$ that attaches a positive integer $d_i$ to each vertex $i$ of $Q$. We refer to $\dtuple$ as the \emph{weight tuple} of $(Q,\dtuple)$ and to each $d_i$ as a \emph{weight attached to $i$}.
\end{defi}

A weighted quiver will be said to be \emph{2-acyclic} if its underlying quiver does not have oriented cycles of length 2.

Let $B\in\mathbb{Z}^{n\times n}$ be a skew-symmetrizable matrix, and fix a skew-symmetrizer $D=\diag(d_1,\ldots,d_n)$ of $B$. As in \cite{LZ}, we associate to $B$ a 2-acyclic weighted quiver $(Q_B,\dtuple)$ as follows. The vertex set of $Q_B$ is $\{1,\ldots,n\}$, and $Q_B$ has exactly
$$
\frac{b_{ij}\gcd(d_i,d_j)}{d_j}=\frac{-b_{ji}\gcd(d_j,d_i)}{d_i}
$$
arrows from $j$ to $i$ whenever $b_{ij}\geq 0$. As for the tuple $\dtuple$, its $i^{\operatorname{th}}$ member is defined to be precisely $d_i$, the $i^{\operatorname{th}}$ diagonal entry of $D$. Since $B$ is skew-symmetrizable, it is clear that $(Q_B,\dtuple)$ is a 2-acyclic weighted quiver.

\begin{remark}
The rational number $\frac{b_{ij}\gcd(d_i,d_j)}{d_j}$ is indeed an integer: write $d_i=d_i'\gcd(d_i,d_j)$ and $d_j=d_j'\gcd(d_i,d_j)$. The equality $d_ib_{ij}=-d_jb_{ji}$ implies the equality $d_i'b_{ij}=-d_j'b_{ji}$, which in turn implies that the rational number
$\frac{b_{ij}}{d_j'}$ is an integer since $\gcd(d_i',d_j')=1$. So, $\frac{b_{ij}\gcd(d_i,d_j)}{d_j}=\frac{b_{ij}}{d_j'}$ is an integer.
\end{remark}

The following lemma is obvious.

\begin{lemma}\cite{LZ}\label{lemma:B<->(Q,d)} Fix a tuple $\dtuple=(d_1,\ldots,d_n)$ of positive integers. The correspondence $B\mapsto(Q_B,\dtuple)$ is a bijection between the set of all skew-symmetrizable matrices that have $D=\diag(d_1,\ldots,d_n)$ as a skew-symmetrizer, and the set of all 2-acyclic weighted quivers whose vertex set is $\{1,\ldots,n\}$ and whose weight tuple is precisely $\dtuple$.
\end{lemma}

Because of this lemma, it makes sense to ask how Fomin-Zelevinsky's mutation rule for skew-symmetrizable matrices (cf. \cite{FZ1} or \cite[Equation (1.3)]{FZ2}) translates to the language of weighted quivers.

\begin{defi}\label{def:weighted-quiver-mutation}\cite{LZ} Let $(Q,\dtuple)$ be a 2-acyclic weighted quiver, and let $k\in Q_0$ be a vertex of $Q$. The \emph{mutation of $(Q,\dtuple)$} in direction $k$ is the weighted quiver $\mu_k(Q,\dtuple)=(\mu_k(Q),\dtuple)$, where $\mu_k(Q)$ is the quiver obtained from $Q$ by applying the following 3 steps:
\begin{itemize}
\item[(Step 1)] For each pair $a,b\in Q_1$ such that $h(a)=k=t(b)$, add $\frac{d_k d_{h(b),t(a)}}{d_{h(b),k}d_{k,t(a)}}$ ``composite'' arrows from $t(a)$ to $h(b)$, where $d_{i,j}=\gcd(d_i,d_j)$ for $i,j\in Q_0$;
\item[(Step 2)] replace each arrow $c$ incident to $k$ with an arrow $c^*$ going in the opposite direction;
\item[(Step 3)] choose a maximal collection of pairwise disjoint 2-cycles and delete it.
\end{itemize}
\end{defi}

\begin{remark}\label{rem:composite-arrows-indeed-integer} For each pair $a,b\in Q_1$ such that $h(a)=k=t(b)$, the rational number $\frac{d_k d_{h(b),t(a)}}{d_{h(b),k}d_{k,t(a)}}$ is indeed a positive integer: $\frac{d_k d_{h(b),t(a)}}{d_{h(b),k}d_{k,t(a)}}
    =\frac{d_k d_{h(b),t(a)}}{\operatorname{lcm}(d_{h(b),k},d_{k,t(a)})\gcd(d_{h(b),k},d_{k,t(a)})}$ and $\operatorname{lcm}(d_{h(b),k},d_{k,t(a)})$ divides $d_k$ since each of $d_{h(b),k}$ and $d_{k,t(a)}$ does, while $\gcd(d_{h(b),k},d_{k,t(a)})$ obviously divides $d_{h(b),t(a)}$.
\end{remark}

The next lemma says that Definition \ref{def:weighted-quiver-mutation} gives the correct translation of Fomin-Zelevinsky's matrix mutation rule to the language of weighted quivers.

\begin{lemma}\cite{LZ} Fix a tuple $\dtuple=(d_1,\ldots,d_n)$ of positive integers, and let $B\in\mathbb{Z}^{n\times n}$ be a skew-symmetrizable matrix having $D=\diag(d_1,\ldots,d_n)$ as a skew-symmetrizer. For every $k\in\{1,\ldots,n\}$, the weighted quivers $\mu_k(Q_B,\dtuple)$ and $(Q_{\mu_k(B)},\dtuple)$ are isomorphic as weighted quivers.
\end{lemma}

\subsection{Species realizations of skew-symmetrizable matrices}\label{subsec:species-realizations}

The following definition is an adaptation of Dlab-Ringel's notion of \emph{modulation of a valued graph} (cf. \cite{DR}) to the setting of skew-symmetrizable matrices.

\begin{defi}\label{def:species-realization} Let $B\in\mathbb{Z}^{n\times n}$ be a skew-symmetrizable matrix, and set $Q_0=\{1,\ldots,n\}$. A \emph{species realization} of $B$ is a pair $(\mathbf{F},\mathbf{A})$ such that:
\begin{enumerate}
\item $\mathbf{F}=(F_i)_{i\in Q_0}$ is a tuple of division rings;
\item $\mathbf{A}$ is a tuple consisting of an $F_i$-$F_j$-bimodule $A_{ij}$ for each pair $(i,j)\in Q_0\times Q_0$ such that $b_{ij}> 0$;
\item for every pair $(i,j)\in Q_0\times Q_0$ such that $b_{ij}> 0$, there are $F_j$-$F_i$-bimodule isomorphisms
$$
\operatorname{Hom}_{F_i}(A_{ij},F_i)\cong\operatorname{Hom}_{F_j}(A_{ij},F_j);
$$
\item for every pair $i,j\in Q_0$ such that $b_{ij}> 0$ we have $\operatorname{dim}_{F_i}(A_{ij})=b_{ij}$ and $\operatorname{dim}_{F_j}(A_{ij})=-b_{ji}$.
\end{enumerate}
\end{defi}

The next question is motivated by Derksen-Weyman-Zelevinsky's mutation theory of quivers with potential.

\begin{question}\label{question:non-deg-SPs?} Can a mutation theory of species with potential be defined so that every skew-symmetrizable matrix $B$ have a species realization which admit a non-degenerate potential?
\end{question}

To the best of our knowledge, this question has not been fully answered. Partial answers to this question have been given\footnote{In chronological order according to appearance in arXiv.} by Derksen-Weyman-Zelevinsky \cite{DWZ1}, Demonet \cite{Demonet}, Nguefack \cite{Nguefack}, Labardini-Fragoso--Zelevinsky \cite{LZ}, Bautista--L\'opez-Aguayo \cite{Bautista-Lopez}, who give partially positive answers by restricting their attention to certain particular classes of skew-symmetrizable matrices and certain specific types of species realizations:

\begin{enumerate}
\item In their seminal paper  \cite{DWZ1}, Derksen-Weyman-Zelevinsky give a full positive answer for the class of skew-symmetric matrices;
\item in \cite{Demonet} Demonet develops a mutation theory of \emph{group species with potentials}, and gives a full positive answer for two classes of matrices, namely, those that are mutation-equivalent to \emph{acyclic} skew-symmetrizable matrices, and those skew-symmetrizable matrices that are of the form $B=CB'$ for some skew-symmetric matrix $B'$ and some diagonal matrix $C=\diag(c_1,\ldots,c_n)$ with $c_1,\ldots,c_n\in\mathbb{Z}_{>0}$ (as mentioned above, Demonet works with \emph{group species}, which is a notion of species different from the one given in Definition \ref{def:species-realization} --it attaches a group algebra to each $i\in Q_0$ instead of a division ring);
\item  Nguefack \cite{Nguefack} has given a general mutation rule for species with potential; however, in that generality the question of existence of species realizations admitting non-degenerate potentials seems hard to address;
\item in \cite{LZ}, Labardini-Fragoso--Zelevinsky give a partially positive answer to Question \ref{question:non-deg-SPs?} for the skew-symmetrizable matrices $B$ that admit a skew-symmetrizer with pairwise coprime diagonal entries. More precisely, they show that for every such $B\in\mathbb{Z}^{n\times n}$ and every finite sequence $(i_1,\ldots,i_\ell)$ of elements of $\{1,\ldots,n\}$ there exist a species realization of $B$ over finite fields and a potential on this species to which the finite SP-mutation sequence $(\mu_{i_1}, \ \mu_{i_2}\mu_{i_1}, \ \ldots \ , \ \mu_{i_\ell}\cdots\mu_{i_2}\mu_{i_1})$ can be applied producing 2-acyclic SPs along the way;
\item in their very recent paper \cite{Bautista-Lopez}, Bautista--L\'opez-Aguayo give a partially positive answer to Question \ref{question:non-deg-SPs?} for the skew-symmetrizable matrices $B=(b_{ij})_{i,j}$ that admit a skew-symmetrizer $D=\diag(d_1,\ldots,d_n)$ with the property that each $d_j$ divides each and every member $b_{ij}$ of the $j^{\operatorname{th}}$ column of $B$. More precisely, stated as one of the main results in \cite{Bautista-Lopez} is the existence, for every such $B\in\mathbb{Z}^{n\times n}$ and every finite sequence $(i_1,\ldots,i_\ell)$ of elements of $\{1,\ldots,n\}$, of a species realization of $B$ and a potential on this species to which the finite SP-mutation sequence $(\mu_{i_1}, \ \mu_{i_2}\mu_{i_1}, \ \ldots \ , \ \mu_{i_\ell}\cdots\mu_{i_2}\mu_{i_1})$ can be applied producing 2-acyclic SPs along the way.
\end{enumerate}

In this paper we will give a general mutation rule for species with potential\footnote{General in the sense that it will be valid for species realizations of arbitrary skew-symmetrizable matrices.}, then we will show, by means of Example \ref{ex:6263-cycle-species}, that this mutation rule does not provide a full positive answer to Question \ref{question:non-deg-SPs?} for arbitrary skew-symmetrizable matrices, after which we will consider a subclass of a class of skew-symmetrizable matrices associated by Felikson-Shapiro-Tumarkin \cite{FeShTu-orbifolds} to (tagged) triangulations of surfaces with orbifold points of order 2  (see Remark \ref{rem:tau=t_1(tau)}). We will construct an explicit species realization for each matrix in this class of matrices, define an explicit potential on this species, and prove that whenever two tagged triangulations are related by a flip, the corresponding species with potential are related by the mutation rule we define. This compatibility between flips and SP-mutations will easily yield a positive answer to Question \ref{question:non-deg-SPs?} for the subclass under consideration.

\section{A few general considerations regarding species with potential}
\label{sec:SPs}

This section is devoted to presenting the algebraic setting for our species realizations of skew-symmetrizable matrices, and to developing a generalization of Derksen-Weyman-Zelevinsky's definition of mutations of quivers with potential to such species realizations. Our definition of mutation of species with potential is general in the sense that any skew-symmetrizable matrix has species realizations to which the definition can be applied. However, we will see that there exist skew-symmetrizable matrices whose species realizations (over certain cyclic Galois extensions) never admit non-degenerate potentials.

Most of the contents of this section are directly motivated by \cite{DWZ1}. In particular, Proposition \ref{prop:automorphisms}, Definitions \ref{def:automorphisms}, \ref{def:cyclic-stuff}, \ref{def:direct-sum-of-SPs}, \ref{def:red-and-trivial-SPs}, \ref{def:reduced-and-trivial-parts},  and \ref{def:SP-premutation}, and Theorems \ref{thm:splitting-theorem}, \ref{thm:SP-mut-well-defined-up-to-re} and \ref{thm:SP-mutation-is-involution} below, try to follow the guidelines of \cite{DWZ1}.

Let $(Q,\dtuple)$ be a weighted quiver. Let $d$ be the least common multiple of the integers that conform the tuple $\dtuple$.
Throughout this section we will suppose that
\begin{equation}\label{eq:root-of-unity}
\text{$F$ is a field containing a primitive $d^{\operatorname{th}}$ root of unity, and}
\end{equation}
\begin{equation}\label{eq:E/F-degree-d}
\text{$E$ is a degree-$d$ cyclic Galois field extension of $F$.}
\end{equation}

The assumption  \eqref{eq:E/F-degree-d} implies that
\begin{equation}\label{eq:Fi/F=degree-di-extension}
\text{for each $i\in Q_0$ there exists a unique degree-$d_i$ cyclic Galois subextension $F_i/F$ of $E/F$,}
\end{equation}
while \eqref{eq:root-of-unity} and \eqref{eq:E/F-degree-d} together imply that
\begin{equation}\label{eq:eigenbasis}
\text{there exists an element $v\in E$ such that $\B_E=\{1,v,\ldots,v^{d-1}\}$ is an eigenbasis of $E/F$,}
\end{equation}
that is, an $F$-vector space basis of $E$ consisting of eigenvectors of all elements of the Galois group $\Gal(E/F)$.

Throughout the paper we will use the following notation:
\begin{center}
{\renewcommand{\arraystretch}{1.5}
\renewcommand{\tabcolsep}{0.5cm}
\begin{tabular}{cc}
$d_{i,j} = \gcd(d_i,d_j)$ & $F_{i,j} = F_i\cap F_j$\\
$v_{i,j} = v^{\frac{d}{d_{i,j}}}$ & $v_i = v^{\frac{d}{d_i}}$\\
$\B_{i,j} = \{1,v_{i,j},v_{i,j}^2,\ldots,v_{i,j}^{d_{i,j}-1}\}=\B_E\cap F_{i,j}$ & $\B_i = \{1,v_i,\ldots,v_i^{d_i-1}\}=\B_E\cap F_i$\\
$G_{i,j} = \Gal(F_{i,j}/F)$ & $G_{i} = \Gal(F_i/F)$
\end{tabular}}
\end{center}


Notice that
$\B_i$ is an eigenbasis of $F_i/F$, and $\B_{i,j}$ is an eigenbasis of $F_{i,j}/F$.
Note also that
$v^d\in F$, from which one can easily deduce that there are functions $f:\B_E\times\B_E\rightarrow F^\times$ and $m:\B_E\times\B_E\rightarrow\B_E$ such that
$uu'=f(u,u')m(u,u')$ for all $u,u'\in\B_E$.
These functions have the extra property that $m(u,u')\in\B_i$ for all $u,u'\in\B_i$.

\begin{ex} For $d=2$, the extension $\C/\R$ obviously satisfies \eqref{eq:root-of-unity} and \eqref{eq:E/F-degree-d}, and the element $v\in\C$ in \eqref{eq:eigenbasis} can be taken to be an imaginary number whose square equals $-1$.
\end{ex}

\begin{defi}\label{def:modulating-function} Let $(Q,\dtuple)$ be a weighted quiver and let $E/F$ be a field extension satisfying \eqref{eq:root-of-unity} and \eqref{eq:E/F-degree-d}. If $g:Q_1\rightarrow\bigcup_{i,j\in Q_0}G_{i,j}$ is a function that to each arrow $a\in Q_1$ associates an element $g_a$ of $G_{h(a),t(a)}=\Gal(F_{h(a)}\cap F_{t(a)}/F)$, we will say that $g$ is a \emph{modulating function for $(Q,\dtuple)$ over $E/F$}.
\end{defi}

Suppose we have a modulating function $g:Q_1\rightarrow\bigcup_{i,j\in Q_0}G_{i,j}$ for $(Q,\dtuple)$ over $E/F$. Setting
\begin{equation}\label{eq:R-and-Ag-for-modulating-function}
R=\bigoplus_{i\in Q_0} F_i \ \ \ \text{and} \ \ \ A=\bigoplus_{a\in Q_1}F_{h(a)}^{g_a}\otimes_{F_{h(a),t(a)}} F_{t(a)},
\end{equation}
we see that $R$ is a semisimple commutative $F$-algebra and $A$ is an $R$-$R$-bimodule. We stress the fact that $A$ depends essentially not only on $(Q,\dtuple)$, but also on $g$ and $E/F$.

\begin{defi}\label{def:species-Ag} We will say that $A$ is the \emph{species} of $(Q,\dtuple,g)$ and $R$ is the \emph{vertex span} of $(Q,\dtuple,g)$.
\end{defi}

We will adopt the following conventions:
\begin{enumerate}
\item We shall identify each arrow $a\in Q_1$ with the element $1\otimes 1$ of the direct summand corresponding to $a$ in \eqref{eq:R-and-Ag-for-modulating-function};
\item we will write $e_i$ to denote the element of $R$ that has a 1 at the $i^{\operatorname{th}}$ component and $0$ elsewhere.
\end{enumerate}
With these conventions, the $R$-$R$-bimodule structure of $A$ referred to above is given by the rules
\begin{equation}\label{eq:Ag-standard-bimod-structure}
\left(\sum_{i\in Q_0}\lambda_ie_i\right)(xay)=\lambda_{h(a)}xay \ \ \ \text{(left action)} \ \ \ \ \ \text{and}\ \ \ \ \ (xay)\left(\sum_{i\in Q_0}\lambda_ie_i\right)=xay\lambda_{t(a)} \ \ \ \text{(right action)}
\end{equation}
for $\sum_{i\in Q_0}\lambda_ie_i\in R$, $a\in Q_1$, $x\in F_{h(a)}$ and $y\in F_{t(a)}$.

\begin{remark}\begin{enumerate}\item Notice that the left and right actions of $R$ on $A$ make no reference to the modulating function $g$. The place where $g$ plays a role is in the passage of scalars through the arrows (from left to right and from right to left); more formally, for $a\in Q_1$ and $x\in F_{h(a),t(a)}$ we have $ax=g_a(x)a$ and $xa=ag_{a}^{-1}(x)$.
\item Our use of the term ``species'' in Definition \ref{def:species-Ag} is motivated by the following: Let $B\in\mathbb{Z}^{n\times n}$ be a skew-symmetrizable matrix, and fix a skew-symmetrizer $D$ of $B$. Suppose $(Q,\dtuple)=(Q_B,\dtuple)$ is the weighted quiver we have associated to $B$ in Subsection \ref{subsec:matrices-vs-quivers}. Take any modulating function $g$ for $(Q,\dtuple)$ over $E/F$, and let $A$ be the species of $(Q,\dtuple,g)$. Then, setting $\mathbf{F}=(F_i)_{i\in Q_0}$ and $\mathbf{A}=(e_iAe_j\suchthat (i,j)\in Q_0\times Q_0, b_{ij}> 0)$, it is not hard to see that $(\mathbf{F},\mathbf{A})$ is a species realization of $B$ (see Definition \ref{def:species-realization}).
\end{enumerate}
\end{remark}

\begin{defi} \begin{itemize}\item The \emph{path algebra} of $A$ (or of $(Q,\dtuple,g)$) is the tensor algebra $R\langle A\rangle$ of $A$ over $R$.
\item The \emph{complete path algebra} of $A$ (or of $(Q,\dtuple,g)$) is the complete tensor algebra $\RA{A}$ of $A$ over $R$.
\end{itemize}
\end{defi}

\begin{defi}\label{def:path} A \emph{path of length $\ell$} on $(Q,\dtuple,g)$ is an element $\omega_0a_1\omega_1a_2\ldots \omega_{\ell-1}a_\ell\omega_\ell\in\RA{A}$, where
\begin{itemize}
\item $a_1,\ldots,a_\ell$, are arrows of $Q$ such that $h(a_{r+1})=t(a_r)$ for $r=1,\ldots,\ell-1$;
\item $\omega_0\in\B_{h(a_1)}$ and $\omega_r\in\B_{t(a_r)}$ for $r=1,\ldots,\ell$.
\end{itemize}
Here we are assuming that the eigenbasis $\B_E$, and hence the eigenbases $\B_i$ for $i\in Q_0$, have been \emph{a priori} fixed.
\end{defi}

Note that a path of length $0$ is just an element of the form $\omega e_i$, where $\omega$ belongs to an eigenbasis $\B_i$ and $e_i$ is the idempotent of $R$ sitting in the $i^{\operatorname{th}}$ component (hence there are $d_i$ paths of length $0$ sitting at each vertex $i$ of $Q$). Note also that while every arrow is a path of length $1$, not every path of length $1$ is an arrow. (That every arrow is a path of length $1$ follows from the fact that the element $1\in F$ belongs to each one of the eigenbases $\B_i$).

The set of all length-$\ell$ paths spans $A^\ell$ as an $F$-vector space, although it is not necessarily linearly independent over $F$.

Let ${\mathfrak m} = {\mathfrak m}(A)$ denote the (two-sided) ideal
of $\RA{A}$ given by
${\mathfrak m}  = {\mathfrak m}(A) = \prod_{\ell=1}^\infty A^\ell$.
Following \cite{DWZ1}, we view $\RA{A}$ as a topological $F$-algebra
via the \emph{${\mathfrak m}$-adic topology}, whose basic system of open neighborhoods around~$0$ is given by
the powers of ${\mathfrak m}$.
So, just as in \cite{DWZ1}, the closure of any subset $W \subseteq \RA{A}$
is given by
\begin{equation}\label{eq:closure}
\overline W = \bigcap_{n=0}^\infty (W + {\mathfrak m}^n).
\end{equation}
It is clear that $\pathalg{A}$ is dense in $\RA{A}$.

The ideal $\maxid$ satisfies the basic properties one would expect (cf.\ \cite[Section 2]{DWZ1}). For instance, $\maxid$ is
maximal amongst the two-sided ideals of $\RA{A}$
that have zero intersection with $R$.
Moreover,
if $(Q,\dtuple)$ and $(Q',\dtuple)$ are weighted quivers on the same vertex set and with the same weight function $\dtuple$, then any $R$-algebra homomorphism $\varphi: \RA{A}\rightarrow \RA{A'}$ sends $\maxid=\maxid(A)$ into $\maxid'=\maxid(A')$, and is hence
continuous.
Furthermore, such a ~$\varphi$ is uniquely determined by its
restriction to $A$, which is an $R$-$R$-bimodule homomorphism
$A \to \maxid' = A' \oplus (\maxid')^2$.
We write $\varphi|_{A} = (\varphi^{(1)}, \varphi^{(2)})$, where
$\varphi^{(1)}:A \to A'$ and $\varphi^{(2)}:A \to (\maxid')^2$ are
$R$-$R$-bimodule homomorphisms.

\begin{prop} \label{prop:automorphisms}
Any pair $(\varphi^{(1)}, \varphi^{(2)})$
of $R$-$R$-bimodule homomorphisms
$\varphi^{(1)}:A \to A'$ and $\varphi^{(2)}:A \to (\maxid')^2$
gives rise to a unique continuous $R$-algebra homomorphism
$\varphi: \RA{A} \to \RA {A'}$
such that $\varphi|_{A} = (\varphi^{(1)}, \varphi^{(2)})$.
Furthermore, $\varphi$ is an isomorphism
if and only if $\varphi^{(1)}$ is an $R$-$R$-bimodule isomorphism, in which case there exists a quiver isomorphism $\psi:Q\rightarrow Q'$ that fixes the vertex set pointwise and has the property that $g_{\psi(a)}=g_a$ for every arrow $a\in Q_1$.
\end{prop}

\begin{proof}
A suitable modification of the proof of Proposition 2.4 of \cite{DWZ1} applies here.
\end{proof}

\begin{remark} Simply laced quivers correspond to weighted quivers $(Q,\dtuple)$ for which the tuple $\dtuple$ consists entirely of 1s. For such a weighted quiver, giving $R$-$R$-bimodule homomorphisms $\varphi^{(1)}:A \to A$ and $\varphi^{(2)}:A \to (\maxid)^2$ is equivalent to providing tuples $(b_a)_{a\in Q_1}$ and $(u_a)_{a\in Q_1}$ subject to the conditions that $b_a\in e_{h(a)}Ae_{t(a)}$ and $u_a\in e_{h(a)}\maxid^2e_{t(a)}$ for every
$a\in Q_1$. The first of these two conditions means that, for every $a$, the element $b_a$ be an $F$-linear combination of the arrows of $Q$ that go from $t(a)$ to $h(a)$, while the second one means that $u_a$ be an $F$-linear combination of the paths of length at least 2 that go from $t(a)$ to $h(a)$. Hence, giving $R$-$R$-bimodule homomorphisms $\varphi^{(1)}:A \to A$ and $\varphi^{(2)}:A \to (\maxid)^2$ becomes rather easy, one just has to be careful with the starting and ending points of the elements $b_a$ and $u_a$ for $a\in Q_1$.

For weighted quivers whose weight tuple $\dtuple$ does not consist only of 1s the situation becomes more complicated, as one is required to obey an extra constraint besides the conditions that $b_a\in e_{h(a)}Ae_{t(a)}$ and $u_a\in e_{h(a)}\maxid^2e_{t(a)}$ for every $a\in Q_1$. Indeed, in order for the tuples $(b_a)_{a\in Q_1}$ and $(u_a)_{a\in Q_1}$ to induce $R$-$R$-bimodule homomorphisms $\varphi^{(1)}:A \to A$ and $\varphi^{(2)}:A \to (\maxid)^2$, one must have $b_a\in\image\pi_{g_a}$ and $u_a\in\image\pi_{g_a}$, where, as in Proposition \ref{prop:properties-of-J-K-bimodules}, $\pi_{g_a}:e_{h(a)}\RA{A}e_{t(a)}\rightarrow e_{h(a)}\RA{A}e_{t(a)}$
is the $F_{h(a)}$-$F_{t(a)}$-bimodule homomorphisms given by the rule
$$
m\mapsto\frac{1}{d_{h(a),t(a)}}\sum_{\omega\in\B_{h(a),t(a)}} g_a(\omega^{-1})m\omega.
$$
The condition that $b_a\in\image\pi_{g_a}$ can be translated to the requirement that $b_a$ be an $F_{h(a)}$-$F_{t(a)}$-bilinear combination of the arrows $c$ that go from $t(a)$ to $h(a)$ and satisfy $g_c=g_a$.
A similar situation occurs with each element $u_a$.

On the other hand, once one has tuples $(b_a)_{a\in Q_1}$ and $(u_a)_{a\in Q_1}$ such that $b_a\in\image\pi_{g_a}\cap A$ and $u_a\in\image\pi_{g_a}\cap\maxid^2$ for every $a\in Q_1$, one can automatically define the corresponding $R$-$R$-bimodule homomorphisms $\varphi^{(1)}$ and $\varphi^{(2)}$ and thus obtain an $R$-algebra homomorphism $\varphi:\RA{A}\rightarrow\RA{A}$. In other words, the possible obstructions for defining $R$-algebra endomorphisms of $\RA{A}$ lie on the difficulty to define $R$-$R$-bimodule homomorphisms $\varphi^{(1)}$ and $\varphi^{(2)}$, and not on the passage from $R$-$R$-bimodule homomorphisms to $R$-algebra homomorphisms, the latter passage is always possible.
\end{remark}

\begin{ex} We illustrate the previous remark with an example. Let $Q$ be the quiver $1\overset{a}{\underset{b}{\rightrightarrows}} 2$ and $\dtuple$ be the pair $(2,2)$. The function $g:\{a,b\}\rightarrow\Gal(\C/\R)$ given by $g_a=\myid$ and $g_b=\theta$, where $\theta$ denotes complex conjugation, clearly is a modulating function for $(Q,\dtuple)$ over $\C/\R$. The ideal $\maxid$ of $\RA{A}$ obviously satisfies $\maxid^2=0$, whence the $R$-algebra endomorphisms of $\RA{A}$ are in bijection with the $R$-$R$-bimodule endomorphisms of $A$. Since $g_a\neq g_b$, any $R$-$R$-bimodule endomorphism of $A$ must map $a$ to a $\C$-multiple of $a$, and $b$ to a $\C$-multiple of $b$. In particular, no $R$-algebra endomorphism of $\RA{A}$ sends $a$ to $b$ or $b$ to $a$.
\end{ex}

\begin{defi}\label{def:automorphisms}
Let $\varphi$ be an automorphism of $\RA{A}$, and let
$(\varphi^{(1)}, \varphi^{(2)})$ be the corresponding pair of $R$-$R$-bimodule homomorphisms.
If~$\varphi^{(2)} = 0$, then we call $\varphi$ a \emph{change of arrows}.
If $\varphi^{(1)}$ is the identity automorphism of~$A$, we say that~$\varphi$
is a \emph{unitriangular} automorphism; furthermore, we
say that $\varphi$ is of \emph{depth} $\delta \geq 1$, if $\varphi^{(2)}(A)
\subset \maxid^{\delta+1}$.
\end{defi}

Just as in the case of complete path algebras of quivers, unitriangular automorphisms possess the following useful property:
\begin{align}
\label{eq:unitriangular-d}
&\text{If~$\varphi$ is a unitriangular automorphism of $\RA{A}$ of depth~$\delta$,}\\
\nonumber
&\text{then
$\varphi(u) - u \in  \maxid^{n+\delta}$ for $u \in \maxid^{n}$.}
\end{align}


\begin{defi}[Potentials, cyclical equivalence, cyclic derivatives, Jacobian algebras]\
\label{def:cyclic-stuff}
\begin{itemize}
\item For each $\ell \geq 1$, we define the \emph{cyclic part} of $A^\ell$ to be
${A^\ell}_{\operatorname{cyc}} = \bigoplus_{i \in Q_0} e_iA^\ell e_i$.
Thus, ${A^\ell}_{\rm cyc}$ is the $F$-span of all paths $\omega_0 a_1\omega_1 \cdots a_\ell\omega_\ell$
with $h(a_1) = t(a_d)$; we call such paths \emph{cycles}.
\item We define a closed $F$-vector subspace
$\RA{A}_{\operatorname{cyc}} \subseteq \RA{A}$
by setting
$
\RA{A}_{\operatorname{cyc}} =
\prod_{\ell=1}^\infty {A^\ell}_{\operatorname{cyc}},
$
and call the elements of $\RA{A}_{\operatorname{cyc}}$
\emph{potentials} on $A$. For any potential $S\in\RA{A}_{\operatorname{cyc}}$, we will say that the pair $(A,S)$ is a \emph{species with potential}, or \emph{SP} for short.
\item\label{item:cyclic-equivalence}
Two potentials $S$ and $S'$ are \emph{cyclically equivalent}
if $S - S'$ lies in the closure of the $F$-span of
all elements of the form $\omega_0 a_1\omega_1 a_2\omega_2  \cdots a_\ell\omega_\ell - \omega_1 a_2\omega_2  \cdots a_\ell\omega_\ell\omega_0a_1$,
where $\omega_0a_1\omega_1 a_2\omega_2  \cdots a_\ell\omega_\ell$ is a cyclic path on $(Q,\dtuple)$.
\item Two SPs $(A,S)$ and $(A',S')$ are \emph{right-equivalent} if there exists a \emph{right-equivalence} $\varphi:(A,S)\rightarrow(A',S')$, which by definition is a ring isomorphism $\varphi:\RA{A}\rightarrow\RA{A'}$ satisfying $\varphi|_{R}=\myid_{R}$ and such that $\varphi(S)$ and $S'$ are cyclically equivalent.
\item For each $a\in Q_1$, we define the \emph{cyclic
derivative} $\partial_a$ as the continuous $F$-linear map
$\RA{A}_{\rm cyc} \to \RA{A}$
acting on individual cycles by
\begin{equation}
\label{eq:cyclic-derivative}
\partial_a (\omega_0 a_1\omega_1 a_2\cdots a_\ell\omega_\ell)=
\sum_{k=1}^\ell \delta_{a,a_k} \pi_{g_a^{-1}}(\omega_k a_{k+1} \cdots a_\ell\omega_\ell\omega_0 a_1 \cdots a_{k-1}\omega_{k-1}),
\end{equation}
where $\delta_{a,a_k}$ is the \emph{Kronecker delta} between $a$ and $a_k$, and $\pi_{g_a^{-1}}(\xi)=\frac{1}{d_{h(a),t(a)}}\sum_{\nu\in\B_{h(a),t(a)}}g_{a}^{-1}(\nu^{-1})\xi\nu$.
\item For every potential~$S$, we
define its \emph{Jacobian ideal} $J(S)$ as the closure of
the (two-sided) ideal in $\RA{A}$
generated by the elements $\partial_a(S)$ for all $a\in Q_1$
(see \eqref{eq:closure}); clearly, $J(S)$ is a two-sided
ideal in $\RA{A}$.
\item We call the quotient $\RA{A}/J(S)$
the \emph{Jacobian algebra} of~$S$, and denote
it by ${\mathcal P}(A,S)$.
\end{itemize}
\end{defi}

\begin{ex}\label{ex:2-cycle-equivalent-to-0}
Let $m \in e_i \RA{A} e_j$ and $n \in e_j \RA{A} e_i$.
 In what follows, we will repeatedly use Proposition~\ref{prop:properties-of-J-K-bimodules}.
 If $n = \pi_\rho(n)$ for some $\rho \in G_{i.j}$, one computes easily
 \[
  \begin{array}{llllllll}
   m n
   &=&
   \displaystyle \frac{1}{d_{i,j}} \sum_{\omega \in \B_{i,j}} m \rho(\omega^{-1}) n \omega
   &\sim_{\operatorname{cyc}}&
   \displaystyle \frac{1}{d_{i,j}} \sum_{\omega \in \B_{i,j}} \omega m \rho(\omega^{-1}) n
   &=&
   \pi_{\rho^{-1}}(m) n \,.
  \end{array}
 \]
 Therefore, we have in general
 \[
  \begin{array}{llllllll}
   m n
   &=&
   \displaystyle \sum_{\rho \in G_{i,j}} m \pi_\rho(n)
   &\sim_{\operatorname{cyc}}&
   \displaystyle \sum_{\rho \in G_{i,j}} \pi_{\rho^{-1}}(m) \pi_\rho(n) \,.
  \end{array}
 \]
 In particular, if $m = x a$ and $n = y b z$ for arrows $a, b \in Q_1$ and elements~$x, z \in F_i$, $y \in F_j$ one has
 \[
  x a y b z \:\sim_{\operatorname{cyc}}\: \displaystyle \sum_{\rho \in G_{i,j}} x \pi_{\rho^{-1}}(a) y \pi_\rho(b) z \:=\: \delta_{g_a,g_b^{-1}} \, x a y b z \,,
 \]
 i.e.\ the potential~$x a y b z$ is cyclically equivalent to $0$ unless $g_a = g_b^{-1}$.
\end{ex}

\begin{ex}\label{ex:6263-cycle-2cycle-species} Consider the weighted quiver
{\small\begin{center}
$(Q,\dtuple)=$\begin{tabular}{cc}
$\xymatrix{1   \ar@{}[ddrrr]^(.1){}="a"^(.94){}="b" \ar@<0.7ex> "a";"b" \ar@<1.7ex> "a";"b" \ar[rrr]^{\delta}&  & & 2 \ar[dd]^{\gamma}\\
 & & \\
 4 \ar[uu]^{\alpha} & & & 3 \ar@{}[uulll]^(.13){}="a"^(.93){}="b"  \ar@<0.5ex> "a";"b" \ar@<1.5ex> "a";"b" \ar@<2.5ex> "a";"b" \ar[lll]^{\beta}}$ &
$\xymatrix{6 & & & 2 \\
 & & &\\
3 & & & 6 }$
\end{tabular}
\end{center}}
\noindent whose diagonal arrows we denote by $\varepsilon_1,\varepsilon_2:1\rightarrow 3$ and $\eta_1,\eta_2,\eta_3:3\rightarrow 1$.
Since $\lcm(1,2,1,3)=6$, we have $[E:F]=6$, $F_1=E=F_3$, $[F_2:F]=2$ and $[F_4:F]=3$ (note that here, just as in \eqref{eq:Fi/F=degree-di-extension}, the subindex $i$ in $F_i$ does not refer to the degree of the extension $F_i/F$, but rather to the vertex $i$ of $Q$ to which $F_i$ is attached). Write $\Aut(F_{i,j})$ to denote the group of field automorphisms of $F_{i,j}=F_i\cap F_j$. Let $g:Q_1\rightarrow\bigcup_{i,j\in Q_0}\Aut(F_{i,j})$ be any function satisfying the following properties simultaneously:
\begin{itemize}
\item[(0)] $g_a\in\Aut(F_{h(a),t(a)})$ for every $a\in Q_1$;
\item[(i)] $g_{\varepsilon_1}\neq g_{\varepsilon_2}$;
\item[(ii)] $g_{\varepsilon_1}|_{F_4}=g_{\varepsilon_2}|_{F_4}$;
\item[(iii)] $g_{\eta_1}|_{F_2}=g_{\eta_2}|_{F_2}=g_{\eta_3}|_{F_2}$.
\end{itemize}
By (0), we have $\varepsilon_1,\varepsilon_2,\eta_1,\eta_2,\eta_3\in\Aut(E)$, while (i), (ii) and the equality $E=F_2F_4$ imply that $g_{\varepsilon_1}^{-1}|_{F_2}\neq g_{\varepsilon_2}^{-1}|_{F_2}$. From this and (iii), we deduce that there necessarily exists $\varepsilon\in \{\varepsilon_1,\varepsilon_2\}$ such that $g_\varepsilon\notin\{g_{\eta_1}^{-1},g_{\eta_2}^{-1},g_{\eta_3}^{-1}\}$. For such $\varepsilon$, every $E$-multiple of each of the 2-cycles $\varepsilon\eta_1$, $\varepsilon\eta_2$ and $\varepsilon\eta_3$ will be cyclically equivalent to 0 in the complete path algebra of the species of $(Q,\dtuple,g)$, and hence, $\varepsilon$ will not belong to the Jacobian ideal of any potential $W$ belonging to the complete path algebra.
\end{ex}


\begin{defi}\label{def:direct-sum-of-SPs} Let $(Q,\dtuple)$ and $(Q',\dtuple)$ be weighted quivers with the same vertex set $Q_0=Q'_0$ and the same weight tuple $\dtuple=(d_i)_{i\in Q_0}$, let $g$ and $g'$ respectively be modulating functions for $(Q,\dtuple)$ and $(Q',\dtuple)$ over $E/F$, and let $A$ and $A'$ be the species of $(Q,\dtuple,g)$ and $(Q',\dtuple,g')$, respectively. Let $S$ and $S'$ be potentials on $A$ and $A'$, respectively. The \emph{direct sum} of the SPs $(A,S)$ and $(A',S')$ is the SP $(A,S)\oplus(A',S')=(A\oplus A',S+S')$, where $A\oplus A'$ is the direct sum of $A$ and $A'$ as $R$-$R$-bimodules, and $S+S'$ is the sum of $S$ and $S'$ in $\RA{A\oplus A'}$.
\end{defi}

\begin{defi}\label{def:red-and-trivial-SPs} Let $A$ be the species of a triple $(Q,\dtuple,g)$, and let $S$ be a potential on $A$. We will say that $(A,S)$ is:
\begin{itemize}
\item a \emph{reduced SP} if the degree-2 component of $S$ according to the decomposition \eqref{eq:def-of-R<<A>>} of $\RA{A}$ is $0$;
\item a \emph{trivial SP} if $S$ equals its degree-2 component and the $R$-$R$-subbimodule of $A$ generated by the cyclic derivatives of $S$ is equal to $A$.
\end{itemize}
\end{defi}

We now state the analog in our context of \cite[Theorem 4.6]{DWZ1}.

\begin{thm}[Splitting Theorem]\label{thm:splitting-theorem} For every species $A$ as in Definition \ref{def:species-Ag} (see also Definitions \ref{def:weighted-quiver} and \ref{def:modulating-function}) and every potential $S\in\RA{A}$ there exist a reduced SP $(A_{\operatorname{red}},S_{\operatorname{red}})$, a trivial SP $(A_{\operatorname{triv}},S_{\operatorname{triv}})$ and a right-equivalence $\varphi:(A,S)\rightarrow (A_{\operatorname{red}},S_{\operatorname{red}})\oplus(A_{\operatorname{triv}},S_{\operatorname{triv}})$. Furthermore, the right-equivalence classes of $(A_{\operatorname{red}},S_{\operatorname{red}})$ and $(A_{\operatorname{triv}},S_{\operatorname{triv}})$ are uniquely determined by $(A,S)$.
\end{thm}

The proof of \cite[Theorem 4.6]{DWZ1} can be adapted to produce a proof of Theorem \ref{thm:splitting-theorem}. The adaptation requires some considerations that may be regarded to be non-trivial; the reader can find these considerations in Subsection \ref{subsec:proof-Splitting-Thm}.

\begin{defi}\label{def:reduced-and-trivial-parts} In the situation of Theorem \ref{thm:splitting-theorem}, the (right-equivalence class of the) SP $(A_{\operatorname{red}},S_{\operatorname{red}})$ is called the \emph{reduced part} of $(A,S)$, while the (right-equivalence class of the) SP $(A_{\operatorname{triv}},S_{\operatorname{triv}})$ is called the \emph{trivial part} of $(A,S)$.
\end{defi}

\begin{remark} Despite Theorem \ref{thm:splitting-theorem}, there exist species $A$ such that, no matter which potential $S\in\RA{A}$ we take, the underlying quiver of the SP $(A_{\operatorname{red}},S_{\operatorname{red}})$ has 2-cycles. Example \ref{ex:6263-cycle-2cycle-species} provides a way of producing explicit examples.
\end{remark}

We now turn to the problem of giving the ``pre-mutation" rule for weighted quivers, modulating functions and potentials. This needs some preparation.

Fix a vertex $k\in Q_0$. For every pair of arrows $a,b\in Q_1$ such that $t(b)=k=h(a)$ we adopt the notations
\begin{eqnarray}\label{eq:Galois-coset-G_ba}
\B_{b,a} & =& \left\{1,v_{k},v_{k}^2,\ldots,v_{k}^{\frac{d_{k}}{\operatorname{lcm}(d_{h(b),k},d_{k,t(a)})}-1}\right\}\\
\nonumber
X_{b,a} &=& \left\{h\in G_{h(b),t(a)}\suchthat h|_{F_{h(b),t(a)}\cap F_k}=g_b|_{F_{h(b),t(a)}\cap F_k}g_a|_{F_{h(b),t(a)}\cap F_k}\right\};
\end{eqnarray}
Note that for any such pair of arrows, the set $\B_{b,a}$ is an eigenbasis of the field extension $F_{k}/(F_{h(b),k})(F_{k,t(a)})$, and the set $X_{b,a}$ is a coset of the subgroup $\Gal(F_{h(b),t(a)}/F_{h(b),t(a)}\cap F_k)$ of the group $G_{h(b),t(a)}=\Gal(F_{h(b),t(a)}/F)$, hence $|X_{b,a}|=\frac{d_{h(b),t(a)}}{\gcd(d_{h(b)},d_{k},d_{t(a)})}$. Notice also that
$$
|\B_{b,a}||X_{b,a}|=
    \frac{d_k}{\operatorname{lcm}(d_{h(b),k},d_{k,t(a)})}\frac{d_{h(b),t(a)}}{\gcd(d_{h(b)},d_k,d_{t(a)})}=
    \frac{d_kd_{h(b),t(a)}\gcd(d_{h(b),k},d_{k,t(a)})}{d_{h(b),k}d_{k,t(a)}\gcd(d_{h(b)},d_k,d_{t(a)})}=
    \frac{d_kd_{h(b),t(a)}}{d_{h(b),k}d_{k,t(a)}},
$$
that is, the integer $|\B_{b,a}||X_{b,a}|$ is precisely the number of ``composite'' arrows from $t(a)$ to $h(b)$ introduced in Definition \ref{def:weighted-quiver-mutation}. Therefore, we can bijectively label these ``composite'' arrows with the symbols $[b\omega a]_{\rho}$, where $\omega$ runs inside $\B_{b,a}$ and $\rho$ runs inside $X_{b,a}$.


\begin{defi}\label{def:SP-premutation} Let $(Q,\dtuple)$ be a 2-acyclic weighted quiver, $k\in Q_0$ a vertex of $Q$, and $E/F$ a field extension satisfying \eqref{eq:root-of-unity} and \eqref{eq:E/F-degree-d}.
\begin{enumerate}
\item Denote by $(\widetilde{\mu}_k(Q),\dtuple)$ the weighted quiver obtained after applying the following two steps:
\begin{itemize}
\item[(Step 1)] For each pair $a,b\in Q_1$ such that $h(a)=k=t(b)$, add $\frac{d_k d_{h(b),t(a)}}{d_{h(b),k}d_{k,t(a)}}$ ``composite'' arrows from $t(a)$ to $h(b)$. Label these arrows with the symbols $[b\omega a]_{\rho}$, where $\omega$ runs inside $\B_{b,a}$ and $\rho$ runs inside the set $X_{b,a}$;
\item[(Step 2)] replace each arrow $c$ incident to $k$ with an arrow $c^*$ going in the opposite direction.
\end{itemize}
\item Given a modulating function $g:Q_1\rightarrow\bigcup_{i,j\in Q_0}G_{i,j}$ for $(Q,\dtuple)$, we define a function $\widetilde{\mu}_k(g):\widetilde{\mu}_k(Q)_1\rightarrow\bigcup_{i,j\in Q_0}G_{i,j}$ in the obvious way, namely, by setting
$$
\widetilde{\mu}_k(g)_{\alpha}=\begin{cases}
g_\alpha & \text{if $\alpha\in Q_1\cap\widetilde{\mu}_k(Q)_1$};\\
\rho & \text{if $\alpha=[b\omega a]_\rho$ is a composite arrow of $\widetilde{\mu}_k(Q)$};\\
g_c^{-1} & \text{if $c=c^*$ for some arrow $c\in Q_1$ incident to $k$.}
\end{cases}
$$
It is clear that $\widetilde{\mu}_k(g)$ is a modulating function for $(\widetilde{\mu}_k(Q),\dtuple)$.
\item Let us denote by $\widetilde{\mu}_k(A)$ the species of $(\widetilde{\mu}_k(Q),\dtuple,\widetilde{\mu}_k(g))$ over $E/F$, where $g$ is a modulating function as in the previous paragraph. Given a potential $S\in\RA{A}$, we assume that $e_kS=0=Se_k$ by replacing $S$ with a cyclically equivalent potential if necessary, and denote by $\widetilde{\mu}_k(S)$ the element of $\RA{\widetilde{\mu}_k(A)}$ given by
$$
\widetilde{\mu}_k(S):=[S]+\triangle_k(A):=[S]+\sum_{\overset{a}{\to}k\overset{b}{\to}}\frac{1}{|\B_{b,a}|}\sum_{\omega\in\B_{b,a}}\sum_{\rho\in X_{b,a}}\omega^{-1}b^*[b\omega a]_{\rho}a^*,
$$
where $[S]$ is obtained from $S$ as follows. Write $S$ as a possibly infinite $F$-linear combination of cyclic paths, say $S=\sum_\xi\alpha_\xi\xi$ with $\alpha_\xi\in F$, where each path $\xi$ is expressed as $\xi=x_0 b_1 \omega_1 a_1 x_1 \cdots b_\ell \omega_\ell a_\ell x_\ell$, where $a_i, b_i \in Q_1$ are arrows with $t(b_i) = k = h(a_i)$ for all $i\in\{1,\ldots,\ell\}$, the $x_i$ are paths not passing through $k$, and $\omega_i \in \B_{b_i,a_i}$ for all $i\in\{1,\ldots,\ell\}$.
  Then we define $[S]=\sum_\xi\alpha_\xi[\xi]$, where
  $$
   [\xi] \:=\: x_0 [b_1 \omega_1 a_1] x_1 \cdots [b_\ell \omega_\ell a_\ell] x_\ell
   \ \ \ \text{with} \ \ \
   [b_i \omega_i a_i] \:=\: \sum_{\rho\in X_{b,a}} [b_i \omega_i a_i]_\rho
   \,.
  $$
\end{enumerate}
\end{defi}

The definition of the potential $[S]$ deserves some comments.

\begin{remark}\label{rem:on-def-of-[S]}\begin{enumerate}\item Under the labeling we have adopted for the ``composite'' arrows from $t(a)$ to $h(b)$, the element $1\otimes 1$ of the direct summand corresponding to the indices $\rho\in X_{b,a}$ and $\omega\in\B_{b,a}$ in the direct sum
    \begin{equation}\label{eq:premut-bimodule-labeled}
    \bigoplus_{\omega\in\B_{b,a}}\bigoplus_{\rho\in X_{b,a}}F_{h(b)}^\rho\otimes_{F_{h(b),t(a)}}F_{t(a)}\subseteq\widetilde{\mu}_k(A)
    \end{equation}
    is identified with the arrow $[b\omega a]_{\rho}$ of $\widetilde{\mu}_k(Q,\dtuple)$ when $1\otimes1$ is considered as an element of the species $\widetilde{\mu}_k(A)$, see the paragraph immediately after Definition \ref{def:species-Ag};
\item Proposition \ref{prop:tensor-prod-decomp} tells us that there is an $F_{h(b)}$-$F_{t(a)}$-bimodule isomorphism
    \begin{equation}\label{eq:tensor-prod-decomp-premut}
    \left(F_{h(b)}^{g_b}\otimes_{F_{h(b),k}}F_k\right)\otimes_{F_k}\left(F_{k}^{g_a}\otimes_{F_{k,t(a)}}F_{t(a)}\right)\cong
    \left(\bigoplus_{\rho\in X_{b,a}}F_{h(b)}^{\rho}\otimes_{F_{h(b),t(a)}}F_{t(a)}\right)^{[F_k:(F_{h(b),k})(F_{k,t(a)})]}.
    \end{equation}
    Thus, the labels of the ``composite arrows'' and the modulating function $\widetilde{\mu}_k(g)$ are compatible with the direct sum decomposition \eqref{eq:tensor-prod-decomp-premut}, and the right-hand side of \eqref{eq:tensor-prod-decomp-premut} coincides with the left-hand side of \eqref{eq:premut-bimodule-labeled};
\item the bimodule isomorphisms \eqref{eq:tensor-prod-decomp-premut} for $a,b\in Q_1$ such that $t(b)=k=h(a)$, considered altogether, induce an $R$-$R$-bimodule homomorphism $[\bullet]:(1-e_k)\RA{A}(1-e_k)\rightarrow\RA{\widetilde{\mu}_k(A)}$. The assumption $e_kS=0=Se_k$ guarantees that $S$ belongs to the domain of $[\bullet]$. The image of $S$ under $[\bullet]$ coincides with the potential $[S]$ defined in Part (3) of Definition \ref{def:SP-premutation}. In particular, this means that $[S]$ is indeed independent of the chosen expression of $S$ as (possibly infinite) $F$-linear combination of cyclic paths.
\item the cyclical equivalence class of $\widetilde{\mu}_k(S)$ is determined by the cyclical equivalence class of $S$.
\end{enumerate}
\end{remark}

\begin{thm}\label{thm:SP-mut-well-defined-up-to-re} If $(A,S)$ and $(A',S')$ are right-equivalent SPs, then $(\widetilde{\mu}_k(A),\widetilde{\mu}_k(S))$ and $(\widetilde{\mu}_k(A'),\widetilde{\mu}_k(S'))$ are right-equivalent SPs too. Consequently, the reduced parts of $(\widetilde{\mu}_k(A),\widetilde{\mu}_k(S))$ and $(\widetilde{\mu}_k(A'),\widetilde{\mu}_k(S'))$ are right-equivalent as well.
\end{thm}

The proof of the first statement of Theorem \ref{thm:SP-mut-well-defined-up-to-re} is similar to the proof of \cite[Theorem 8.3]{LZ} (which in turn follows the main idea of the proof of \cite[Theorem 5.2]{DWZ1}), but requires an extra consideration that, although elementary, may be regarded to be non-obvious. This consideration, not present (nor necessary) in \cite{DWZ1} nor in \cite{LZ}, can be found in Section \ref{subsec:proof-SP-mut-well-defined-up-to-re}.

\begin{defi}\label{def:SP-mutation} The (right-equivalence class of the) reduced part of $(\widetilde{\mu}_k(A),\widetilde{\mu}_k(S))$ will be called \emph{the mutation of $(A,S)$ in direction $k$} and denoted $\mu_k(A,S)$.
\end{defi}

\begin{remark} Since the underlying weighted quiver of $\mu_k(A,S)$ does not necessarily coincide with the weighted quiver obtained from $(Q,\dtuple)$ through the $k^{\operatorname{th}}$ weighted-quiver mutation, we are not allowed to always write $\mu_k(A,S)=(\mu_k(A),\mu_k(S))$.
\end{remark}

\begin{thm}\label{thm:SP-mutation-is-involution} The SPs $(A,S)$ and $\mu_k\mu_k(A,S)$ are right-equivalent. In other words, SP-mutation is an involution up to right-equivalence.
\end{thm}

The proof of Theorem \ref{thm:SP-mutation-is-involution} is a minor modification of the proof of \cite[Theorem 8.10]{LZ}, which in turn follows the main idea of proof of \cite[Theorem 5.7]{DWZ1}

The following example was found by A. Zelevinsky and the second author a few years ago.

\begin{ex}\label{ex:6263-cycle-species} Example \ref{ex:6263-cycle-2cycle-species} has a very unpleasant consequence. Consider the matrix
$$
B=\left[\begin{array}{rrrr}
0 & 1 & 0 & -1\\
-3 & 0 & 3 & 0\\
0 & -1 & 0 & 1 \\
2 & 0 & -2 & 0
\end{array}
\right] .
$$
A straightforward check shows that if we set $D=\diag(6,2,6,3)$, then $DB$ is skew-symmetric. That is, $B$ is skew-symmetrizable and $D=\diag(6,2,6,3)$ is a skew-symmetrizer of $B$. The associated weighted quiver is
{\small\begin{center}
$(Q',\dtuple)=$\begin{tabular}{cc}
$\xymatrix{1  \ar[dd]
& & 2 \ar[ll]
\\
 & & \\
4 \ar[rr]
& & 3 \ar[uu]
}$ &
$\xymatrix{6  & & 2 \\
 & & \\
3 & & 6 }$
\end{tabular}
\end{center}}
Take any modulating function $h:Q'_1\rightarrow\bigcup_{i,j\in Q'_0}G_{i,j}$. Then $\widetilde{\mu}_2(Q',\dtuple)=\mu_2(Q',\dtuple)$ and the weighted quiver $(Q,\dtuple)=\widetilde{\mu}_3\mu_2(Q',\dtuple)$ is precisely the one in Example \ref{ex:6263-cycle-2cycle-species}. Furthermore, the modulating function $\widetilde{\mu}_3\mu_2(h):Q_1\rightarrow\bigcup_{i,j\in Q_0}G_{i,j}$ satisfies the conditions (0), (i), (ii) and (iii) stated in that same example. Whence, for any potential $S\in\RA{A}$, where $A$ is the species of $(Q',\dtuple,h)$,
the underlying quiver of the reduced part of the SP $\widetilde{\mu}_3\mu_2(A,S)$ will fail to be 2-acyclic. In other words, the underlying quiver of $\mu_3\mu_2(A,S)$ will not be 2-acyclic, no matter which potential $S\in\RA{A}$ we take. In the terminology of \cite{DWZ1}, this means that, no matter which modulating function $h:Q'_1\rightarrow\bigcup_{i,j\in Q'_0}G_{i,j}$ we take, the species with potential $(A,S)$ will be degenerate for every potential $S$ on the species $A$ of $(Q',\dtuple,h)$.

Notice that, setting $C=\diag(1,3,1,2)$, the matrix $C^{-1}B$ is skew-symmetric. As shown in \cite[Example 14.3]{LZ}, from this it can be deduced that $B$ admits a global unfolding (for the definition of what a global unfolding is, see \cite[Section 4]{FeShTu-unfoldings} or \cite[Definition 14.2]{LZ}).
\end{ex}


We finish this section with a couple of technical results that will be used later.

\begin{lemma}\label{lemma:non-triv-automorphism} Let $(Q,\dtuple)$ be a weighted quiver, $E/F$ a field extension satisfying \eqref{eq:root-of-unity} and \eqref{eq:E/F-degree-d}, and $g$ a modulating function for $(Q,\dtuple)$ over $E/F$. Suppose that we are given a tuple $(\phi_j)_{j\in Q_0}$ and a bijection $\psi:Q_1\rightarrow Q_1$ such that:
\begin{itemize}
\item $\phi_j\in G_j$ for every $j\in Q_0$;
\item $t(a)=t(\psi(a))$ and $h(a)=h(\psi(a))$ for every $a\in Q_1$;
\item $g_{\psi(a)}=\phi_{h(a)}|_{F_{h(a),t(a)}}g_a\phi_{t(a)}^{-1}|_{F_{h(a),t(a)}}$ for every $a\in Q_1$.
\end{itemize}
Then the function $\Psi^{(0)}:R\rightarrow R$ given by
$
\sum_{j\in Q_0}x_je_j\mapsto\sum_{j\in Q_0}\phi_j(x_j)e_j
$
is a ring automorphism of $R$, and the rule
\begin{equation}\label{eq:Psi^1-general-version}
xay\mapsto\phi_{h(a)}(x)\psi(a)\phi_{t(a)}(y), \ \ \ \text{for $a\in Q_1$, $x\in F_{h(a)}$ and $y\in F_{t(a)}$},
\end{equation}
produces a well-defined group automorphism $\Psi^{(1)}:A\rightarrow A$ such that $\Psi^{(1)}(ru)=\Psi^{(0)}(r)\Psi^{(1)}(u)$ and $\Psi^{(1)}(ur)=\Psi^{(1)}(u)\Psi^{(0)}(r)$ for all $r\in R$ and all $u\in A$. Consequently, there exists a unique ring automorphism $\Psi:\RA{A}\rightarrow\RA{A}$ such that $\Psi|_R=\Psi^{(0)}$ and $\Psi|_{A}=\Psi^{(1)}$.
\end{lemma}

\begin{proof} It is clear that $\Psi^{(0)}$ is a ring automorphism of $R$.
For $a\in Q_1$, $x\in F_{h(a)}$, $y\in F_{t(a)}$ and $z\in F_{h(a),t(a)}$, we have
\begin{eqnarray}\nonumber
\phi_{h(a)}(x)\psi(a)\phi_{t(a)}(zy) &=& \phi_{h(a)}(x)g_{\psi(a)}(\phi_{t(a)}(z))\psi(a)\phi_{t(a)}(y)\\
\nonumber
&=& \phi_{h(a)}(x)\phi_{h(a)}(g_a(z))\psi(a)\phi_{t(a)}(y)\\
\nonumber
&=&\phi_{h(a)}(xg_a(z))\psi(a)\phi_{t(a)}(y),
\end{eqnarray}
hence \eqref{eq:Psi^1-general-version} produces a well-defined group homomorphism $\Psi^{(1)}:A\rightarrow A$.
It is clear that $\Psi^{(1)}$ is bijective and that it satisfies $\Psi^{(1)}(ru)=\Psi^{(0)}(r)\Psi^{(1)}(u)$ and $\Psi^{(1)}(ur)=\Psi^{(1)}(u)\Psi^{(0)}(r)$. The existence and uniqueness of a ring automorphism $\Psi:\RA{A}\rightarrow\RA{A}$ such that $\Psi|_R=\Psi^{(0)}$ and $\Psi|_{A}=\Psi^{(1)}$ follow now from Proposition \ref{prop:ring-endomorphisms-of-RA}.
\end{proof}

\begin{coro}\label{coro:non-triv-automorphism-Psi_i} Let $(Q,\dtuple)$ be a weighted quiver such that the least common multiple of $\dtuple$ is 2, $E/F$ a field extension satisfying \eqref{eq:root-of-unity} and \eqref{eq:E/F-degree-d}, and $g$ a modulating function for $(Q,\dtuple)$ over $E/F$. Suppose that $i$ is a vertex of $Q$ for which $d_i=2$ and $\psi:Q_1\rightarrow Q_1$ is a bijection such that:
\begin{itemize}
\item $\psi^2$ is the identity function of $Q_1$;
\item $t(a)=t(\psi(a))$ and $h(a)=h(\psi(a))$ for every $a\in Q_1$;
\item for $a\in Q_1$, if $a$ is not incident to $i$, then $\psi(a)=a$;
\item for $a\in Q_1$, if $a$ is incident to $i$ and one of the integers $\dtuple_{h(a)}$ and $\dtuple_{t(a)}$ is equal to $1$, then $\psi(a)=a$;
\item for $a\in Q_1$, if $a$ is incident to $i$ and both $\dtuple_{h(a)}$ and $\dtuple_{t(a)}$ are equal to 2, then $g_{\psi(a)}\neq g_{a}$.
\end{itemize}
Then there exists a unique ring automorphism $\Psi_i$ of $\RA{A}$ with the properties that:
\begin{itemize}
\item its restriction to $R$ obeys the rule
\begin{equation}\label{eq:automorphism-of-R}
\sum_{j\in Q_0}\alpha_je_j\mapsto
\theta(\alpha_i)e_i+\sum\limits_{\substack{j\in Q_0 \\ j\neq i}}\alpha_je_j,
\end{equation}
where $\theta$ is the unique non-identity element of $\Gal(E/F)$.
\item its restriction to $Q_1$ equals $\psi$.
\end{itemize}
The automorphism $\Psi_i$ whose existence has just been claimed satisfies $\Psi_i^2=\myid_{\RA{A}}$.
\end{coro}

\begin{proof} Uniqueness is obvious, we prove existence. For $j\in Q_0$ let $\phi_j$ be the field automorphism of $F_j$ given by
\begin{equation}\label{eq:tuple-phi_j-for-one-orbpop}
\phi_j=\begin{cases}
\theta & \text{if $j=i$},\\
\myid & \text{if $j\neq i$}.
\end{cases}
\end{equation}
The ring automorphism $\Psi^{(0)}:R\rightarrow R$ induced by the tuple $(\phi_j)_{j\in Q_0}$ according to the statement of Lemma \ref{lemma:non-triv-automorphism} is clearly given by \eqref{eq:automorphism-of-R}. Furthermore, it is easy to check that
$g_{\psi(a)}=\phi_{h(a)}|_{F_{h(a),t(a)}}g_a\phi_{t(a)}^{-1}|_{F_{h(a),t(a)}}$ for every $a\in Q_1$. The existence of the desired automorphism $\Psi_i:\RA{A}\rightarrow\RA{A}$ thus follows from Lemma \ref{lemma:non-triv-automorphism}.
\end{proof}

\begin{ex} Let $(Q,\dtuple)$ be the weighted quiver
$$
\xymatrix{ & 2 \ar[dr]^{a} &\\
1 \ar[ur]^{b} & & 3 \ar[ll]^{{c}}} \ \ \ \xymatrix{ & 2 &\\
1  & & 1 }
$$
The least common multiple of $\dtuple$ is obviously 2, and the field extension $\C/\R$ clearly satisfies \eqref{eq:E/F-degree-d} and \eqref{eq:root-of-unity}. Since $F_{h(\alpha),t(\alpha)}=\R$ for every arrow $\alpha\in Q_1$, the weighted quiver $(Q,\dtuple)$ admits only one modulating function $g$, namely, the one given by $g_a=g_b=g_c=\myid_\R\in\Gal(\R/\R)$. The vertex $2\in Q_0$ and the identity function of $Q_1$ satisfy the hypotheses of Corollary \ref{coro:non-triv-automorphism-Psi_i}. The action on $R$ of the corresponding ring automorphism $\Psi_2:\RA{A}\rightarrow\RA{A}$ is given by
$$
\Psi_2(x_1e_1+x_2e_2+x_3e_3)=x_1e_1+\theta(x_2)e_2+x_3e_3 \ \ \ \text{for $x_1,x_3\in\R$ and $x_2\in\C$},
$$
where $\theta:\C\rightarrow\C$ is the usual complex conjugation. As for positive-length paths, we have, for example,
$$\xymatrix{
\Psi_2((abc)^n)=(abc)^n & \text{and} &
\Psi_2((avbc)^n)=(-avbc)^n}
$$
for $n\geq 1$, where $v\in\C$ is an imaginary number whose square equals $-1$.
\end{ex}

\begin{ex} Let $(Q,\dtuple)$ be the weighted quiver
$$
\xymatrix{1 \ar@<.75ex>[rr]^{a_1} \ar@<-.75ex>[rr]_{a_2} & & 2 \ar[dl]^{c}\\
 & 3 \ar[ul]^{b} & } \ \ \ \xymatrix{2 & & 2\\ & 1 & }
$$
The least common multiple of $\dtuple$ is obviously 2, and the field extension $\C/\R$ clearly satisfies \eqref{eq:E/F-degree-d} and \eqref{eq:root-of-unity}. The function $g:Q_1\rightarrow\Gal(\C/\R)\cup\Gal(\R/\R)$ given by $g_{a_1}=\myid_{\C}$, $g_{a_2}=\theta\neq\myid_{\C}$, $g_{b}=\myid_{\R}=g_{c}$ is a modulating function for $(Q,\dtuple)$ over $E/F$. The vertex $1\in Q_0$ and the function $\psi:Q_1\rightarrow Q_1$ that swaps $a_1$ and $a_2$ and acts as the identity on $b$ and $c$ are easily seen to satisfy the hypotheses of Corollary \ref{coro:non-triv-automorphism-Psi_i}. The action on $R$ of the corresponding ring automorphism $\Psi_1:\RA{A}\rightarrow\RA{A}$ is given by
$$
\Psi_1(x_1e_1+x_2e_2+x_3e_3)=\theta(x_1)e_1+x_2e_2+x_3e_3 \ \ \ \text{for $x_1\in\C$ and $x_2,x_3\in\R$}.
$$
As for positive-length paths, we have, for example,
\begin{eqnarray*}
\Psi_1(a_1bc)=a_2bc, &&
\Psi_1(a_1vbc)=-a_2vbc=va_2bc=\Psi_1(va_1bc), \\
\Psi_1(a_2bc)=a_1bc, &&
\Psi_1(a_2vbc)=-a_1vbc=-va_1bc=\Psi_1(-va_2bc),
\end{eqnarray*}
where $v\in\C$ is an imaginary number whose square equals $-1$.
\end{ex}

\section{The weighted quiver and the species of a triangulation}
\label{sec:species-of-a-triangulation}

Let $\surf$ be a surface with marked points and orbifold points, and let $\tau$ be an ideal triangulation of $\surf$. We define a weighted quiver $(\unredQtau,\dtuple(\tau))$ as follows. The vertices of the quiver $\unredQtau$ are the arcs in $\tau$, and to each $i\in\tau$ the weight tuple $\dtuple(\tau)$ attaches the weight
$$
\dtuple(\tau)_i=\begin{cases} 1 & \text{if no orbifold point lies on $i$};\\
2 & \text{otherwise}.
\end{cases}
$$
To define the number of arrows between any two given vertices of $\unredQtau$, we need an auxiliary function $\pi_\tau:\tau\rightarrow\tau$ which we now define. Given an arc $i\in\tau$, let $\pi_\tau(i)$ be the arc in $\tau$ defined as follows. If $i$ is not the folded side of a self-folded triangle, then $\pi_\tau(i)=i$, whereas if $i$ is the folded side of a self-folded triangle $\triangle$, and $k\in\tau$ is the loop that encloses $\triangle$, then $\pi_\tau(i)=k$.

Given arcs $i,j\in \tau$, let $\triangle(i,j)$ be the number of ideal triangles $\triangle$ of $\tau$ that satisfy the following conditions:
\begin{itemize}
\item $\triangle$ is not a self-folded triangle;
\item $\pi_\tau(i)$ and $\pi_\tau(j)$ are contained in $\triangle$;
\item inside $\triangle$, $\pi_\tau(j)$ directly precedes $\pi_\tau(i)$ according to the clockwise orientation of $\triangle$ which is inherited from the orientation of $\Sigma$.
\end{itemize}
The number of arrows in $\unredQtau$ that go from $j$ to $i$ is then defined to be
$$
\widehat{c}_{ij}=\begin{cases}
\triangle(i,j) & \text{if at least one of $i$ and $j$ contains no orbifold points};\\
2\triangle(i,j) & \text{otherwise}.
\end{cases}
$$

\begin{defi}\label{def:weighted-quiver-of-ideal-triangulation} For an ideal triangulation $\tau$ of $\surf$, the weighted quiver $(\unredQtau,\dtuple(\tau))$ just defined receives the name of \emph{unreduced weighted quiver of $\tau$}, while the weighted quiver $(\Qtau,\dtuple(\tau))$ obtained from it by deleting all 2-cycles is called the \emph{weighted quiver of $\tau$}.
\end{defi}

Note that because of the way $\unredQtau$ is defined, every arrow in $\unredQtau$ (resp. $\Qtau$) has a marked point canonically associated to it.

The reader can find examples of ideal triangulations $\tau$ and their associated weighted quivers $(\Qtau,\dtuple(\tau))$ in Examples \ref{ex:type-C_4}, \ref{ex:hex_1punct_2orbs_two_together} and \ref{ex:hex_2punct_1orb_2_improved}.


To define the weighted quivers of arbitrary tagged triangulations, one first passes through ideal triangulations by means of a procedure of ``notch deletion" (first introduced by Fomin-Shapiro-Thurston), whose description requires some preparation.

\begin{defi}
Let $\tau$ be a tagged triangulation of $\surf$.
\begin{enumerate}\item Following \cite[Definition 9.1]{FST}, we define the \emph{signature of $\tau$} to be the function $\delta_\tau:\punct\rightarrow\{-1,0,1\}$ given by
$$
\delta_\tau(p)=\begin{cases}
1 & \text{if all the tags at $p$ of tagged arcs in $\tau$ that are incident to $p$ are plain};\\
-1 & \text{if all the tags at $p$ of tagged arcs in $\tau$ that are incident to $p$ are notches};\\
0 & \text{otherwise}.
\end{cases}
$$
Note that if $\delta_\tau(p) = 0$, then there are precisely two tagged arcs in $\tau$ incident to $p$, the untagged versions
of these arcs coincide and they carry the same tag at the end different from $p$.
\item Following \cite[Definition 2.10]{Labardini-potsnoboundaryrevised}, we define the \emph{weak signature of $\tau$} to be the function $\epsilon_\tau:\punct\rightarrow\{-1,1\}$ given by
$$
\epsilon_\tau(p)=\begin{cases}
1 & \text{if $\delta_\tau(p)\in\{0,1\}$};\\
-1 & \text{otherwise}.
\end{cases}
$$
\end{enumerate}
\end{defi}

\begin{defi}\label{def:tau-circ}  Let $\tau$ be a tagged triangulation of $\surf$. We define a set of ordinary arcs by replacing each tagged arc
 $k$
in $\tau$ with an ordinary arc
 $k^\circ$
via the following procedure:
\begin{enumerate}
\item delete all tags at the punctures $p$ with non-zero signature;
\item for each puncture $p$ with $\delta_\tau(p)=0$, replace the tagged arc $\arc\in\tau$ which is notched at $p$ by a loop closely enclosing $\arc$.
\end{enumerate}
The resulting collection $\{k^\circ\suchthat k\in \tau\}$ of ordinary arcs will be denoted by $\tau^\circ$.
\end{defi}

\begin{defi}
\label{def:tagfunction(ideal-arc)} Let $\epsilon:\punct\rightarrow\{-1,1\}$ be any function. We define a function $\tagfunction_\epsilon:\arcsinsurf\rightarrow\taggedinsurf$ that represents ordinary arcs by tagged ones as follows.
\begin{enumerate}\item If $\arc$ is an ordinary arc that is not a loop enclosing a once-punctured monogon, set $\arc$ to be the underlying ordinary arc of the tagged arc $\tagfunction_\epsilon(\arc)$. An end of $\tagfunction_\epsilon(\arc)$ will be tagged notched if and only if the corresponding marked point is an element of $\punct$ where $\epsilon$ takes the value $-1$.
\item If $\arc$ is a loop, based at a marked point $q$, that encloses a once-punctured monogon, being $p$ the puncture inside this monogon, then the underlying ordinary arc of $\tagfunction(\arc)$ is the arc that connects $q$ with $p$ inside the monogon. The end at $q$ will be tagged notched if and only if $q\in\punct$ and $\epsilon(q)=-1$, and the end at $p$ will be tagged notched if and only if $\epsilon(p)=1$.
\end{enumerate}
\end{defi}

\begin{prop}\label{prop:tagfunction-and-circ} Let $\surf$ be a surface.
\begin{enumerate}
\item For every function $\epsilon:\punct\rightarrow\{-1,1\}$, the function $\tagfunction_\epsilon:\arcsinsurf\to\taggedinsurf$ is injective and preserves compatibility. Thus, if $\arc_1$ and $\arc_2$ are compatible ordinary arcs, then $\tagfunction_\epsilon(\arc_1)$ and $\tagfunction_\epsilon(\arc_2)$ are compatible tagged arcs. Consequently, if $T$ is an ideal triangulation of $\surf$, then $\tagfunction_\epsilon(T)=\{\tagfunction_\epsilon(\arc)\suchthat\arc\in T\}$ is a tagged triangulation of $\surf$. Moreover, if $T_1$ and $T_2$ are ideal triangulations such that $T_2=f_\arc(T_1)$ for an arc $\arc\in T_1$, then $\tagfunction_\epsilon(T_2)=f_{\tagfunction_\epsilon(\arc)}(\tagfunction_{\epsilon}(T_1))$.
\item If $\tau$ is a tagged triangulation of $\surf$, then $\tau^\circ$ is an ideal triangulation of $\surf$ and $\arc\mapsto\arc^\circ$ is a bijection between $\tau$ and $\tau^\circ$.
\item For every ideal triangulation $T$, we have $\tagfunction_\mathbf{1}(T)^\circ=T$, where $\mathbf{1}:\punct\to\{-1,1\}$ is the constant function taking the value 1.
\item For every tagged triangulation $\tau$ and every tagged arc $\arc\in\tau$ we have $\tagfunction_{\epsilon_\tau}(\arc^\circ)=\arc$ (see
Definition \ref{def:tau-circ} for the definition of $\arc^\circ$).
Consequently, $\tagfunction_{\epsilon_\tau}(\tau^\circ)=\tau$.
\item Let $\tau$ and $\sigma$ be tagged triangulations such that $\epsilon_\tau=\epsilon_\sigma$.
    If $\sigma=f_{\arc}(\tau)$ for a tagged arc $\arc\in\tau$, then $\sigma^\circ=f_{\arc^\circ}(\tau^\circ)$ (see Definition \ref{def:tau-circ}).
    Moreover, the diagram of functions
    \begin{equation}\label{eq:commutative-diagram-ordinary-flips}
    \xymatrix{\tau \ar[r]^{?^\circ} \ar[d] & \tau^\circ \ar[d] \ar[r]^{\tagfunction_{\epsilon_\tau}} & \tagfunction_{\epsilon_\tau}(\tau^\circ) \ar[d] \\
    \sigma \ar[r]_{?^\circ} & \sigma^\circ\ar[r]_{\tagfunction_{\epsilon_\sigma}} & \tagfunction_{\epsilon_\sigma}(\sigma^\circ)
    }
    \end{equation}
    commutes, where the three vertical arrows are bijections canonically induced by the operation of flip.
\item Let $\tau$ and $\sigma$ be tagged triangulations such that $\sigma=f_\arc(\tau)$ for some tagged arc $\arc\in\tau$. Then $\epsilon_\tau$ and $\epsilon_\sigma$ either are equal or differ at exactly one puncture $q$. In the latter case, if $\epsilon_\tau(q)=1=-\epsilon_\sigma(q)$, then $\arc^\circ$ is a folded side of $\tau^\circ$ incident to the puncture $q$, and $\tagfunction_{\epsilon_\tau\epsilon_\sigma}(\sigma^\circ)=f_{\arc^\circ}(\tagfunction_{\mathbf{1}}(\tau^\circ))$, where\\ $\epsilon_\tau\epsilon_\sigma:\punct\to\{-1,1\}$ is the function defined by $p\mapsto\epsilon_\tau(p)\epsilon_\sigma(p)$.
    Moreover, the diagram of functions
    \begin{equation}\label{eq:commutative-diagram-tagged-flips}
    \xymatrix{\tau \ar[r]^{?^\circ} \ar[d] & \tau^\circ \ar[r]^{\tagfunction_{\mathbf{1}}} & \tagfunction_{\mathbf{1}}(\tau^\circ) \ar[d]\\
\sigma=f_k(\tau) \ar[r]_{\phantom{xxx}?^\circ} & \sigma^\circ \ar[r]_{\tagfunction_{\epsilon_\tau\epsilon_\sigma}\phantom{xxxxxxxxxx}} & \tagfunction_{\epsilon_\tau\epsilon_\sigma}(\sigma^\circ)=f_{k^\circ}(\tagfunction_{\mathbf{1}}(\tau^\circ))
}
    \end{equation}
    commutes, where the two vertical arrows are bijections canonically induced by the operation of flip.
\end{enumerate}
\end{prop}

\begin{defi}\label{def:weighted-quivers-of-triangulations} Let $\tau$ be an arbitrary tagged triangulation of $\surf$. By Proposition \ref{prop:tagfunction-and-circ}, the assignments $\arc\mapsto\arc^\circ$ and $\arc\mapsto\tagfunction_{\epsilon_\tau}(\arc)$ constitute mutually inverse bijections between $\tau$ and $\tau^\circ$. We define $(\unredQtau,\dtuple(\tau))$ and $(\Qtau,\dtuple(\tau))$, respectively, as the result of replacing each $k\in\tau^\circ$ with $\tagfunction_{\epsilon_\tau}(k)$ as a vertex of $(\widehat{Q}(\tau^\circ),\dtuple(\tau^\circ))$ and $(Q(\tau^\circ),\dtuple(\tau^\circ))$.
\end{defi}

\begin{remark}\label{rem:tau=t_1(tau)}\begin{enumerate}\item When $\tau$ is an ideal triangulation, Proposition \ref{prop:tagfunction-and-circ} allows us to make no distinction between $\tau$ and the tagged triangulation $\tagfunction_{\mathbf{1}}(\tau)$. We will thus refer to $\tau$ and to $\tagfunction_{\mathbf{1}}(\tau)$ as being the same ideal triangulation and the same tagged triangulation with no further apology. Similarly, given an ordinary arc $i$ we will refer to $i$ and to the corresponding tagged arc $\tagfunction_{\mathbf{1}}(i)$ as being the same ordinary arc and the same tagged arc.
\item By Lemma \ref{lemma:B<->(Q,d)}, $(\Qtau,\dtuple(\tau))$ corresponds to a skew-symmetrizable matrix $B(\tau)$ whose columns and rows are indexed by the arcs in $\tau$. So, Definition \ref{def:weighted-quivers-of-triangulations} associates a skew-symmetrizable matrix to each tagged triangulation of a surface with marked points and orbifold points of order 2. This matrix was originally defined by Felikson-Shapiro-Tumarkin \cite{FeShTu-orbifolds}. They showed that whenever two tagged triangulations $\tau$ and $\sigma$ are related by the flip of a tagged arc $k\in\tau$, the associated matrices satisfy $B(\sigma)=\mu_k(B(\tau))$.
\item We must point out that in the presence of at least one orbifold point, Felikson-Shapiro-Tumarkin associate not one, but several skew-symmetrizable matrices to each given tagged triangulation. In this paper we will restrict our attention to only one of these matrices, namely, the matrix $B(\tau)$ corresponding to the weighted quiver $(Q(\tau),\dtuple(\tau))$ defined above. In the forthcoming \cite{Geuenich-Labardini-2} we will consider (the weighted quiver of) other skew-symmetrizable matrices.
\item Let us describe more precisely how Felikson-Shapiro-Tumarkin associate several skew-symmetrizable matrices to each tagged triangulation $\tau$. Start with a pair $(\tau,w)$ consisting of a tagged triangulation and an arbitrary function $w:\orb\rightarrow \{\frac{1}{2},2\}$. If $w(q)=2$ for all $q\in\orb$, set $d_k'=2$ for every pending arc $k\in\tau$ and $d_i'=1$ for every non-pending arc $i\in\tau$. If $w$ is not the constant function taking the value $2$ on all $\orb$, set $d'_k=2w(q_k)$ for each pending arc $k\in\tau$, where $q_k\in \orb$ is the unique orbifold point lying on $k$, and set $d_i'=2$ for each non-pending arc $i\in\tau$. With the tuple $(d'_i\suchthat i\in\tau)$ at hand, define $d_k=\frac{\operatorname{lcm}(d'_i\suchthat i\in\tau)}{d_k'}$ for every $k\in\tau$, and set $\dtuple_{w}=(d_k)_{k\in\tau}$. Furthermore, let $Q_w(\tau)$ be the quiver obtained from $Q(\tau)$ by replacing each pair of arrows $j\rightrightarrows i$ such that $d_i=1$, $d_j=4$, or $d_i=4$, $d_j=1$, with a single arrow. The matrix Felikson-Shapiro-Tumarkin associate to $(\tau,w)$ is the matrix corresponding to the weighted quiver $(Q_w(\tau),\dtuple_w)$ according to Lemma \ref{lemma:B<->(Q,d)}. In this paper we are considering the matrix that arises from choosing the function $w:\orb\rightarrow\{\frac{1}{2},2\}$ that takes the constant value $\frac{1}{2}$.
\end{enumerate}
\end{remark}

\begin{ex}\label{ex:type-C_4}
   Consider the following triangulation of the pentagon with one orbifold point:
  \begin{center}
  \begin{tikzpicture}[scale=0.4]
    \def \n {7}
    \def \s {360/\n}
    \def \r {3cm}

    \foreach \i in {2,...,\n} {
      \coordinate (\i) at (\i*360/\n+90:\r);
    }
    \coordinate (O) at ($(2)!0.3!(\n)$);
    \node[rotate=45] at (O) {$\times$};

    \foreach \i in {3,...,\n}
      \node at (\i) {$\bullet$};

    \draw[dotted] (0,0) circle (\r);

    \draw (O) to (\n);
    \foreach \i [evaluate=\i as \is using int(\i-2)] in {5,...,\n} {
      \draw (\is) to (\n);
    }
  \end{tikzpicture}
  \end{center}
Its associated weighted quiver is $(\bullet \to \bullet \to \bullet \to \bullet, \ \ 2 \ \ 1 \ \ 1 \ \ 1)$, whose corresponding matrix under Lemma \ref{lemma:B<->(Q,d)} is
$$
\left[\begin{array}{rrrr}
0 & -1 & 0 & 0 \\
2 & 0 & -1 & 0\\
0 & 1 & 0 & -1\\
0 & 0 & 1 & 0
\end{array}\right] .
$$
\end{ex}

To be able to define a species for $(\unredQtau,\dtuple(\tau))$ we need a field extension and a modulating function over this extension.
Since $\dtuple(\tau)$ consists only of 1s and 2s, its least common multiple $d$ is either 1 or 2. Let $E/F$ be a field extension satisfying \eqref{eq:root-of-unity} and \eqref{eq:E/F-degree-d}. Then either $E=F$ or $[E:F]=2$. If $[E:F]=2$, we denote by $\theta$ the unique element of $\Gal(E/F)$ which is different from the identity, so that $\Gal(E/F)=\{\myid_E,\theta\}$.

Let $g=g(\tau):\unredQtau\rightarrow\Gal(E/F)$ be a modulating function satisfying the following conditions:
\begin{itemize}
\item For every arrow $a$ of $\unredQtau$, if either of the numbers $d(\tau)_{h(a)}$ and $d(\tau)_{t(a)}$ is equal to $1$, then $G_{h(a),t(a)}=\Gal(F/F)$, so we take $g_a=\myid_F\in\Gal(F/F)$;
\item for every two arcs $i$ and $j$ of $\tau$ such that $d(\tau)_i=2=d(\tau)_j$, if $i$ and $j$ are connected by at least one arrow of $\unredQtau$, then they are connected by exactly two arrows of $\unredQtau$, say $a$ and $b$. We take $g_a=\myid_E$ and $g_b=\theta$.
\end{itemize}
We fix one such modulating function once and for all, and call it the \emph{modulating function of $\tau$ over $E/F$}.

\begin{defi}\label{def:unred-species-of-triangulation} We define the \emph{unreduced species of $\tau$ over $E/F$} to be the species of $(\unredQtau,\dtuple(\tau),g(\tau))$, and denote it $\widehat{A}(\tau)$ (see Definitions \ref{def:species-Ag}, \ref{def:weighted-quiver-of-ideal-triangulation} and \ref{def:weighted-quivers-of-triangulations}).
\end{defi}

We will define the reduced species of $\tau$ over $E/F$ in the next section. For the moment we only observe that $\unredQtau$ may have 2-cycles, but it never has 2-cycles incident to pending arcs.

\section{The potential associated to a triangulation}
\label{sec:potential-of-a-triangulation}

Let $\surf$ be a surface with marked points and orbifold points, and let $\tau$ be an ideal triangulation of $\surf$. Let $E/F$ and $g$ be as in the two paragraphs that follow Definition \ref{def:weighted-quivers-of-triangulations}. Recall that if $[E:F]=2$, then $E/F$ has an eigenbasis of the form $\{1,v\}$. From this moment to the end of the paper we assume that
\begin{equation}\label{eq:v^2=-1}
\text{if $[E:F]=2$, then the element $v$ of the chosen eigenbasis satisfies $v^2=-1$.}
\end{equation}
For example, if $E/F$ is the extension $\C/\R$, we can take $v$ to be one of the two imaginary numbers whose square is $-1$.

It is our intention to define a potential $\unredStau$, which will be an element of $\RA{\widehat{A}(\tau)}$.  The following definitions are aimed at locating some cycles on the unreduced species $\widehat{A}(\tau)$. Throughout Definitions \ref{def:cycles-from-int-triangles}, \ref{def:cycles-from-orb-triangles}, \ref{def:cycles-from-self-folded-triangles} and \ref{def:cycles-from-punctures}, $\tau$ will be an ideal triangulation of $\surf$.

\begin{defi}[Cycles from interior triangles]\label{def:cycles-from-int-triangles} Every interior ideal triangle $\triangle$ of $\tau$ which is neither self-folded nor orbifolded gives rise to an oriented 3-cycle $\alpha^\triangle\beta^\triangle\gamma^\triangle$ of $\unredQtau$. We set $\widehat{S}^\triangle(\tau)=\alpha^\triangle\beta^\triangle\gamma^\triangle\in\RA{\widehat{A}(\tau)}$ up to cyclical equivalence. See Figure \ref{Fig:normal_triangle}.
        \begin{figure}[!ht]
                \centering
                \includegraphics[scale=.30]{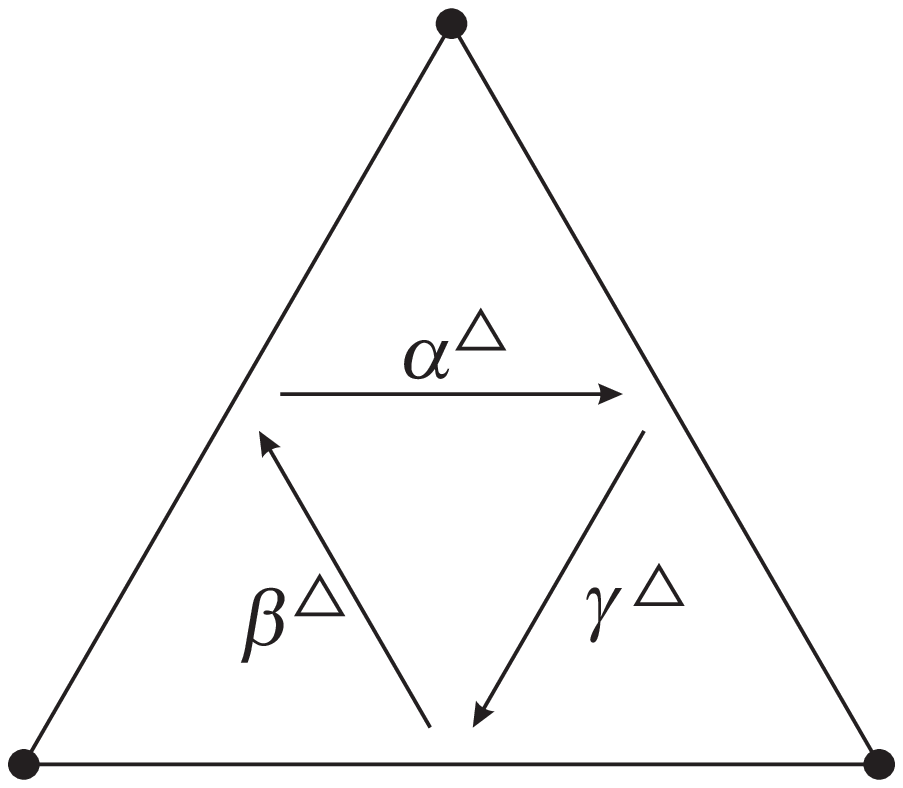}
                \caption{}\label{Fig:normal_triangle}
        \end{figure}

\end{defi}

\begin{defi}[Cycles from orbifolded triangles]\label{def:cycles-from-orb-triangles} Suppose $\triangle$ is an interior orbifolded triangle of $\tau$.
 \begin{enumerate}
 \item If $\triangle$ contains exactly one orbifold point, then, with the notation of Figure \ref{Fig:orb_triangles_1and2pts} (left), we define $\widehat{S}^{\triangle}(\tau,\mathbf{x})=\alpha^\triangle\beta^\triangle\gamma^\triangle-\alpha^\triangle\beta^\triangle v\gamma^\triangle$.
        \begin{figure}[!ht]
                \centering
                \includegraphics[scale=.45]{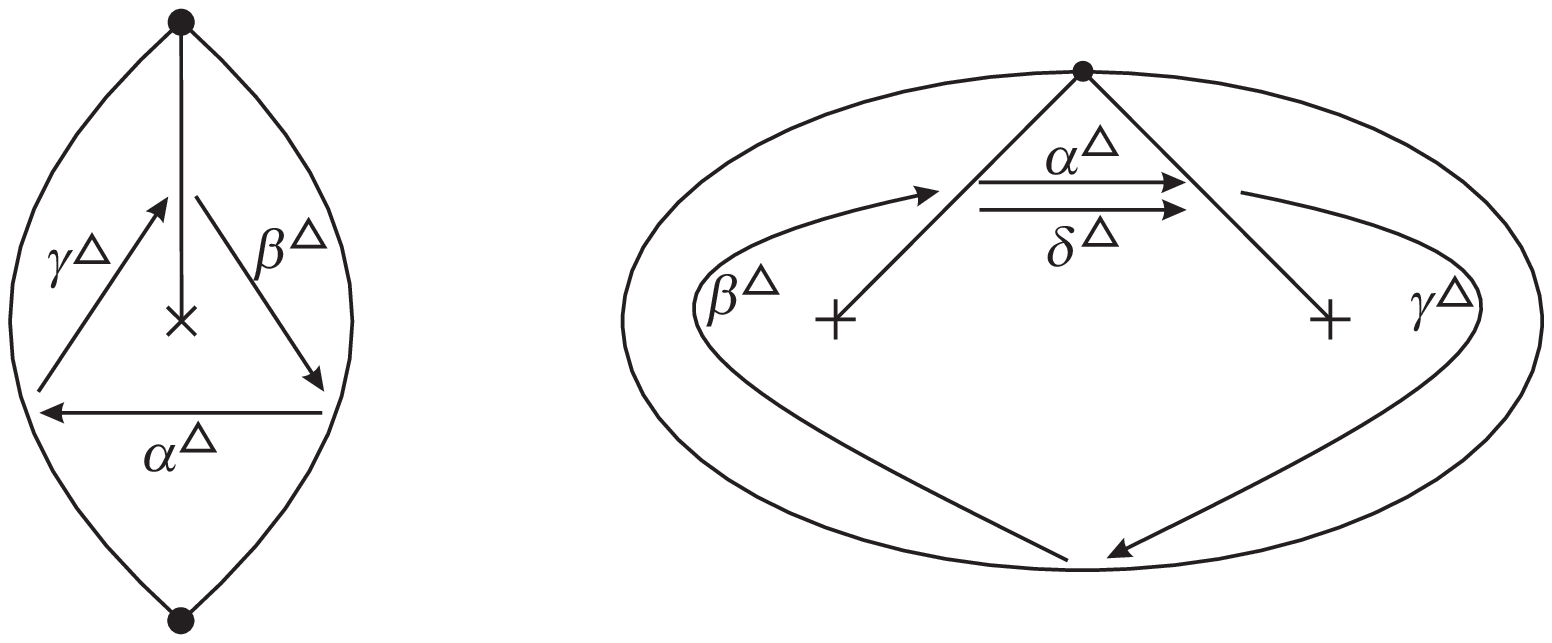}
                \caption{}\label{Fig:orb_triangles_1and2pts}
        \end{figure}
 \item If $\triangle$ contains exactly two orbifold points, then, with the notation of Figure \ref{Fig:orb_triangles_1and2pts} (right) and the convention that $g(\tau)_{\alpha^\triangle}=\myid_E\in\Gal(E/F)$ and $g(\tau)_{\delta^\triangle}=\theta\in\Gal(E/F)$, where $\Gal(E/F)=\{\myid_E,\theta\}$, define $\widehat{S}^{\triangle}(\tau,\mathbf{x})=2v^{-1}\alpha^\triangle\beta^\triangle\gamma^\triangle+2\delta^\triangle\beta^\triangle\gamma^\triangle$,.
 \end{enumerate}
\end{defi}

\begin{defi}[Cycles from triangles adjacent to self-folded triangles]
\label{def:cycles-from-self-folded-triangles}
Suppose $\triangle$ is an interior non-self-folded triangle of $\tau$.
\begin{enumerate}
\item If $\triangle$ is not an orbifolded triangle, and does not share sides with exactly two self-folded triangles, we set $\widehat{U}^\triangle(\tau,\mathbf{x})=0$.
\item If $\triangle$ is not an orbifolded triangle, and is adjacent to two self-folded triangles like in the configuration of Figure \ref{Fig:adj_1sf_triang_merged} (left),
        \begin{figure}[!ht]
                \centering
                \includegraphics[scale=.4]{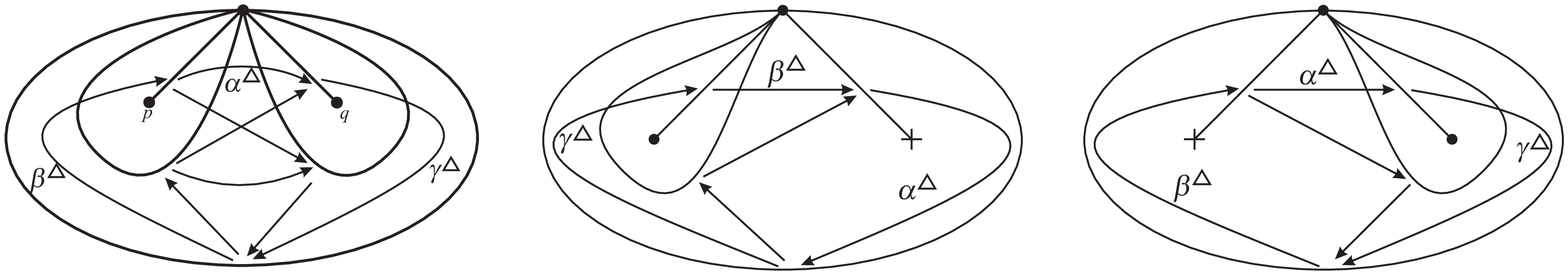}
                \caption{}\label{Fig:adj_1sf_triang_merged}
        \end{figure}
we set
$\widehat{U}^\triangle(\tau,\mathbf{x})=x_p^{-1}x_q^{-1}\alpha^\triangle\beta^\triangle\gamma^\triangle$ up to cyclical equivalence, where $p$ and $q$ are the punctures enclosed in the self-folded triangles adjacent to $\triangle$.
\item If $\triangle$ is an orbifolded triangle which is adjacent to a self-folded triangle like in the configurations shown in Figure \ref{Fig:adj_1sf_triang_merged} (center and right), we set $\widehat{U}^\triangle(\tau,\mathbf{x})=x_p^{-1}\alpha^\triangle v\beta^\triangle\gamma^\triangle$
\end{enumerate}
\end{defi}

\begin{defi}[Cycles from punctures]
\label{def:cycles-from-punctures}
Let $p\in\punct$ be a puncture.
\begin{enumerate}
\item Suppose $p$ is adjacent to exactly one arc $i$ of $\tau$; then $i$ is the folded side of a self-folded triangle $\triangle$ of $\tau$. Let $\triangle'$ be the non-self-folded ideal triangle of $\tau$ which is adjacent to $\triangle$. If $\triangle'$ is an interior triangle, then around $i$ we have one of the three configurations shown in Figure \ref{Fig:sftriangle_edit}, and
        \begin{figure}[!ht]
                \centering
                \includegraphics[scale=.5]{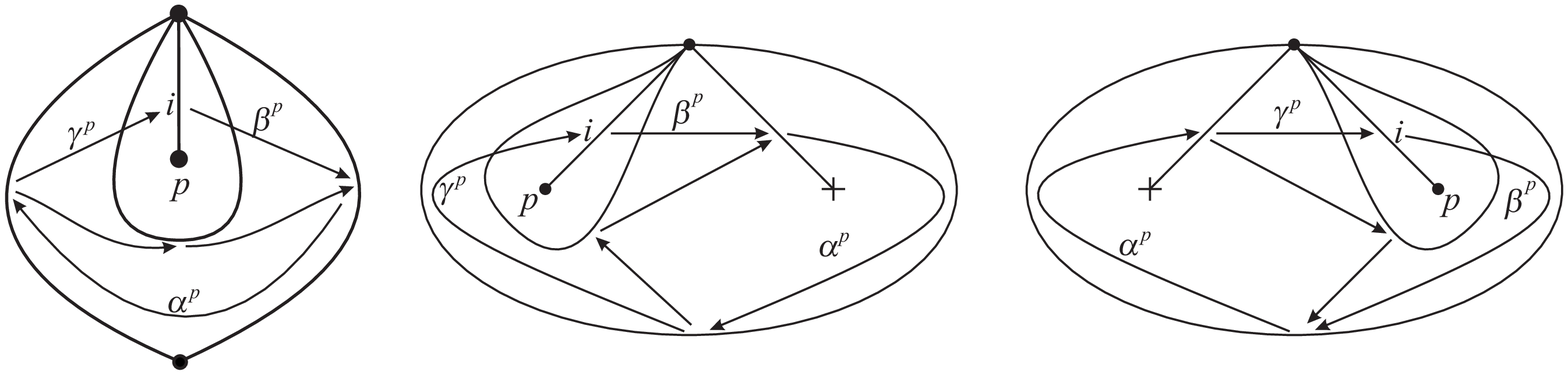}
                \caption{}\label{Fig:sftriangle_edit}
        \end{figure}
we set
$\widehat{S}^p(\tau,\mathbf{x})=-x_p^{-1}\alpha^p\beta^p\gamma^p$ up to cyclical equivalence.
Otherwise, if $\triangle'$ is not an interior triangle, we set $\widehat{S}^{p}(\tau,\mathbf{x})=0$.
\item
Suppose $p$ is adjacent to more than one arc. Let $\tau^p$ be the set of arcs incident to $p$ that are not loops enclosing self-folded triangles.
The orientation of $\Sigma$ induces a cyclic ordering~$(i_1, \ldots, i_\ell)$ on $\tau^p$, i.e.\ moving in a small closed curve counter-clockwise around the puncture $p$ the arc $i_{s+1}$ is crossed directly after crossing $i_s$ (we may have $i_s=i_{s'}$ for different indices $s$ and $s'$).
Set $\widehat{S}^p(\tau,\mathbf{x})=x_p\Lambda^p_\ell\cdots\Lambda^p_{1} \in \RA{A}_{\rm cyc}$ where $\Lambda^p_s$ is the sum of all arrows that go from $i_{s-1}$ to $i_{s}$ and have $p$ as associated puncture (see the comment right after Definition \ref{def:weighted-quiver-of-ideal-triangulation}).
\end{enumerate}
\end{defi}

\begin{remark} If no orbifolded triangle with exactly two orbifold points contains the puncture $p$, then $\widehat{S}^p(\tau,\mathbf{x})$ is a non-zero scalar multiple of a concatenation of actual arrows of $\widehat{Q}(\tau)$.
\end{remark}

Now that we have located some ``obvious" cycles on (the species of) the weighted quiver of an ideal triangulation, we are ready to define the potential of an arbitrary tagged triangulation.
Recall that given a tagged triangulation $\tau$, the function $\tagfunction_{\epsilon_\tau}:\tau^\circ\rightarrow\tau$ is a specific bijection. By its very definition, $\Qtau$ is the weighted quiver obtained from $Q(\tau^\circ)$ by replacing each vertex $i\in\tau^\circ$ with $\tagfunction_{\epsilon_\tau}(i)\in\tagfunction_{\epsilon_\tau}(\tau^\circ)=\tau$. During this replacement, the arrow set does not suffer any significant change: if $a$ is an arrow in $Q(\tau^\circ)$ going from $j\in\tau^\circ$ to $i\in\tau^\circ$, then $a$ itself is an arrow in $\Qtau$ going from $\tagfunction_{\epsilon_\tau}(j)$ to $\tagfunction_{\epsilon_\tau}(i)$. The weight tuple $\dtuple(\tau)$ and the modulating function $g(\tau)$ are defined in the obvious way, namely, by setting $\dtuple(\tau)_i=\dtuple(\tau^\circ)_{i^\circ}$ and $g(\tau)=g(\tau^\circ)$. Therefore, the function $\tagfunction_{\epsilon_\tau}:\tau^\circ\rightarrow\tau$, together with the identity function of the arrow set, is a specific weighted-quiver isomorphism $(Q(\tau^\circ),\dtuple(\tau^\circ)\rightarrow(\Qtau,\dtuple(\tau))$ that preserves the modulating functions and hence yields an $R$-$R$-bimodule isomorphism $A(\tau^\circ)\rightarrow\Atau$, which we denote by $\tagfunction_{\epsilon_\tau}$ also. Being an $R$-$R$-bimodule isomorphism, $\tagfunction_{\epsilon_\tau}$ induces an $F$-algebra isomorphism $\RA{A(\tau^\circ)}\rightarrow\RA{\Atau}$, which, in a slight abuse of notation, we denote by $\tagfunction_{\epsilon_\tau}$ as well.

\begin{defi}\label{def:SP-of-tagged-triangulation} Let $\tau$ be a tagged triangulation of $\surf$, and $\mathbf{x}=(x_p)_{p\in\punct}$ be a choice of non-zero scalars (one scalar $x_p\in F$ per puncture $p$). Let $\widehat{A}(\tau)$ be the unreduced species of $\tau$ (see Definition \ref{def:unred-species-of-triangulation}).
\begin{enumerate}
\item The \emph{unreduced potential $\unredStau\in \RA{\widehat{A}(\tau)}$ associated to $\tau$ with respect to the choice $\mathbf{x}=(x_p)_{p\in\punct}$} is:
\begin{equation}\unredStau=
\tagfunction_{\epsilon_\tau}\left(\sum_\triangle\left(\widehat{S}^\triangle(\tau^\circ)+
\widehat{U}^\triangle(\tau^\circ,\mathbf{x})\right)+
\sum_{p\in\punct}\left(\epsilon_\tau(p)\widehat{S}^p(\tau^\circ,\mathbf{x})\right)\right)
\end{equation}
where the first sum runs over all interior non-self-folded triangles of $\tau^\circ$, and $\epsilon_\tau:\punct\rightarrow\{-1,1\}$ is the weak signature of $\tau$.
\item We define $\AStau$ to be the reduced part of $\unredAStau$.
\end{enumerate}
\end{defi}

\begin{remark}
\begin{enumerate}\item If $\orb=\varnothing$, that is, if $\tau$ is a tagged triangulation of a surface with marked points but without orbifold points, the QP $\AStau$ defined above coincides with the one defined in \cite{Labardini-potsnoboundaryrevised}. If $\orb=\varnothing$ and the boundary of $\Sigma$ is not empty, $\AStau$ is right-equivalent to the QP defined in~\cite{CI-LF}.
\item If $\tau$ is an ideal triangulation, the only situation where one needs to apply reduction to $(\widehat{A}(\tau),\unredStau)$ in order to obtain $\Stau$ is when there is some puncture incident to exactly two arcs of $\tau$. It is not hard to see that the SP $\AStau$ is always 2-acyclic, and that the underlying weighted quiver of $\AStau$ is $(\Qtau,\dtuple(\tau))$.
\end{enumerate}
\end{remark}


\begin{ex}\label{ex:hex_1punct_2orbs_two_together}
In Figure \ref{Fig:hex_1punct_2orbs_two_together} we can see two triangulations of a hexagon with 1 puncture and 2 orbifold points.
        \begin{figure}[!ht]
                \centering
                \includegraphics[scale=.65]{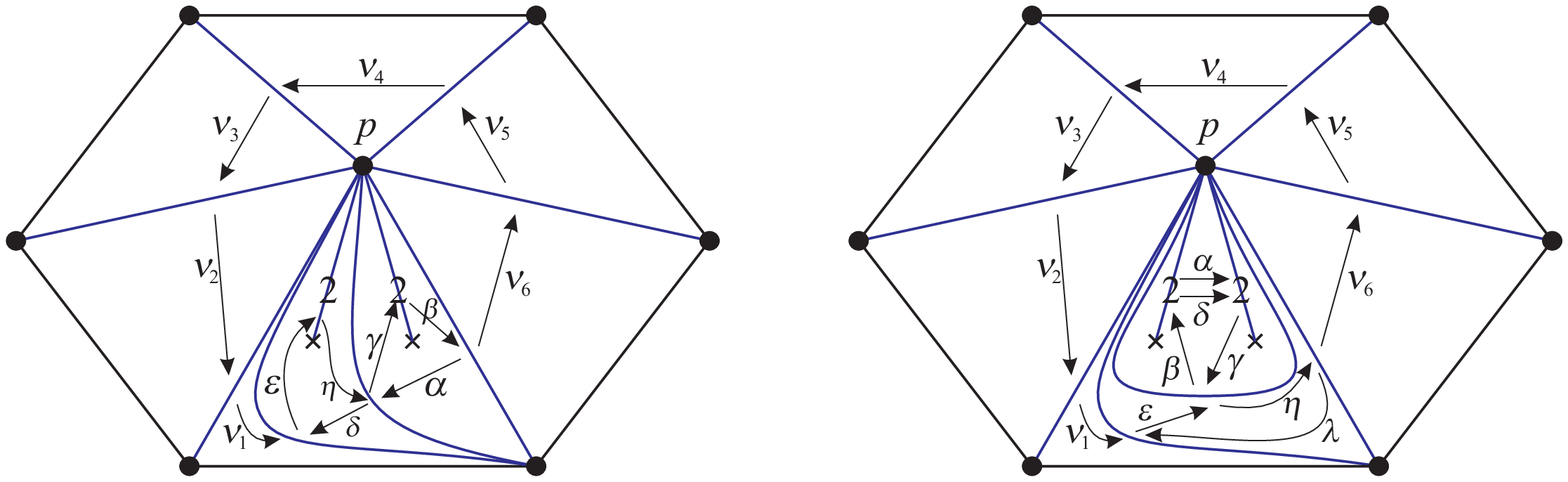}
                \caption{}\label{Fig:hex_1punct_2orbs_two_together}
        \end{figure}
Let $\tau$ be the triangulation depicted on the left and $\sigma$ be the triangulation depicted on the right.
Taking $\C/\R$ as our extension $E/F$, we have $g(\tau)_{a}=\myid_{\R}\in\Gal(\R/\R)$ for every $a\in\Qtau_1$. As for the potential, we have
$$
\Stau=\alpha\beta\gamma-\alpha\beta v\gamma+\delta\eta\varepsilon-\delta\eta v\varepsilon+x_p\beta\gamma\eta\varepsilon\nu_1\nu_2\nu_3\nu_4\nu_5\nu_6.
$$

Similarly, we have $g(\sigma)_{\alpha}=\myid_{\C}\in\Gal(\C/\R)$ and $g(\sigma)_\delta=\theta\in\Gal(\C/\R)=\{\myid_{\C},\theta\}$, while $g(\sigma)_{a}=\myid_{\R}\in\Gal(\R/\R)$ for every $a\in\Qsigma_1\setminus\{\alpha,\delta\}$. As for the potential, we have
\begin{eqnarray}\nonumber
\Ssigma & = & \eta\varepsilon\lambda-2v\alpha\beta\gamma+2\delta\beta\gamma+x_p(\alpha+\delta)\beta\varepsilon\nu_1\nu_2\nu_3\nu_4\nu_5\nu_6\eta\gamma\\
\nonumber
& = & \eta\varepsilon\lambda-2v\alpha\beta\gamma+2\delta\beta\gamma+x_p\alpha\beta\varepsilon\nu_1\nu_2\nu_3\nu_4\nu_5\nu_6\eta\gamma
+x_p\delta\beta\varepsilon\nu_1\nu_2\nu_3\nu_4\nu_5\nu_6\eta\gamma.
\end{eqnarray}
\end{ex}

\begin{ex}\label{ex:hex_2punct_1orb_2_improved}  In Figure \ref{Fig:hex_2punct_1orb_2_improved} we can see a triangulation $\tau$ of a hexagon with 2 punctures and 1 orbifold point.
        \begin{figure}[!ht]
                \centering
                \includegraphics[scale=.65]{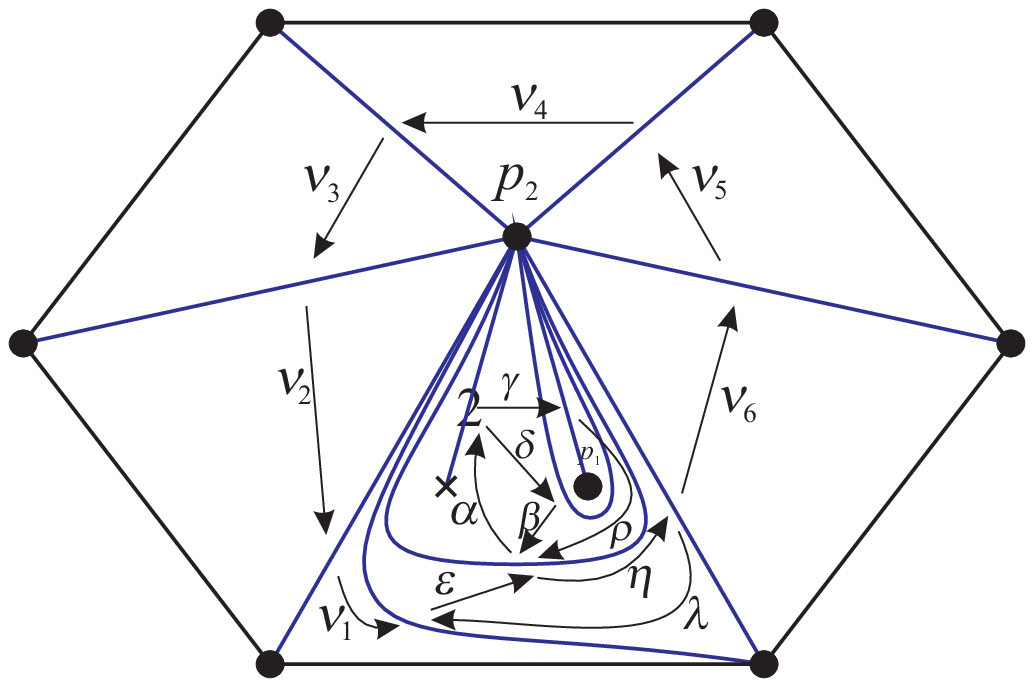}
                \caption{}\label{Fig:hex_2punct_1orb_2_improved}
        \end{figure}
The triangulation has an orbifolded triangle adjacent to a self-folded triangle. Taking $\C/\R$ as our extension $E/F$, we have $g(\tau)_{a}=\myid_{\R}\in\Gal(\R/\R)$ for every $a\in\Qtau_1$. As for the potential, we have
\begin{eqnarray}\nonumber
\Stau & = & \eta\varepsilon\lambda+\alpha\beta\delta-v\alpha\beta\delta+x_{p_1}^{-1}v\alpha\rho\gamma-x_{p_1}^{-1}\alpha\rho\gamma
+x_{p_2}\eta\rho\gamma\alpha\varepsilon\nu_1\nu_2\nu_3\nu_4\nu_5\nu_6.
\end{eqnarray}
\end{ex}

\begin{remark}\label{rem:ultimate-goal-and-steps}
The ultimate goal of this paper is to prove that whenever two tagged triangulations $\tau$ and $\sigma$ are related by a flip, then some function $\zeta:\punct\rightarrow\{-1,1\}$ exists that takes the value $-1$ at most once and has the property that the SPs $\AStau$ and $(A(\sigma),S(\sigma,\zeta\mathbf{x}))$ are related by the SP-mutation given in Definition \ref{def:SP-mutation}, where $\zeta\mathbf{x}=(\zeta(p)x_p)_{p\in\punct}$. Our strategy of proof will consist of 2 steps:
\begin{enumerate}
\item Showing that the assertion holds for surfaces with empty boundary (with an arbitrary number of punctures and an arbitrary number of orbifold points), see Theorem \ref{thm:tagged-flips<->SP-mutation};
\item deducing that the assertion holds for all surfaces with marked points and orbifold points, regardless of the emptiness or non-emptiness of their boundary, by using restriction of SPs and gluing/cutting of discs along boundary components, see Theorem \ref{thm:tagged-flips<->SP-mutation--non-empty-boundary}.
\end{enumerate}
Among these two steps, the second one will be more or less straightforward (see Section \ref{sec:empty=>general}). The first step will be a lot more involved, and it will be achieved in four steps:
\begin{itemize}
\item[(i)] Showing that, in the empty-boundary case, the desired result is true when both $\tau$ and $\sigma$ are ideal triangulations related by the flip of an arc which is not self-folded nor pending; the precise statement can be found in Theorem \ref{thm:ideal-non-pending-flips<->SP-mutation};
\item[(ii)] proving the desired assertion in the empty-boundary case when $\tau$ is ideal and $\sigma$ is the tagged triangulation obtained by flipping a self-folded arc of $\tau$, the reader can find the precise statement in Theorem \ref{thm:leaving-positive-stratum};
\item[(iii)] showing the desired statement to be true in the empty-boundary case when $\tau$ is ideal and $\sigma$ is the ideal triangulation obtained by flipping a pending arc of $\tau$; the precise statement can be found in Theorem \ref{thm:flip-pending-arc<->SP-mut};
\item[(iv)] proving that the the desired assertion is always true for tagged triangulations of surfaces with empty boundary by using weak signatures and notch deletions to reduce the proof to the cases covered by the first three steps; see Theorem \ref{thm:tagged-flips<->SP-mutation}.
\end{itemize}
Due to technical reasons, to accomplish steps (ii) and (iii) we will need to introduce what we call \emph{popped potentials} (see Section \ref{sec:pops}).
\end{remark}

The next result is step (i) of the strategy just described.

\begin{thm}\label{thm:ideal-non-pending-flips<->SP-mutation} Let $\surf$ be a surface with empty boundary, and let $\tau$ and $\sigma$ be ideal triangulations of $\surf$. If $\sigma$ can be obtained from $\tau$ by the flip of a non-pending arc $k\in\tau$, then the SP $\mu_k\ASsigma$ is right-equivalent to the SP $\AStau$.
\end{thm}

Note that in Theorem \ref{thm:ideal-non-pending-flips<->SP-mutation}, the fact that both $\tau$ and $\sigma$ are ideal triangulations implies that the arc $k$ cannot be a folded side of any self-folded triangle of $\tau$.

Our proof of Theorem \ref{thm:ideal-non-pending-flips<->SP-mutation} is rather long since it is done through a case-by-case verification of the statement for all the possible combinatorial configurations that $\tau$ and $\sigma$ can present around $k$. For this reason we have deferred the alluded proof to Section \ref{sec:technical-proofs}.

\section{Pops}\label{sec:pops}

\subsection{Pop at a self-folded triangle}\label{subsec:self-folded-pop}

In this subsection we define the \emph{popped potential} associated to an ideal triangulation with respect to a fixed self-folded triangle contained in it.
Our main aim in this subsection is to prove that these potentials are right-equivalent to the potentials defined in Section \ref{sec:potential-of-a-triangulation} (the reason for wanting this result will become clear in the proof of Theorem \ref{thm:flip-pending-arc<->SP-mut}). To achieve this we will first show that if we do not flip/mutate at pending arcs, nor at any of the two arcs in the fixed self-folded triangle under consideration, the resulting \emph{popped SPs} present the same compatibility between flips and SP-mutations stated in Theorem \ref{thm:ideal-non-pending-flips<->SP-mutation} for the SPs $\AStau$; then we will show that, for any fixed self-folded triangle and any set consisting of exactly $|\orb|$ pairwise compatible pending arcs, there always exists an ideal triangulation that contains it and for which the associated SP $\AStau$ is right-equivalent to the corresponding popped SP. These two facts will imply that the popped potentials to be defined shortly are right-equivalent to the potentials defined in Section \ref{sec:potential-of-a-triangulation}, since any two ideal triangulations sharing a fixed self-folded triangle and a fixed set of exactly $|\orb|$ pairwise compatible pending arcs are related by a sequence of flips that involves only ideal triangulations that also share the fixed self-folded triangle and the fixed set of $|\orb|$ pairwise compatible pending arcs.

The constructions and results presented in this subsection are, in essence, those contained in \cite[Sections 5 and 6]{Labardini-potsnoboundaryrevised}. However, we will need the uniqueness-up-to-right-equivalence of reduced and trivial parts of SPs in an essential way, and in our current species setup this uniqueness result is not covered by the statement of \cite[Theorem 4.6]{DWZ1}. Moreover, as said in the previous paragraph, we will work not only with an arbitrarily given self-folded triangle, but also with an arbitrarily given collection of $|\orb|$ pairwise compatible pending arcs, neither of which will be flipped throughout this subsection. Hence we prefer not to omit any details, thus consciously running the risk of seeming to literally repeat \cite[Sections 5 and 6]{Labardini-potsnoboundaryrevised}.

Suppose $i$ is the folded side of a self-folded triangle of an ideal triangulation $\tau$. Let $j\in\tau$ be the loop that cuts out a once-punctured monogon and encloses $i$, and $q\in\punct$ be the puncture that lies inside the monogon cut out by $j$. Let $\pi_{i,j}^\tau:\tau\to\tau$ be the bijection that fixes $\tau\setminus\{i,j\}$ pointwise and interchanges $i$ and $j$. Then $\pi_{i,j}^\tau$ extends uniquely to a quiver automorphism of $\Qtau$ that acts as the identity on all the arrows of $\Qtau$ that are not incident to $i$ or $j$. This quiver automorphism yields an $F$-algebra automorphism of $\RA{\Atau}$. In a slight notational abuse we shall denote both this quiver and algebra automorphisms by $\pi_{i,j}^\tau$ as well.


\begin{defi}\label{def:popped-potential} Let $\mathbf{x}=(x_p)_{p\in\punct}$ be a choice of non-zero scalars. With the notations from the preceding paragraph, let $S(\tau,\mathbf{y})$ be the potential associated to $\tau$ with respect to the choice of scalars $\mathbf{y}=(y_p)_{p\in\punct}$ defined by $y_p=(-1)^{\delta_{p,q}}x_p$ for all $p\in\punct$, where $\delta_{p,q}$ is the \emph{Kronecker delta}. The \emph{popped potential} of the quadruple $(\tau,\mathbf{x},i,j)$ is the potential $\Wtau=\pi_{i,j}^\tau(S(\tau,\mathbf{y}))\in \RA{\Atau}$.
\end{defi}

\begin{remark}\label{rem:on-popped-pots}\begin{enumerate}\item
The $F$-algebra automorphism $\pi_{i,j}^\tau$ of $\RA{\Atau}$ has only been used to define $\Wtau$, and it is actually possible to define $\Wtau$ without any mention of $\pi_{i,j}^\tau$ (we have not done so in order to avoid an unnecessarily cumbersome definition), but $\pi_{i,j}^\tau$ is not a right-equivalence because it does not act as the identity on $R$ (see Definition \ref{def:cyclic-stuff}; 
 compare to \cite[Definition 4.2]{DWZ1}).
\item As mentioned in Remark \ref{rem:ultimate-goal-and-steps}, the ultimate goal of this paper is to prove that whenever two tagged triangulations $\tau$ and $\sigma$ are related by a flip, the SPs $\AStau$ and $(A(\sigma),S(\sigma,\zeta\mathbf{x}))$ are related by the corresponding SP-mutation for some function $\zeta:\punct\rightarrow\{-1,1\}$ that takes the value $-1$ on at most one puncture (at the moment we know this to be true only when both $\tau$ and $\sigma$ are ideal triangulations of a surface with empty boundary, see Theorem \ref{thm:ideal-non-pending-flips<->SP-mutation} and the sentence right after it). The popped potentials $\Wtau$ will be used as a tool to achieve our goal: showing that the flip of a folded side of an ideal triangulation $\tau$ results in an SP related to $\AStau$ by the corresponding SP-mutation may be way less straightforward or obvious than we would like (the problem one faces is precisely the one described in \cite[Section 4]{Labardini-potsnoboundaryrevised}). So, what we will do is show that whenever $i$ and $j$ are arcs forming a self-folded triangle of an ideal triangulation $\tau$, with $i$ as folded side and $j$ as enclosing loop, the SPs $\AStau$ and $\AWtau$ are right-equivalent (see Theorem \ref{thm:popping-is-right-equiv}). This means that proving that $\mu_i\AWtau$ is right-equivalent to $\ASsigma$ will suffice in order to show that $\mu_i\AStau$ is right-equivalent to $\ASsigma$ (see Theorem \ref{thm:leaving-positive-stratum} and its proof). To say it in few words: the popped potentials $\Wtau$ associated to self-folded triangles will help us to prove that the flip of the self-folded arc $i$ produces a tagged triangulation $\sigma$ whose SP is related to the SP of $\tau$ by the corresponding SP-mutation.
\end{enumerate}
\end{remark}

\begin{prop}\label{prop:pop-commutes-with-mutation} Let $\tau$ be an ideal triangulation, and $\mathbf{x}=(x_p)_{p\in\punct}$ and $i,j\in\tau$ be as above. If $k\in\tau\setminus\{i,j\}$ is a non-pending arc such that $\sigma=f_k(\tau)$ happens to be an ideal triangulation, then $\mu_k(\Atau,\Wtau)$ is right-equivalent to $(\Asigma,\Wsigma)$.
\end{prop}

\begin{proof} This is analogous to the proof of \cite[Lemma 5.5]{Labardini-potsnoboundaryrevised}. We reproduce the latter (with modifications) here because we need to appeal to a result not covered by the statement of \cite[Theorem 4.6]{DWZ1}, namely, the one which asserts the uniqueness of reduced parts in the non-simply laced case (cf.\ Theorem \ref{thm:splitting-theorem} above). Note that the hypotheses that both $\tau$ and $\sigma=f_k(\tau)$ be ideal triangulations implies that $k$ cannot be a folded side of $\tau$ nor of $\sigma$.

Let $\mathbf{y}=((-1)^{\delta_{p,q}}x_p)_{p\in\punct}$ be as in Definition \ref{def:popped-potential}.
The bijection $\pi=\pi_{i,j}^\tau:\tau\to\tau$ induces a weighted-quiver automorphism $\psi$ of $\widetilde{\mu}_k(\Qtau)$. Following a slight notational abuse, we denote also by $\psi$ the induced $F$-algebra automorphism of $\RA{\widetilde{\mu}_k(\Atau)}$ (which is not an $R$-algebra automorphism). It is straightforward to see that $\psi([S(\tau,\mathbf{y})])=[\pi(S(\tau,\mathbf{y}))]$ and $\psi(\triangle_k(\Atau))=\triangle_k(\Atau)$ (see Part (3) of Definition \ref{def:SP-premutation}), and hence
\begin{eqnarray}\nonumber
\psi(\widetilde{\mu}_{k}(S(\tau,\mathbf{y}))) & = & \psi\left([S(\tau,\mathbf{y})]+\triangle_k(\Atau)\right) \ \ \ \ \ \text{by Part (3) of Definition \ref{def:SP-premutation}}\\
\nonumber & = & [\pi(S(\tau,\mathbf{y}))]+\triangle_k(\Atau)\\
\nonumber & = & \widetilde{\mu}_k(\pi(S(\tau,\mathbf{y}))) \ \ \ \ \ \ \ \ \ \ \ \ \ \ \ \ \ \ \ \text{by Part (3) of Definition \ref{def:SP-premutation}}\\
\nonumber & = &\widetilde{\mu}_k(\Wtau)  \ \ \ \ \ \ \ \ \ \ \ \ \ \ \ \ \ \ \ \  \text{by Definition \ref{def:popped-potential}}.
\end{eqnarray}

Since $\sigma=f_k(\tau)$ is an ideal triangulation and $k$ is not a pending arc, Theorem \ref{thm:ideal-non-pending-flips<->SP-mutation} guarantees the existence of a right-equivalence
$\varphi:\widetilde{\mu}_{k}(\Atau,S(\tau,\mathbf{y}))\rightarrow(A(\sigma),S(\sigma,\mathbf{y}))\oplus(C,T)$, where $(C,T)$ is a trivial SP.  This right-equivalence identifies the weighted quiver $(\Qsigma,\dtuple(\sigma))$ with a weighted subquiver of $\widetilde{\mu}_k(\Qtau,\dtuple(\tau))$, and the species $\Asigma$ with an $R$-$R$-subbimodule of $\widetilde{\mu}_k(\Atau)$. Under the former identification, the restriction of the quiver automorphism $\psi$ of $\widetilde{\mu}_k(\Qtau)$ to $Q(\sigma)$ is precisely $\pi^\sigma_{i,j}$.
This, together with the obvious fact that the $F$-algebra isomorphism $\psi\varphi\psi^{-1}:\RA{\widetilde{\mu}_k(\Atau)}\rightarrow \RA{\Asigma\oplus C}$ acts as the identity on the vertex span $R$, implies that
$\psi\varphi\psi^{-1}$ is an $R$-algebra isomorphism such that
$\psi\varphi\psi^{-1}(\widetilde{\mu}_k(\Wtau))=\psi\varphi(\widetilde{\mu}_{k}(S(\tau,\mathbf{y})))$ is cyclically-equivalent to
$\psi(S(\sigma,\mathbf{y})+T)=\Wsigma+\psi(T)$. That is, $\psi\varphi\psi^{-1}$ is a right-equivalence
$\psi\varphi\psi^{-1}:\widetilde{\mu}_{k}(\Atau,\Wtau)\rightarrow(\Asigma,\Wsigma)\oplus(\psi(C),\psi(T))$. Since $(\Asigma,\Wsigma)$ is a
reduced SP and $(\psi(C),\psi(T))$ is a trivial SP, Theorem \ref{thm:splitting-theorem} allows us to conclude that the SPs $\mu_k(\Atau,\Wtau)$ and $(\Asigma,\Wsigma)$
are right-equivalent.
\end{proof}

The next proposition, whose proof we defer to Section \ref{sec:technical-proofs}, guarantees the existence of triangulations $\tau$ for which the SPs $\AStau$ and $\AWtau$ are right-equivalent.

\begin{prop}\label{prop:pop-existence} Let $\surf$ be a surface with empty boundary such that $|\marked|\geq 7$ if $\Sigma$ is a sphere, let $\mathcal{S}$ be a collection of $|\orb|$ distinct pending arcs, and let $i,j$ be a pair of arcs forming a self-folded triangle with $i$ as its folded side. If the elements $\mathcal{S}\cup\{i,j\}$ are all pairwise compatible, then there exists an ideal triangulation $\tau$ of $\surf$ such that $\AStau$ and $\AWtau$ are right-equivalent and $\{i,j\}\cup\mathcal{S}\subseteq\tau$.
\end{prop}

\begin{thm}\label{thm:popping-is-right-equiv} Let $\surf$ be a surface with empty boundary. For every ideal triangulation $\tau$ of $\surf$ and every pair of arcs $i,j\in\tau$ forming a self-folded triangle with $i$ as its folded side the SPs $\AStau$ and $\AWtau$ are right-equivalent.
\end{thm}

\begin{proof} Let $\mathcal{S}$ be the set of pending arcs that belong to $\tau$. Thus $\mathcal{S}$ is a collection of $|\orb|$ distinct pending arcs and the elements of $\{i,j\}\cup\mathcal{S}$ are all pairwise compatible. By Proposition \ref{prop:pop-existence}, there exists an ideal triangulation $\sigma$ of $\surf$ such that $\ASsigma$ and $\AWsigma$ are right-equivalent and $\{i,j\}\cup\mathcal{S}\subseteq\sigma$. A straightforward modification of Mosher's algorithm \cite[Pages 36-41]{Mosher} shows that there exists a finite sequence $(\tau_0,\tau_1,\ldots,\tau_\ell)$ of ideal triangulations such that:
\begin{itemize}
\item $\tau_0=\tau$ and $\tau_\ell=\sigma$;
\item for every $s\in\{1,\ldots,\ell\}$, the triangulation $\tau_{s-1}$ can be obtained from $\tau_{s}$ by flipping an arc $k_s\in\tau_{s}$;
\item $\{i,j\}\cup\mathcal{S}\subseteq\tau_s$ for every $s\in\{0,1,\ldots,\ell\}$.
\end{itemize}
Therefore, using the symbol $\simeq$ to abbreviate ``is right-equivalent to",
\begin{eqnarray}
\nonumber\AStau&\simeq&\mu_{k_1}\mu_{k_2}\ldots\mu_{k_\ell}\ASsigma \ \ \ \ \ \ \ \ \text{{\footnotesize(by Theorem \ref{thm:ideal-non-pending-flips<->SP-mutation} and Theorem \ref{thm:SP-mut-well-defined-up-to-re})}}\\
\nonumber&\simeq&\mu_{k_1}\mu_{k_2}\ldots\mu_{k_\ell}\AWsigma \ \ \ \ \ \text{{\footnotesize(by Theorem \ref{thm:SP-mut-well-defined-up-to-re}, since $\ASsigma\simeq\AWsigma$)}}\\
\nonumber&\simeq&\AWtau \ \ \ \ \ \  \ \ \ \ \ \ \ \ \ \ \ \ \ \ \ \ \ \text{{\footnotesize(by Proposition \ref{prop:pop-commutes-with-mutation} and Theorem \ref{thm:SP-mut-well-defined-up-to-re}).}}
\end{eqnarray}
Theorem \ref{thm:popping-is-right-equiv} is proved.
\end{proof}

\subsection{Pop at an orbifold point}

Throughout this subsection we keep assuming that $\surf$ is a surface with empty boundary. We also suppose that $\orb\neq\varnothing$, for the entire content of this subsection will be vacuous when $\orb=\varnothing$.

Let $\tau$ be an ideal triangulation of $\surf$. Since $\orb\neq\varnothing$, the least common multiple of the tuple $\dtuple(\tau)$ is 2. Let $E/F$ be a field extension satisfying \eqref{eq:E/F-degree-d}, \eqref{eq:root-of-unity} and \eqref{eq:v^2=-1}, and $g=g(\tau)$ be the modulating function for $\tau$ over $E/F$ (see the paragraph preceding Definition \ref{def:unred-species-of-triangulation}). Let $q\in\orb$ be an orbifold point and $i$ be the unique arc in $\tau$ containing $q$. Then we have $\dtuple(\tau)_i=2$. We claim that there exists a unique bijection $\psi=\psi_i^\tau:\Qtau_1\rightarrow\Qtau_1$ such that the hypotheses of Corollary \ref{coro:non-triv-automorphism-Psi_i} are satisfied for $(Q,\dtuple)=(\Qtau,\dtuple(\tau))$. This is clear if the (unique) orbifolded triangle of $\tau$ where $q$ lies contains only one orbifold point, for in such situation we can take $\psi$ to be the identity of $\Qtau_1$. If the (unique) orbifolded triangle of $\tau$ where $q$ lies contains an orbifold point $q'\neq q$, then $\Qtau$ has exactly two arrows connecting $i$ and the unique arc $j\in\tau$ that contains $q'$. Let these two arrows be $a$ and $b$. By the definition of the modulating function $g=g(\tau)$, we have $g_a\neq g_b$. Since any arc $k\in\tau\setminus\{j\}$ connected to $i$ by means of an arrow of $\Qtau$ necessarily satisfies $\dtuple(\tau)_k=1$, we can take $\psi=\psi_i^\tau$ to be the permutation of $Q_1$ that interchanges $a$ and $b$ and fixes every other arrow of $Q$.

Thus, regardless of whether the number of orbifold points in the orbifolded triangle containing $q$ is 1 or 2, we have a bijection $\psi=\psi_i^\tau:\Qtau_1\rightarrow\Qtau_1$ such that the hypotheses of Corollary \ref{coro:non-triv-automorphism-Psi_i} are all met. Therefore, we have a unique ring automorphism $\Psi=\Psi_i^\tau:\RA{\Atau}\rightarrow\RA{\Atau}$ whose restriction to $R$ is given by \eqref{eq:automorphism-of-R} and whose restriction to $\Qtau_1$ equals $\psi=\psi_i^\tau$. Notice that $\Psi=\Psi_i^\tau$ is $F$-linear.

\begin{defi}\label{def:orb-popped-potential} Let $\tau$ be an ideal triangulation of $\surf$, and $q\in\orb$ be an orbifold point. Let $i$ be the unique arc of $\tau$ in which $q$ is contained, $p$ be the unique puncture contained in $i$, and $\Psi_i^\tau$ be the ring automorphism of $\RA{\Atau}$ from the previous paragraph.
We define the \emph{popped potential of $\tau$ with respect to the orbifold point $q$ (under the choice $\mathbf{x}=(x_p)_{p\in\punct}$ of non-zero elements of $F$)} to be the potential $\Vtauq=\Psi_i^\tau(S(\tau,\zeta_{q,\tau}\mathbf{x}))$, where $\zeta_{q,\tau}:\punct\rightarrow\{1,-1\}$ is the function that takes the value $-1$ at the puncture $p$ and the value $1$ at all the punctures different from $p$.
\end{defi}

\begin{remark}
Notice that $\Psi_i^\tau$ is not an $R$-algebra automorphism of $\RA{\Atau}$ since it does not restrict to the identity of $R$. Therefore, $\Psi_i^\tau$ is not a right-equivalence between $\AStau$ and $\AVtauq$ (nor between $(A(\tau),S(\tau,\zeta_{q,\tau}\mathbf{x}))$ and $\AVtauq$) because, by definition, a right-equivalence between two SPs on the same species must restrict to the identity of the vertex span.
\end{remark}

Our ultimate goal in this paper is to show that whenever $\tau$ and $\sigma$ are tagged triangulations of $\surf$ related by the flip of a tagged arc $k$, the associated SPs $\AStau$ and $(A(\sigma),S(\sigma,\zeta\mathbf{x}))$ are related by the corresponding SP-mutation $\mu_k$ for some function $\zeta:\punct\rightarrow\{1,-1\}$ that takes the value $-1$ on at most one puncture. The popped potentials $\AVtauq$ defined above will be a tool to pursue this goal. As a first step, we will show that if $\tau$ and $\sigma$ are ideal triangulations such that $\sigma=f_k(\tau)$ for an arc $k\in\tau$ which is neither a pending arc nor a folded side, then for any orbifold point $q$ the SPs $\AVtauq$ and $\AVsigmaq$ are related by the SP-mutation $\mu_k$. The proof of this fact requires some preparation.

Let $\surf$ be a surface with empty boundary, and let $\tau$ and $\sigma$ be ideal triangulations of $\surf$. Suppose that $\sigma$ can be obtained from $\tau$ by the flip of an arc $k\in\tau$ which is not a pending arc (since $\sigma$ is an ideal triangulation, $k$ cannot be a folded side of $\tau$). Pick any orbifold point $q\in\orb$ and let $i$ be the unique arc in $\tau$ that contains $q$. Then $i\neq k$ and hence $i$ belongs to $\sigma$ as well. Moreover, we have $\dtuple(\tau)_k=1$ since $k$ is not a pending arc.

Let $(\phi_j:F_j\rightarrow F_j)_{j\in \Qtau_0}$ be the tuple of field automorphisms given by \eqref{eq:tuple-phi_j-for-one-orbpop}, and $\psi^\tau_i:\Qtau_1\rightarrow \Qtau_1$ and $\Psi^\tau_i:\RA{\Atau}\rightarrow\RA{\Atau}$ be the bijection and the ring automorphism defined in the second and third paragraphs of the ongoing subsection. Via the bijection $\tau\rightarrow\sigma$ induced by the flip operation, we can view the tuple $(\phi_j:F_j\rightarrow F_j)_{j\in \Qtau_0}$ as a tuple indexed by the arcs in $\sigma$, i.e., $(\phi_j:F_j\rightarrow F_j)_{j\in \Qsigma_0}$. So, $\sigma$ has its own bijection $\psi^\sigma_i:\Qsigma_1\rightarrow \Qsigma_1$ and its own ring automorphism $\Psi^\sigma_i:\RA{\Asigma}\rightarrow\RA{\Asigma}$. We shall construct a bijection $\widetilde{\psi}_i^\tau:\widetilde{\mu}_k(\Qtau)_1\rightarrow\widetilde{\mu}_k(\Qtau)_1$ and a ring automorphism $\widetilde{\Psi}^\tau_i:\RA{\widetilde{\mu}_k(\Atau)}\rightarrow\RA{\widetilde{\mu}_k(\Atau)}$, and compare these with $\psi_i^\sigma$ and $\Psi_i^\sigma$.

Define a function $\widetilde{\psi} = \widetilde{\psi}^\tau_i:\widetilde{\mu}_k(\Qtau)_1\rightarrow\widetilde{\mu}_k(\Qtau)_1$ by
$$
\widetilde{\psi}(c)=\begin{cases}
\psi(c) & \text{if $c\in\widetilde{\mu}_k(\Qtau)_1\cap\Qtau_1$};\\
\gamma^* & \text{if $c=\gamma^*$ for an arrow $\gamma\in\Qtau_1$ incident to $k$};\\
[ab]_{\phi_{h(a)}|\rho\phi_{t(b)}^{-1}|} & \text{if $c=[ab]_\rho$ for $a,b\in\Qtau_1$ with $t(a)=k=h(b)$, and $\rho\in G_{h(a),t(b)}$};
\end{cases}
$$
where we are using the notations $\phi_{h(a)}|=\phi_{h(a)}|_{F_{h(a),t(b)}}$ and $\phi_{t(b)}^{-1}|=\phi_{t(b)}^{-1}|_{F_{h(a),t(b)}}$. It is clear that $\widetilde{\mu}_k(\Qtau)_1\cap\Qtau_1$, which is the set of all arrows that are common to $\Qtau$ and $\widetilde{\mu}_k(\Qtau)_1$, is contained in the image of $\widetilde{\psi}$. It is also clear that the arrows of $\widetilde{\mu}_k(\Qtau)_1$ that are incident to $k$ belong to the image of $\widetilde{\psi}$. If $a$ and $b$ are arrows of $\Qtau$ such that $t(a)=k=h(b)$, and $\rho\in G_{h(a),t(b)}$, then $\phi_{h(a)}^{-1}|\rho\phi_{t(b)}|\in G_{h(a),t(b)}$, which means that $[ab]_{\phi_{h(a)}^{-1}|\rho\phi_{t(b)}|}$ is an arrow of $\widetilde{\mu}_k(\Qtau)$; this arrow satisfies $[ab]_{\rho}=\widetilde{\psi}([ab]_{\phi_{h(a)}^{-1}|\rho\phi_{t(b)}|})$. Hence $\widetilde{\psi}=\widetilde{\psi}^\tau_i:\widetilde{\mu}_k(\Qtau)_1\rightarrow\widetilde{\mu}_k(\Qtau)_1$ is a surjective function and therefore bijective.

By the previous paragraph, $\widetilde{\psi}$ is indeed well defined. We leave in the hands of the reader to verify that the weighted quiver $(\widetilde{\mu}_k(\Qtau),\dtuple(\tau))$, its vertex $i$, the modulating function  $\widetilde{\mu}_k(g):\widetilde{\mu}_k(\Qtau)_1\rightarrow\Gal(E/F)\cup\Gal(F/F)$ given by applying Definition \ref{def:SP-premutation}, and the function $\widetilde{\psi}$ defined above satisfy the hypotheses of Corollary \ref{coro:non-triv-automorphism-Psi_i}, and hence yield a ring automorphism $\widetilde{\Psi}_i:\RA{\widetilde{\mu}_k(\Atau)}\rightarrow\RA{\widetilde{\mu}_k(\Atau)}$.

Recall that Theorem \ref{thm:ideal-non-pending-flips<->SP-mutation} states that, under the hypotheses we have currently imposed on $\tau$, $\sigma$ and $k$, there exists a trivial SP $(C,T)$ with the property that there exists a right-equivalence $\varphi:\ASsigma\oplus(C,T)\rightarrow\widetilde{\mu}_k\AStau$. A careful look at Subsection \ref{subsec:proof-ideal-non-pending-flips<->SP-mutation} reveals that, as part of the Theorem's proof strategy, we identify $\Qsigma$ with a specific subquiver of $\widetilde{\mu}_k(\Qtau)$ (recall that $\widetilde{\mu}_k(\Qtau)$ is the underlying quiver of the weighted quiver $\widetilde{\mu}_k(\Qtau,\dtuple(\tau))$).
Under this identification, a case-by-case check can be used to show that $\widetilde{\psi}_i^\tau(\Qsigma_1)=\Qsigma_1$ and that, actually, the restriction of $\widetilde{\psi}_i^\tau$ to $\Qsigma_1$ is equal to $\psi_i^\sigma$ (the cases of the case-by-case check coincide with the cases in the proof of Theorem \ref{thm:ideal-non-pending-flips<->SP-mutation}). Therefore, the restriction of $\widetilde{\Psi}_i^\tau$ to $\RA{\Asigma}$ is equal to $\Psi_i^\sigma$.

\begin{prop}\label{prop:orbipop-commutes-with-SP-mutation} Let $\surf$ be a surface with empty boundary, and let $\tau$ and $\sigma$ be ideal triangulations of $\surf$. If $\sigma$ can be obtained from $\tau$ by the flip of an arc $k\in\tau$ which is not a pending arc, then for any orbifold point $q\in\orb$ the SP $\mu_k\AVsigmaq$ is right-equivalent to the SP $\AVtauq$.
\end{prop}

\begin{proof} As before, let $i$ be the unique arc in $\tau$ containing $q$.
The ring automorphism $\widetilde{\Psi}_i^\tau:\RA{\widetilde{\mu}_k(\Atau)}\rightarrow\RA{\widetilde{\mu}_k(\Atau)}$ is $F$-linear and satisfies
$$
\widetilde{\Psi}_i^\tau([S(\tau,\zeta_{q,\tau}\mathbf{x})])=[\Psi_i^\tau(S(\tau,\zeta_{q,\tau}\mathbf{x}))], \ \ \ \text{and} \ \ \ \widetilde{\Psi}_i^\tau\left(\sum_{\overset{a}{\to}k\overset{b}{\to}}\sum_{\rho}b^*[ba]_{\rho}a^*\right)= \sum_{\overset{a}{\to}k\overset{b}{\to}}\sum_{\rho}b^*[ba]_{\rho}a^*,
$$
from which it follows that $\widetilde{\Psi}_i^\tau(\widetilde{\mu}_k(S(\tau,\zeta_{q,\tau}\mathbf{x})))=\widetilde{\mu}_k(\Psi^\tau_i(S(\tau,\zeta_{q,\tau}\mathbf{x})))=
\widetilde{\mu}_k(\Vtauq)$.

Since $\sigma=f_k(\tau)$ is an ideal triangulation and $k$ is not a pending arc, Theorem \ref{thm:ideal-non-pending-flips<->SP-mutation} guarantees the existence of a right-equivalence
$\varphi:\widetilde{\mu}_{k}(\Atau,S(\tau,\zeta_{q,\tau}\mathbf{x}))\rightarrow(\Asigma,S(\tau,\zeta_{q,\tau}\mathbf{x}))\oplus(C,T)$, where $(C,T)$ is a trivial SP.
This, the fact that the restriction of $\widetilde{\Psi}_i^\tau$ to $\RA{\Asigma}$ is equal to $\Psi_i^\sigma$, and the obvious fact that the $F$-algebra isomorphism
$\widetilde{\Psi}_i^\tau\varphi(\widetilde{\Psi}_i^{\tau})^{-1}:\RA{\widetilde{\mu}_k(\Atau)}\rightarrow \RA{\Asigma\oplus C}$ acts as the identity on the vertex span $R$, imply that
$\widetilde{\Psi}_i^\tau\varphi(\widetilde{\Psi}_i^{\tau})^{-1}$ is an $R$-algebra isomorphism such that
$\widetilde{\Psi}_i^\tau\varphi(\widetilde{\Psi}_i^{\tau})^{-1}(\widetilde{\mu}_k(\Vtauq))=
\widetilde{\Psi}_i^\tau\varphi(\widetilde{\mu}_{k}(S(\tau,\zeta_{q,\tau}\mathbf{x})))$ is cyclically-equivalent to
$\widetilde{\Psi}_i^\tau(S(\sigma,\zeta_{q,\tau}\mathbf{x})+T)=\Vsigmaq+\widetilde{\Psi}_i^\tau(T)$. That is, $\widetilde{\Psi}_i^\tau\varphi(\widetilde{\Psi}_i^{\tau})^{-1}$ is a right-equivalence
$\widetilde{\Psi}_i^\tau\varphi(\widetilde{\Psi}_i^{\tau})^{-1}:
\widetilde{\mu}_{k}(\Atau,\Vtauq)\rightarrow(\Asigma,\Vsigmaq)\oplus(\widetilde{\Psi}_i^\tau(C),\widetilde{\Psi}_i^\tau(T))$. Since $(\Asigma,\Vsigmaq)$ is a
reduced SP and $(\widetilde{\Psi}_i^\tau(C),\widetilde{\Psi}_i^\tau(T))$ is a trivial SP, Theorem \ref{thm:splitting-theorem} allows us to conclude that the SPs $\mu_k(\Atau,\Vtauq)$ and $(\Asigma,\Vsigmaq)$
are right-equivalent.
\end{proof}

The next proposition, whose proof we defer to Section \ref{sec:technical-proofs}, guarantees the existence of triangulations $\tau$ for which the SPs $\AStau$ and $\AVtauq$ are right-equivalent.

\begin{prop}\label{prop:orb-pop-stronger} Let $\surf$ be a surface with empty boundary such that $|\orb|\geq 1$, $\mathcal{S}$ an arbitrary collection of $|\orb|$ pairwise compatible pending arcs on $\surf$, and $q\in\orb$ any orbifold point. There exists an ideal triangulation $\tau$ of $\surf$ such that $\AStau$ and $\AVtauq$ are right-equivalent, and $\mathcal{S}\subseteq\tau$.
\end{prop}

We now combine Proposition \ref{prop:orbipop-commutes-with-SP-mutation} with Proposition \ref{prop:orb-pop-stronger} to deduce that the popped SP $\AVtauq$ is always right-equivalent to $(\Atau,\Stau)$.

\begin{thm}\label{thm:orbi-pop} Let $\surf$ be a surface with empty boundary. For any ideal triangulation $\tau$ of $\surf$ and any orbifold point $q\in\orb$, the SPs $(\Atau,\Stau)$ and $\AVtauq$ are right-equivalent.
\end{thm}

\begin{proof} Let $\mathcal{S}$ be the set of all arcs in $\tau$ that are pending arcs. Then $\mathcal{S}$ is a collection of $|\orb|$ pairwise compatible arcs. By Proposition \ref{prop:orb-pop-stronger}, there exists an ideal triangulation $\sigma$ of $\surf$ such that $(\Asigma,\Ssigma)$ and $\AVsigmaq$ are right-equivalent and $\mathcal{S}\subseteq\sigma$. Since $\mathcal{S}\subseteq\sigma\cap\tau$, a minor modification of Mosher's algorithm \cite[Pages 36-41]{Mosher} shows that there exists a finite sequence $(\tau_0,\tau_1,\ldots,\tau_n)$ of ideal triangulations such that:
\begin{itemize}
\item $\tau_0=\sigma$ and $\tau_n=\tau$;
\item for every $t\in\{1,\ldots,n\}$, $\tau_t=f_{k_t}(\tau_{t-1})$ for some arc $k_t\in\tau_{t-1}$ which is not pending.
\end{itemize}
Since all $\tau_0,\tau_1,\ldots,\tau_n$ are ideal triangulations, there is no index $t\in\{1,\ldots,n\}$ for which $k_t$ be a folded side of $\tau_{t-1}$. Therefore, using the symbol $\simeq$ to abbreviate ``is right-equivalent to", we have
\begin{eqnarray}\nonumber
(\Atau,\Stau) & \simeq & \mu_{k_n}\ldots\mu_{k_1}(\Asigma,\Ssigma) \ \ \ \ \ \ \  \text{{\footnotesize (by Theorem \ref{thm:ideal-non-pending-flips<->SP-mutation} and Theorem \ref{thm:SP-mut-well-defined-up-to-re})}}\\
\nonumber
&\simeq& \mu_{k_n}\ldots\mu_{k_1}\AVsigmaq \ \ \ \ \ \text{{\footnotesize (by Theorem \ref{thm:SP-mut-well-defined-up-to-re}, since $(\Asigma,\Ssigma)\simeq\AVsigmaq$)}}\\
\nonumber
&\simeq& \AVtauq \ \ \ \ \ \ \ \ \ \ \ \ \ \ \ \ \ \ \ \text{{\footnotesize(by Proposition \ref{prop:orbipop-commutes-with-SP-mutation} and Theorem \ref{thm:SP-mut-well-defined-up-to-re}).}}
\end{eqnarray}
Theorem \ref{thm:orbi-pop} is proved.
\end{proof}

\section{Empty boundary: Flip is compatible with SP-mutation}\label{sec:flip-and-mutation-compatibility}

\subsection{Flip of a folded side of an ideal triangulation}

\begin{thm}\label{thm:leaving-positive-stratum} Let $\surf$ be a surface with empty boundary. Suppose $\tau$ is an ideal triangulation of $\surf$, and $i\in\tau$ is the folded side of a self-folded triangle of $\tau$. Then $\mu_{i}(A(\tau),S(\tau),\mathbf{x}))$ is right-equivalent to $\ASsigma$, where $\sigma=f_i(\tau)$ is the tagged triangulation obtained from $\tau$ by flipping the arc $i$.
\end{thm}

\begin{proof}Let $l$ be the unique tagged arc such that $\sigma=(\tau\setminus\{i\})\cup\{l\}$. Since SP-mutations are involutive up to right-equivalence, to prove Theorem \ref{thm:leaving-positive-stratum} it suffices to show that $\mu_l\ASsigma$ and $(A(\tau),S(\tau,\mathbf{x}))$ are right-equivalent.

Let $j$ be the arc in $\tau$ that together with $i$ forms a self-folded triangle. By Theorem \ref{thm:popping-is-right-equiv}, the SPs $\AStau$ and $\AWtau$ are right-equivalent.  We shall show that $(A(\tau),\Wtau)$ is right-equivalent to $\mu_{l}\ASsigma$.

Let $\triangle$ be the self-folded triangle formed by $i$ and $j$, and let $\triangle'$ be the unique triangle of $\tau$ which is different from $\triangle$ and contains $j$. Then $\triangle'$ contains exactly two arcs $k$ and $k'$ different from $j$, see Figure \ref{Fig:pos_stratum_abandoned_notation}. Note that $\triangle'$ may be an orbifolded triangle, but it cannot be a self-folded triangle nor an orbifolded triangle with exactly two orbifold points.
        \begin{figure}[!ht]
                \centering
                \includegraphics[scale=.4]{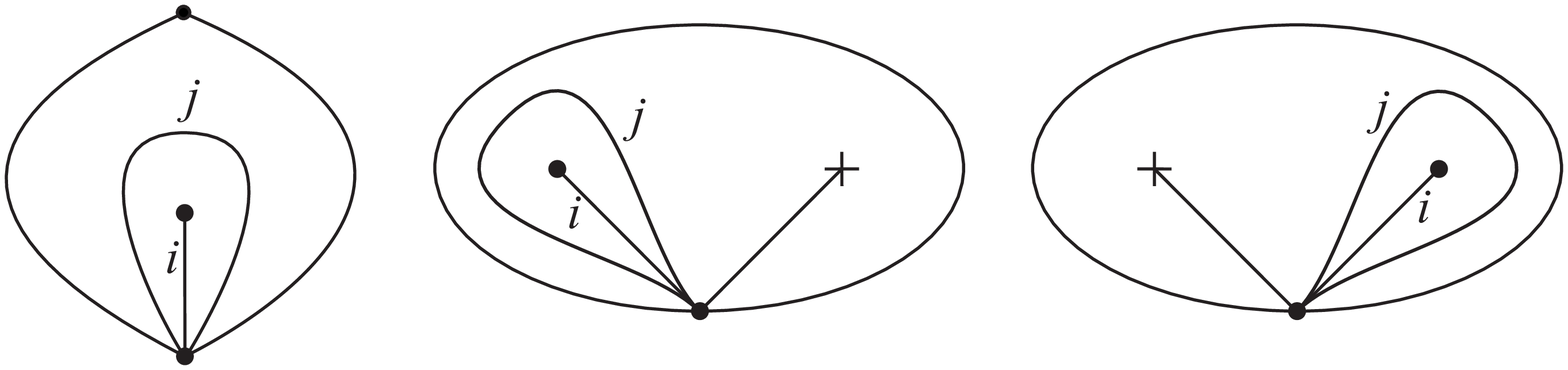}
                \caption{}\label{Fig:pos_stratum_abandoned_notation}
        \end{figure}
We have 4 possibilities:
\begin{enumerate}
\item Either $\triangle'$ is not orbifolded, and neither of $k$ and $k'$ is a loop enclosing a self-folded triangle of $\tau$; or
\item $\triangle'$ is not orbifolded, and exactly one of $k$ and $k'$ is a loop enclosing a self-folded triangle of $\tau$; or
\item $\triangle'$ is orbifolded, and $\Qtau$ has an arrow going from the (unique) pending arc contained in $\triangle'$ to $i$; or
\item $\triangle'$ is orbifolded, and $\Qtau$ has an arrow going from $i$ to the (unique) pending arc contained in $\triangle'$.
\end{enumerate}

Note that $k$ and $k'$ cannot both happen to be loops enclosing self-folded triangles, for otherwise $\surf$ would be a sphere with exactly 4 punctures (situation that we have disallowed all throughout the paper, see Definition \ref{def:surf-with-orb-points}). We will prove that $(A(\tau),\Wtau)$ is right-equivalent to $\mu_{l}\ASsigma$ in the 4 cases that we have just listed as possibilities.

\begin{case}
If $\triangle'$ is not orbifolded, and neither of $k$ and $k'$ is a loop enclosing a self-folded triangle of $\tau$, then the computation made in \cite[Proof of Proposition 7.1]{Labardini-potsnoboundaryrevised} can be applied here as is in order to show $(A(\tau),\Wtau)$ and $\mu_{l}\ASsigma$ are right-equivalent.
\end{case}

\begin{case} If $\triangle'$ is not orbifolded, and exactly one of $k$ and $k'$ is a loop enclosing a self-folded triangle of $\tau$, then the computation made in \cite[Proof of Proposition 7.1]{Labardini-potsnoboundaryrevised} can be applied here as is in order to show $(A(\tau),\Wtau)$ and $\mu_{l}\ASsigma$ are right-equivalent.
\end{case}

\begin{case}\label{case:pos-strat-abandoned-orb-triang-adjacent} Suppose $\triangle'$ is orbifolded, and $\Qtau$ has an arrow going from the (unique) pending arc contained in $\triangle'$ to $i$. Without loss of generality, we can assume that the unique pending arc contained in $\triangle'$ is $k$. Then, the configurations that $\sigma$ and $\tau$ respectively present around $l$ and $i$ are the ones shown in Figure \ref{Fig:pos_stratum_abandoned_orb_1},
        \begin{figure}[!ht]
                \centering
                \includegraphics[scale=.5]{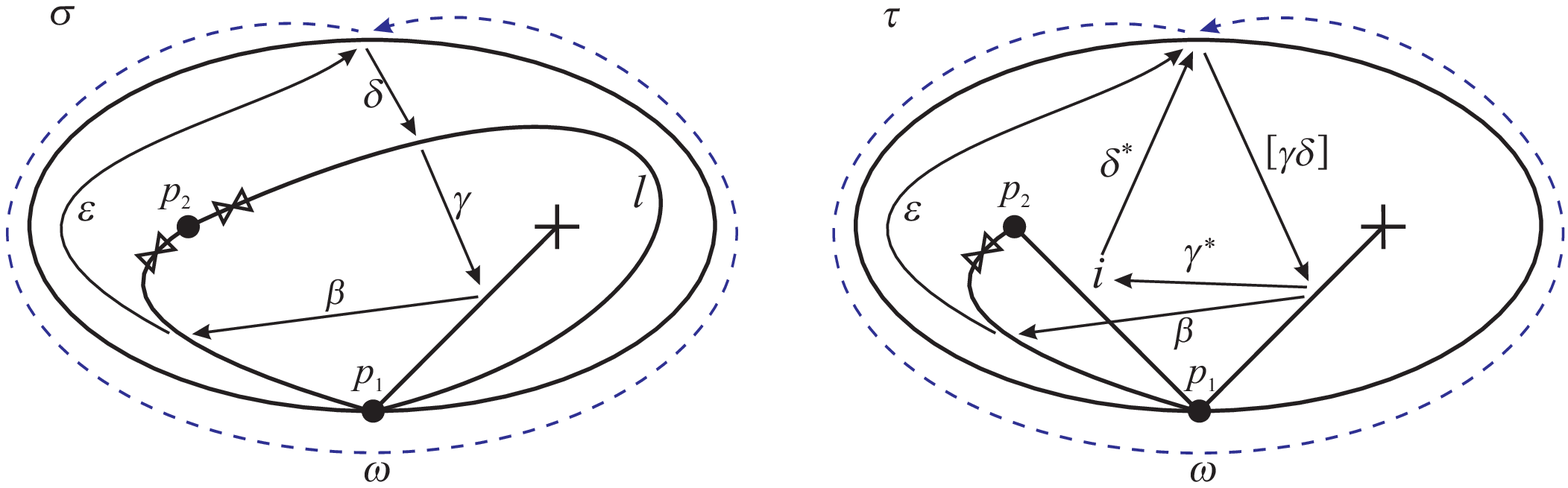}
                \caption{Left: $\sigma$. Right: $\tau$.}\label{Fig:pos_stratum_abandoned_orb_1}
        \end{figure}
and the potentials $\Ssigma$ and $\Wtau$ are given by
\begin{eqnarray}\nonumber
\Ssigma & = & x_{p_1}\omega\varepsilon\beta\gamma\delta + x_{p_2}^{-1}\delta\varepsilon(\beta\gamma-\beta v\gamma) + S'(\sigma) \ \ \ \text{and}\\
\nonumber
\Wtau & = & \gamma^*[\gamma\delta]\delta^*
-\gamma^* v[\gamma\delta]\delta^*
-x_{p_2}^{-1}\beta v[\gamma\delta]\varepsilon+x_{p_2}^{-1}\beta[\gamma\delta]\varepsilon
+x_{p_1}\omega\varepsilon\beta[\gamma\delta]
+S'(\sigma),
\end{eqnarray}
where $\omega$ and $S'(\sigma)$ are
potentials that can be written as an $F$-linear combination of paths not involving any of the arrows $\beta$, $\gamma$, $\delta$ and $\varepsilon$.
Furthermore, $\mu_l(\Qsigma,\dtuple(\sigma))=\widetilde{\mu}_l(\Qsigma,\dtuple(\sigma))=(Q(\tau),\dtuple(\tau))$  (with the vertex $i\in Q(\tau)_0$ replacing the vertex $l\in\Qsigma_0$), which means that $\mu_l(\Qsigma,\dtuple(\sigma))$ is 2-acyclic and that hence we can unambiguously write $\mu_l(\Asigma)$ to denote the underlying species of $\mu_l\ASsigma$. At the level of potentials, we have
$$
\mu_l(\Ssigma)=\widetilde{\mu}_l(\Ssigma)= x_{p_1}\omega\varepsilon\beta[\gamma\delta] + x_{p_2}^{-1}\varepsilon\beta[\gamma\delta] -x_{p_2}^{-1}\varepsilon\beta v[\gamma\delta]+S'(\sigma)+\gamma^*[\gamma\delta]\delta^*.
$$
Therefore, the $R$-algebra isomorphism $\varphi:\RA{\mu_l(\Asigma)}\rightarrow\RA{A(\tau)}$ given by $\gamma^*\mapsto \gamma^*-\gamma^*v$ and the identity on the rest of the arrows of $\mu_l(\Qsigma,\dtuple(\sigma))$ is a right-equivalence $\mu_l\ASsigma\rightarrow(A(\tau),\Wtau)$.
\end{case}

\begin{case} Suppose $\triangle'$ is orbifolded, and $\Qtau$ has an arrow going from $i$ to the (unique) pending arc contained in $\triangle'$. Without loss of generality, we can assume that the unique pending arc contained in $\triangle'$ is $k$. Then, the configurations that $\sigma$ and $\tau$ respectively present around $l$ and $i$ are the ones shown in Figure \ref{Fig:pos_stratum_abandoned_orb_2}.
        \begin{figure}[!ht]
                \centering
                \includegraphics[scale=.5]{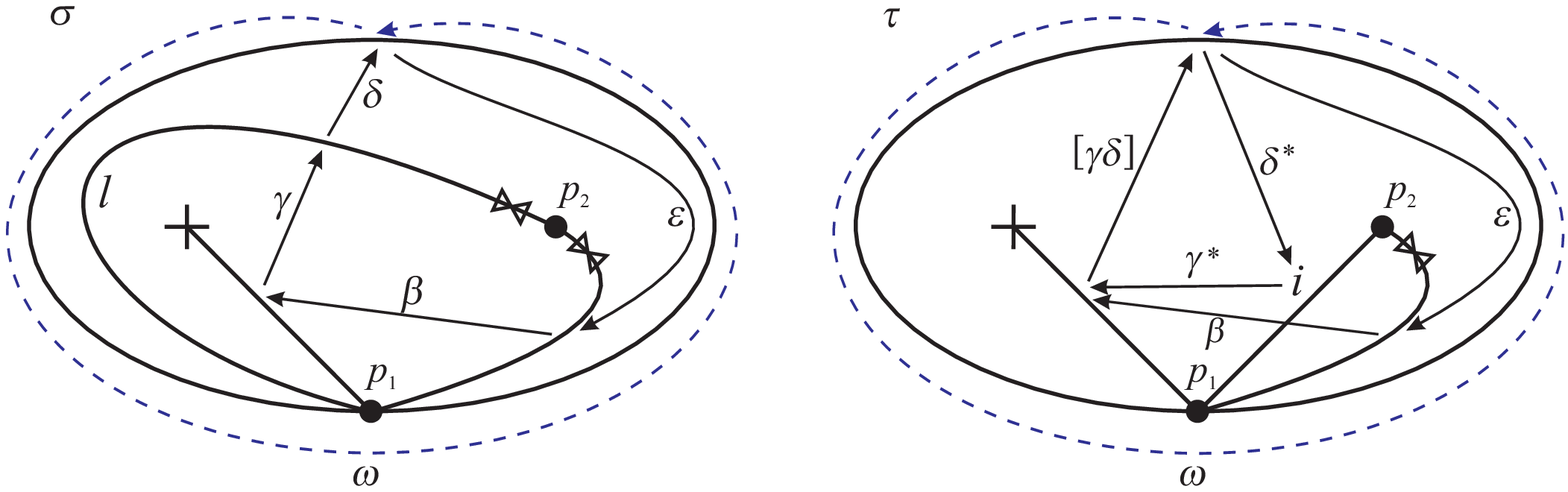}
                \caption{Left: $\sigma$. Right: $\tau$.}\label{Fig:pos_stratum_abandoned_orb_2}
        \end{figure}
The calculation showing that $\mu_l\ASsigma$ and $(A(\tau),\Wtau)$ are right-equivalent is very similar to the calculation made in Case \ref{case:pos-strat-abandoned-orb-triang-adjacent}.
\end{case}

We have thus proved that $\mu_l(\Asigma,\Ssigma)$ and $(A(\tau),\Wtau)$ are right-equivalent in all cases. Since $(A(\tau),\Wtau)$ is right-equivalent to $(A(\tau),S(\tau,\mathbf{x}))$ by Theorem \ref{thm:popping-is-right-equiv}, and since the composition of right equivalences is again a right equivalence, it follows that $\mu_l(\Asigma,\Ssigma)$ is right-equivalent to $(A(\tau),S(\tau,\mathbf{x}))$.
\end{proof}

\subsection{Flip of a pending arc of an ideal triangulation}

\begin{thm}\label{thm:flip-pending-arc<->SP-mut} Let $\surf$ be a surface with empty boundary. Let $\tau$ be any ideal triangulation of $\surf$ and $k$ be any pending arc of $\tau$. The SP $\mu_k\AStau$ is right-equivalent to the SP $(A(\sigma),S(\sigma,\zeta_{q,\sigma}\mathbf{x}))$, where $\sigma=f_k(\tau)$ and $q$ is the unique orbifold point lying on $k$.
\end{thm}

\begin{proof}
By Theorem \ref{thm:orbi-pop}, we know that $(A(\sigma),S(\sigma,\zeta_{q,\sigma}\mathbf{x}))$ is right-equivalent to $(A(\sigma),V^q(\sigma,\zeta_{q,\sigma}\mathbf{x}))$. We will show that $\mu_k\AStau$ is right-equivalent to $(A(\sigma),V^q(\sigma,\zeta_{q,\sigma}\mathbf{x}))$.

Let $q$ be the unique orbifold point lying on $k$, and let $\triangle$ be the unique orbifolded triangle containing $q$. We have the following possibilities:
\begin{enumerate}
\item Either $\triangle$ contains exactly one orbifold point, and $\triangle$ does not share sides with any self-folded triangle; or
\item $\triangle$ contains exactly one orbifold point, and the arc in $\triangle$ from which an arrow goes to $k$ is contained in a self-folded triangle; or
\item $\triangle$ contains exactly one orbifold point, and the arc in $\triangle$ to which an arrow goes from $k$ is contained in a self-folded triangle; or
\item $\triangle$ contains exactly two orbifold points, and $k$ is the tail of exactly two arrows of $\Qtau$; or
\item $\triangle$ contains exactly two orbifold points, and $k$ is the head of exactly two arrows of $\Qtau$.
\end{enumerate}

\setcounter{case}{0}
\begin{case} Suppose $\triangle$ contains exactly one orbifold point, and $\triangle$ does not share sides with any self-folded triangle. Then $\triangle$ is either homeomorphic to the puzzle piece IV from Figure \ref{Fig:puzzle_pieces_2}, or it can be obtained from it by identifying the two marked points in the boundary of the latter. These two possibilities have been taken into account in Figure \ref{Fig:flip_mut_orb_triangle_1} (the figure consists of two parts, the left corresponds to the situation where $\triangle$ is homeomorphic to the puzzle piece IV, and the right corresponds to the situation where $\triangle$ is obtained from the puzzle piece IV by identifying the two marked points on the boundary of the latter; in each of them we have drawn both $\triangle$ and the configuration obtained from it by applying the flip of $k$).
        \begin{figure}[!ht]
                ~\\
                \centering
                \includegraphics[scale=.5]{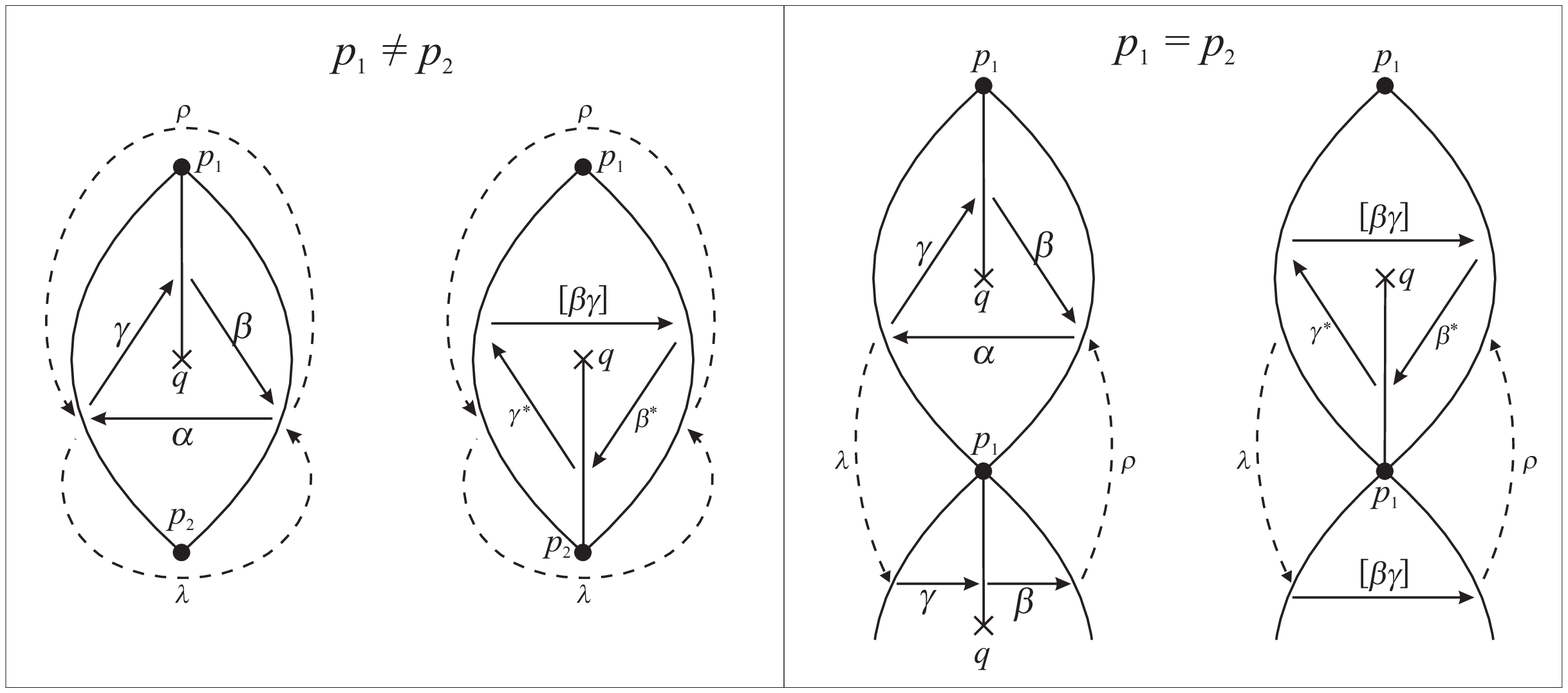}
                \caption{}\label{Fig:flip_mut_orb_triangle_1}
        \end{figure}
With the notation of Figure \ref{Fig:flip_mut_orb_triangle_1}, we have
\begin{eqnarray*}
\Stau &=& \alpha\beta\gamma-\alpha\beta v\gamma+X(x_{p_1}\rho\beta\gamma+x_{p_2}\alpha\lambda)+Yx_{p_1}\alpha\rho\beta\gamma\lambda+S'(\tau)\\
\widetilde{\mu}_k(\Stau) &=& \alpha[\beta\gamma]-\alpha[\beta v\gamma]+X(x_{p_1}\rho[\beta\gamma]+x_{p_2}\alpha\lambda)+Yx_{p_1}\alpha\rho[\beta\gamma]\lambda+S'(\tau)+\frac{1}{2}\left(\gamma^*\beta^*[\beta\gamma]-\gamma^*v\beta^*[\beta v\gamma]\right)\\
V^q(\sigma,\zeta_{q,\sigma}\mathbf{x}) &=& \gamma^*\beta^*[\beta\gamma]+\gamma^*v\beta^*[\beta\gamma]+X(x_{p_1}\rho[\beta\gamma]+x_{p_2}\gamma^*\beta^*\lambda)
+Yx_{p_1}\gamma^*\beta^*\rho[\beta\gamma]\lambda+S'(\tau),
\end{eqnarray*}
where $X=1$ and $Y=0$ if $\triangle$ is homeomorphic to puzzle piece IV, and $X=0$ and $Y=1$ if $\triangle$ is obtained from puzzle piece IV by identifying the two marked points on the boundary of the latter.

Define $R$-algebra automorphisms $\varphi_0,\varphi_1,\varphi_2,\varphi_3:\RA{\widetilde{\mu}_k(\Atau)}\rightarrow\RA{\widetilde{\mu}_k(\Atau)}$ by means of the rules
\begin{center}
\begin{tabular}{ccll}
$\varphi_0$ &:& $\beta^*\mapsto 2\beta^*$ & \\
$\varphi_1$ &:& $[\beta v\gamma]\mapsto -[\beta v\gamma]+[\beta\gamma]$ & \\
$\varphi_2$ &:& $\alpha\mapsto\alpha-\gamma^*v\beta^*$, & $[\beta v\gamma]\mapsto[\beta v\gamma]-Xx_{p_2}\lambda-Yx_{p_1}\rho[\beta\gamma]\lambda$\\
$\varphi_3$ &:& $\beta^*\mapsto v\beta^*$. &
\end{tabular}
\end{center}
Direct computation shows that $\varphi_3\varphi_2\varphi_1\varphi_0(\widetilde{\mu}_k(\Stau)) = \alpha[\beta v\gamma]+V^q(\sigma,\zeta_{q,\sigma}\mathbf{x})$, and this implies that $\mu_k\AStau$ is right-equivalent to $(A(\sigma),V^q(\sigma,\zeta_{q,\sigma}\mathbf{x}))$.
\end{case}

\begin{case}\label{case:mut-pend-arc-orb-&-sf-triangles} Suppose $\triangle$ contains exactly one orbifold point, and the arc in $\triangle$ from which an arrow goes to $k$ is adjacent to a self-folded triangle.
\begin{figure}[!ht]
                \centering
                \includegraphics[scale=.5]{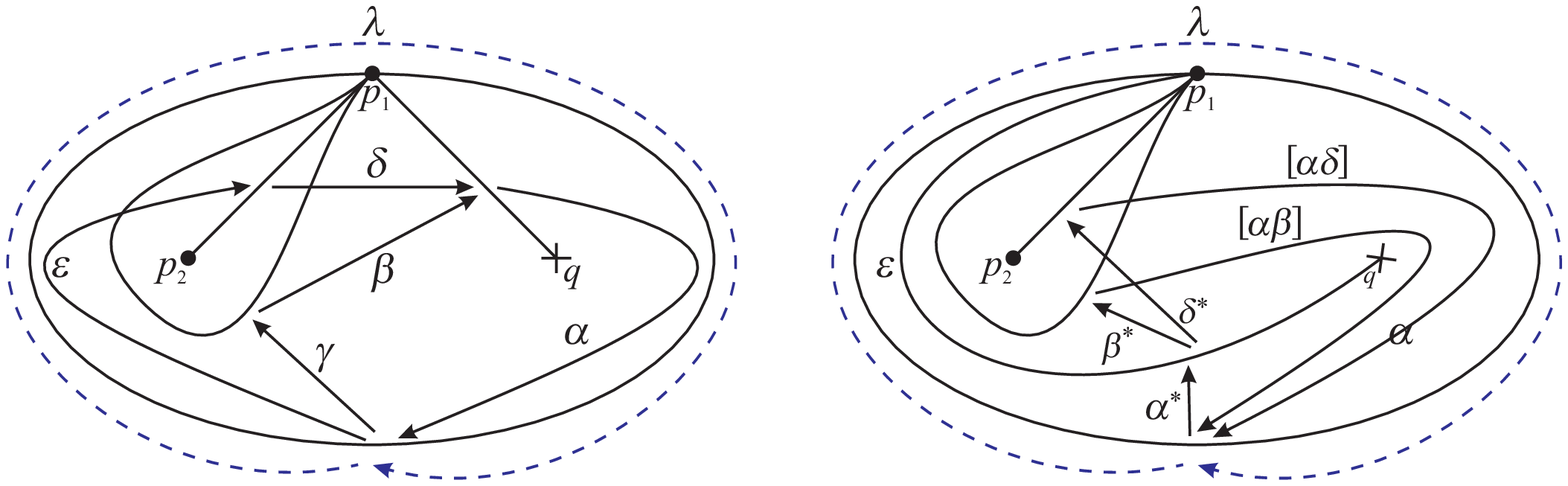}
                \caption{}\label{Fig:flip_mut_orb_triangle_3}
        \end{figure}
With the notation of Figure \ref{Fig:flip_mut_orb_triangle_3} we then have
\begin{eqnarray*}
\Stau &=& \alpha\beta\gamma-\alpha v\beta\gamma-x_{p_2}^{-1}\alpha\delta\varepsilon+ x_{p_2}^{-1}\alpha v\delta\varepsilon+x_{p_1}\alpha\delta\varepsilon\lambda+S'(\tau),\\
\widetilde{\mu}_k(\Stau) &=& [\alpha\beta]\gamma-[\alpha v\beta]\gamma-x_{p_2}^{-1}[\alpha\delta]\varepsilon+ x_{p_2}^{-1}[\alpha v\delta]\varepsilon+x_{p_1}[\alpha\delta]\varepsilon\lambda+S'(\tau)\\
&&
+\frac{1}{2}\left(\beta^*\alpha^*[\alpha\beta]-\beta^*v\alpha^*[\alpha v\beta]+\delta^*\alpha^*[\alpha\delta]-\delta^*v\alpha^*[\alpha v\delta]\right),\\
V^q(\sigma,\zeta_{q,\sigma}\mathbf{x}) &=& [\alpha\beta]\beta^*\alpha^*+[\alpha\beta]\beta^*v\alpha^*-x_{p_2}^{-1}[\alpha\delta]\delta^*\alpha^*-x_{p_2}^{-1}[\alpha\delta]\delta^*v\alpha^*
+x_{p_1}[\alpha\delta]\delta^*\alpha^*\lambda+S'(\tau).
\end{eqnarray*}
Define $R$-algebra automorphisms $\varphi_0,\varphi_1,\varphi_2,\varphi_3:\RA{\widetilde{\mu}_k(\Atau)}\rightarrow\RA{\widetilde{\mu}_k(\Atau)}$ by means of the rules
\begin{center}
\begin{tabular}{cclll}
$\varphi_0$ &:& $\alpha^*\mapsto 2\alpha^*$, & &\\
$\varphi_1$ &:& $[\alpha v\beta]\mapsto-[\alpha v\beta]+[\alpha\beta]$, & $[\alpha v\delta]\mapsto[\alpha v\delta]+[\alpha\delta]$, &\\
$\varphi_2$ &:& $\gamma\mapsto\gamma-\beta^*v\alpha^*$, & $[\alpha v\delta]\mapsto[\alpha v\delta]-x_{p_1}x_{p_2}\lambda[\alpha\delta]$, & $\varepsilon\mapsto\varepsilon+x_{p_2}\delta^*v\alpha^*$,\\
$\varphi_3$ &:& $\alpha^*\mapsto v\alpha^*$, & $\delta^*\mapsto-x_{p_2}^{-1}\delta^*$. &
\end{tabular}
\end{center}
Direct computation shows that $\varphi_3\varphi_2\varphi_1\varphi_0(\widetilde{\mu}_k(\Stau))=V^q(\sigma,\zeta_{q,\sigma}\mathbf{x})$.
\end{case}

\begin{case} The case where $\triangle$ contains exactly one orbifold point, and where the arc in $\triangle$ to which an arrow goes from $k$ is adjacent to a self-folded triangle, can be established in a way similar to Case \ref{case:mut-pend-arc-orb-&-sf-triangles} above.
\end{case}

\begin{case} Suppose $\triangle$ contains exactly two orbifold points, and $k$ is the tail of exactly two arrows of $\Qtau$.
        \begin{figure}[!ht]
                \centering
                \includegraphics[scale=.5]{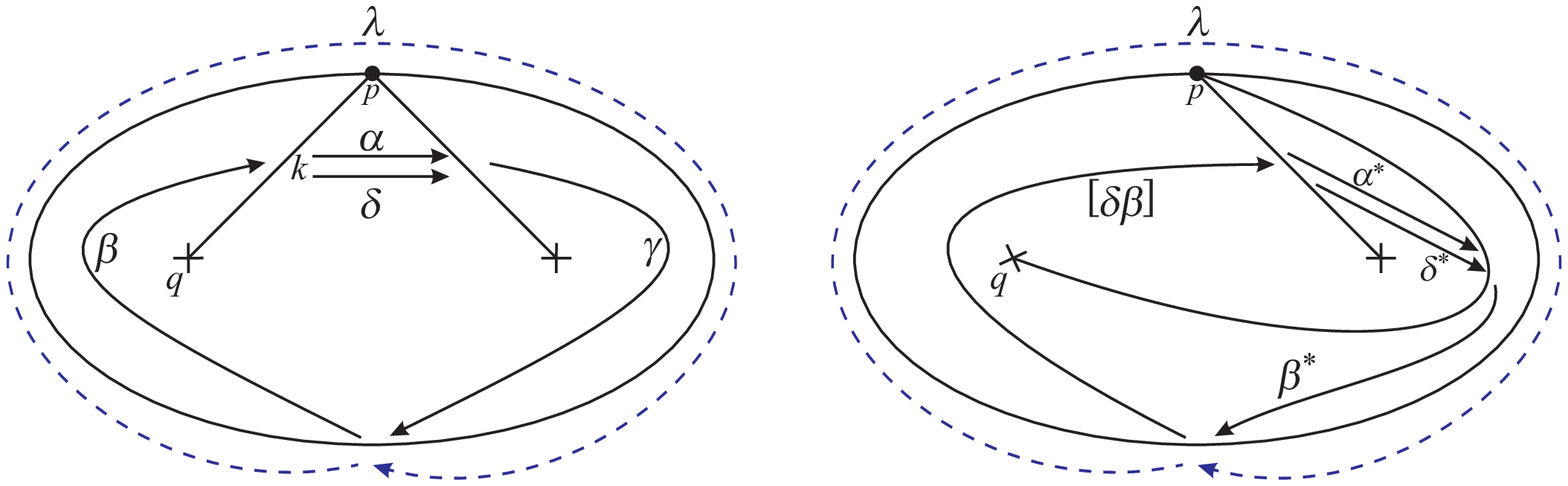}
                \caption{}\label{Fig:flip_mut_orb_triangle_2}
        \end{figure}
With the notation of Figure \ref{Fig:flip_mut_orb_triangle_2} we then have
\begin{eqnarray*}
\Stau &=& 2v^{-1}\alpha\beta\gamma+2\delta\beta\gamma+ x_p\alpha\beta\lambda\gamma+x_p\delta\beta\lambda\gamma+S'(\tau)\\
\widetilde{\mu}_k(\Stau) &=& 2v^{-1}[\alpha\beta]\gamma+2[\delta\beta]\gamma+ x_p[\alpha\beta]\lambda\gamma+x_p[\delta\beta]\lambda\gamma+S'(\tau)+[\alpha\beta]\beta^*\alpha^*+[\delta\beta]\beta^*\delta^*\\
V^q(\sigma,\zeta_{q,\sigma}\mathbf{x}) &=& 2[\delta\beta]\beta^*\alpha^*+2[\delta\beta]\beta^*v\delta^*+x_p[\delta\beta]\lambda\beta^*\alpha^*+x_p[\delta\beta]\lambda\beta^*\delta^*+S'(\tau),
\end{eqnarray*}
where $S'(\tau)$ is a potential not involving any of the arrows $\alpha$, $\beta$, $\gamma$ and $\delta$.
\end{case}
Define $R$-algebra automorphisms $\varphi_1,\varphi_2,\varphi_3,\varphi_4,\varphi_5:\RA{\widetilde{\mu}_k(\Atau)}\rightarrow\RA{\widetilde{\mu}_{k}(\Atau)}$ through the rules
\begin{center}
\begin{tabular}{ccll}
$\varphi_1$ &:& $[\delta\beta]\mapsto[\delta\beta]-[\alpha\beta]$, & \\
$\varphi_2$ &:& $[\alpha\beta]\mapsto\left(\frac{-1+v}{4}\right)[\alpha\beta]$, & \\
$\varphi_3$ &:& $[\alpha\beta]\mapsto[\alpha\beta]-2[\delta\beta]$, & \\
$\varphi_4$ &:& $[\alpha\beta]\mapsto[\alpha\beta]-x_p[\delta\beta]\lambda$, & $\gamma\mapsto\gamma-\beta^*\alpha^*\left(\frac{-1+v}{4}\right)+\beta^*\delta^*\left(\frac{-1+v}{4}\right)$,\\
$\varphi_5$ &:& $\alpha^*\mapsto\alpha^*\left(\frac{4}{1-v}\right)$, &  $\delta^*\mapsto\delta^*\left(\frac{4}{-1+v}\right)$.
\end{tabular}
\end{center}
Direct computation shows that
$\varphi_5\varphi_4\varphi_3\varphi_2\varphi_1(\widetilde{\mu}_k(\Stau))\sim_{\operatorname{cyc}}[\alpha\beta]\gamma+V^q(\sigma,\zeta_{q,\sigma}\mathbf{x})$, and this means that $\mu_k\AStau$ is right-equivalent to $(A(\sigma),V^q(\sigma,\zeta_{q,\sigma}\mathbf{x}))$.

\begin{case} The case where $\triangle$ contains exactly two orbifold points, and $k$ is the head of exactly two arrows of $\Qtau$, can be deduced from the previous case using the fact that flips and SP-mutations are involutive operations (the latter up to right-equivalence). Alternatively, it can be treated directly, without any reference to the alluded involutions, in a way similar to the previous case.
\end{case}%
%
\end{proof}

\subsection{Flip of an arbitrary arc of a tagged triangulation}

\begin{thm}\label{thm:tagged-flips<->SP-mutation} Let $\surf$ be a surface with empty boundary, and let $\mathbf{x}=(x_p)_{p\in\marked}$ be a choice of non-zero elements of the base field $F$. Let $\tau$ and $\sigma$ be tagged triangulations of $\surf$. If $\sigma$ is obtained from $\tau$ by the flip of a tagged arc $k\in\tau$, then the SPs $\mu_k(\Atau,\Stau)$ and $(\Asigma,S(\sigma,\mathbf{y}))$ are right-equivalent for some tuple $\mathbf{y}=(y_p)_{p\in\marked}$ of non-zero elements of $F$.
\end{thm}

Before entering the proof, let us give a completely explicit description of the tuple $\mathbf{y}=(y_p)_{p\in\marked}$:
\begin{enumerate}
\item If $k$ is not a pending arc, then $\mathbf{y}=\mathbf{x}$, which means that $\mu_k(\Atau,\Stau)$ is right-equivalent to $\ASsigma$;
\item if $k$ is a pending arc, $k'$ is the (unique) arc in $\sigma$ such that $\sigma=(\tau\setminus\{k\})\cup\{k'\}$, and $p'$ is the unique marked point lying on $k'$, then $y_p=(-1)^{\delta_{p,p'}}x_p$ for all $p\in\marked$, where $\delta_{p,p'}$ is \emph{Kronecker's delta}.
\end{enumerate}

\begin{proof}[Proof of Theorem \ref{thm:tagged-flips<->SP-mutation}] We have two possibilities: the weak signatures $\epsilon_\tau$ and $\epsilon_\sigma$ either are equal, or they differ at exactly one puncture $p_0$. Let $\mathbf{z}=(\epsilon_\tau(p)x_{p})_{p\in\marked}$.

Suppose that $\epsilon_\tau$ and $\epsilon_\sigma$ differ at exactly one puncture $p_0$. Since flips and SP-mutations are involutive up to right-equivalence, without loss of generality we can assume that $\epsilon_\tau(p_0)=1=-\epsilon_\sigma(p_0)$. Then $k^\circ$ is a folded side of $\tau^\circ$ incident to $p_0$, and $\tagfunction_{\epsilon_\tau\epsilon_\sigma}(\sigma^\circ)=f_{k^\circ}(\tagfunction_{\mathbf{1}}(\tau^\circ))$. By Theorem \ref{thm:leaving-positive-stratum}, the SP $\mu_{k^\circ}(A(\tagfunction_{\mathbf{1}}(\tau^\circ)),S(\tagfunction_{\mathbf{1}}(\tau^\circ),\mathbf{z}))$ is right-equivalent to
$(A(\tagfunction_{\epsilon_\tau\epsilon_\sigma}(\sigma^\circ)),S(\tagfunction_{\epsilon_\tau\epsilon_\sigma}(\sigma^\circ),\mathbf{z}))$.
By the definition of the species $A(\sigma)$ and of the underlying species of $\mu_{k}(A(\tau),S(\tau,\mathbf{x}))$ on the one hand, and of the species $A(\tagfunction_{\epsilon_\tau\epsilon_\sigma}(\sigma^\circ))$ and of the underlying species of $\mu_{k^\circ}(A(\tagfunction_{\mathbf{1}}(\tau^\circ)),S(\tagfunction_{\mathbf{1}}(\tau^\circ),\mathbf{z}))$ on the other, we see that $A(\sigma)$ and the underlying species of $\mu_{k}(A(\tau),S(\tau,\mathbf{x}))$ can be respectively obtained from $A(\tagfunction_{\epsilon_\tau\epsilon_\sigma}(\sigma^\circ))$ and the underlying species of $\mu_{k^\circ}(A(\tagfunction_{\mathbf{1}}(\tau^\circ)),S(\tagfunction_{\mathbf{1}}(\tau^\circ),\mathbf{z}))$ using the same relabeling of the vertices of the quiver $Q(\tagfunction_{\epsilon_\tau\epsilon_\sigma}(\sigma^\circ))$. When performing this relabeling of vertices, it is clear that the potentials we obtain from $S(\tagfunction_{\epsilon_\tau\epsilon_\sigma}(\sigma^\circ),\mathbf{z})$ and the underlying potential of $\mu_{k^\circ}(A(\tagfunction_{\mathbf{1}}(\tau^\circ)),S(\tagfunction_{\mathbf{1}}(\tau^\circ),\mathbf{z}))$ are precisely $S(\sigma,\mathbf{x})$ and the underlying potential of $\mu_{k}(A(\tau),S(\tau,\mathbf{x}))$. Furthermore, the right-equivalence between $\mu_{k^\circ}(A(\tagfunction_{\mathbf{1}}(\tau^\circ)),S(\tagfunction_{\mathbf{1}}(\tau^\circ),\mathbf{z}))$ and
$(A(\tagfunction_{\epsilon_\tau\epsilon_\sigma}(\sigma^\circ)),S(\tagfunction_{\epsilon_\tau\epsilon_\sigma}(\sigma^\circ),\mathbf{z}))$ becomes a right-equivalence between $\mu_{k}(A(\tau),S(\tau,\mathbf{x}))$ and $(A(\sigma),S(\sigma,\mathbf{x}))$.

Suppose now that $\varepsilon_\tau=\varepsilon_\sigma$. Then $\sigma^\circ=f_{k^\circ}(\tau^\circ)$. We have two possible situations to consider, namely, either $k$ is a (tagged) pending arc, or it is not.

If $k$ is a (tagged) pending arc, let $q\in\orb$ be the unique orbifold point lying on $k$. Then Theorem \ref{thm:flip-pending-arc<->SP-mut} tells us that $\mu_{k^\circ}(A(\tau^\circ),S(\tau^\circ,\mathbf{z}))$ is right-equivalent to $(A(\sigma^\circ),S(\sigma^\circ,\zeta_{q,\sigma^\circ}\mathbf{z}))$.
By the definition of the species $A(\sigma)$ and of the underlying species of $\mu_{k}(A(\tau),S(\tau,\mathbf{x}))$ on the one hand, and of the species $A(\sigma^\circ)$ and of the underlying species of $\mu_{k^\circ}(A(\tau^\circ),S(\tau^\circ,\mathbf{z}))$ on the other, we see that $A(\sigma)$ and the underlying species of $\mu_{k}(A(\tau),S(\tau,\mathbf{x}))$ can be respectively obtained from $A(\sigma^\circ)$ and the underlying species of $\mu_{k^\circ}(A(\tau^\circ),S(\tau^\circ,\mathbf{z}))$ using the same relabeling of the vertices of the quiver $Q(\sigma^\circ)$. When performing this relabeling of vertices, it is clear that the potentials we obtain from $S(\sigma^\circ,\zeta_{q,\sigma^\circ}\mathbf{z})$ and the underlying potential of $\mu_{k^\circ}(A(\tau^\circ),S(\tau^\circ,\mathbf{z}))$ are precisely $S(\sigma,\zeta_{q,\sigma}\mathbf{x})$ and the underlying potential of $\mu_{k}(A(\tau),S(\tau,\mathbf{x}))$. Furthermore, the right-equivalence between $\mu_{k^\circ}(A(\tau^\circ),S(\tau^\circ,\mathbf{z}))$ and
$(A(\sigma^\circ),S(\sigma^\circ,\zeta_{q,\sigma^\circ}\mathbf{z}))$ becomes a right-equivalence between $\mu_{k}(A(\tau),S(\tau,\mathbf{x}))$ and $(A(\sigma),S(\sigma,\zeta_{q,\sigma}\mathbf{x}))$.

Finally, if $\varepsilon_\tau=\varepsilon_\sigma$ and $k$ is not a (tagged) pending arc, then Theorem \ref{thm:ideal-non-pending-flips<->SP-mutation} tells us that $\mu_{k^\circ}(A(\tau^\circ),S(\tau^\circ,\mathbf{z}))$ is right-equivalent to $(A(\sigma^\circ),S(\sigma^\circ,\mathbf{z}))$.
By the definition of the species $A(\sigma)$ and of the underlying species of $\mu_{k}(A(\tau),S(\tau,\mathbf{x}))$ on the one hand, and of the species $A(\sigma^\circ)$ and of the underlying species of $\mu_{k^\circ}(A(\tau^\circ),S(\tau^\circ,\mathbf{z}))$ on the other, we see that $A(\sigma)$ and the underlying species of $\mu_{k}(A(\tau),S(\tau,\mathbf{x}))$ can be respectively obtained from $A(\sigma^\circ)$ and the underlying species of $\mu_{k^\circ}(A(\tau^\circ),S(\tau^\circ,\mathbf{z}))$ using the same relabeling of the vertices of the quiver $Q(\sigma^\circ)$. When performing this relabeling of vertices, it is clear that the potentials we obtain from $S(\sigma^\circ,\mathbf{z})$ and the underlying potential of $\mu_{k^\circ}(A(\tau^\circ),S(\tau^\circ,\mathbf{z}))$ are precisely $S(\sigma,\mathbf{x})$ and the underlying potential of $\mu_{k}(A(\tau),S(\tau,\mathbf{x}))$. Furthermore, the right-equivalence between $\mu_{k^\circ}(A(\tau^\circ),S(\tau^\circ,\mathbf{z}))$ and
$(A(\sigma^\circ),S(\sigma^\circ,\mathbf{z}))$ becomes a right-equivalence between $\mu_{k}(A(\tau),S(\tau,\mathbf{x}))$ and $(A(\sigma),S(\sigma,\mathbf{x}))$.
%
\end{proof}

\section{From empty boundary to the general case}\label{sec:empty=>general}

In this section we show that Theorem \ref{thm:tagged-flips<->SP-mutation} still holds if we remove the assumption of $\Sigma$ having empty boundary. The strategy of proof consists in realizing the SPs associated to triangulations of surfaces with non-empty boundary as restrictions of SPs associated to triangulations of surfaces with empty boundary.

The notion of restriction of an SP to a non-empty subset of the vertex set of its underlying quiver is a straightforward generalization of {\cite[Definition 8.8]{DWZ1}}.

\begin{defi}
 \label{def:restriction}
 Let $(A,S)$ be an SP with underlying weighted quiver~$(Q,\dtuple)$ and modulating function~$g$, and let $I$ be a non-empty subset of the vertex set $Q_0$.
 The \emph{restriction} of $(A,S)$ to $I$ is the SP $(A,S)|_I = (A|_I,S|_I)$ with underlying weighted quiver~$(Q|_I,\dtuple)$ and modulating function~$g|_I$ where
 \begin{itemize}
  \item $Q|_I$ is the quiver obtained from $Q$ by deleting all arrows incident to vertices outside $I$;
  \item $g|_I$ is defined to be the function obtained by restricting $g$ to the arrow set of $Q|_I$;
  \item $S|_I$ is the image of $S$ under the $R$-algebra homomorphism $\psi_I: \RA{A} \to \RA{A|_I}$ defined on $a\in Q_1$ as $\psi_I(a) = a$ if both head and tail of $a$ belong to $I$, and as $\psi_I(a) = 0$ otherwise.
 \end{itemize}
\end{defi}

The next lemma is an adaptation of \cite[Lemma~6.8]{Labardini-potsnoboundaryrevised} to the SP setup. The proof is easy.

\begin{lemma}
 \label{lemma:res-re=re-res}
 Let $(A,S)$ and $(A,S')$ be SPs with underlying weighted quiver $(Q,\dtuple)$, and let $I$ be a non-empty subset of $Q_0$.
 If $\varphi : (A,S) \rightarrow (A,S')$ is a right-equivalence, then the $R$-algebra automorphism $\RA{A|_I} \rightarrow \RA{A|_I}$ given by $u \mapsto \psi_I\circ\varphi(u)$ defines a right-equivalence $(A|_I,S|_I) \to (A|_I,S'|_I)$.
 Moreover, if $k\in I$ and $Q$ has no $2$-cycles incident to $k$, then the SP $\mu_k(A,S)|_I$ is right-equivalent to $\mu_k((A,S)|_I)$.
\end{lemma}

Both statement and proof of the next lemma are similar to \cite[Lemma~29]{Labardini1} and \cite[Lemma 6.9]{Labardini-potsnoboundaryrevised}.

\begin{lemma}
 \label{lemma:all-are-restrictions}
 For every tagged triangulation $\tau$ of a surface~$\surf$ there exists a tagged triangulation~$\widetilde{\tau} \supseteq \tau$ of an empty-boundary surface $(\widetilde{\Sigma}, \widetilde{\marked}, \widetilde{\orb}) \supseteq \surf$ such that $\AStau = (A(\widetilde{\tau}), S(\widetilde{\tau}, \widetilde{\mathbf{x}}))|_\tau$ for any pair $\widetilde{\mathbf{x}} = (\widetilde{x}_m)_{m\in\widetilde{M}}$ and $\mathbf{x}=(x_p)_{p\in\punct}$ of tuples of non-zero elements of $F$ such that $x_{p}=\widetilde{x}_p$ for $p\in\punct$.
\end{lemma}

\begin{proof}
 The same proof strategy of~\cite[Lemma 6.9]{Labardini-potsnoboundaryrevised} works here. Namely, cap each boundary component~$b$ of $\Sigma$ with a $5$-punctured $m_b$-gon, where $m_b$ is the number of points in $\marked$ that lie on $b$. This yields a surface $(\widetilde{\Sigma},\widetilde{\marked},\widetilde{\orb})$ with empty boundary.
 Now complete $\tau$ to a tagged triangulation~$\widetilde{\tau}$ of $\widetilde{\Sigma}$.
\end{proof}

The following is the main result of this article. It extends Theorem~\ref{thm:tagged-flips<->SP-mutation} by allowing the surface $\Sigma$ to have non-empty boundary. We remind the reader that according to Definition \ref{def:surf-with-orb-points} we always assume that $\surf$ satisfies that if $\Sigma$ is a sphere, then $|\marked|\geq 7$.

\begin{thm}
 \label{thm:tagged-flips<->SP-mutation--non-empty-boundary}
 Let $\surf$ be a surface with marked points and orbifold points of order 2, and let $\mathbf{x}=(x_p)_{p\in\marked}$ be a choice of non-zero elements of the base field $F$. Let $\tau$ and $\sigma$ be tagged triangulations of $\surf$. If $\sigma$ is obtained from $\tau$ by the flip of a tagged arc $k\in\tau$, then the SPs $\mu_k(\Atau,\Stau)$ and $(\Asigma,S(\sigma,\mathbf{y}))$ are right-equivalent for some tuple $\mathbf{y}=(y_p)_{p\in\marked}$ of non-zero elements of $F$.
\end{thm}

\begin{proof} This follows from a combination of Lemma \ref{lemma:res-re=re-res}, Lemma \ref{lemma:all-are-restrictions} and Theorem \ref{thm:tagged-flips<->SP-mutation}.
\end{proof}

The following is a direct consequence of Theorem \ref{thm:tagged-flips<->SP-mutation--non-empty-boundary}.

\begin{coro}\label{coro:non-degeneracy-of-S(tau)} Let $\surf$ be a surface with marked points and orbifold points of order 2 (see Definition \ref{def:surf-with-orb-points}). For any tagged triangulation $\tau$ of $\surf$ and any choice $\mathbf{x}=(x_p)_{p\in\punct}$ of non-zero elements of $F$, the SP $\AStau$ is non-degenerate in the sense of Derksen-Weyman-Zelevinsky (cf. \cite[Definition 7.2]{DWZ1}). That is, for any finite sequence $(i_1,\ldots,i_\ell)$ of vertices of $\Qtau$, each of the SPs $\mu_{i_1}\AStau$, $\mu_{i_2}\mu_{i_1}\AStau$, $\ldots$, $\mu_{i_\ell}\ldots\mu_{i_2}\mu_{i_1}\AStau$ is 2-acyclic.
\end{coro}

The next two results show that if the boundary of $\Sigma$ is not empty, then we can take $\mathbf{y}=\mathbf{x}$ in Theorem~\ref{thm:tagged-flips<->SP-mutation--non-empty-boundary}.

\begin{prop} If $\surf$ is a surface with non-empty boundary, then for every tagged triangulation $\tau$ of $\surf$ and any two tuples $\mathbf{x}=(x_{p})_{p\in\punct}$ and $\mathbf{y}=(y_{p})_{p\in\punct}$ of non-zero elements of $F$, the SPs $\AStau$ and $(A(\tau),S(\tau,\mathbf{y}))$ are right-equivalent.
\end{prop}

\begin{proof} A minor modification of the proof of \cite[Proposition 10.2]{Labardini-potsnoboundaryrevised} works here. Roughly speaking, the proof consists in showing that for an ideal triangulation $\sigma$ of $\surf$ without self-folded triangles, $\ASsigma$ is right-equivalent to $(A(\sigma),S(\sigma,\mathbf{1}))$ through a composition of right-equivalences that ``pull the scalars $x_p$ to the boundary of $\Sigma$''. The statement for arbitrary tagged triangulations then follows from Theorem \ref{thm:tagged-flips<->SP-mutation--non-empty-boundary}, the fact that any two tagged triangulations are related by a sequence of flips, and the fact that the compositions of right-equivalences is a right-equivalence.
\end{proof}

\begin{coro} If $\surf$ is a surface with non-empty boundary, $\tau$ is any tagged triangulation of $\surf$ and $k$ is any tagged arc in $\tau$, then $\mu_k\AStau$ is right-equivalent to $(A(f_k(\tau)),S(f_k(\tau),\mathbf{x}))$.
\end{coro}

\section{Examples}\label{sec:examples}

\subsection{Unpunctured surfaces with orbifold points}
\label{subsec:unpunctured-surfaces}

\noindent
For triangulations~$\tau$ of an unpunctured surface the associated potential can be put in a particularly simple form:
Due to the absence of punctures, $\unredStau$ contains no summands of the form~$\widehat{S}^p(\tau,\mathbf{x})$ or $\widehat{U}^\triangle(\tau,\mathbf{x})$ (cf.\ Definitions~\ref{def:cycles-from-self-folded-triangles} and \ref{def:cycles-from-punctures}),
and one can easily see that $\AStau$ is right-equivalent to $(A(\tau),P(\tau))$, where
 $P(\tau) = \sum_{\triangle} P^\triangle(\tau)$,
with the sum running over all interior triangles~$\triangle$ of $\tau$. Here, $P^\triangle(\tau) = \alpha^\triangle\beta^\triangle\gamma^\triangle$ if $\triangle$ is a triangle containing at most one orbifold point (see Figure \ref{Fig:normal_triangle} and the left side of Figure \ref{Fig:orb_triangles_1and2pts}), and $P^\triangle(\tau) = (\alpha^\triangle+\delta^\triangle)\beta^\triangle\gamma^\triangle$ if $\triangle$ is a triangle containing exactly two orbifold points (see the right side of Figure \ref{Fig:orb_triangles_1and2pts}).

\subsection{Types $\C_n$ and $\tildeC_n$}
\label{subsec:type-C}
Subsection \ref{subsec:unpunctured-surfaces} has the unpunctured $n$-gons with at most two orbifold points as particular examples.
Following the convention of \cite[Section 3]{Musiker-classical} (see also \cite[Remark 6]{Musiker-classical}), let us say that the matrices
{\small $$
\left[\begin{array}{rrrrrr}
0 & -1 & 0 &        & 0 & 0\\
2 & 0 & -1 & \ldots & 0 & 0\\
0 & 1 & 0 &         &0 & 0\\
 & \vdots &  & \ddots & & \\
0 & 0 & 0 &  & 0 & -1\\
0  & 0 & 0 &  & 1 & 0
\end{array}\right]\in\mathbb{Z}^{(n-1)\times (n-1)}
\ \ \ \text{and} \ \ \
\left[\begin{array}{rrrrrrr}
0 & -1 & 0 &        & 0 & 0 & 0\\
2 & 0 & -1 & \ldots & 0 & 0 &0\\
0 & 1 & 0 &         &0 & 0&0\\
 & \vdots &  & \ddots & & &\\
0 & 0 & 0 &  & 0 & -1 & 0 \\
0  & 0 & 0 &  & 1 & 0 & -2\\
0  & 0 & 0 &  & 0 & 1 & 0
\end{array}\right]\in\mathbb{Z}^{(n+1)\times (n+1)}
$$}
are of type $\CC_{n-1}$ and of type $\tildeC_n$, respectively. The weighted quivers of the triangulations of an unpunctured~$n$-gon with one orbifold point are the weighted quivers which correspond to the skew-symmetrizable matrices that are mutation-equivalent to the matrix of type $\CC_{n-1}$ above. See Example \ref{ex:type-C_4}. Similarly, the weighted quivers of the triangulations of an unpunctured~$n$-gon with two orbifold points are the weighted quivers which correspond to the skew-symmetrizable matrices that are mutation-equivalent to the matrix of type $\tildeC_{n}$.

Thus, for any  matrix $B$ mutation-equivalent to a matrix of type $\CC_n$ or $\tildeC_n$
we have a way to construct a species realization of $B$ and an explicit non-degenerate potential on this species.

\subsection{Once-punctured torus with one orbifold point}
\label{subsec:torus-1-orb-point}
Every triangulation of the empty-boundary torus with one puncture~$p$ and one orbifold point has the following form:
\begin{center}
  \begin{tikzpicture}[scale=0.4]
    \coordinate (0) at ( 0, 0);
    \coordinate (1) at (-3, 3);
    \coordinate (2) at (-3,-3);
    \coordinate (3) at ( 3,-3);
    \coordinate (4) at ( 3, 3);

    \node[rotate=45] at (0) {$\times$};
    \foreach \i in {1,...,4}
      \node at (\i) {$\bullet$};

    \draw (0) to (4);
    \draw (1) to node {\rotatebox{90}{$<$}} (2);
    \draw (2) to node {$\gg$} (3);
    \draw (3) to node {\rotatebox{90}{$<$}} (4);
    \draw (4) to node {$\gg$} (1);

    \draw[bend right] (2) to (4);
    \draw[bend left]  (2) to (4);
  \end{tikzpicture}
\end{center}
Hence, for any such triangulation $\tau$, every weighted quiver in the mutation class of $(\Qtau,\dtuple(\tau))$ is isomorphic to the following weighted quiver:
\[
 \xygraph{
   !{<0cm,0cm>;<1.5cm,0cm>:<0cm,1.5cm>::}
   !{( 0, 0)}*+{\bullet}="1"
   !{( 0,-1)}*+{\bullet}="2"
   !{(-1, 0)}*+{\bullet}="3"
   !{( 1, 0)}*+{\bullet}="4"
   !{( 0, 1)}*+{\bullet}="5"
   "1" :                 "3"  ^/-3pt/*-<2pt>{\labelstyle \gamma_1}
   "3" :                 "2"  _/0pt/*-<8pt>{\labelstyle \beta_1}
   "2" :@<-2pt>          "1"  ^/0pt/*+<2pt>{\labelstyle \alpha_1}
   "1" :                 "4"  _/-3pt/*-<2pt>{\labelstyle \gamma_2}
   "4" :                 "2"  ^/-1pt/*-<6pt>{\labelstyle \beta_2}
   "2" :@<+2pt>          "1"  _/0pt/*+<5pt>{\labelstyle \alpha_2}
   "3" :                 "5"  ^/0pt/*-<4pt>{\labelstyle \zeta}
   "5" :                 "4"  ^/0pt/*-<4pt>{\labelstyle \varepsilon}
   "4" :@<+3pt>@(lu,ru)  "3"  _/0pt/*-<3pt>{\labelstyle \delta}
 } \ \ \ \ \
 \xygraph{
   !{<0cm,0cm>;<1.5cm,0cm>:<0cm,1.5cm>::}
   !{( 0, 0)}*+{1}="1"
   !{( 0,-1)}*+{1}="2"
   !{(-1, 0)}*+{1}="3"
   !{( 1, 0)}*+{1}="4"
   !{( 0, 1)}*+{2}="5"
 }
\]
The potential associated with $\tau$ is~$\Stau = \alpha_1\beta_1\gamma_1 + \alpha_2\beta_2\gamma_2 + \delta\varepsilon(1-v)\zeta + x_p \alpha_1\beta_2\varepsilon\zeta\gamma_1\alpha_2\beta_1\delta\gamma_2$.

\section{Technical proofs}\label{sec:technical-proofs}

\subsection{On the proof of the Splitting Theorem}\label{subsec:proof-Splitting-Thm}

Every potential $\RA{A}$ is right-equivalent to a potential of the form
\begin{equation}
 \label{eq:potential-canonical-form}
 \sum_{k=1}^N (a_k b_k + a_k u_k + v_k b_k) + S'
\end{equation}
where
$a_1, b_1,\ldots,a_N, b_N$ are $2N$ different arrows of $Q$ such that each $a_k b_k$ is a $2$-cycle with $g_{a_k} g_{b_k} = \myid\in G_{h(a_k),t(a_k)}$, the elements $u_k$ and $v_k$ satisfy
\smash{$u_k \in \pi_{g_{b_k}} \big(e_{t(a_k)} {\mathfrak m}^2 e_{h(a_k)} \big)$}, \smash{$v_k \in \pi_{g_{a_k}} \big(e_{t(b_k)} {\mathfrak m}^2 e_{h(b_k)} \big)$} for all $k\in\{1,\ldots,N\}$, and $S' \in {\mathfrak m}^3$ is a possibly infinite $F$-linear combination of cyclic paths containing none of the arrows $a_k$ or $b_k$. Compare this with \cite[Equation (4.6)]{DWZ1}.

Since \smash{$u_k \in \pi_{g_{b_k}} \big(e_{t(a_k)} {\mathfrak m}^2 e_{h(a_k)} \big)$}, \smash{$v_k \in \pi_{g_{a_k}} \big(e_{t(b_k)} {\mathfrak m}^2 e_{h(b_k)} \big)$} for all $k\in\{1,\ldots,N\}$, the rule
$$
\psi(a_k) = a_k - v_k, \ \ \ \psi(b_k) = b_k - u_k \ \ \ \text{for $k\in\{1,\ldots,N\}$}, \ \ \ \psi(c)=c \ \ \ \text{for $c\in Q_1\setminus\{a_1,b_1,\ldots,a_N,b_N\}$},
$$
defines an $R$-algebra automorphism $\psi:\RA{A}\rightarrow\RA{A}$. Knowing this, the proof of \cite[Lemma~4.7]{DWZ1} can be applied here with no change to obtain the following.

\begin{lemma}
 \label{lem:splitting-theorem-existence}
 For any potential $S \in \RA{A}$ of the form~(\ref{eq:potential-canonical-form}) there exists an $R$-algebra automorphism~$\varphi$ of $\RA{A}$ such that $\varphi(S)$ is cyclically equivalent to a potential of the form~(\ref{eq:potential-canonical-form}) with $u_k = v_k = 0$ for all $k$.
\end{lemma}

The existence of $(A_{\operatorname{red}},S_{\operatorname{red}})$, $(A_{\operatorname{triv}},S_{\operatorname{triv}})$ and the right-equivalence~$\varphi$ in Theorem~\ref{thm:splitting-theorem} is a direct consequence of Lemma~\ref{lem:splitting-theorem-existence}.
Since the uniqueness of $(A_{\operatorname{triv}},S_{\operatorname{triv}})$ up to right-equivalence is clear, it remains to verify that the right-equivalence class of $(A_{\operatorname{red}},S_{\operatorname{red}})$ is determined by the right-equivalence class of $(A,S)$.
To achieve this goal, we formulate a version of the cyclic chain rule from~\cite[Lemma~3.9]{DWZ1} for our species setting.

Let us write $\RA{A} \,\widehat{\otimes}\, \RA{A} = \prod_{x,y \in \mathbb{Z}_{\geq 0}} (A^x \otimes_F A^y)$.
As a generalization of \cite[Section~3]{DWZ1} we define a continuous $R$-$R$-bimodule homomorphism $\Delta_a : \RA{A} \to \RA{A} \,\widehat{\otimes}\, \RA{A}$
by requiring $\Delta_a(R) = 0$ and by setting
\[
  \Delta_a(\omega_0 a_1 \omega_1 \cdots a_\ell \omega_\ell) \:=\: \sum_{k : a_k = a} \omega_0 \cdot a_1 \omega_1 \cdots a_{k-1} \omega_{k-1} \cdot \xi_a \cdot \omega_k a_{k+1} \cdots \omega_{\ell-1} a_\ell \cdot \omega_\ell
\]
for any path~$\omega_0 a_1 \omega_1 \cdots a_\ell \omega_\ell$ of length $\ell \geq 1$,
where $\xi_a = \frac{1}{d_{h(a),t(a)}} \sum_{\nu \in \B_{h(a),t(a)}} g_a(\nu) \otimes \nu^{-1}$.
Note that $\Delta_a$ is well-defined because we have $\xi_a x = g_a(x) \xi_a$ for all $x \in F_{h(a),t(a)}$.
Furthermore, we use the notation $f \sqr g$ for the image of $f \otimes g$ under the continuous $F$-linear map $(\RA{A} \,\widehat{\otimes}\, \RA{A}) \otimes_F \RA{A} \to \RA{A}$ defined by
$ (u \otimes w) \sqr g \:=\: wgu$
for $u, w \in R\langle{A}\rangle$ and $g \in \RA{A}$.

The following lemma is similar to~\cite[Lemma~3.8]{DWZ1}.

\begin{lemma}[Cyclic Leibniz rule]
  \label{lem:cyclic-leibniz-rule}
  For any finite sequence $(i_1,\ldots,i_\ell,i_{\ell+1})$ of vertices of $Q$ such that $i_{\ell+1}=i_1$, and for any elements~$f_k \in e_{i_k} \RA{A} e_{i_{k+1}}$, $k\in\{1,\ldots,\ell\}$, we have
  \[
    \partial_a(f_1 \cdots f_\ell) \:=\: \sum_{k=1}^\ell \Delta_a(f_k) \sqr f_{\widehat{k}} \,,
  \]
  where $f_{\widehat{k}} = f_{k+1} \cdots f_\ell f_1 \cdots f_{k-1}$.
  If, moreover, for all $k \in\{1,\ldots,\ell\}$ we have $f_k = \pi_{g_k}(f_k)$ for some $g_k \in G_{i_k,i_{k+1}}$, the same equality holds with $f_{\widehat{k}} = \pi_{g_k^{-1}}(f_{k+1} \cdots f_\ell f_1 \cdots f_{k-1})$.
\end{lemma}

\begin{proof}
 By $F$-linearity and continuity of $\partial_a$, $\Delta_a$ and $\sqr$, it suffices to prove the first statement when the $f_k$ are paths.
 Then it is obvious that $\Delta_a(f_1 \cdots f_\ell) = \sum_k f_1 \cdots f_{k-1} \cdot \Delta_a(f_k) \cdot f_{k+1} \cdots f_\ell$.
 Thus
 \[
  \begin{array}{lllll}
   \partial_a(f_1 \cdots f_\ell)
   &=&
   \Delta_a(f_1 \cdots f_\ell) \sqr 1
   &=&
   \sum_k (f_1 \cdots f_{k-1} \cdot \Delta_a(f_k) \cdot f_{k+1} \cdots f_\ell) \sqr 1
   \\ [0.5em]
   &&&=&
   \sum_k \Delta_a(f_k) \sqr f_{k+1} \cdots f_\ell f_1 \cdots f_{k-1} \,.
  \end{array}
 \]
 The last statement follows from the identity $\pi_{g_k}(\Delta_a(f_k)) \sqr f_{\widehat{k}} = \Delta_a(f_k) \sqr \pi_{g_k^{-1}}(f_{\widehat{k}})$ for $g_k \in G_{i_k,i_{k+1}}$.
\end{proof}

As an easy consequence of Lemma~\ref{lem:cyclic-leibniz-rule} we get the cyclic chain rule (cf.\ \cite[Lemma~3.9]{DWZ1}), which implies that the Jacobian ideal of any potential~$S \in \RA{A}$ is mapped onto the Jacobian ideal of~$\varphi(S)$ under any $R$-algebra automorphism~$\varphi$ of~$\RA{A}$ (cf.\ \cite[Proposition~3.7]{DWZ1}).

\begin{lemma}[Cyclic chain rule]
  Let $\varphi : \RA{A} \to \RA{A}$ be an $R$-algebra endomorphism.
  Then for any $a \in Q_1$ and for any potential $S \in \RA{A}$ we have
  \[
    \partial_a(\varphi(S)) \:=\: \sum_{b \in Q_1} \Delta_a(\varphi(b)) \sqr \varphi(\partial_b(S)) \,.
  \]
  In particular, $J(\varphi(S)) = \varphi(J(S))$ if $\varphi$ is an automorphism.
\end{lemma}

\begin{proof}
 By the $F$-linearity and continuity of the maps $\partial_a$, $\partial_b$, $\Delta_a$, $\sqr$ and $\varphi$, we may assume $S = \omega_0 a_1 \omega_1 \cdots a_\ell \omega_\ell$ is a path.
 With the notation $S_{\widehat{k}} = \omega_k a_{k+1} \cdots a_\ell \cdot \omega_\ell \omega_0 \cdot a_1 \cdots a_{k-1} \omega_{k-1}$ one has \smash{$\partial_b(S) = \pi_{g_b^{-1}} \big( \sum_{k:a_k =b} S_{\widehat{k}} \big)$}.
 By Lemma~\ref{lem:cyclic-leibniz-rule} and straightforward calculations (note that $\Delta_a(\varphi(\omega_k)) = 0$ for all $0 \leq k \leq \ell$) we obtain
 \[
  \begin{array}{llllll}
   \partial_a(\varphi(S))
   &=&
   \sum_{k=1}^\ell \Delta_a(\varphi(a_k)) \sqr \varphi \big( \pi_{g_{a_k}^{-1}} \big( S_{\widehat{k}} \big) \big)
   \\ [0.5em]
   &=&
   \sum_{b \in Q_1} \Delta_a(\varphi(b)) \sqr \varphi \big( \pi_{g_{b}^{-1}} \big( \sum_{k : a_k = b} S_{\widehat{k}} \big) \big)
   &=&
   \sum_{b \in Q_1} \Delta_a(\varphi(b)) \sqr \varphi(\partial_b(S)) \,.
  \end{array}
 \]
 ~\\ [-1.7em]
\end{proof}

With the Leibniz and cyclic chain rule at hand, it possible to see that the statements as well as the proofs of Propositions~4.9, 4.10 and Lemmas~4.11, 4.12 from \cite{DWZ1} remain correct word for word in our species setting.
In particular, the statement of Proposition~4.9 from \cite{DWZ1}, interpreted in our setting, implies the uniqueness up to right-equivalence of $(A_{\operatorname{red}},S_{\operatorname{red}})$ in Theorem~\ref{thm:splitting-theorem}.

\subsection{Proof of Theorem \ref{thm:SP-mut-well-defined-up-to-re}}\label{subsec:proof-SP-mut-well-defined-up-to-re}

The proof of Theorem \ref{thm:SP-mut-well-defined-up-to-re} is similar to the proof of \cite[Theorem 8.3]{LZ}, and just as the latter, it is based on the same main idea of Derksen-Weyman-Zelevinsky's proof of \cite[Theorem 5.2]{DWZ1}. However, it requires an extra consideration that, although elementary, may be considered non-obvious. This consideration was not present (nor necessary) neither in \cite{DWZ1} nor in \cite{LZ}, hence we have decided not to omit it here.

Fix vertices $i,j\in Q_0$. For $\rho\in G_{i,j}$,
and $\myid_{F_{i}}\in G_{i}$ we define functions $\pi_{\rho^{-1}}:e_i \RA{A} e_j\rightarrow e_i \RA{A} e_j$, $\pi_{\rho}:e_j \RA{A} e_i\rightarrow e_j \RA{A} e_i$
and $\pi_{\myid_{F_{i}}}:e_i \RA{A} e_i\rightarrow e_i \RA{A} e_i$ according to the rules
$$
\pi_{\rho^{-1}}(m)=\frac{1}{d_{i,j}}\sum_{\omega\in\B_{i,j}}\rho^{-1}(\omega^{-1})m\omega,
\ \ \
\pi_{\rho}(n)=\frac{1}{d_{i,j}}\sum_{\omega\in\B_{i,j}}\rho(\omega^{-1})n\omega,\ \ \
\pi_{\myid_{F_i}}(x) = \frac{1}{d_i}\sum_{\omega\in\B_i}\omega^{-1}x\omega.
$$
These functions are well defined since $F$ acts centrally on $e_i \RA{A} e_j$, $e_j \RA{A} e_i$ and $e_i \RA{A} e_i$, and since the product of any two elements of $\B_{i,j}$ (resp. $\B_i$) is an $F$-multiple of some element of $\B_{i,j}$ (resp. $\B_i$).
Compare to Part (2) of Proposition \ref{prop:properties-of-J-K-bimodules}.

Take $\rho_1,\rho_2\in G_{i,j}$, $y\in\image\pi_{\rho_{1}^{-1}}$, and $z\in\image\pi_{\rho_2}$. We have
$yzv_{i,j}=y\rho_2(v_{i,j})z=\rho_1^{-1}\rho_2(v_{i,j})yz=\zeta v_{i,j} yz$ for some $d^{\operatorname{th}}$ root of unity $\zeta\in F$. If $\zeta\neq 1$, a straightforward calculation shows that $\pi_{\myid_{F_i}}(yz)=0$. Furthermore, we have $\zeta\neq 1$ if and only if $\rho_1\neq\rho_2$. We deduce that $\pi_{\myid_{F_i}}(yz)=0$ if $\rho_1\neq\rho_2$. Therefore, for all $m\in e_i \RA{A} e_j$ and $n\in e_j \RA{A} e_i$ we have
\begin{equation}\label{eq:iterated-isotypical-projections}
\pi_{\myid_{F_i}}(mn)=\pi_{\myid_{F_i}}\left(\left(\sum_{\rho_1\in G_{i,j}}\pi_{\rho_1^{-1}}(m)\right)\left(\sum_{\rho_2\in G_{i,j}}\pi_{\rho_2}(n)\right)\right)=
\sum_{\rho\in G_{i,j}}\pi_{\myid_{F_i}}\left(\pi_{\rho^{-1}}(m)\pi_{\rho}(n)\right).
\end{equation}

Equation \eqref{eq:iterated-isotypical-projections} is part of the extra consideration mentioned at the beginning of Subsection \ref{subsec:proof-SP-mut-well-defined-up-to-re}.

Consider the $R$-$R$-bimodule
$   \widehat{A} \:=\: A \oplus (e_k A)^* \oplus (A e_k)^* \,$,
where $(e_k A)^*=\operatorname{Hom}_F(e_kA,F)$ and $(A e_k)^*=\operatorname{Hom}_F(Ae_k,F)$.
The inclusion of $A$ as a direct summand of $\widehat{A}$ identifies $\RA{A}$ with an $R$-subalgebra of $\RA{\widehat{A}}$.
Furthermore, the function
$\iota:\widetilde{\mu}_k(Q)_1\rightarrow\RA{\widehat{A}}$ defined by
$$
\iota(c)=\begin{cases}
\pi_{\rho}(b\omega a) & \text{if $c=[b\omega a]_{\rho}$ is a composite arrow of $\widetilde{\mu}_k(A)$;}\\
c & \text{if $c$ is not a composite arrow in $\widetilde{\mu}_k(A)$},
\end{cases}
$$
where $\pi_{\rho}(b\omega a)=\frac{1}{d_{h(b),t(a)}}\sum_{\nu\in\B_{h(b),t(a)}}\rho(\nu^{-1})b\omega a\nu$, identifies $\widetilde{\mu}_k(A)$ with an $R$-$R$-subbimodule of $\RA{\widehat{A}}$ and hence identifies $\RA{\widetilde{\mu}_k(A)}$ with an $R$-subalgebra of $\RA{\widehat{A}}$.
Henceforth, we view $\RA{A}$ and $\RA{\widetilde{\mu}_k(A)}$ as $R$-subalgebras of $\RA{\widehat{A}}$. Under these identifications we have
$
\widetilde{\mu}_k(S) = S + \Delta_k(A),
$
where
\begin{equation}
  \Delta_k(A)
  \::=\:
  \sum_{\overset{a}{\to}k\overset{b}{\to}} \frac{1}{|\B_{b,a}|} \sum_{\omega\in\B_{b,a}} \omega^{-1} b^* b \omega a a^*
  \,.
 \end{equation}

Notice that
 \[
  \frac{1}{|\B_{b,a}|} \sum_{\omega\in\B_{b,a}} \omega^{-1} b^* b \omega a a^*
  \:=\: \frac{1}{|\B_{k}||\B_{k}|} \sum_{u \in \B_{k}, \, \nu \in \B_{k}} \nu^{-1} u^{-1} b^* b u \nu a a^*
  \:\sim_{\operatorname{cyc}}\:
  \pi_{\myid_{F_k}}(b^* b) \pi_{\myid_{F_k}}(a a^*),
  \, .
 \]
 which means that if we set $\Delta_{k,\mathrm{out}} = \sum_{b \in Q_1 \cap A e_k} \pi_{\myid_{F_k}}(b^* b)$ and $\Delta_{k,\mathrm{in}} = \sum_{a \in Q_1 \cap e_k A} \pi_{\myid_{F_k}}(a a^*)$, then we have
 \begin{equation}
  \label{eq:delta_k-ce-to-delta_out-delta_in}
  \Delta_k(A)
  \:\sim_{\operatorname{cyc}}\:
  \Delta_{k,\mathrm{out}} \Delta_{k,\mathrm{in}}
  \,.
 \end{equation}
 The observation~(\ref{eq:delta_k-ce-to-delta_out-delta_in}) makes it clear that Theorem~\ref{thm:SP-mut-well-defined-up-to-re} is a consequence the following lemma.

 \begin{lemma}
  Every $R$-algebra automorphism~$\varphi$ of $\RA{A}$ can be extended to an $R$-algebra automorphism~$\overline{\varphi}$ of $\RA{\widehat{A}}$ satisfying \ \
  $\overline{\varphi} \big ( \RA{\widetilde{\mu}_k(A)} \big) = \RA{\widetilde{\mu}_k(A)}$, \ \   $\overline{\varphi}(\Delta_{k,\mathrm{out}}) \:=\: \Delta_{k,\mathrm{out}}$ \ \ and \ \
   $\overline{\varphi}(\Delta_{k,\mathrm{in}}) \:=\: \Delta_{k,\mathrm{in}}$.
 \end{lemma}

 \begin{proof}
  Set $\overline{\varphi}(c) = \varphi(c)$ for all $c \in Q_1$.
  This will ensure that $\overline{\varphi}$ is indeed an extension of $\varphi$.
  Now it only remains to define \smash{$\overline{\varphi}(a^*) \in \pi_{g_a^{-1}}(e_{t(a)} \mathfrak{m} e_{h(a)})$} and \smash{$\overline{\varphi}(b^*) \in \pi_{g_b^{-1}}(e_{t(b)} \mathfrak{m} e_{h(b)})$} for $a \in Q_1 \cap e_k A$ and $b \in Q_1 \cap A e_k$ such that $\overline{\varphi}$ is invertible and leaves the elements $\Delta_{k,\mathrm{out}}$ and $\Delta_{k,\mathrm{in}}$ and the subalgebra~\smash{$\RA{\widetilde{\mu}_k(A)}$} invariant.

  For $a \in Q_1 \cap e_k A$ and $b \in Q_1 \cap A e_k$ let $\widehat{\varphi}(a^*) \in e_{t(a)} \RA{\widetilde{\mu}_k(A)} e_{h(a)}$ and $\widehat{\varphi}(b^*) \in e_{t(b)} \RA{\widetilde{\mu}_k(A)} e_{h(b)}$ be the elements defined in the exact same way as in the proof of~\cite[Lemma~8.4]{LZ}.
  Then, the arguments given there guarantee the identities
  \begin{equation}
   \label{eq:phi-hat-leaves-delta_in-invariant}
   \sum_{a \in Q_1 \cap e_k A} \sum_{\omega \in \B_k} \omega a \widehat{\varphi}(a^*) \omega^{-1}
   \:=\:
   \sum_{a \in Q_1 \cap e_k A} \sum_{\omega \in \B_k} \omega a a^* \omega^{-1}
  \end{equation}
  \begin{equation}
   \label{eq:phi-hat-leaves-delta_out-invariant}
   \sum_{b \in Q_1 \cap A e_k} \sum_{\omega \in \B_k} \omega^{-1} \widehat{\varphi}(b^*) b \omega
   \:=\:
   \sum_{b \in Q_1 \cap A e_k} \sum_{\omega \in \B_k} \omega^{-1} b^* b \omega
  \end{equation}
  and they also imply that the map~\smash{$\widehat{\varphi}^{(1)}_{k}:(e_k A)^* \oplus (A e_k)^*\rightarrow (e_k A)^* \oplus (A e_k)^*$} given by \smash{$c^* \mapsto (\widehat{\varphi}(c^*))^{(1)}$} for $c \in Q_1 \cap (A e_k \cup e_k A)$ is an automorphism of the $R$-$R$-bimodule~\smash{$(e_k A)^* \oplus (A e_k)^*$}. In particular, \smash{$\pi_{g_c^{-1}}((\widehat{\varphi}_{k}(c^*))^{(1)}) = (\widehat{\varphi}_{k}(c^*))^{(1)}$}.

  Note that the identities (\ref{eq:phi-hat-leaves-delta_in-invariant}) and (\ref{eq:phi-hat-leaves-delta_out-invariant}) are stated as (8.15) and (8.20) in \cite{LZ} whereas the fact that \smash{$\widehat{\varphi}^{(1)}_{k}$} is invertible follows from the invertibility of the matrices $C_0$ and $D_0$ in \cite{LZ}.

  Let us now define \smash{$\overline{\varphi}(a^*) = \pi_{g_a^{-1}}(\widehat{\varphi}(a^*))$} and \smash{$\overline{\varphi}(b^*) = \pi_{g_b^{-1}}(\widehat{\varphi}(b^*))$} for $a \in Q_1 \cap e_k A$ and $b \in Q_1 \cap A e_k$.
  Then the invertibility of $\widehat{\varphi}^{(1)}_{k}$ and the observation that for $c \in Q_1 \cap (A e_k \cup e_k A)$ one has
  \[
   (\overline{\varphi}(c^*))^{(1)} = \pi_{g_c^{-1}}\big((\widehat{\varphi}(c^*))^{(1)}\big) = \widehat{\varphi}^{(1)}_{k}(c^*)
  \]
  imply that the $R$-$R$-bimodule endomorphism $\overline{\varphi}^{(1)} = \varphi^{(1)} \oplus \widehat{\varphi}^{(1)}_{k}$ of $\widehat{A} = A \oplus M$ is invertible.
  Thus $\overline{\varphi}$ is indeed an automorphism by Proposition~\ref{prop:automorphisms}.
  It is clear from the definition that
   $\overline{\varphi} \big ( \RA{\widetilde{\mu}_k(A)} \big) = \RA{\widetilde{\mu}_k(A)} \,$.
  So all that remains to be checked is $\overline{\varphi}(\Delta_{k,\mathrm{in}}) = \Delta_{k,\mathrm{in}}$ and $\overline{\varphi}(\Delta_{k,\mathrm{out}}) = \Delta_{k,\mathrm{out}}$.
  We will only verify the first of these identities because the proof of the other is similar.

  By definition of the projection~$\pi_{\myid_{F_k}}$ the identity (\ref{eq:phi-hat-leaves-delta_in-invariant}) can be rewritten as
  \[
   \sum_{a \in Q_1 \cap e_k A} \pi_{\myid_{F_k}} (a \widehat{\varphi}(a^*)) \:=\: \sum_{a \in Q_1 \cap e_k A} \pi_{\myid_{F_k}} (a a^*) \,.
  \]
  Applying (\ref{eq:iterated-isotypical-projections}) to each summand of the left hand side, one obtains
  \begin{equation}
   \label{eq:phi-bar-leaves-delta_in-invariant}
   \overline{\varphi}(\Delta_{k,\mathrm{in}})
   \:=\:
   \sum_{a \in Q_1 \cap e_k A} \pi_{\myid_{F_k}} (a \overline{\varphi}(a^*))
   \:=\:
   \sum_{a \in Q_1 \cap e_k A} \pi_{\myid_{F_k}} (a a^*)
   \:=\:
   \Delta_{k,\mathrm{in}}
   \ \,.
  \end{equation}
  ~\\ [-1em]
 \end{proof}

\subsection{Proof of Theorem \ref{thm:ideal-non-pending-flips<->SP-mutation}}\label{subsec:proof-ideal-non-pending-flips<->SP-mutation}

The proof will be achieved through a case-by-case check. We shall start with a description of all the cases and with a couple of considerations that will allow us to focus in a relatively small number of cases. In all cases, the general strategy will consist in identifying $(\Qsigma,\dtuple(\sigma))$ with a weighted subquiver of $\widetilde{\mu}_k(\Qtau,\dtuple(\tau))$ and exhibiting a right equivalence between $(\widetilde{\mu}_k(\Atau),\widetilde{\mu}_k(\Stau))$ and an SP $(\widetilde{\mu}_k(\Atau),\Ssigma^\sharp)$ that has $\ASsigma$ as reduced part (recall that $\widetilde{\mu}_k(\Atau)$ is the species of $(\widetilde{\mu}_k(\Qtau),\dtuple(\tau),\widetilde{\mu}_k(g(\tau)))$, see Definition \ref{def:SP-premutation}). The uniqueness of reduced parts up to right equivalence (cf. Theorem \ref{thm:splitting-theorem}) will then imply that $\mu_k\AStau$ is right-equivalent to $\ASsigma$ as desired.

Consider a puzzle-piece decomposition of $\tau$. In order to prove Theorem \ref{thm:ideal-non-pending-flips<->SP-mutation} it suffices to verify that its statement is true in the cases where $k$ is either an arc shared by two puzzle pieces, or a non-folded non-pending arc inside one of the puzzle pieces II, III, VI, VII from Figure \ref{Fig:puzzle_pieces_2}. The full set of possibilities for an arc shared by two puzzle pieces is given in Figure \ref{Fig:all_possible_matchings_2} (that this is the full set of possibilities can be seen by analyzing all possible ways of gluing outer sides of pairs of puzzle pieces).
        \begin{figure}[!ht]
                \centering
                \includegraphics[scale=.5]{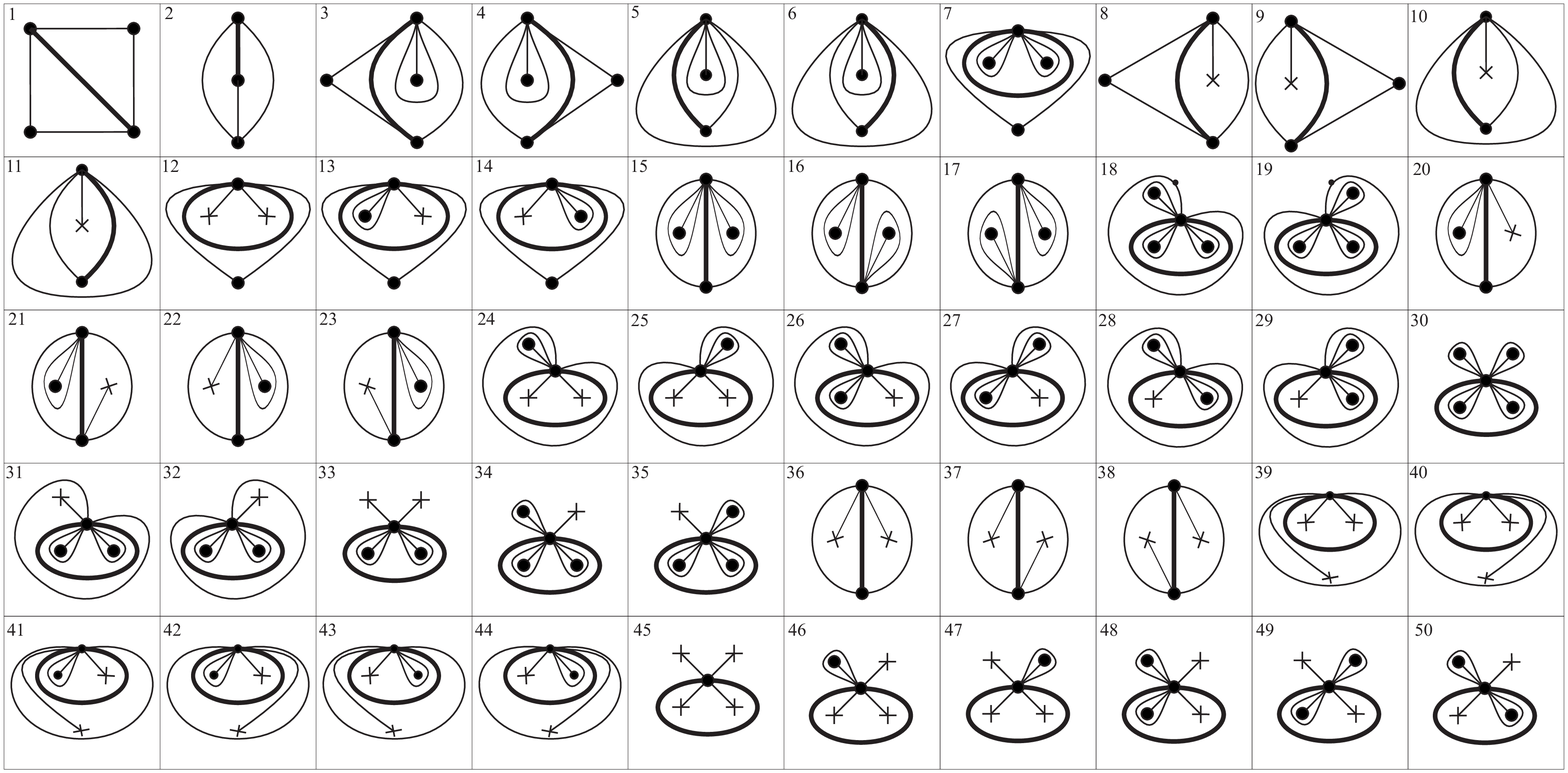}
                \caption{}\label{Fig:all_possible_matchings_2}
        \end{figure}
Note that in the configurations with numbers 2, 5, 6, 10, 11 in Figure \ref{Fig:all_possible_matchings_2}, the flip of the indicated arc results in a puzzle piece from Figure \ref{Fig:puzzle_pieces_2}.

Since we are assuming that $|\marked\cup\orb|>5$ in case $\Sigma$ is a (closed) sphere, we can, and will, ignore the configurations 30, 33, 34, 35, 45, 46, 47, 48, 49 and 50 in Figure \ref{Fig:all_possible_matchings_2}. Amongst the remaining configurations we notice the following flip correspondences:
\begin{equation}\label{eq:matching-cases}
\begin{array}{ccccc}
1 \leftrightarrow 1, &
3 \leftrightarrow 4, &
7 \leftrightarrow 15, &
8 \leftrightarrow 9, &
12 \leftrightarrow 36,\\
13 \leftrightarrow 20, &
14 \leftrightarrow 22, &
16 \leftrightarrow 17, &
18 \leftrightarrow 19, &
21 \leftrightarrow 23,\\
24 \leftrightarrow 42,&
25 \leftrightarrow 43,&
26 \leftrightarrow 32,&
27 \leftrightarrow 28,&
29 \leftrightarrow 31,\\
37 \leftrightarrow 38,&
39 \leftrightarrow 40,&
41 \leftrightarrow 44& &
\end{array}
\end{equation}

Thus, since flips and SP-mutations are involutive (the latter up to right-equivalence, cf.\ Theorem \ref{thm:SP-mutation-is-involution}), and since SP-mutations of right-equivalent SPs are right-equivalent (Theorem \ref{thm:SP-mut-well-defined-up-to-re}), in order to prove that the statement of Theorem \ref{thm:ideal-non-pending-flips<->SP-mutation} is true in the cases where $k$ is either an arc shared by two puzzle pieces, or a non-folded non-pending arc inside one of the puzzle pieces II, III, VI and VII from Figure \ref{Fig:puzzle_pieces_2}, it suffices to verify the statement of Theorem \ref{thm:ideal-non-pending-flips<->SP-mutation} for the flips of non-folded non-pending arcs inside the puzzle pieces II, III, VI and VII, and for the flips indicated in the configurations 1, 3, 7, 8, 13, 14, 16, 18, 21, 24, 25, 26, 27, 29, 36, 37, 39 and 41 from Figure \ref{Fig:all_possible_matchings_2}.

It is easy to see that if $k$ is a non-folded arc inside one of the puzzle pieces II and III, or if $k$ is the bold arc in one of the configurations 1, 3, 7, 16 and 18, then the computations made in \cite{Labardini1} to prove \cite[Theorem 30]{Labardini1} (which is nothing but the particular case $\orb=\varnothing$ of Theorem \ref{thm:ideal-non-pending-flips<->SP-mutation}) can be applied \emph{as are} in our current, more general, situation to show that
$\mu_k\AStau$ and $\ASsigma$ are right-equivalent. Thus, in order to prove Theorem \ref{thm:ideal-non-pending-flips<->SP-mutation} we can restrict our attention to the flips of non-folded non-pending arcs inside the puzzle pieces VI and VII, and to the flips indicated in the configurations 8, 13, 14, 21, 24, 25, 26, 27, 29, 36, 37, 39 and 41 from Figure \ref{Fig:all_possible_matchings_2}. These are the cases we shall check in a few moments.

\setcounter{case}{0}
\begin{case}\label{case:VI}\emph{Flip inside puzzle piece VI}. We adopt the notation from Figure \ref{Fig:flip_VI}.
\begin{figure}[!ht]
                \centering
                \includegraphics[scale=.5]{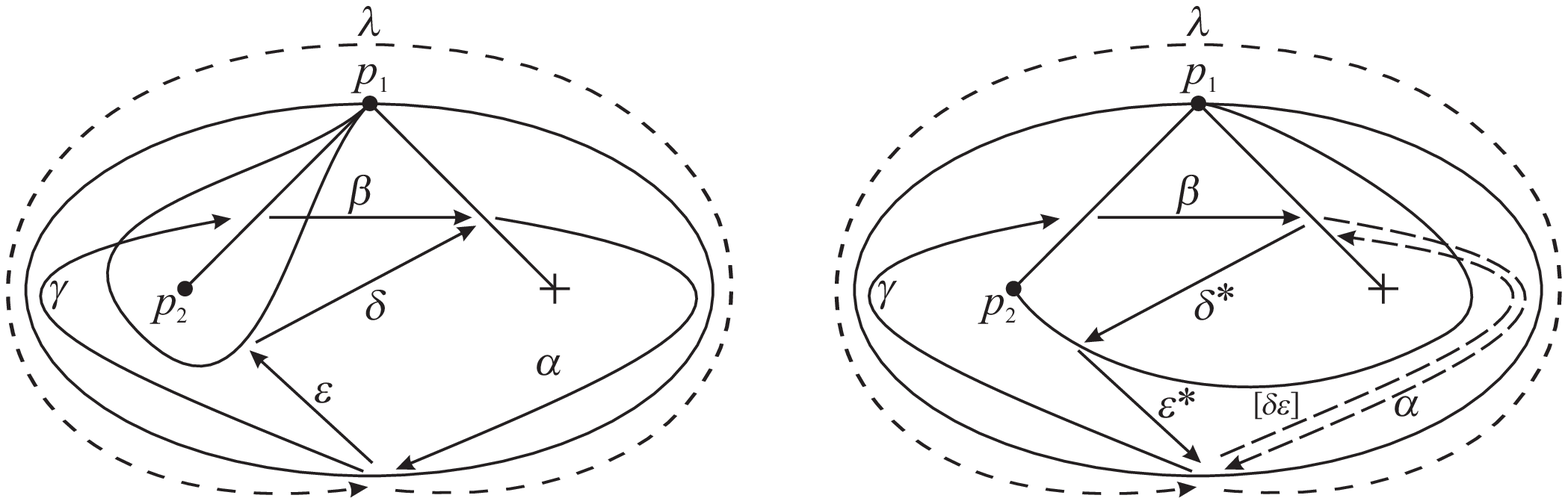}
                \caption{Flip of the unique non-pending non-folded arc inside puzzle piece VI.\ \ \ \ \ \ \ \ \ \ \ \ \ \ \ \ \ \ \ \
                Left: $\tau$ and $\Qtau$.
                Right: $\sigma$ and $\widetilde{\mu}_k(\Qtau)$.}
                \label{Fig:flip_VI}
        \end{figure}
Note that $\ASsigma$ is the reduced part of $(\widetilde{\mu}_k(\Atau),\Ssigma^\sharp)$, where
\begin{eqnarray}\nonumber
\Ssigma^\sharp  & = &
\alpha[\delta\varepsilon]+
-x_{p_2}^{-1}\delta^*\beta \gamma\varepsilon^*+x_{p_2}^{-1}\delta^*v\beta \gamma\varepsilon^*+x_{p_1}\varepsilon^*\delta^*\beta\gamma\lambda
+S(\tau,\sigma)
\in\RA{\widetilde{\mu}_k(\Atau)},
\end{eqnarray}
with $S(\tau,\sigma)\in\RA{\Atau}\cap\RA{\Asigma}$. 
Furthermore,
\begin{eqnarray}
\nonumber
\Stau & = &
\alpha\delta\varepsilon - \alpha v\delta\varepsilon
-x_{p_2}^{-1}\alpha\beta\gamma + x_{p_2}^{-1}\alpha v\beta\gamma
+
x_{p_1}\alpha\beta\gamma\lambda
+S(\tau,\sigma)\\
\nonumber
\text{and} \ \ \ \ \widetilde{\mu}_k(\Stau)
& \sim_{\operatorname{cyc}} &
[\delta\varepsilon]\alpha -  v[\delta\varepsilon]\alpha
-x_{p_2}^{-1}\alpha\beta\gamma + x_{p_2}^{-1}\alpha v\beta\gamma
+
x_{p_1}\alpha\beta\gamma\lambda
+
[\delta\varepsilon]\varepsilon^*\delta^*
+S(\tau,\sigma).
\end{eqnarray}

Define $R$-algebra automorphisms $\psi,\varphi,\Phi:\RA{\widetilde{\mu}_k(\Atau)}\rightarrow\RA{\widetilde{\mu}_k(\Atau)}$ according to the rules
\begin{center}
\begin{tabular}{ccll}
$\psi$ & : &
$[\delta\varepsilon]\mapsto\left(\frac{1+v}{2}\right)[\delta\varepsilon]$ &\\
$\varphi$ &:&
$[\delta\varepsilon]\mapsto[\delta\varepsilon]
+x_{p_2}^{-1}\beta\gamma - x_{p_2}^{-1} v\beta\gamma-x_{p_1}\beta\gamma\lambda$,&
$\alpha\mapsto\alpha-\varepsilon^*\delta^*\left(\frac{1+v}{2}\right)$\\
$\Phi$ &:&
$\delta^*\mapsto\delta^*(1-v)$,&
$\varepsilon^*\mapsto-\varepsilon^*$.
\end{tabular}
\end{center}

Direct computation shows that
the composition $\Phi\varphi\psi$ is a right-equivalence $(\widetilde{\mu}_k(\Atau),\widetilde{\mu}_k(\Stau))\rightarrow(\widetilde{\mu}_k(\Atau),\Ssigma^\sharp)$. Whence the reduced parts of these two SPs are right-equivalent. But these reduced parts are precisely $\mu_k\AStau$ and $\ASsigma$, respectively.
\end{case}

\begin{case}\emph{Flip inside puzzle piece VII}. Similar to Case \ref{case:VI}.
\end{case}

\begin{case}\label{case:flip-8}\emph{Flip of configuration 8}.
Configuration 8 has been drawn in Figure \ref{Fig:all_possible_matchings_2} as a planar figure; more precisely, as a planar disc with 3 marked points on its boundary and 1 orbifold point. However, different marked points on the boundary of the disc may actually represent the same puncture in $\surf$. Thus, when proving that the statement of Theorem \ref{thm:ideal-non-pending-flips<->SP-mutation} holds for Configuration 8, we have to consider all the possible ways of identifying the 3 different marked points. This consideration is reflected in Figure \ref{Fig:flip_8},
where the 6 possible identifications are drawn. The reader can easily see that the terms contributed to the potentials $\Stau$ and $\Ssigma$ by the 3 marked points on the boundary of Configuration 8 really depend on how these marked points are identified on the surface $\surf$. The reader can also see that the terms contributed to $\Stau$ and $\Ssigma$ by the ideal triangles contained in Configuration 8  are independent from the identifications. Consequently, we will write each of the potentials $\Stau$ and $\Ssigma^\sharp$ as a sum $P+P'$, where $P$ is independent of the identifications and $P'$ takes all possible identifications into account. To be able to encode in $P'$ all 6 identifications, we will write $P'$ as a sum $Y_1P_1''+Y_2P_2''+Y_3P_3''+Y_4P_4''+Y_5P_5''+Y_6P_6''$, where each $P_\ell''$ is determined by an identification, and $Y_1,Y_2,Y_3,Y_4,Y_5,Y_6\in\{0,1\}\subseteq F$ are scalars with the property that exactly one of them is equal to $1$ and all the rest are equal to $0$.

We warn the reader that a similar consideration about identifications will be (tacitly) made in Cases \ref{case:13}, \ref{case:flip-21}, \ref{case:flip-36} and \ref{case:flip-37} below, which respectively deal with the configurations  13, 
21, 36 and 37 of Figure \ref{Fig:all_possible_matchings_2}. In those cases, the potentials will be directly written in the form $P+P'$ as above, but no further explanation or apology will be given.

In the case currently under analysis, with the notation from Figure \ref{Fig:flip_8} we see that $\ASsigma$ is the reduced part of $(\widetilde{\mu}_k(\Atau),\Ssigma^\sharp)$, where
\begin{eqnarray}\nonumber
\Ssigma^\sharp  & = &  [\varepsilon\eta]\delta+[\gamma\alpha]\beta+
[\varepsilon\alpha]\alpha^*\varepsilon^*+[\gamma\eta]\eta^*\gamma^*-[\gamma\eta]\eta^*\gamma^*v\\
\nonumber
& + & Y_1(x_{p_1}\alpha^*\gamma^*[\gamma\eta]\nu + x_{p_2}\eta^*\varepsilon^*\lambda + x_{p_3}[\varepsilon\alpha]\rho)
\ \ + \ \ Y_2(x_{p_1} \alpha^*\gamma^*[\gamma\eta]\lambda\eta^*\varepsilon^*\nu + x_{p_3}[\varepsilon\alpha]\rho)\\
\nonumber
& + & Y_3(x_{p_1}\alpha^*\gamma^*[\gamma\eta]\nu[\varepsilon\alpha]\rho +x_{p_2}\eta^*\varepsilon^*\lambda)
\ \ + \ \ Y_4(x_{p_1}\alpha^*\gamma^*[\gamma\eta]\nu + x_{p_2}\eta^*\varepsilon^*\lambda[\varepsilon\alpha]\rho )\\
\nonumber
& + & Y_5(x_{p_1}\alpha^*\gamma^*[\gamma\eta]\lambda[\varepsilon\alpha]\nu\eta^*\varepsilon^*\rho)
\ \ + \ \ Y_6(x_{p_1}\alpha^*\gamma^*[\gamma\eta]\rho\eta^*\varepsilon^*\lambda[\varepsilon\alpha]\nu)
\ \ +\ \ S(\tau,\sigma)
\in\RA{\widetilde{\mu}_k(\Atau)},
\end{eqnarray}
with $S(\tau,\sigma)\in\RA{\Atau}\cap\RA{\Asigma}$. 
\begin{figure}[!ht]
                \centering
                \includegraphics[scale=.3]{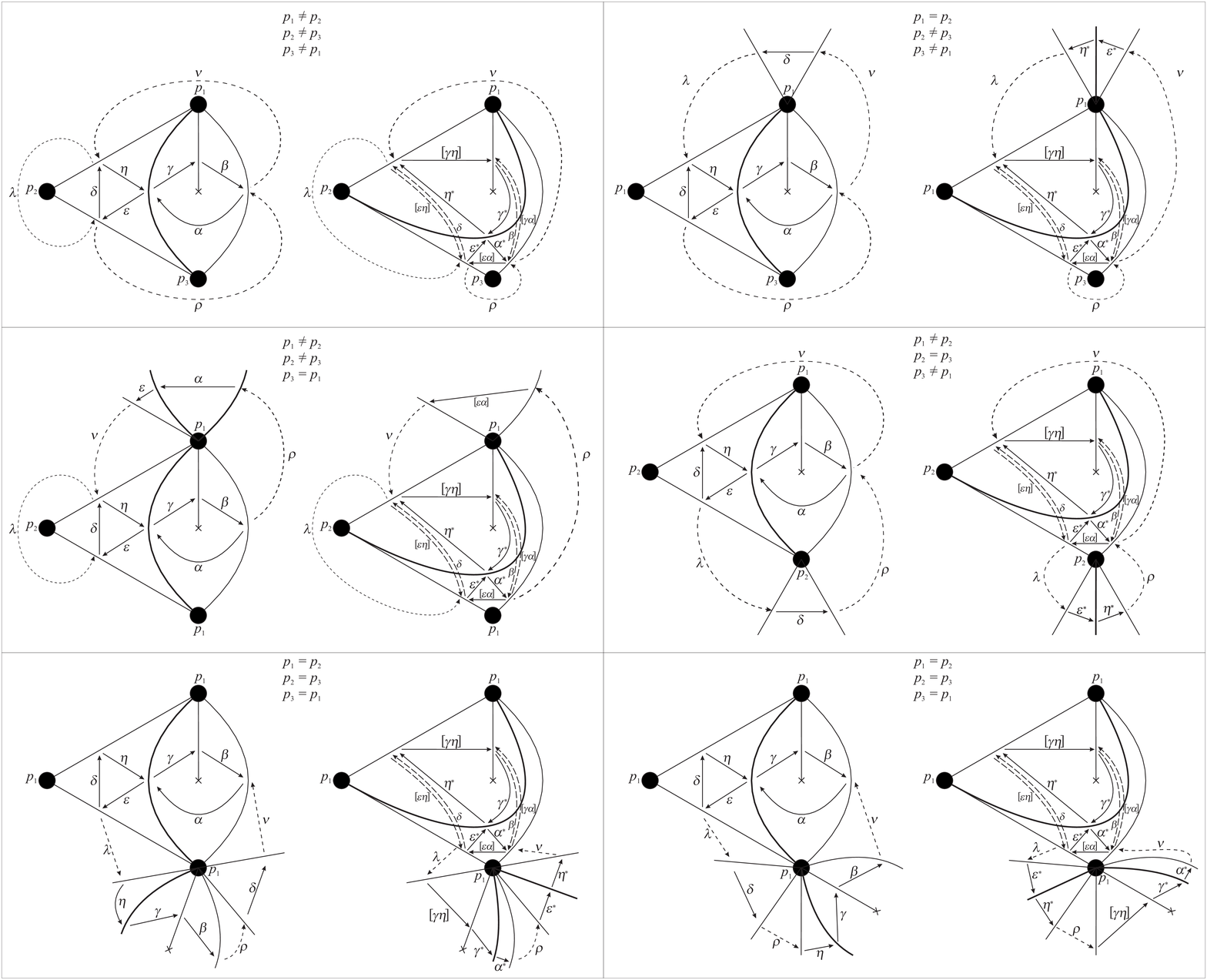}
                 \caption{Flip of the $8^{\operatorname{th}}$ configuration of Figure \ref{Fig:all_possible_matchings_2}.
                 Here we see the six possible ways to identify the three marked points in the boundary of Configuration 8. Inside each rectangle, $\tau$ and $\Qtau$ appear on the left, and $\sigma$ and $\widetilde{\mu}_k(\Qtau)$ appear on the right.}
                \label{Fig:flip_8}
        \end{figure}
Furthermore,
\begin{eqnarray}\nonumber
\Stau & = &
\varepsilon\eta\delta + \alpha\beta\gamma -\alpha\beta v\gamma\\
\nonumber
& + & Y_1(x_{p_1}\beta\gamma\eta\nu + x_{p_2}\delta\lambda + x_{p_3}\varepsilon\alpha\rho)
 \ \ + \ \ Y_2(x_{p_1}\beta\gamma\eta\lambda\delta\nu+x_{p_3}\varepsilon\alpha\rho)\\
\nonumber
& + & Y_3(x_{p_1}\beta\gamma\eta\nu\varepsilon\alpha\rho+x_{p_2}\delta\lambda)
\ \ + \ \ Y_4(x_{p_1}\beta\gamma\eta\nu+x_{p_2}\varepsilon\alpha\rho\delta\lambda)\\
\nonumber
& + & Y_5(x_{p_1}\beta\gamma\eta\lambda\varepsilon\alpha\nu\delta\rho)
\ \ + \ \ Y_6(x_{p_1}\beta\gamma\eta\rho\delta\lambda\varepsilon\alpha\nu)
\ \ +\ \ S(\tau,\sigma)\\
\nonumber
\text{and} \ \ \ \ \
\widetilde{\mu}_k(\Stau) & = &
[\varepsilon\eta]\delta + [\gamma\alpha]\beta -[\gamma\alpha]\beta v\\
\nonumber
& + & Y_1(x_{p_1}\beta[\gamma\eta]\nu + x_{p_2}\delta\lambda + x_{p_3}[\varepsilon\alpha]\rho)
\ \  + \ \  Y_2(x_{p_1}\beta[\gamma\eta]\lambda\delta\nu+x_{p_3}[\varepsilon\alpha]\rho)\\
\nonumber
& + & Y_3(x_{p_1}\beta[\gamma\eta]\nu[\varepsilon\alpha]\rho+x_{p_2}\delta\lambda)
\ \  + \ \ Y_4(x_{p_1}\beta[\gamma\eta]\nu+x_{p_2}[\varepsilon\alpha]\rho\delta\lambda)\\
\nonumber
& + & Y_5(x_{p_1}\beta[\gamma\eta]\lambda[\varepsilon\alpha]\nu\delta\rho)
\ \  + \ \ Y_6(x_{p_1}\beta[\gamma\eta]\rho\delta\lambda[\varepsilon\alpha]\nu)\\
\nonumber
& + & [\varepsilon\eta]\eta^*\varepsilon^*+
[\gamma\alpha]\alpha^*\gamma^*+
[\varepsilon\alpha]\alpha^*\varepsilon^*+
[\gamma\eta]\eta^*\gamma^*
+S(\tau,\sigma).
\end{eqnarray}

Define $R$-algebra automorphisms $\psi,\varphi_1,\varphi_2,\Phi:\RA{\widetilde{\mu}_k(\Atau)}\rightarrow\RA{\widetilde{\mu}_k(\Atau)}$ according to the rules
\begin{center}
{\renewcommand{\arraystretch}{1.125}
\begin{tabular}{ccl}
$\psi$ & : & $\beta\mapsto \beta\left(\frac{1+v}{2}\right)$,\\
$\varphi_1$ & : &
$[\varepsilon\eta]\mapsto[\varepsilon\eta]-\left(Y_1( x_{p_2}\lambda)+Y_2(x_{p_1}\nu\beta\left(\frac{1+v}{2}\right)[\gamma\eta]\lambda) +  Y_3(x_{p_2}\lambda) +  Y_4(x_{p_2}\lambda[\varepsilon\alpha]\rho)\right.$ \\
&&
 $\left. + Y_5(x_{p_1}\rho\beta\left(\frac{1+v}{2}\right)[\gamma\eta]\lambda[\varepsilon\alpha]\nu)
 +  Y_6(x_{p_1}\lambda[\varepsilon\alpha]\nu\beta\left(\frac{1+v}{2}\right)[\gamma\eta]\rho)\right)$, \\
&&$\delta\mapsto\delta-\eta^*\varepsilon^*$,\\
&&$[\gamma\alpha]\mapsto[\gamma\alpha]-\left(Y_1(x_{p_1}\left(\frac{1+v}{2}\right)[\gamma\eta]\nu) +  Y_3(x_{p_1}\left(\frac{1+v}{2}\right)[\gamma\eta]\nu[\varepsilon\alpha]\rho) +  Y_4(x_{p_1}\left(\frac{1+v}{2}\right)[\gamma\eta]\nu)\right)$,\\
&&$\beta\mapsto\beta-\alpha^*\gamma^*$,\\
$\varphi_2$ & : &
$[\varepsilon\eta]\mapsto[\varepsilon\eta]-\left( Y_2(-x_{p_1}\nu\alpha^*\gamma^*\left(\frac{1+v}{2}\right)[\gamma\eta]\lambda )
 +
Y_5(-x_{p_1}\rho\alpha^*\gamma^*\left(\frac{1+v}{2}\right)[\gamma\eta]\lambda[\varepsilon\alpha]\nu)\right.$ \\
&& $\left.+Y_6(-x_{p_1}\lambda[\varepsilon\alpha]\nu\alpha^*\gamma^*\left(\frac{1+v}{2}\right)[\gamma\eta]\rho)\right)$,\\
&&
$[\gamma\alpha]\mapsto[\gamma\alpha]-\left(
   Y_2(-x_{p_1}\left(\frac{1+v}{2}\right)[\gamma\eta]\lambda\eta^*\varepsilon^*\nu )
   +
Y_5(-x_{p_1}\left(\frac{1+v}{2}\right)[\gamma\eta]\lambda[\varepsilon\alpha]\nu\eta^*\varepsilon^*\rho)\right.$\\
&& $\left.+  Y_6(-x_{p_1}\left(\frac{1+v}{2}\right)[\gamma\eta]\rho\eta^*\varepsilon^*\lambda[\varepsilon\alpha]\nu)\right)$,\\
$\Phi$ &:&
$\gamma^*\mapsto\gamma^*(1-v), \ \ \ \alpha^*\mapsto-\alpha^*, \ \ \ \varepsilon^*\mapsto-\varepsilon^*$.
\end{tabular}
}
\end{center}

Direct computation shows that
the composition $\Phi\varphi_2\varphi_1\psi$ is a right-equivalence $(\widetilde{\mu}_k(\Atau),\widetilde{\mu}_k(\Stau))\rightarrow(\widetilde{\mu}_k(\Atau),\Ssigma^\sharp)$. Whence the reduced parts of these two SPs are right-equivalent. But these reduced parts are precisely $\mu_k\AStau$ and $\ASsigma$, respectively.
\end{case}

\begin{case}\label{case:13}\emph{Flip of configuration 13}. See the discussion in the first and second paragraphs of Case \ref{case:flip-8}. We adopt the notation from Figure \ref{Fig:flip_13}.
Note that $\ASsigma$ is the reduced part of $(\widetilde{\mu}_k(\Atau),\Ssigma^\sharp)$, where
\begin{eqnarray}\nonumber
\Ssigma^\sharp  & = &
[\varepsilon\lambda]\eta  + [\kappa\alpha]\delta + [\gamma\alpha]\beta
+ [\kappa\lambda]\lambda^*\kappa^* + [\varepsilon\alpha]\alpha^*\varepsilon^* - [\varepsilon\alpha]v\alpha^*\varepsilon^*- x_{p_3}^{-1}[\gamma\lambda]\lambda^*\gamma^*\\
\nonumber
&+& Y_1(x_{p_1}[\varepsilon\alpha]\alpha^*\gamma^*[\gamma\lambda]\nu+x_{p_2}\lambda^*\varepsilon^*\rho)
\ \ + \ \ Y_2(x_{p_1}[\varepsilon\alpha]\alpha^*\gamma^*[\gamma\lambda]\nu\lambda^*\varepsilon^*\rho)
\ \ +\ \ S(\tau,\sigma)
\in\RA{\widetilde{\mu}_k(\Atau)},
\end{eqnarray}
with $S(\tau,\sigma)\in\RA{\Atau}\cap\RA{\Asigma}$. 
Here, $Y_1$ and $Y_2$ are elements of $F$ with the property that exactly one of them is equal to $1$ and exactly one is equal to $0$.
Furthermore,
\begin{eqnarray}\nonumber
\Stau & = &
\varepsilon\lambda\eta + \kappa\alpha\delta-\kappa\alpha v\delta
+ x_{p_3}^{-1}\gamma\alpha v\beta
- x_{p_3}^{-1}\gamma\alpha\beta\\
\nonumber
&+& Y_1(x_{p_1}\varepsilon\alpha\beta\gamma\lambda\nu+x_{p_2}\eta\rho)
\ \ + \ \ Y_2(x_{p_1}\varepsilon\alpha\beta\gamma\lambda\nu\eta\rho)
\ \ + \ \ S(\tau,\sigma)\\
\nonumber
\text{and} \ \ \ \ \
\widetilde{\mu}_k(\Stau) & = &
[\varepsilon\lambda]\eta + [\kappa\alpha]\delta-[\kappa\alpha] v\delta
+ x_{p_3}^{-1}[\gamma\alpha] v\beta
- x_{p_3}^{-1}[\gamma\alpha]\beta\\
\nonumber
&+& Y_1(x_{p_1}[\varepsilon\alpha]\beta[\gamma\lambda]\nu+x_{p_2}\eta\rho)
\ \ + \ \ Y_2(x_{p_1}[\varepsilon\alpha]\beta[\gamma\lambda]\nu\eta\rho)\\
\nonumber
&+& [\varepsilon\lambda]\lambda^*\varepsilon^*
+ [\kappa\alpha]\alpha^*\kappa^*
+[\gamma\alpha]\alpha^*\gamma^*
+[\varepsilon\alpha]\alpha^*\varepsilon^*
+[\gamma\lambda]\lambda^*\gamma^*
+[\kappa\lambda]\lambda^*\kappa^*
+S(\tau,\sigma).
\end{eqnarray}
\begin{figure}[!ht]
                \centering
                \includegraphics[scale=.5]{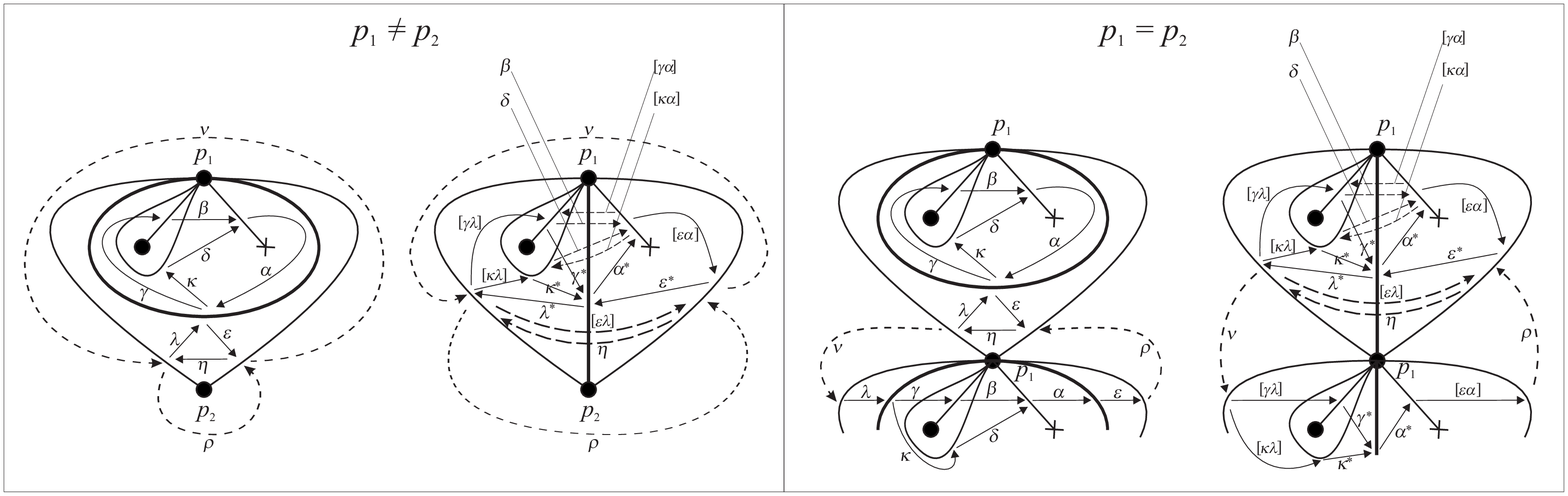}
                \caption{Flip of the $13^{\operatorname{th}}$ configuration of Figure \ref{Fig:all_possible_matchings_2}. Here we see the two possible ways to identify the two marked points in the boundary of Configuration 13. Inside each rectangle, $\tau$ and $\Qtau$ appear on the left, and $\sigma$ and $\widetilde{\mu}_k(\Qtau)$ appear on the right.}
                \label{Fig:flip_13}
        \end{figure}
Define $R$-algebra automorphisms $\psi,\varphi_1,\varphi_2,\Phi:\RA{\widetilde{\mu}_k(\Atau)}\rightarrow\RA{\widetilde{\mu}_k(\Atau)}$ according to the rules
\begin{center}
{\renewcommand{\arraystretch}{1.125}
\begin{tabular}{ccll}
$\psi$ &:& $[\kappa\alpha]\mapsto[\kappa\alpha]\left(\frac{1+v}{2}\right)$, &$[\gamma\alpha]\mapsto[\gamma\alpha]\left(\frac{-1-v}{2}\right)$, \\
$\varphi_1$ & : &
$[\varepsilon\lambda]\mapsto[\varepsilon\lambda]-Y_1x_{p_2}\rho-Y_2x_{p_1}\rho[\varepsilon\alpha]\beta[\gamma\lambda]\nu$,&
$\eta\mapsto\eta-\lambda^*\varepsilon^*$, \\
&&
$\delta\mapsto\delta-\left(\frac{1+v}{2}\right)\alpha^*\kappa^*$, &
$[\gamma\alpha]\mapsto [\gamma\alpha]-Y_1x_{p_1}x_{p_3}[\gamma\lambda]\nu[\varepsilon\alpha]$,\\
&&
$\beta\mapsto\beta-x_{p_3}\left(\frac{-1-v}{2}\right)\alpha^*\gamma^*$,\\
$\varphi_2$ &:&
$[\varepsilon\lambda]\mapsto[\varepsilon\lambda]+Y_2x_{p_1}x_{p_3}\rho[\varepsilon\alpha]\left(\frac{-1-v}{2}\right)\alpha^*\gamma^*[\gamma\lambda]\nu$,&
$[\gamma\alpha]\mapsto [\gamma\alpha]+Y_2x_{p_1}x_{p_3}[\gamma\lambda]\nu\lambda^*\varepsilon^*\rho[\varepsilon\alpha]$, \\
$\Phi$ &:&
$[\gamma\alpha]\mapsto x_{p_3}[\gamma\alpha], \ \ \
\lambda^*\mapsto -\lambda^*, \ \ \
[\gamma\lambda]\mapsto x_{p_3}^{-1}[\gamma\lambda]$,
& $[\varepsilon\alpha]\mapsto[\varepsilon\alpha](1-v), \ \ \
[\kappa\lambda]\mapsto-[\kappa\lambda]$.
\end{tabular}
}
\end{center}

Direct computation shows that
the composition $\Phi\varphi_2\varphi_1\psi$ is a right-equivalence $(\widetilde{\mu}_k(\Atau),\widetilde{\mu}_k(\Stau))\rightarrow(\widetilde{\mu}_k(\Atau),\Ssigma^\sharp)$. Whence the reduced parts of these two SPs are right-equivalent. But these reduced parts are precisely $\mu_k\AStau$ and $\ASsigma$, respectively.
\end{case}

\begin{case}\emph{Flip of configuration 14}. Similar to Case \ref{case:13}.
\end{case}

\begin{case}\label{case:flip-21}\emph{Flip of configuration 21}. See the discussion in the first and second paragraph of Case \ref{case:flip-8}. We adopt the notation from Figure \ref{Fig:flip_21}.
\begin{figure}[!ht]
                \centering
                \includegraphics[scale=.65]{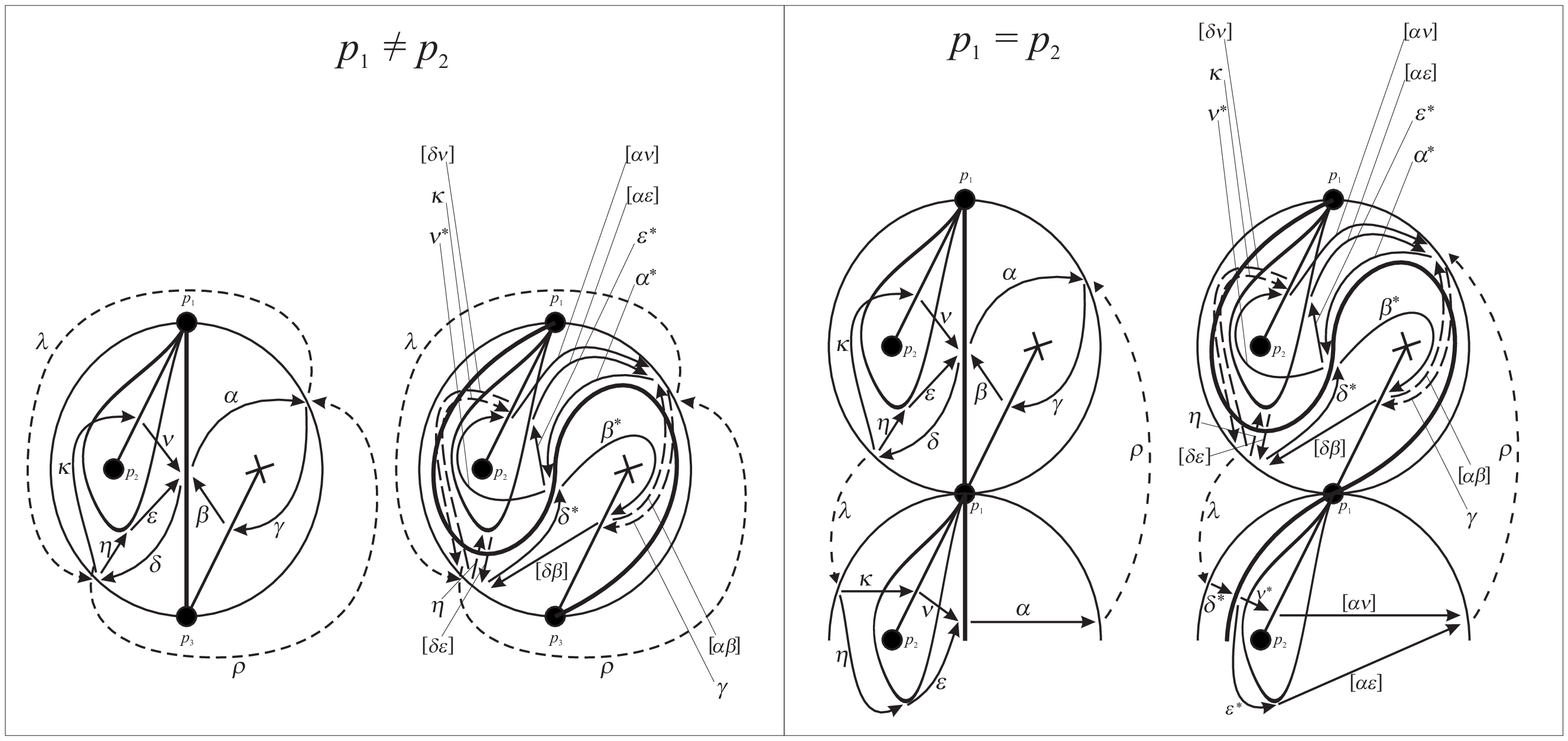}
                \caption{Flip of the $21^{\operatorname{st}}$ configuration of Figure \ref{Fig:all_possible_matchings_2}. Here we see the two possible ways to identify the two marked points in the boundary of Configuration 21. Inside each rectangle, $\tau$ and $\Qtau$ appear on the left, and $\sigma$ and $\widetilde{\mu}_k(\Qtau)$ appear on the right.}
                \label{Fig:flip_21}
        \end{figure}
Note that $\ASsigma$ is the reduced part of $(\widetilde{\mu}_k(\Atau),\Ssigma^\sharp)$, where
\begin{eqnarray}\nonumber
\Ssigma^\sharp  & = &
[\alpha\beta]\gamma
+[\delta\varepsilon]\eta-x_{p_2}^{-1}[\delta\nu]\kappa
+
[\delta\beta]\beta^*\delta^*-[\delta\beta]v\beta^*\delta^*
+[\alpha\varepsilon]\varepsilon^*\alpha^*
- x_{p_2}^{-1} [\alpha\nu]\nu^*\alpha^*\\
\nonumber
&+&
Y_1(x_{p_1}[\alpha\nu]\nu^*\delta^*\lambda + x_{p_3}[\delta\beta]\beta^*\alpha^*\rho )
\ \ +\ \
Y_2(x_{p_1}[\alpha\nu]\nu^*\delta^*\lambda[\delta\beta]\beta^*\alpha^*\rho )
\ \ + \ \ S(\tau,\sigma)
\in\RA{\widetilde{\mu}_k(\Atau)},
\end{eqnarray}
with $S(\tau,\sigma)\in\RA{\Atau}\cap\RA{\Asigma}$. 
Here, $Y_1$ and $Y_2$ are elements of $F$ with the property that exactly one of them is equal to $1$ and exactly one is equal to $0$.
Furthermore,
\begin{eqnarray}\nonumber
\Stau & = &
\alpha\beta\gamma-\alpha\beta v\gamma
+\delta\varepsilon\eta
- x_{p_2}^{-1}\delta\nu\kappa\\
\nonumber
&+&Y_1(x_{p_1}\alpha\nu\kappa\lambda+x_{p_3}\delta\beta\gamma\rho)
\ \ + \ \  Y_2(x_{p_1}\alpha\nu\kappa\lambda\delta\beta\gamma\rho)
\ \ + \ \ S(\tau,\sigma)\\
\nonumber
\text{and} \ \ \ \ \
\widetilde{\mu}_k(\Stau) & = &
[\alpha\beta]\gamma-[\alpha\beta] v\gamma
+[\delta\varepsilon]\eta
- x_{p_2}^{-1}[\delta\nu]\kappa\\
\nonumber
&+& Y_1(x_{p_1}[\alpha\nu]\kappa\lambda+x_{p_3}[\delta\beta]\gamma\rho)
\ \ + \ \  Y_2(x_{p_1}[\alpha\nu]\kappa\lambda[\delta\beta]\gamma\rho)\\
\nonumber
&+& [\alpha\beta]\beta^*\alpha^*
+[\delta\varepsilon]\varepsilon^*\delta^*
+ [\delta\nu]\nu^*\delta^*
+[\alpha\nu]\nu^*\alpha^*
+[\delta\beta]\beta^*\delta^*
+[\alpha\varepsilon]\varepsilon^*\alpha^*
+S(\tau,\sigma).
\end{eqnarray}

Define $R$-algebra automorphisms $\psi,\varphi_1,\varphi_2,\Phi:\RA{\widetilde{\mu}_k(\Atau)}\rightarrow\RA{\widetilde{\mu}_k(\Atau)}$ according to the rules
\begin{center}
{\renewcommand{\arraystretch}{1.125}
\begin{tabular}{ccll}
$\psi$ &:& $[\alpha\beta]\mapsto[\alpha\beta]\left(\frac{1+v}{2}\right)$, & \\
$\varphi_1$ &:&
$[\alpha\beta]\mapsto[\alpha\beta]-Y_1x_{p_3}\rho[\delta\beta]-Y_2x_{p_1}\rho[\alpha\nu]\kappa\lambda[\delta\beta]$,&
$\gamma\mapsto\gamma-\left(\frac{1+v}{2}\right)\beta^*\alpha^*$, \\
&&
$\eta\mapsto\eta-\varepsilon^*\delta^*$, \ \ \
$[\delta\nu]\mapsto[\delta\nu]+Y_1x_{p_1}x_{p_2}\lambda[\alpha\nu]$, &
$\kappa\mapsto\kappa+x_{p_2} \nu^*\delta^*$, \\
$\varphi_2$ &:&
$[\alpha\beta]\mapsto[\alpha\beta]-Y_2x_{p_1}x_{p_2}\rho[\alpha\nu]\nu^*\delta^*\lambda[\delta\beta]$, &
$[\delta\nu]\mapsto[\delta\nu]-Y_2x_{p_1}x_{p_2}\lambda[\delta\beta]\left(\frac{1+v}{2}\right)\beta^*\alpha^*\rho[\alpha\nu]$,\\
$\Phi$ &:&
$[\delta\beta]\mapsto[\delta\beta](1-v)$, &
$\alpha^*\mapsto-\alpha^*$,\\
&&
$[\alpha\nu]\mapsto x_{p_2}^{-1}[\alpha\nu]$,&
$[\alpha\varepsilon]\mapsto-[\alpha\varepsilon]$.
\end{tabular}
}
\end{center}

Direct computation shows that
the composition $\Phi\varphi_2\varphi_1\psi$ is a right-equivalence $(\widetilde{\mu}_k(\Atau),\widetilde{\mu}_k(\Stau))\rightarrow(\widetilde{\mu}_k(\Atau),\Ssigma^\sharp)$. Whence the reduced parts of these two SPs are right-equivalent. But these reduced parts are precisely $\mu_k\AStau$ and $\ASsigma$, respectively.
\end{case}

\begin{case}\label{case:24}\emph{Flip of configuration 24}. We adopt the notation from Figure \ref{Fig:flip_24}).
Note that $\ASsigma$ is the reduced part of $(\widetilde{\mu}_k(\Atau),\Ssigma^\sharp)$, where
\begin{eqnarray}\nonumber
\Ssigma^\sharp  & = &
[\rho\varepsilon]\eta -x_{p_2}^{-1}[\rho\nu]\kappa + 2v^{-1}\alpha[\beta\gamma]_{\myid} + 2\delta[\beta\gamma]_{\theta}
+
[\rho\gamma]\gamma^*\rho^* - [\rho\gamma]v\gamma^*\rho^*
+[\beta\varepsilon]\varepsilon^*\beta^*-[\beta\varepsilon]\varepsilon^*\beta^*v\\
\nonumber
&+&
x_{p_2}^{-1}[\beta\nu]\nu^*\beta^*v- x_{p_2}^{-1}[\beta\nu]\nu^*\beta^*
+x_{p_1}[\rho\gamma]\gamma^*\beta^*[\beta\nu]\nu^*\rho^*\lambda
+S(\tau,\sigma)
\in\RA{\widetilde{\mu}_k(\Atau)},
\end{eqnarray}
with $S(\tau,\sigma)\in\RA{\Atau}\cap\RA{\Asigma}$. 
Furthermore,
\begin{eqnarray}\nonumber
\Stau & = &
\rho\varepsilon\eta -x_{p_2}^{-1}\rho\nu\kappa + 2v^{-1}\alpha\beta\gamma + 2\delta\beta\gamma
+
x_{p_1}\rho\gamma(\alpha+\delta)\beta\nu\kappa\lambda
+S(\tau,\sigma),
\end{eqnarray}
\begin{eqnarray}\nonumber
\text{and} \ \ \ \ \
\widetilde{\mu}_k(\Stau)
&\sim_{\operatorname{cyc}}&
[\rho\varepsilon]\eta -x_{p_2}^{-1}[\rho\nu]\kappa + 2v^{-1}\alpha[\beta\gamma]_{\myid} + 2\delta[\beta\gamma]_{\theta}\\
\nonumber
&+&
x_{p_1}\alpha\left(\frac{1}{2}([\beta\nu]\kappa\lambda[\rho\gamma]+v^{-1}[\beta\nu]\kappa\lambda[\rho\gamma]v)\right)\\
\nonumber
&+&x_{p_1}\delta\left(\frac{1}{2}([\beta\nu]\kappa\lambda[\rho\gamma]+\theta(v^{-1})[\beta\nu]\kappa\lambda[\rho\gamma]v)\right)\\
\nonumber
&+&
[\rho\varepsilon]\varepsilon^*\rho^*
+[\rho\nu]\nu^*\rho^*
+[\rho\gamma]\gamma^*\rho^*
+[\beta\nu]\nu^*\beta^*
+[\beta\varepsilon]\varepsilon^*\beta^*\\
\nonumber
&+&
[\beta\gamma]_{\myid}\left(\frac{1}{2}(\gamma^*\beta^*+v^{-1}\gamma^*\beta^*v)\right)
+[\beta\gamma]_{\theta}\left(\frac{1}{2}(\gamma^*\beta^*+\theta(v^{-1})\gamma^*\beta^*v)\right)+
S(\tau,\sigma).
\end{eqnarray}
\begin{figure}[!ht]
                \centering
                \includegraphics[scale=.65]{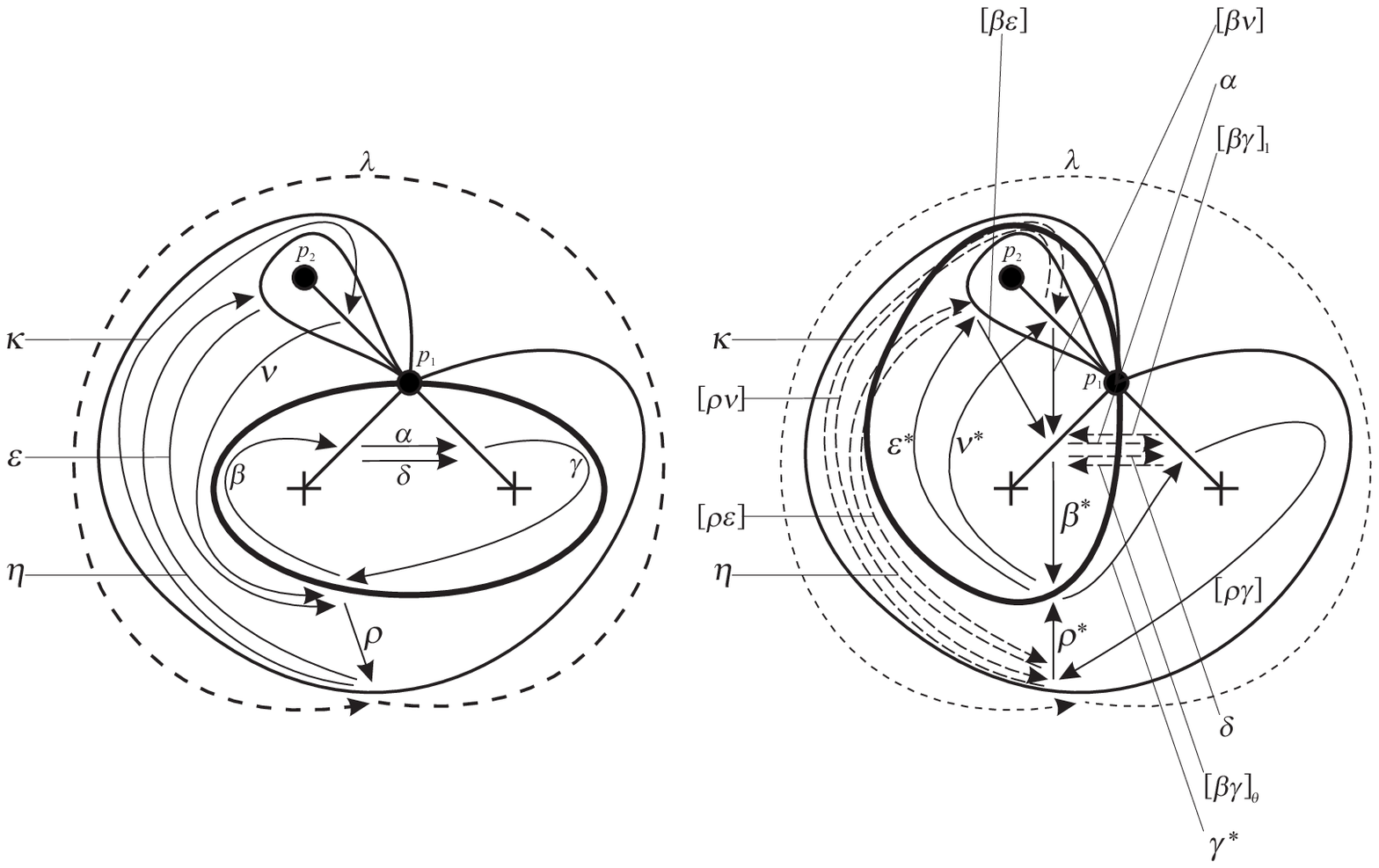}
                 \caption{Flip of the $24^{\operatorname{th}}$ configuration of Figure \ref{Fig:all_possible_matchings_2}.
                Left: $\tau$ and $\Qtau$.
                Right: $\sigma$ and $\widetilde{\mu}_k(\Qtau)$.}
                \label{Fig:flip_24}
        \end{figure}

Define $R$-algebra automorphisms $\varphi_1,\varphi_2,\Phi:\RA{\widetilde{\mu}_k(\Atau)}\rightarrow\RA{\widetilde{\mu}_k(\Atau)}$ according to the rules
\begin{center}
{\renewcommand{\arraystretch}{1.125}
\begin{tabular}{ccl}
$\varphi_1$ &:&
$\eta\mapsto\eta-\varepsilon^*\rho^*, \ \ \ [\rho\nu]\mapsto[\rho\nu]+x_{p_1}x_{p_2}\lambda[\rho\gamma](\alpha+\delta)[\beta\nu], \ \ \
\kappa\mapsto\kappa +x_{p_2}\nu^*\rho^*$,\\
&&
$\alpha\mapsto \alpha-\left(\frac{v}{4}(\gamma^*\beta^*+v^{-1}\gamma^*\beta^*v)\right), \ \ \
\delta\mapsto \delta- \left(\frac{1}{4}(\gamma^*\beta^*+\theta(v^{-1})\gamma^*\beta^*v)\right)$,\\
$\varphi_2$ &:&
$[\rho\nu]\mapsto
[\rho\nu] - x_{p_1}x_{p_2}\lambda[\rho\gamma]\left(\frac{v}{4}(\gamma^*\beta^*+v^{-1}\gamma^*\beta^*v)\right)[\beta\nu]
-x_{p_1}x_{p_2}\lambda[\rho\gamma]\left(\frac{1}{4}(\gamma^*\beta^*+\theta(v^{-1})\gamma^*\beta^*v)\right)[\beta\nu]$,\\
&&
$[\beta\gamma]_{\myid}\mapsto
[\beta\gamma]_{\myid} - \frac{x_{p_1}x_{p_2}v}{4}\left([\beta\nu]\nu^*\rho^*\lambda[\rho\gamma]+v^{-1}[\beta\nu]\nu^*\rho^*\lambda[\rho\gamma]v\right)$,\\
&&
$[\beta\gamma]_{\theta}\mapsto
[\beta\gamma]_{\theta}-\frac{x_{p_1}x_{p_2}}{4}\left([\beta\nu]\nu^*\rho^*\lambda[\rho\gamma]+\theta(v^{-1})[\beta\nu]\nu^*\rho^*\lambda[\rho\gamma]v\right)$,\\
\nonumber
$\Phi$ &:&
$[\beta\nu]\mapsto-x_{p_2}^{-1}(1-v)[\beta\nu], \ \ \
[\rho\gamma]\mapsto[\rho\gamma](1-v), \ \ \
[\beta\varepsilon]\mapsto(1-v)[\beta\varepsilon]$.
\end{tabular}
}
\end{center}

Direct computation shows that
the composition $\Phi\varphi_2\varphi_1$ is a right-equivalence $(\widetilde{\mu}_k(\Atau),\widetilde{\mu}_k(\Stau))\rightarrow(\widetilde{\mu}_k(\Atau),\Ssigma)$. Whence the reduced parts of these two SPs are right-equivalent. But these reduced parts are precisely $\mu_k\AStau$ and $\ASsigma$, respectively.
\end{case}

\begin{case}{Flip of configuration 25}. Similar to Case \ref{case:24}.
\end{case}

\begin{case}\label{case:26}\emph{Flip of configuration 26}. We adopt the notation from Figure \ref{Fig:flip_26}.
\begin{figure}[!ht]
                \centering
                \includegraphics[scale=.5]{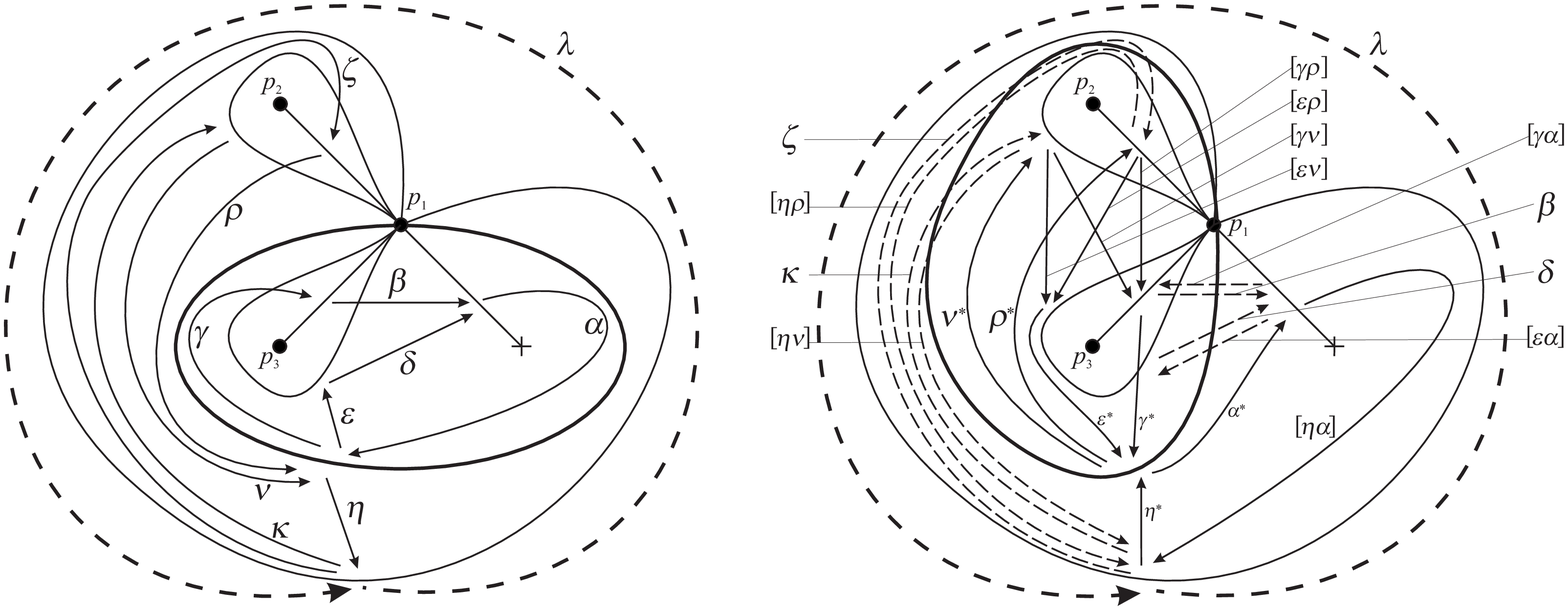}
                \caption{Flip of the $26^{\operatorname{th}}$ configuration of Figure \ref{Fig:all_possible_matchings_2}.
                Left: $\tau$ and $\Qtau$.
                Right: $\sigma$ and $\widetilde{\mu}_k(\Qtau)$.}
                \label{Fig:flip_26}
        \end{figure}
Note that $\ASsigma$ is the reduced part of $(\widetilde{\mu}_k(\Atau),\Ssigma^\sharp)$, where
\begin{eqnarray}\nonumber
\Ssigma^\sharp  & = &
[\eta\nu]\kappa
+ [\varepsilon\alpha]\delta
-x_{p_2}^{-1}[\eta\rho]\zeta
+[\gamma\alpha]\beta
+
[\eta\alpha]\alpha^*\eta^*-[\eta\alpha]v\alpha^*\eta^*
+[\varepsilon\nu]\nu^*\varepsilon^*\\
\nonumber
&-&x_{p_2}^{-1}[\varepsilon\rho]\rho^*\varepsilon^*
-x_{p_3}^{-1}[\gamma\nu]\nu^*\gamma^*
+x_{p_2}^{-1}x_{p_3}^{-1}[\gamma\rho]\rho^*\gamma^*
+x_{p_1}[\eta\alpha]\alpha^*\gamma^*[\gamma\rho]\rho^*\eta^*\lambda\\
\nonumber
&+&S(\tau,\sigma)
\in\RA{\widetilde{\mu}_k(\Atau)},
\end{eqnarray}
with $S(\tau,\sigma)\in\RA{\Atau}\cap\RA{\Asigma}$. 
Furthermore,
\begin{eqnarray}\nonumber
\Stau & = &
\eta\nu\kappa + \alpha\delta\varepsilon - \alpha v\delta\varepsilon+
x_{p_1}\eta\alpha\beta\gamma\rho\zeta\lambda
-x_{p_2}^{-1}\eta\rho\zeta + x_{p_3}^{-1}\alpha v\beta\gamma - x_{p_3}^{-1}\alpha\beta\gamma
+S(\tau,\sigma)\\
\nonumber
\text{and} \ \ \ \ \ \widetilde{\mu}_k(\Stau)
&\sim_{\operatorname{cyc}}&
[\eta\nu]\kappa + [\varepsilon\alpha]\delta - [\varepsilon\alpha] v\delta+
x_{p_1}[\eta\alpha]\beta[\gamma\rho]\zeta\lambda
-x_{p_2}^{-1}[\eta\rho]\zeta + x_{p_3}^{-1}[\gamma\alpha] v\beta - x_{p_3}^{-1}[\gamma\alpha]\beta\\
\nonumber
&+&
[\eta\nu]\nu^*\eta^*
+[\varepsilon\alpha]\alpha^*\varepsilon^*
+[\eta\rho]\rho^*\eta^*
+[\gamma\alpha]\alpha^*\gamma^*
+[\eta\alpha]\alpha^*\eta^*\\
\nonumber
&+&
[\varepsilon\nu]\nu^*\varepsilon^*
+[\varepsilon\rho]\rho^*\varepsilon^*
+[\gamma\nu]\nu^*\gamma^*
+[\gamma\rho]\rho^*\gamma^*
+S(\tau,\sigma).
\end{eqnarray}

Define $R$-algebra automorphisms $\psi,\varphi_1,\varphi_2,\Phi:\RA{\widetilde{\mu}_k(\Atau)}\rightarrow\RA{\widetilde{\mu}_k(\Atau)}$ according to the rules
\begin{center}
{\renewcommand{\arraystretch}{1.125}
\begin{tabular}{ccll}
$\psi$ &:&
$[\varepsilon\alpha]\mapsto[\varepsilon\alpha]\left(\frac{1+v}{2}\right), \ \ \ \ \ [\gamma\alpha]\mapsto-x_{p_3}[\gamma\alpha]\left(\frac{1+v}{2}\right)$, & \\
$\varphi_1$ &:&
$\kappa\mapsto\kappa- \nu^*\eta^*, \ \ \ \ \ \delta\mapsto\delta-\left(\frac{1+v}{2}\right)\alpha^*\varepsilon^*$, &
$[\eta\rho]\mapsto [\eta\rho]+x_{p_1}x_{p_2}\lambda[\eta\alpha]\beta[\gamma\rho]$,\\
&&
$\zeta\mapsto\zeta+x_{p_2}\rho^*\eta^*, \ \ \ \ \ \beta\mapsto\beta+x_{p_3}\left(\frac{1+v}{2}\right)\alpha^*\gamma^*$, & \\
$\varphi_2$ &:&
$[\eta\rho]\mapsto[\eta\rho]+x_{p_1}x_{p_2}x_{p_3}\lambda[\eta\alpha]\left(\frac{1+v}{2}\right)\alpha^*\gamma^*[\gamma\rho]$, &
$[\gamma\alpha]\mapsto[\gamma\alpha]-x_{p_1}x_{p_2}[\gamma\rho]\rho^*\eta^*\lambda[\eta\alpha]$,\\
$\Phi$ &:&
$[\eta\alpha]\mapsto[\eta\alpha](1-v), \ \ \
\rho^*\mapsto-x_{p_2}^{-1}\rho^*, \ \ \
\gamma^*\mapsto-x_{p_3}^{-1}\gamma^*$.
\end{tabular}
}
\end{center}

Direct computation shows that
the composition $\Phi\varphi_2\varphi_1\psi$ is a right-equivalence $(\widetilde{\mu}_k(\Atau),\widetilde{\mu}_k(\Stau))\rightarrow(\widetilde{\mu}_k(\Atau),\Ssigma^\sharp)$. Whence the reduced parts of these two SPs are right-equivalent. But these reduced parts are precisely $\mu_k\AStau$ and $\ASsigma$, respectively.
\end{case}

\begin{case}\emph{Flip of configuration 27}. We adopt the notation from Figure \ref{Fig:flip_27}.
Note that $\ASsigma$ is the reduced part of $(\widetilde{\mu}_k(\Atau),\Ssigma^\sharp)$, where
\begin{eqnarray}\nonumber
\Ssigma^\sharp  & = &
[\nu\eta]\kappa
+ [\varepsilon\alpha]\delta
+[\gamma\alpha]\beta
+[\rho\eta]\zeta
+
[\varepsilon\eta]\eta^*\varepsilon^*+[\nu\alpha]\alpha^*\nu^*-[\nu\alpha]v\alpha^*\nu^*\\
\nonumber
&+&
x_{p_1} \eta^*\rho^*[\rho\alpha]\alpha^*\gamma^*[\gamma\eta]\lambda
-x_{p_2}^{-1}[\gamma\eta]\eta^*\gamma^*
 + x_{p_3}^{-1}[\rho\alpha]v\alpha^*\rho^* - x_{p_3}^{-1}[\rho\alpha]\alpha^*\rho^*\\
\nonumber
&+&
S(\tau,\sigma)
\in\RA{\widetilde{\mu}_k(\Atau)},
\end{eqnarray}
with $S(\tau,\sigma)\in\RA{\Atau}\cap\RA{\Asigma}$. 
Furthermore,
\begin{eqnarray}\nonumber
\Stau & = &
\eta\kappa\nu + \alpha\delta\varepsilon -\alpha v\delta\varepsilon+
x_{p_1}\zeta\rho\alpha\beta\gamma\eta\lambda
+ x_{p_2}^{-1}\alpha v\beta\gamma - x_{p_2}^{-1}\alpha\beta\gamma
-x_{p_3}^{-1}\eta\zeta\rho
+S(\tau,\sigma)\\
\nonumber
\text{and} \ \ \ \ \ \widetilde{\mu}_k(\Stau)
&\sim_{\operatorname{cyc}}&
[\nu\eta]\kappa + [\varepsilon\alpha]\delta - [\varepsilon\alpha] v\delta+
x_{p_1}\zeta[\rho\alpha]\beta[\gamma\eta]\lambda
+ x_{p_2}^{-1}[\gamma\alpha] v\beta - x_{p_2}^{-1}[\gamma\alpha]\beta
-x_{p_3}^{-1}[\rho\eta]\zeta\\
\nonumber
&+&
[\nu\eta]\eta^*\nu^*
+[\varepsilon\alpha]\alpha^*\varepsilon^*
+[\gamma\alpha]\alpha^*\gamma^*
+[\rho\eta]\eta^*\rho^*
+[\rho\alpha]\alpha^*\rho^*\\
\nonumber
&+&
[\gamma\eta]\eta^*\gamma^*
+[\varepsilon\eta]\eta^*\varepsilon^*
+[\nu\alpha]\alpha^*\nu^*+
S(\tau,\sigma).
\end{eqnarray}
\begin{figure}[!ht]
                \centering
                \includegraphics[scale=.5]{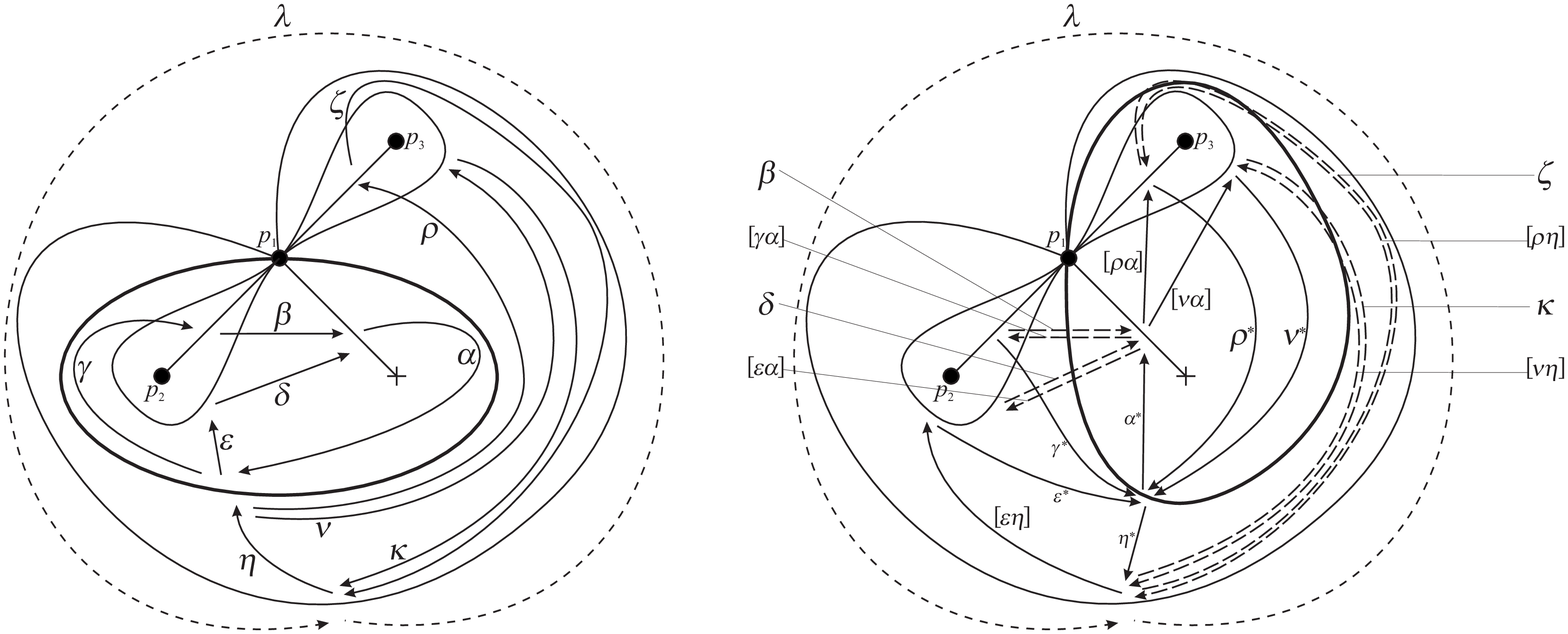}
                \caption{Flip of the $27^{\operatorname{th}}$ configuration of Figure \ref{Fig:all_possible_matchings_2}.
                Left: $\tau$ and $\Qtau$.
                Right: $\sigma$ and $\widetilde{\mu}_k(\Qtau)$.}
                \label{Fig:flip_27}
        \end{figure}

Define $R$-algebra automorphisms $\psi,\varphi_1,\varphi_2,\Phi:\RA{\widetilde{\mu}_k(\Atau)}\rightarrow\RA{\widetilde{\mu}_k(\Atau)}$ according to the rules
\begin{center}
{\renewcommand{\arraystretch}{1.125}
\begin{tabular}{ccll}
$\psi$ &:&
$[\varepsilon\alpha]\mapsto[\varepsilon\alpha]\left(\frac{1+v}{2}\right), \ \ \
[\gamma\alpha]\mapsto -x_{p_2}[\gamma\alpha]\left(\frac{1+v}{2}\right)$,
&
$[\rho\eta]\mapsto-x_{p_3}[\rho\eta]$,\\
$\varphi_1$ &:&
$\kappa\mapsto\kappa-\eta^*\nu^*$, &
$\delta\mapsto\delta- \left(\frac{1+v}{2}\right)\alpha^*\varepsilon^*$,\\
&&
$[\gamma\alpha]\mapsto[\gamma\alpha]-x_{p_1}[\gamma\eta]\lambda\zeta[\rho\alpha]$, &
$\beta\mapsto\beta+x_{p_2}\left(\frac{1+v}{2}\right)\alpha^*\gamma^*, \ \ \ \
\zeta\mapsto\zeta+x_{p_3}\eta^*\rho^*$,\\
$\varphi_2$ &:&
$[\rho\eta]\mapsto[\rho\eta]-x_{p_1}x_{p_2}[\rho\alpha]\left(\frac{1+v}{2}\right)\alpha^*\gamma^*[\gamma\eta]\lambda$, &
$[\gamma\alpha]\mapsto[\gamma\alpha]-x_{p_1}x_{p_3}[\gamma\eta]\lambda\eta^*\rho^*[\rho\alpha]$,\\
$\Phi$ &:&
$\alpha^*\mapsto(1-v)\alpha^*, \ \ \ \ \ \  [\rho\alpha]\mapsto-x_{p_3}^{-1}[\rho\alpha]$, &
$[\gamma\eta]\mapsto-x_{p_2}^{-1}[\gamma\eta]$.
\end{tabular}
}
\end{center}

Direct computation shows that
the composition $\Phi\varphi_2\varphi_1\psi$ is a right-equivalence $(\widetilde{\mu}_k(\Atau),\widetilde{\mu}_k(\Stau))\rightarrow(\widetilde{\mu}_k(\Atau),\Ssigma^\sharp)$. Whence the reduced parts of these two SPs are right-equivalent. But these reduced parts are precisely $\mu_k\AStau$ and $\ASsigma$, respectively.
\end{case}

\begin{case}\emph{Flip of configuration 29}. Similar to Case \ref{case:26}.
\end{case}

\begin{case}\label{case:flip-36}\emph{Flip of configuration 36}. See the discussion in the first and second paragraphs of Case \ref{case:flip-8}. We adopt the notation from Figure \ref{Fig:flip_36}.
Note that $\ASsigma$ is the reduced part of $(\widetilde{\mu}_k(\Atau),\Ssigma^\sharp)$, where
\begin{eqnarray}\nonumber
\Ssigma^\sharp  & = &
[\alpha\beta]\gamma
+[\varepsilon\delta]\eta\\
\nonumber
&+&
Y_1( x_{p_1}\nu\delta^*\varepsilon^*([\varepsilon\beta]_{\myid}+[\varepsilon\beta]_{\theta})\beta^*\alpha^*  +  x_{p_2}[\alpha\delta]\rho  )
\ \ + \ \
Y_2(
x_{p_1}\nu[\alpha\delta]\rho\delta^*\varepsilon^*([\varepsilon\beta]_{\myid}+[\varepsilon\beta]_{\theta})\beta^*\alpha^*
)\\
\nonumber
&+&
[\alpha\delta]\delta^*\alpha^*
+2v^{-1}[\varepsilon\beta]_{\myid}\beta^*\varepsilon^*+2[\varepsilon\beta]_{\theta}\beta^*\varepsilon^*
+S(\tau,\sigma)
\in\RA{\widetilde{\mu}_k(\Atau)},
\end{eqnarray}
with $S(\tau,\sigma)\in\RA{\Atau}\cap\RA{\Asigma}$. 
Here, $Y_1$ and $Y_2$ are elements of $F$ with the property that exactly one of them is equal to $1$ and exactly one is equal to $0$.
Furthermore,
\begin{eqnarray}
\nonumber
\Stau & = &
\alpha\beta\gamma - \alpha\beta v\gamma +\varepsilon\delta\eta-\varepsilon\delta\eta v+
Y_1(  x_{p_1}\eta\varepsilon\beta\gamma\nu  +  x_{p_2}\alpha\delta\rho  )
+
Y_2(  x_{p_1}\eta\varepsilon\beta\gamma\nu\alpha\delta\rho  )+S(\tau,\sigma)\\
\nonumber
\text{and}  \ \ \ \widetilde{\mu}_k(\Stau)
& \sim_{\operatorname{cyc}} &
[\alpha\beta]\gamma - [\alpha\beta] v\gamma + [\varepsilon\delta]\eta - v[\varepsilon\delta]\eta\\
\nonumber
&+&
Y_1(  x_{p_1}\eta([\varepsilon\beta]_{\myid}+[\varepsilon\beta]_{\theta})\gamma\nu  +  x_{p_2}[\alpha\delta]\rho  )
\ \ + \ \
Y_2(  x_{p_1}\eta([\varepsilon\beta]_{\myid}+[\varepsilon\beta]_{\theta})\gamma\nu[\alpha\delta]\rho  )\\
\nonumber
&+&
[\alpha\beta]\beta^*\alpha^*
+[\alpha\delta]\delta^*\alpha^*
+\frac{1}{2}[\varepsilon\beta]_{\myid}\left(\beta^*\varepsilon^* + v^{-1}\beta^*\varepsilon^*v\right)\\
\nonumber
&+&\frac{1}{2}[\varepsilon\beta]_{\theta}\left(\beta^*\varepsilon^* + \theta(v^{-1})\beta^*\varepsilon^*v\right)
+[\varepsilon\delta]\delta^*\varepsilon^*+
S(\tau,\sigma).
\end{eqnarray}
\begin{figure}[!ht]
                \centering
                \includegraphics[scale=.5]{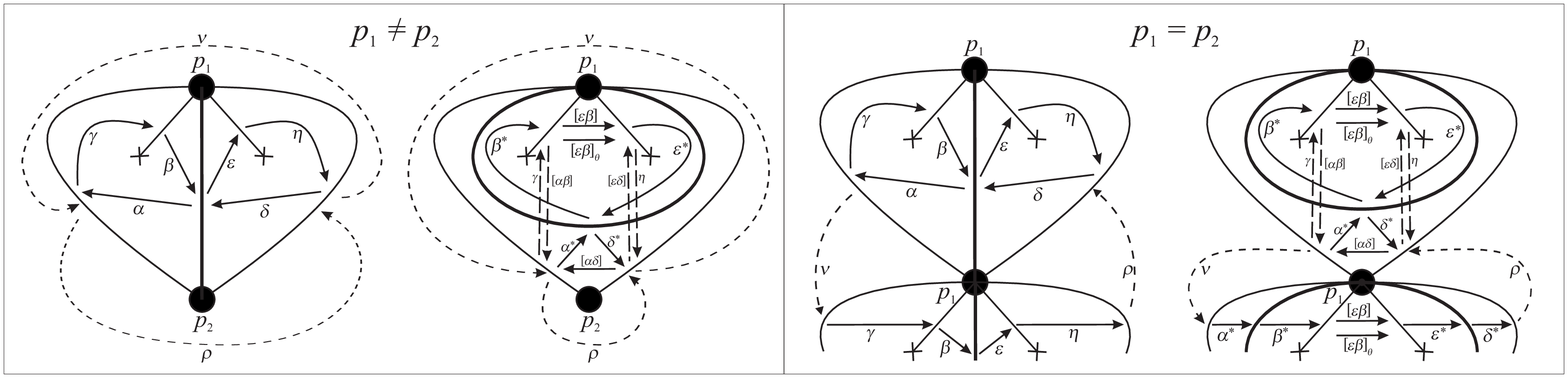}
                \caption{Flip of the $36^{\operatorname{th}}$ configuration of Figure \ref{Fig:all_possible_matchings_2}. Here we see the two possible ways to identify the two marked points in the boundary of Configuration 36. Inside each rectangle, $\tau$ and $\Qtau$ appear on the left, and $\sigma$ and $\widetilde{\mu}_k(\Qtau)$ appear on the right.}
                \label{Fig:flip_36}
        \end{figure}

Define $R$-algebra automorphisms $\psi,\varphi_1,\varphi_2,\Phi:\RA{\widetilde{\mu}_k(\Atau)}\rightarrow\RA{\widetilde{\mu}_k(\Atau)}$ according to the rules
\begin{eqnarray}
\nonumber
\psi & : &
[\alpha\beta]\mapsto[\alpha\beta]\left(\frac{1+v}{2}\right), \ \ \ \ \
[\varepsilon\delta]\mapsto\left(\frac{1+v}{2}\right)[\varepsilon\delta],\\
\nonumber
\varphi_1 &:&
[\alpha\beta]\mapsto[\alpha\beta]-Y_1  x_{p_1}\nu\eta([\varepsilon\beta]_{\myid}+[\varepsilon\beta]_{\theta})
-Y_2  x_{p_1}\nu[\alpha\delta]\rho\eta([\varepsilon\beta]_{\myid}+[\varepsilon\beta]_{\theta}),  \\
\nonumber
&&
\gamma\mapsto\gamma-\left(\frac{1+v}{2}\right)\beta^*\alpha^*, \ \ \ \ \
\eta\mapsto\eta-\delta^*\varepsilon^*\left(\frac{1+v}{2}\right),\\
\nonumber
\varphi_2 &:&
[\alpha\beta]\mapsto[\alpha\beta]
+Y_1x_{p_1}\nu\delta^*\varepsilon^*\left(\frac{1+v}{2}\right)([\varepsilon\beta]_{\myid}+[\varepsilon\beta]_{\theta})
+Y_2x_{p_1}\nu[\alpha\delta]\rho\delta^*\varepsilon^*\left(\frac{1+v}{2}\right)([\varepsilon\beta]_{\myid}+[\varepsilon\beta]_{\theta}),\\
\nonumber
&&
[\varepsilon\delta]\mapsto[\varepsilon\delta]
+Y_1x_{p_1}([\varepsilon\beta]_{\myid}+[\varepsilon\beta]_{\theta})\left(\frac{1+v}{2}\right)\beta^*\alpha^*\nu
+Y_2x_{p_1}([\varepsilon\beta]_{\myid}+[\varepsilon\beta]_{\theta})\left(\frac{1+v}{2}\right)\beta^*\alpha^*\nu[\alpha\delta]\rho,\\
\nonumber
\Phi &:&
\beta^*\mapsto(1-v)\beta^*, \ \ \ \ \
\varepsilon^*\mapsto\varepsilon^*(1-v).
\end{eqnarray}

Direct computation shows that
the composition $\Phi\varphi_2\varphi_1\psi$ is a right-equivalence $(\widetilde{\mu}_k(\Atau),\widetilde{\mu}_k(\Stau))\rightarrow(\widetilde{\mu}_k(\Atau),\Ssigma^\sharp)$. Whence the reduced parts of these two SPs are right-equivalent. But these reduced parts are precisely $\mu_k\AStau$ and $\ASsigma$, respectively.
\end{case}

\begin{case}\label{case:flip-37}\emph{Flip of configuration 37}. See the discussion in the first and second paragraphs of Case \ref{case:flip-8}. We adopt the notation from  Figure \ref{Fig:flip_37}.
Note that $\ASsigma$ is the reduced part of $(\widetilde{\mu}_k(\Atau),\Ssigma^\sharp)$, where
\begin{eqnarray}\nonumber
\Ssigma^\sharp  & = &
[\alpha\beta]\gamma
+[\delta\varepsilon]\eta\\
\nonumber
&+&
Y_1 ( x_{p_1} \lambda[\alpha\varepsilon]\varepsilon^*\delta^* +x_{p_2}\rho[\delta\beta]\beta^*\alpha^* )
\ \ + \ \
 Y_2( x_{p_1} [\alpha\varepsilon]\varepsilon^*\delta^*\lambda[\delta\beta]\beta^*\alpha^*\rho  )\\
\nonumber
&+&
[\alpha\varepsilon](1-v)\varepsilon^*\alpha^*
+[\delta\beta](1-v)\beta^*\delta^*
+S(\tau,\sigma)
\in\RA{\widetilde{\mu}_k(\Atau)},
\end{eqnarray}
with $S(\tau,\sigma)\in\RA{\Atau}\cap\RA{\Asigma}$. 
Here, $Y_1$ and $Y_2$ are elements of $F$ with the property that exactly one of them is equal to $1$ and exactly one is equal to $0$.
Furthermore,
\begin{eqnarray}
\nonumber
\Stau & = &
\alpha\beta\gamma-\alpha\beta v\gamma + \delta\varepsilon\eta -\delta\varepsilon v\eta
\ \ + \ \
Y_1 ( x_{p_1} \alpha\varepsilon\eta\lambda + x_{p_2}\delta\beta\gamma\rho ) \ \ + \ \
Y_2 ( x_{p_1} \alpha\varepsilon\eta\lambda\delta\beta\gamma\rho  )
\ \ + \ \ S(\tau,\sigma)\\
\nonumber
\text{and} \ \ \
\widetilde{\mu}_k(\Stau)
& = &
[\alpha\beta]\gamma-[\alpha\beta] v\gamma + [\delta\varepsilon]\eta -[\delta\varepsilon] v\eta+
Y_1 ( x_{p_1} [\alpha\varepsilon]\eta\lambda + x_{p_2}[\delta\beta]\gamma\rho )
\ \ + \ \
Y_2 ( x_{p_1} [\alpha\varepsilon]\eta\lambda[\delta\beta]\gamma\rho  )\\
\nonumber
&+&
[\alpha\beta]\beta^*\alpha^*
+[\delta\varepsilon]\varepsilon^*\delta^*
+[\alpha\varepsilon]\varepsilon^*\alpha^*
+[\delta\beta]\beta^*\delta^*+
S(\tau,\sigma).
\end{eqnarray}
\begin{figure}[!ht]
                \centering
                \includegraphics[scale=.45]{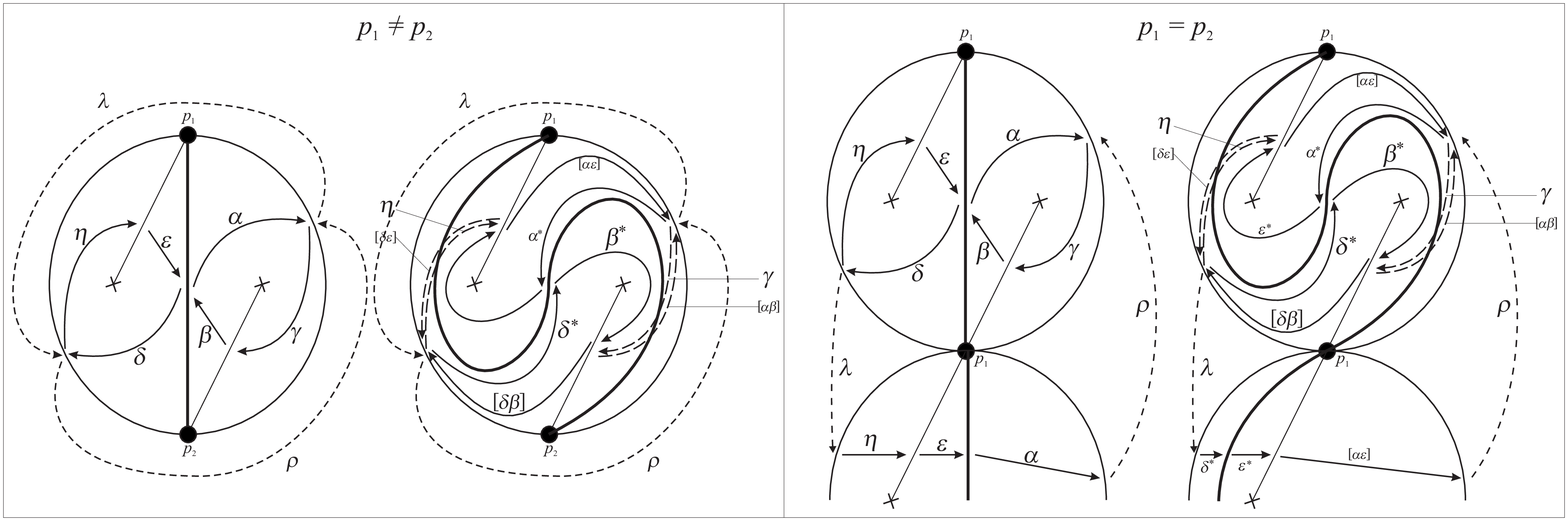}
                 \caption{Flip of the $37^{\operatorname{th}}$ configuration of Figure \ref{Fig:all_possible_matchings_2}. Here we see the two possible ways to identify the two marked points in the boundary of Configuration 37. Inside each rectangle, $\tau$ and $\Qtau$ appear on the left, and $\sigma$ and $\widetilde{\mu}_k(\Qtau)$ appear on the right.}
                \label{Fig:flip_37}
        \end{figure}

Define $R$-algebra automorphisms $\psi,\varphi_1,\varphi_2,\Phi:\RA{\widetilde{\mu}_k(\Atau)}\rightarrow\RA{\widetilde{\mu}_k(\Atau)}$ according to the rules
\begin{center}
{\renewcommand{\arraystretch}{1.125}
\begin{tabular}{ccll}
$\psi$ & : &
$[\alpha\beta]\mapsto[\alpha\beta]\left(\frac{1+v}{2}\right)$,&
$[\delta\varepsilon]\mapsto[\delta\varepsilon]\left(\frac{1+v}{2}\right)$,\\
$\varphi_1$ &:&
$[\alpha\beta]\mapsto[\alpha\beta]- Y_1 x_{p_2}\rho[\delta\beta] -Y_2  x_{p_1} \rho[\alpha\varepsilon]\eta\lambda[\delta\beta]$, &
$\gamma\mapsto\gamma-\left(\frac{1+v}{2}\right)\beta^*\alpha^*$,\\
&&
$[\delta\varepsilon]\mapsto[\delta\varepsilon]-Y_1  x_{p_1} \lambda[\alpha\varepsilon]$, &
$\eta\mapsto\eta-\left(\frac{1+v}{2}\right)\varepsilon^*\delta^*$,\\
$\varphi_2$ & : &
$[\alpha\beta]\mapsto[\alpha\beta]+Y_2x_{p_1} \rho[\alpha\varepsilon]\left(\frac{1+v}{2}\right)\varepsilon^*\delta^*\lambda[\delta\beta]$, &
$[\delta\varepsilon]\mapsto[\delta\varepsilon]+Y_2x_{p_1} \lambda[\delta\beta]\left(\frac{1+v}{2}\right)\beta^*\alpha^*\rho[\alpha\varepsilon]$,\\
$\Phi$ &:&
$\beta^*\mapsto-\beta^*, \ \ \ \
\delta^*\mapsto-\delta^*, \ \ \ \  [\alpha\varepsilon]\mapsto[\alpha\varepsilon](1-v)$, &
$[\delta\beta]\mapsto[\delta\beta](1-v)$.
\end{tabular}
}
\end{center}

Direct computation shows that
the composition $\Phi\varphi_2\varphi_1\psi$ is a right-equivalence $(\widetilde{\mu}_k(\Atau),\widetilde{\mu}_k(\Stau))\rightarrow(\widetilde{\mu}_k(\Atau),\Ssigma^\sharp)$. Whence the reduced parts of these two SPs are right-equivalent. But these reduced parts are precisely $\mu_k\AStau$ and $\ASsigma$, respectively.
\end{case}

\begin{case}\emph{Flip of configuration 39}. We adopt the notation from Figure \ref{Fig:flip_39}.
\begin{figure}[!ht]
                \centering
                \includegraphics[scale=.5]{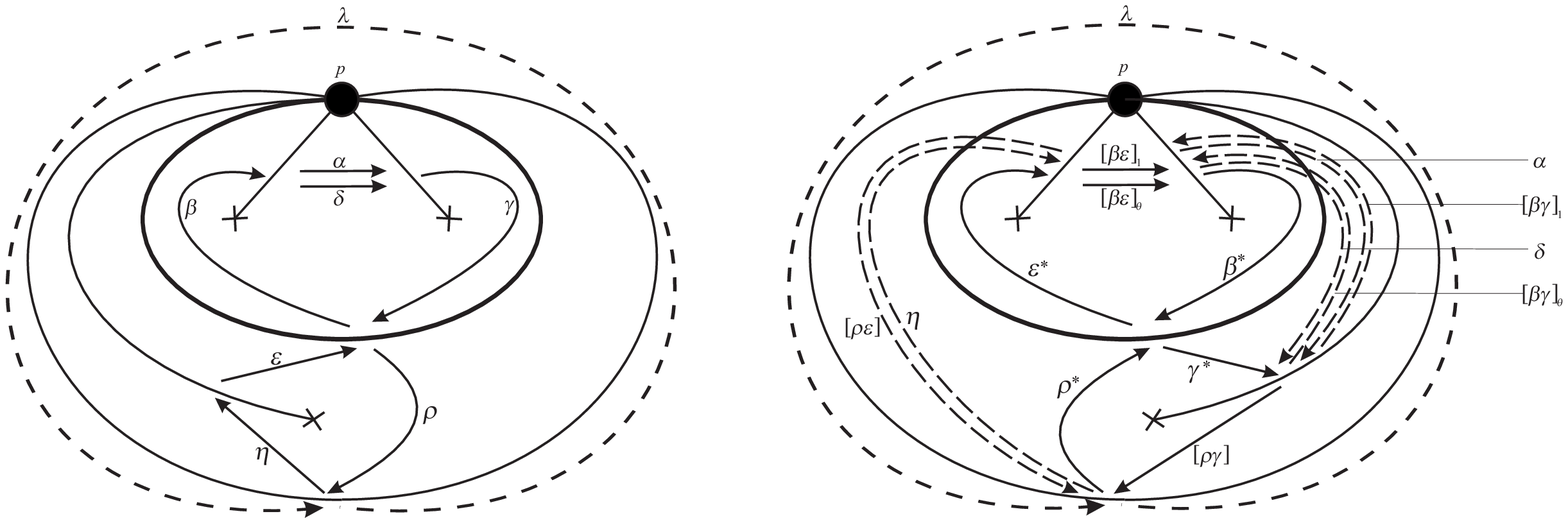}
                \caption{Flip of the $39^{\operatorname{th}}$ configuration of Figure \ref{Fig:all_possible_matchings_2}.
                Left: $\tau$ and $\Qtau$.
                Right: $\sigma$ and $\widetilde{\mu}_k(\Qtau)$.}
                \label{Fig:flip_39}
        \end{figure}
Note that $\ASsigma$ is the reduced part of $(\widetilde{\mu}_k(\Atau),\Ssigma^\sharp)$, where
\begin{eqnarray}\nonumber
\Ssigma^\sharp  & = &
[\rho\varepsilon]\eta
+2[\beta\gamma]_{\myid}v^{-1}\alpha
+2[\beta\gamma]_{\theta}\delta+
x_{p}([\beta\varepsilon]_{\myid}+[\beta\varepsilon]_{\theta}) \varepsilon^*\rho^*\lambda[\rho\gamma]\gamma^*\beta^*\\
\nonumber
&+&
[\rho\gamma](1-v)\gamma^*\rho^*
+2v^{-1}[\beta\varepsilon]_{\myid}\varepsilon^*\beta^*
+2[\beta\varepsilon]_{\theta}\varepsilon^*\beta^*+
S(\tau,\sigma)
\in\RA{\widetilde{\mu}_k(\Atau)},
\end{eqnarray}
with $S(\tau,\sigma)\in\RA{\Atau}\cap\RA{\Asigma}$. 
Furthermore,
\begin{eqnarray}
\nonumber
\Stau & = &
\rho\varepsilon\eta - \rho\varepsilon v\eta + 2v^{-1}\alpha\beta\gamma + 2\delta\beta\gamma+
x_{p}\rho\gamma(\alpha+\delta)\beta\varepsilon\eta\lambda
+S(\tau,\sigma)
\end{eqnarray}
\begin{eqnarray}
\nonumber
\text{and} \ \ \ \ \
\widetilde{\mu}_k(\Stau)
&\sim_{\operatorname{cyc}}&
[\rho\varepsilon]\eta - [\rho\varepsilon] v\eta + 2[\beta\gamma]_{\myid}v^{-1}\alpha + 2[\beta\gamma]_{\theta}\delta\\
\nonumber
&+&
x_{p}\alpha\left(\frac{1}{2}([\beta\varepsilon]\eta\lambda[\rho\gamma]+v^{-1}[\beta\varepsilon]\eta\lambda[\rho\gamma]v\right)+
x_{p}\delta\left(\frac{1}{2}([\beta\varepsilon]\eta\lambda[\rho\gamma]+\theta(v^{-1})[\beta\varepsilon]\eta\lambda[\rho\gamma]v\right) \\
\nonumber
&+&
[\rho\varepsilon]\varepsilon^*\rho^*
+[\beta\gamma]_{\myid}\left(\frac{1}{2}(\gamma^*\beta^*+v^{-1}\gamma^*\beta^*v\right)\\
\nonumber
&+&[\beta\gamma]_{\theta}\left(\frac{1}{2}(\gamma^*\beta^*+\theta(v^{-1})\gamma^*\beta^*v\right)
+[\rho\gamma]\gamma^*\rho^*
+[\beta\varepsilon]\varepsilon^*\beta^*+
S(\tau,\sigma).
\end{eqnarray}

Define $R$-algebra automorphisms $\psi,\varphi_1,\varphi_2,\Phi:\RA{\widetilde{\mu}_k(\Atau)}\rightarrow\RA{\widetilde{\mu}_k(\Atau)}$ according to the rules
\begin{eqnarray}
\nonumber
\psi & : &
[\rho\varepsilon]\mapsto[\rho\varepsilon]\left(\frac{1+v}{2}\right),\\
\nonumber
\varphi_1 & : &
\eta\mapsto\eta- \left(\frac{1+v}{2}\right)\varepsilon^*\rho^*, \ \ \ \ \ \ \ \ \ \ \ \ \ \ \
[\beta\gamma]_{\myid}\mapsto
[\beta\gamma]_{\myid}-x_{p}\left(\frac{1}{4}([\beta\varepsilon]\eta\lambda[\rho\gamma]+v^{-1}[\beta\varepsilon]\eta\lambda[\rho\gamma]v)\right)v,\\
\nonumber
&&
\alpha\mapsto\alpha-\left(\frac{v}{4}(\gamma^*\beta^*+v^{-1}\gamma^*\beta^*v\right),\ \ \ \ \
[\beta\gamma]_{\theta}\mapsto
[\beta\gamma]_{\theta}-x_{p}\left(\frac{1}{4}([\beta\varepsilon]\eta\lambda[\rho\gamma]+\theta(v^{-1})[\beta\varepsilon]\eta\lambda[\rho\gamma]v)\right),\\
\nonumber
&&
\delta\mapsto\delta-\left(\frac{1}{4}(\gamma^*\beta^*+\theta(v^{-1})\gamma^*\beta^*v\right),\\
\nonumber
\varphi_2 &:&
[\rho\varepsilon]\mapsto
[\rho\varepsilon]+
x_{p}\left(\frac{1}{2}\left[\lambda[\rho\gamma]\left(\frac{v}{4}(\gamma^*\beta^*+v^{-1}\gamma^*\beta^*v)\right)[\beta\varepsilon]+\lambda[\rho\gamma]v
\left(\frac{v}{4}(\gamma^*\beta^*+v^{-1}\gamma^*\beta^*v)\right)v^{-1}[\beta\varepsilon]\right]\right),\\
\nonumber
&&
+x_{p}\left(\frac{1}{2}\left[\left(\frac{1}{4}\lambda[\rho\gamma](\gamma^*\beta^*+\theta(v^{-1})\gamma^*\beta^*v)\right) [\beta\varepsilon]+\left(\frac{1}{4}\lambda[\rho\gamma]v(\gamma^*\beta^*+\theta(v^{-1})\gamma^*\beta^*v)\right) \theta(v^{-1})[\beta\varepsilon]\right]\right),\\
\nonumber
&&
[\beta\gamma]_{\myid}\mapsto
[\beta\gamma]_{\myid}+
x_{p}\left(\frac{1}{4}\left[[\beta\varepsilon]\left(\frac{1+v}{2}\right)\varepsilon^*\rho^*\lambda[\rho\gamma]+v^{-1}[\beta\varepsilon] \left(\frac{1+v}{2}\right)\varepsilon^*\rho^*\lambda[\rho\gamma]v\right]\right)v,\\
\nonumber
&&
[\beta\gamma]_{\theta}\mapsto
[\beta\gamma]_{\theta}+
x_{p}\left(\frac{1}{4}\left[[\beta\varepsilon] \left(\frac{1+v}{2}\right)\varepsilon^*\rho^*\lambda[\rho\gamma]+\theta(v^{-1})[\beta\varepsilon] \left(\frac{1+v}{2}\right)\varepsilon^*\rho^*\lambda[\rho\gamma]v\right]\right),\\
\nonumber
\Phi &:&
[\rho\gamma]\mapsto[\rho\gamma](1-v), \ \ \ \ \
[\beta\varepsilon]_{\myid}\mapsto
2v^{-1}[\beta\varepsilon]_{\myid}, \ \ \ \ \
[\beta\varepsilon]_{\theta}\mapsto
2[\beta\varepsilon]_{\theta}.
\end{eqnarray}

Direct computation shows that
the composition $\Phi\varphi_2\varphi_1\psi$ is a right-equivalence $(\widetilde{\mu}_k(\Atau),\widetilde{\mu}_k(\Stau))\rightarrow(\widetilde{\mu}_k(\Atau),\Ssigma^\sharp)$. Whence the reduced parts of these two SPs are right-equivalent. But these reduced parts are precisely $\mu_k\AStau$ and $\ASsigma$, respectively.
\end{case}

\begin{case}\emph{Flip of configuration 41}. We adopt the notation from Figure \ref{Fig:flip_41}.
\begin{figure}[!ht]
                \centering
                \includegraphics[scale=.5]{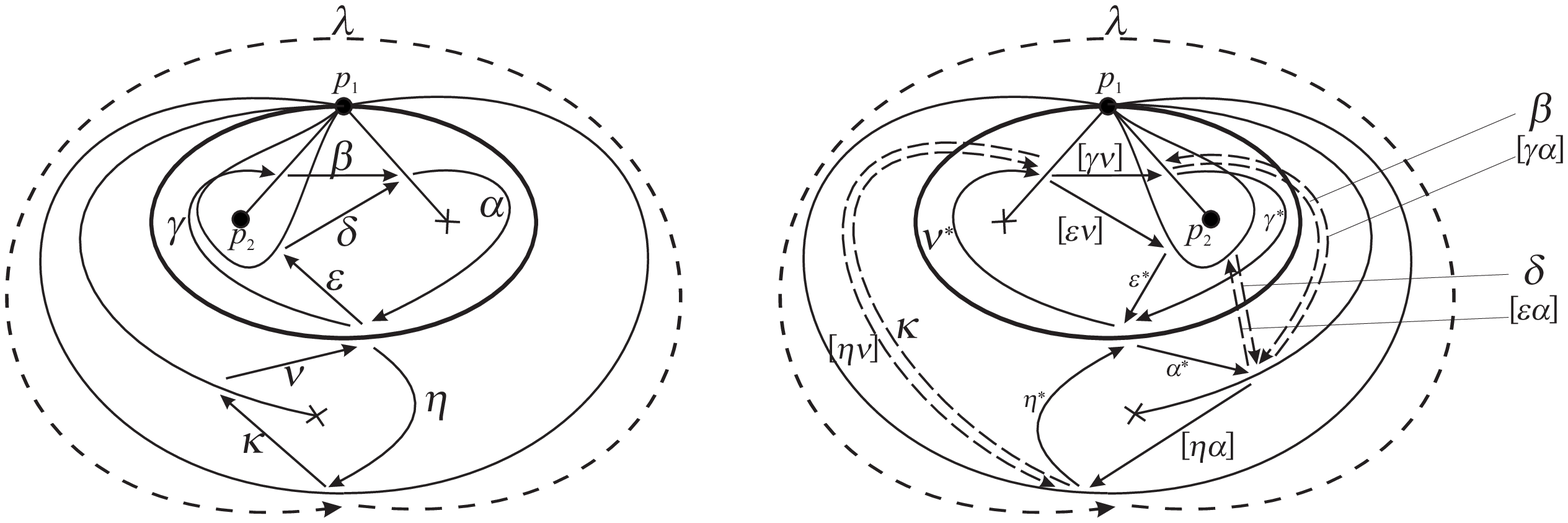}
                \caption{Flip of the $41^{\operatorname{th}}$ configuration of Figure \ref{Fig:all_possible_matchings_2}.
                Left: $\tau$ and $\Qtau$.
                Right: $\sigma$ and $\widetilde{\mu}_k(\Qtau)$.}
                \label{Fig:flip_41}
        \end{figure}
Note that $\ASsigma$ is the reduced part of $(\widetilde{\mu}_k(\Atau),\Ssigma^\sharp)$, where
\begin{eqnarray}\nonumber
\Ssigma^\sharp  & = &
[\eta\nu]\kappa
+[\varepsilon\alpha]\delta
+[\gamma\alpha]\beta\\
\nonumber
&+&
x_{p_1}[\eta\alpha]\alpha^*\gamma^*[\gamma\nu]\nu^*\eta^*\lambda
-x_{p_2}^{-1}[\gamma\nu]\nu^*\gamma^*+x_{p_2}^{-1}[\gamma\nu]v\nu^*\gamma^*
\\
\nonumber
&+&
[\eta\alpha]\alpha^*\eta^*-[\eta\alpha]v\alpha^*\eta^*+[\varepsilon\nu]\nu^*\varepsilon^*-[\varepsilon\nu]v\nu^*\varepsilon^*+S(\tau,\sigma)
\in\RA{\widetilde{\mu}_k(\Atau)},
\end{eqnarray}
with $S(\tau,\sigma)\in\RA{\Atau}\cap\RA{\Asigma}$. 
Furthermore,
\begin{eqnarray}
\nonumber
\Stau & = &
\eta\nu\kappa - \eta\nu v\kappa + \varepsilon\alpha\delta - \varepsilon\alpha v\delta
+x_{p_2}^{-1}\alpha v\beta\gamma - x_{p_2}^{-1}\alpha\beta\gamma+
x_{p_1}\eta\alpha\beta\gamma\nu\kappa\lambda
+S(\tau,\sigma)\\
\nonumber
\text{and} \ \ \
\widetilde{\mu}_k(\Stau)
& = &
[\eta\nu]\kappa - [\eta\nu] v\kappa + [\varepsilon\alpha]\delta - [\varepsilon\alpha] v\delta
+x_{p_2}^{-1}[\gamma\alpha] v\beta - x_{p_2}^{-1}[\gamma\alpha]\beta
+
x_{p_1}[\eta\alpha]\beta[\gamma\nu]\kappa\lambda\\
\nonumber
&+&
[\eta\nu]\nu^*\eta^*
+[\varepsilon\alpha]\alpha^*\varepsilon^*
+[\gamma\alpha]\alpha^*\gamma^*
+[\eta\alpha]\alpha^*\eta^*
+[\gamma\nu]\nu^*\gamma^*
+[\varepsilon\nu]\nu^*\varepsilon^*+
S(\tau,\sigma).
\end{eqnarray}

Define $R$-algebra automorphisms $\psi,\varphi_1,\varphi_2,\Phi:\RA{\widetilde{\mu}_k(\Atau)}\rightarrow\RA{\widetilde{\mu}_k(\Atau)}$ according to the rules
\begin{center}
{\renewcommand{\arraystretch}{1.125}
\begin{tabular}{ccll}
$\psi$ & : &
$[\eta\nu]\mapsto[\eta\nu]\left(\frac{1+v}{2}\right), \ \ \ \ \
[\varepsilon\alpha]\mapsto[\varepsilon\alpha]\left(\frac{1+v}{2}\right)$, &
$[\gamma\alpha]\mapsto-x_{p_2}[\gamma\alpha]\left(\frac{1+v}{2}\right)$,\\
$\varphi_1$ &:&
$[\eta\nu]\mapsto[\eta\nu] - x_{p_1}\lambda[\eta\alpha]\beta[\gamma\nu]$, &
$\kappa\mapsto\kappa-\left(\frac{1+v}{2}\right)\nu^*\eta^*$,\\
&&
$\delta\mapsto\delta-\left(\frac{1+v}{2}\right)\alpha^*\varepsilon^*$, &
$\beta\mapsto\beta+x_{p_2}\left(\frac{1+v}{2}\right)\alpha^*\gamma^*$,\\
$\varphi_2$ &:&
$[\gamma\alpha]\mapsto[\gamma\alpha]+x_{p_1}[\gamma\nu]\left(\frac{1+v}{2}\right)\nu^*\eta^*\lambda[\eta\alpha]$,&
$[\eta\nu]\mapsto[\eta\nu]- x_{p_1}x_{p_2}\lambda[\eta\alpha]\left(\frac{1+v}{2}\right)\alpha^*\gamma^*[\gamma\nu]$,\\
$\Phi$ &:&
$[\eta\alpha]\mapsto[\eta\alpha](1-v), \ \ \ \
[\gamma\nu]\mapsto-x_{p_2}^{-1}[\gamma\nu](1-v)$, &
$[\varepsilon\nu]\mapsto[\varepsilon\nu](1-v)$.
\end{tabular}
}
\end{center}

Direct computation shows that
the composition $\Phi\varphi_2\varphi_1\psi$ is a right-equivalence $(\widetilde{\mu}_k(\Atau),\widetilde{\mu}_k(\Stau))\rightarrow(\widetilde{\mu}_k(\Atau),\Ssigma^\sharp)$. Whence the reduced parts of these two SPs are right-equivalent. But these reduced parts are precisely $\mu_k\AStau$ and $\ASsigma$, respectively.
\end{case}

Theorem \ref{thm:ideal-non-pending-flips<->SP-mutation} is proved.

\subsection{Proof of Proposition \ref{prop:pop-existence}}

\begin{defi}\label{def:val_X,O}
Given any set $X$ of pairwise compatible arcs on $\surf$, and given any marked point $p\in\marked$, we define $\val_{X,\orb}(p)$ to be the number of arcs in $X$ that are pending and incident to $p$. This rule clearly defines a function $\val_{X,\orb}:\marked\rightarrow\mathbb{Z}_{\geq0}$.
\end{defi}

\begin{lemma}\label{lemma:existence-certain-triang} Let $\surf$ be a surface with empty boundary such that $|\marked|\geq 2$ and with the property that $|\marked|\geq 7$ if $\Sigma$ happens to be a sphere.
For every pair of distinct punctures~$p_1,p_2 \in \marked$ and every function~$\vv : \marked \to \Z_{\geq 0}$ satisfying $\sum_{p \in \marked} \vv(p) = |\orb|$ and $\vv(p_2)=0$, there exists an ideal triangulation $\tau$ of $\surf$ and a set of non-pending arcs $\{i_1,i_2,i_3,i_4,i_5,i_6,i_7\}\subseteq\tau$ with the following properties:
\begin{enumerate}
\item $\val_{\tau,\orb} = \vv$
\item $i_1$, $i_2$, $i_3$, $i_4$, $i_5$ and $i_6$ are six different arcs in $\tau$;
\item $i_2$, $i_3$, $i_4$, $i_5$, $i_6$ and $i_7$ are six different arcs in $\tau$;
\item the arcs $i_1$ and $i_7$ are loops;
\item the arcs $i=i_3$ and $j=i_2$ form a self-folded triangle of $\tau$, with $i$ as folded side and $j$ as enclosing loop, the latter being based at $p_1$ and enclosing $p_2$;
\item the full subquiver of $\Qtau$ determined by $\{i_1,i_2,i_3,i_4,i_5,i_6,i_7\}$ is
\begin{center}
$Q \ = \ $\begin{tabular}{c}
$ \xymatrix{
  & i_2 \ar[dl]_{\beta} & & i_5 \ar@<0.5ex>[dd]_{\eta_3 \ } \ar@<-0.5ex>[dd]^{\ \lambda_1} & \\
  i_1  \ar[rr]^{\alpha} & & i_4 \ar[ul]_{\gamma} \ar[dl]^{\varepsilon} \ar[ur]^{\eta_1}  & &  i_7 \ar[ul]_{\lambda_2}\\
 & i_3 \ar[ul]^{\delta} & & i_6 \ar[ul]^{\eta_2} \ar[ur]_{\lambda_3} &
}$
\end{tabular}
\end{center}
\item under the notation just established for the arrows of the full subquiver $Q$ of $\Qtau$, the potentials $\Stau$ and $\Wtau$ are given by the formulas
\begin{eqnarray}\nonumber
\Stau &=& -x_{p_2}^{-1}\alpha\delta\varepsilon+\eta_1\eta_2\eta_3+\lambda_1\lambda_2\lambda_3+\alpha\beta\gamma\\
\nonumber
&&
+Y(x_{p_1}\delta\varepsilon\eta_2\lambda_1\eta_1\alpha\Lambda+x_{p_3}\lambda_3\eta_3\lambda_2\Omega)\\
\nonumber
&&
+
Z(x_{p_1} \delta\varepsilon\eta_2\lambda_1\eta_1\alpha\Lambda\lambda_3\eta_3\lambda_2\Omega)+S'(\tau),\\
\nonumber
\Wtau &=& \alpha\delta\varepsilon+\eta_1\eta_2\eta_3+\lambda_1\lambda_2\lambda_3+x_{p_2}^{-1}\alpha\beta\gamma\\
\nonumber
&&
+Y(x_{p_1}\beta\gamma\eta_2\lambda_1\eta_1\alpha\Lambda+x_{p_3}\lambda_3\eta_3\lambda_2\Omega)\\
\nonumber
&&
+
Z(x_{p_1} \beta\gamma\eta_2\lambda_1\eta_1\alpha\Lambda\lambda_3\eta_3\lambda_2\Omega)+S'(\tau),
\end{eqnarray}
where
\begin{enumerate}
\item $p_1$ is the puncture the loop $i_1$ is based at, $p_2$ is the puncture inside the self-folded triangle formed by $i=i_2$ and $j=i_3$, and
$p_3$ is the puncture the loop $i_7$ is based at;
\item if $p_1\neq p_3$, then $\Lambda\in e_{i_1}\mathfrak{m}e_{i_1}$ and $\Omega\in e_{i_7}\mathfrak{m}e_{i_7}$;
\item if $p_1\neq p_3$, then $Y=1\in K$ and $Z=0\in K$;
\item if $p_1=p_3$, then $\Lambda\in e_{i_1}\RA{\Atau}e_{i_7}$ and $\Omega\in e_{i_7}\RA{\Atau}e_{i_1}$;
\item if $p_1= p_3$, then $Y=0\in K$ and $Z=1\in K$;
\item regardless of whether $p_1\neq p_3$ or $p_1=p_3$, the elements $\Lambda$ and $\Omega$ can be written as finite $F$-linear combinations of paths where none of the arrows $\alpha$, $\beta$, $\gamma$, $\eta_1$, $\eta_2$, $\eta_3$, $\lambda_1$, $\lambda_2$ and $\lambda_3$ appears;
    \item $S'(\tau)\in\RA{\Atau}$ is a potential that can be written as a finite $F$-linear combination of paths where none of the arrows $\alpha$, $\beta$, $\gamma$, $\eta_1$, $\eta_2$, $\eta_3$, $\lambda_1$, $\lambda_2$ and $\lambda_3$ appears.
\end{enumerate}
\end{enumerate}
\end{lemma}

\begin{proof}
Let $g \geq 0$, $m \geq 2$, and $o \geq 0$ be any triple of non-negative integers such that $m \geq 7$ if $g = 0$. Because of the well-known classification of compact connected oriented surfaces without boundary up to diffeomorphisms (in which the diffeomorphism class of such a surface is determined by its genus alone), in order to prove the lemma it is sufficient to show the
existence of a genus-$g$ empty-boundary surface~$\surf$ with $|\marked| = m$ and $|\orb| = o$ satisfying the statement of Lemma \ref{lemma:existence-certain-triang}.

Take a once-punctured cylinder~$(\Sigma_0,\marked_0,\varnothing)$,
and let $\tau_0$ be the ideal triangulation of $(\Sigma_0,\marked_0,\varnothing)$ depicted in Figure \ref{Fig:pop_basic_cylinder}.

\begin{figure}[!h]
  \centering
   \includegraphics[scale=.75]{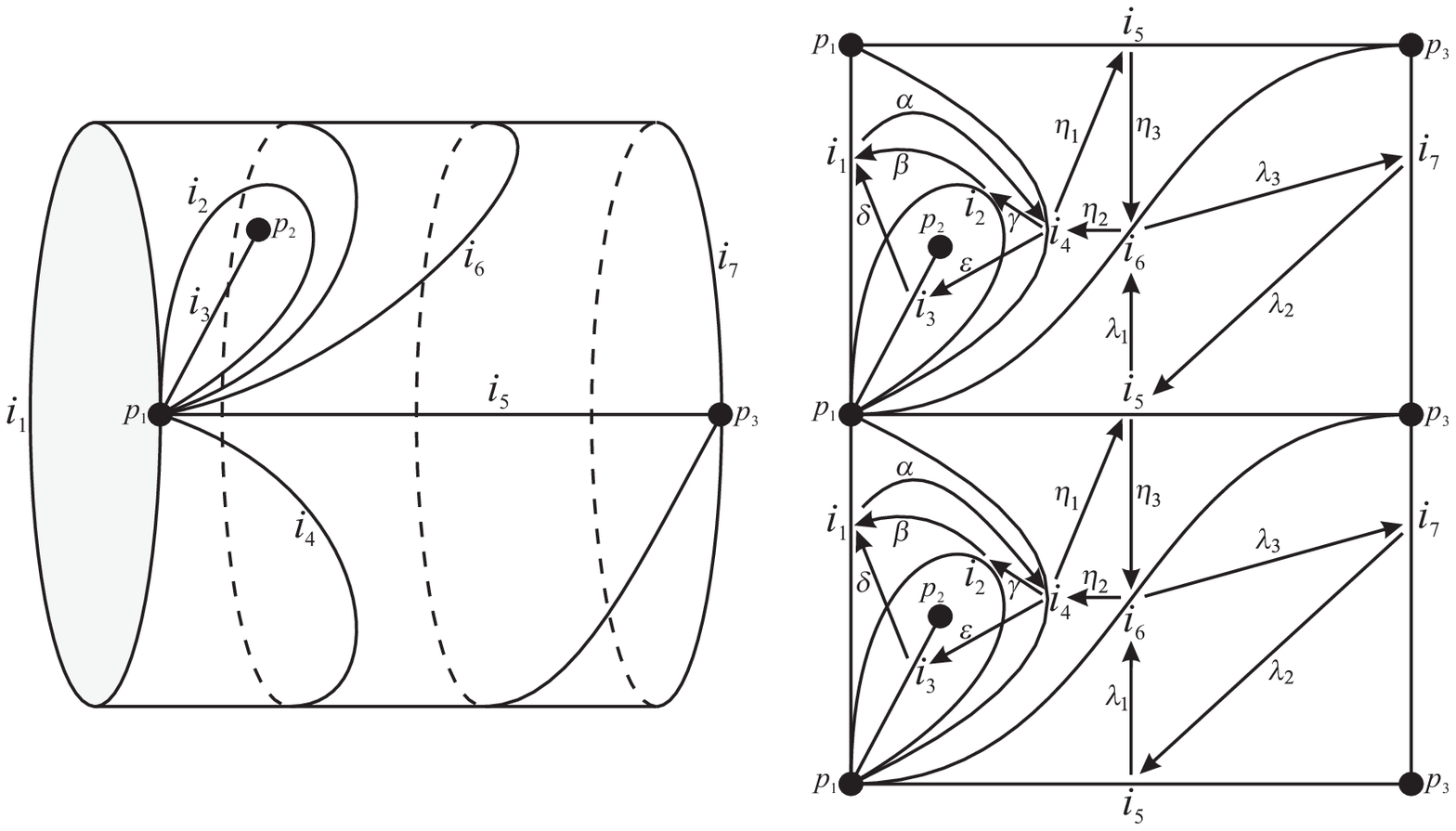}
   \caption{The once-punctured cylinder~$(\Sigma_0,\marked_0,\varnothing)$ and its ideal triangulation~$\tau_0$.}
  \label{Fig:pop_basic_cylinder}
\end{figure}

(a)
If $g > 0$ and $g + m + o > 3$ (equivalently, if $g\geq 1$ and either $g\geq 2$ or $m\geq 3$ or $o\geq 1$), let $(\Sigma_1,\marked_1,\orb)$ be a surface of genus~$g-1$ with exactly one boundary component, exactly $2$ marked points on such component, exactly $m-2$ punctures and exactly $o$ orbifold points.
Take the cylinder $(\Sigma_0,\marked_0,\varnothing)$ and glue $p_1$ and $p_3$, without gluing any other point in $i_1$ to any point in~$i_7$.
To the result~$\Sigma_0'$ of this gluing we then glue the boundary component of $(\Sigma_1,\marked_1,\orb)$ along~$i_1$ and $i_7$ of $\Sigma_0$, making sure that the two marked points on the boundary of  $\Sigma_1$ are glued to $p_1 = p_3 \in \Sigma_0'$ and that the result of this gluing is an oriented empty-boundary surface (for this we take suitable orientations of $i_1$ and $i_7)$. This way we obtain an empty-boundary surface $\surf$ of genus~$g$ with exactly $m$ punctures and exactly $o$ orbifold points.

(b)
If $g > 0$ and $g + m + o \leq 3$ (equivalently, if $g = 1$, $m = 2$, and $o = 0$), glue $i_1$ to $i_7$, making sure that $p_1$ gets identified with $p_3$, and that the gluing is made along orientations of $i_1$ and $i_7$ that ensure that the result of the gluing is an oriented surface. This way we obtain a surface $\surf$ which is a torus with exactly $2$ punctures and without orbifold points.

(c)
If $g = 0$ (which implies $m \geq 7$),
take monogons $(\Sigma_1,\marked_1,\orb_1)$ and $(\Sigma_2,\marked_2,\orb_2)$ satisfying $|\marked_1|\geq 3$, $|\marked_2|\geq 3$, $|\marked_1|+|\marked_2|=m-1$ and $|\orb_1|+|\orb_2|=o$.
Now glue $(\Sigma_1,\marked_1,\orb_1)$ to  $(\Sigma_0,\marked_0,\varnothing)$ and $(\Sigma_2,\marked_2,\orb_2)$ to $(\Sigma_0,\marked_0,\varnothing)$ along boundary components, in such a way that the marked point on the boundary of~$\Sigma_1$ is glued to $p_1 \in \Sigma_0$ and the marked point on the boundary of $\Sigma_2$ is glued to $p_3 \in \Sigma_0$. The result is a surface $\surf$ which is a sphere with exactly $m$ punctures and exactly $o$ orbifold points.

In the cases (a) and (b) it is easy to see that given any function $\vv:\marked\rightarrow\Z_{\geq0}$ such that $\sum_{p\in\marked}\vv(p)=|\orb|$ and $\vv(p_2)=0$, (the collection of arcs on $\Sigma$ induced by) our initial triangulation $\tau_0$ of the cylinder $(\Sigma_0,\marked_0,\varnothing)$ can be completed to a triangulation $\tau$ of $\surf$ satisfying all conditions in the conclusion of the lemma.
In the case (c) it is easy to see that given any function $\vv:\marked\rightarrow\Z_{\geq0}$ such that $\sum_{p\in\marked}\vv(p)=|\orb|$ and $\vv(p_2)=0$, the sets $\orb_1$ and $\orb_2$ can always be chosen in such a way that (the collection of arcs on $\Sigma$ induced by) our initial triangulation $\tau_0$ of the cylinder $(\Sigma_0,\marked_0,\varnothing)$ can be completed to a triangulation $\tau$ of $\surf$ satisfying all conditions in the conclusion of the lemma.
\end{proof}

\begin{prop}\label{prop:specific-pop}
 Let $\tau$ be an ideal triangulation as in the conclusion of Lemma \ref{lemma:existence-certain-triang}. For any choice $\mathbf{x}=(x_p)_{p\in\punct}$ of non-zero elements of $F$, the SP $\AStau$ is right-equivalent to the SP $\AWtau$.
\end{prop}

\begin{proof} This is identical to the proof of \cite[Proposition 6.4]{Labardini-potsnoboundaryrevised}. We emphasize that it relies on a limit process.
\end{proof}

\begin{proof}[Proof of Proposition \ref{prop:pop-existence}] Let $\surf$ be a surface with empty boundary such that $|\marked|\geq 7$ if $\Sigma$ is a sphere, let $\mathcal{S}$ be a collection of $|\orb|$ distinct pending arcs, and let $i,j$ be a pair of arcs forming a self-folded triangle with $i$ as its folded side. Suppose that the elements of $\mathcal{S}\cup\{i,j\}$ are all pairwise compatible. Write $\vv=\val_{\mathcal{S}\cup\{i,j\},\orb}$ (see Definition \ref{def:val_X,O}). By Lemma \ref{lemma:existence-certain-triang} and Proposition \ref{prop:specific-pop}, there exists an ideal triangulation $\tau'$ of $\surf$ such that:
\begin{itemize}
\item $\val_{\tau',\orb}=\val_{\mathcal{S}\cup\{i,j\},\orb}$;
\item $\tau'$ contains arcs $i'$ and $j'$ that form a self-folded triangle;
\item the SPs $(A(\tau'),S(\tau',\mathbf{x}))$ and $(A(\tau'),W^{i'j'}(\tau',\mathbf{x}))$ are right-equivalent.
\end{itemize}
Let $\mathcal{S'}$ be the set consisting of all arcs in $\tau'$ that are pending.
It is not hard to see that there exists an auto-diffeomorphism $\varphi$ of $\Sigma$ satisfying $\varphi(\mathcal{S}')=\mathcal{S}$, $\varphi(i')=i$ and $\varphi(j')=j$. For such a $\varphi$, the collection $\tau=\varphi(\tau')$ is easily seen to be a triangulation of $\surf$ such that $\mathcal{S}\cup\{i,j\}\subseteq\tau$ and with the property that the SPs $\AStau$ and $\AWtau$ are right-equivalent.
\end{proof}

\subsection{Proof of Proposition \ref{prop:orb-pop-stronger}}\label{subsec:orbi-pop}

\begin{lemma}\label{lemma:orb-pop-existence-lemma}
 Let $\surf$ be a surface with empty boundary. If $|\orb|\geq 1$ and $\surf$ satisfies $|\marked|\geq 7$ if $\Sigma$ is a sphere (see Definition \ref{def:surf-with-orb-points}), then for every function $\vv:\marked\rightarrow\mathbb{Z}_{\geq0}$ such that $\sum_{p\in\marked}\vv(p)=|\orb|$ and every orbifold point $q$ there exists an ideal triangulation $\tau$ of $\surf$ and a set of arcs $\{i_1,i_2,i_3,i_4,i_5,i_6\}\subseteq\tau$ with the following properties:
\begin{enumerate}
\item $\val_{\tau,\orb}=\vv$ (see Definition \ref{def:val_X,O});
\item $i_1,i_2,i_3,i_4$ and $i_5$ are five different arcs in $\tau$;
\item $i_2,i_3,i_4,i_5$ and $i_6$ are five different arcs in $\tau$;
\item the arcs $i_1$ and $i_6$ are loops (with basepoints that may or may not coincide; the arcs $i_1$ and $i_6$ themselves may or may not coincide);
\item the arc $i_2$ is a pending arc, and $q$ is the unique orbifold point lying on $i_2$;
\item none of the arcs $i_1,i_3,i_4,i_5$ and $i_6$ is a pending arc;
\item the full subquiver of $Q(\tau)$ determined by $\{i_1,i_2,i_3,i_4,i_5,i_6\}$ is
\begin{center}
$Q \ = \ $\begin{tabular}{c}
$ \xymatrix{
  &  & & i_4 \ar@<0.5ex>[dd]_{\eta_3 \ } \ar@<-0.5ex>[dd]^{\ \lambda_1} & \\
  i_1  \ar[rr]^{\alpha} & & i_3 \ar[dl]^{\varepsilon} \ar[ur]^{\eta_1}  & &  i_6 \ar[ul]_{\lambda_2}\\
 & i_2 \ar[ul]^{\delta} & & i_5 \ar[ul]^{\eta_2} \ar[ur]_{\lambda_3} &
}$
\end{tabular}
\end{center}
\item under the notation just established for the arrows of the full subquiver $Q$ of $\Qtau$, the potentials $\Stau$ and $\Vtauq$ are given by the formulas
    \begin{eqnarray*}
    \Stau &=& \alpha\delta\varepsilon-\alpha\delta v\varepsilon+\eta_1\eta_2\eta_3+\lambda_1\lambda_2\lambda_3\\
    && + Y(x_{p_1}\alpha\Lambda\delta\varepsilon\eta_2\lambda_1\eta_1+x_{p_2}\lambda_3\eta_3\lambda_2\Omega)\\
    && + Z(x_{p_1}\alpha\Lambda\lambda_3\eta_3\lambda_2\Omega\delta\varepsilon\eta_2\lambda_1\eta_1)
    +S'(\tau),\\
    \Vtauq &=& \alpha\delta\varepsilon+\alpha\delta v\varepsilon+\eta_1\eta_2\eta_3+\lambda_1\lambda_2\lambda_3\\
    && + Y((-x_{p_1})\alpha\Lambda\delta\varepsilon\eta_2\lambda_1\eta_1+x_{p_2}\lambda_3\eta_3\lambda_2\Omega)\\
    && + Z((-x_{p_1})\alpha\Lambda\lambda_3\eta_3\lambda_2\Omega\delta\varepsilon\eta_2\lambda_1\eta_1)
    +S'(\tau),
    \end{eqnarray*}
    where
    \begin{enumerate}
    \item $p_1$ is the puncture the loop $i_1$ is based at and the unique puncture contained in the pending arc $i_2$, and $p_2$ is the puncture the loop $i_6$ is based at;
    \item if $p_1\neq p_2$, then $\Lambda\in e_{i_1}\mathfrak{m} e_{i_1}$ and $\Omega\in e_{i_6}\mathfrak{m} e_{i_6}$;
    \item if $p_1\neq p_2$, then $Y=1\in F$ and $Z=0\in F$;
    \item if $p_1=p_2$, then $\Lambda\in e_{i_1}\RA{\Atau}e_{i_6}$ and $\Omega\in e_{i_6}\RA{\Atau}e_{i_1}$;
    \item if $p_1=p_2$, then $Y=0\in F$ and $Z=1\in F$;
    \item regardless of whether $p_1\neq p_2$ or $p_1=p_2$, the elements $\Lambda$ and $\Omega$ can be written as finite $F$-linear combinations of paths where none of the arrows $\alpha$, $\beta$, $\gamma$, $\eta_1$, $\eta_2$, $\eta_3$, $\lambda_1$, $\lambda_2$ and $\lambda_3$ appears;
    \item $S'(\tau)\in\RA{\Atau}$ is a potential that can be written as a finite $F$-linear combination of paths where none of the arrows $\alpha$, $\beta$, $\gamma$, $\eta_1$, $\eta_2$, $\eta_3$, $\lambda_1$, $\lambda_2$ and $\lambda_3$ appears.
    \end{enumerate}
\end{enumerate}
\end{lemma}

\begin{proof} This is extremely similar to the proof of Lemma \ref{lemma:existence-certain-triang}. Indeed, if instead of taking the once-punctured cylinder from Figure \ref{Fig:pop_basic_cylinder} we take  as our initial surface the triangulated cylinder with one orbifold point shown in Figure \ref{Fig:orb_pop_basic_cylinder},
        \begin{figure}[!h]
                \centering
                \includegraphics[scale=.75]{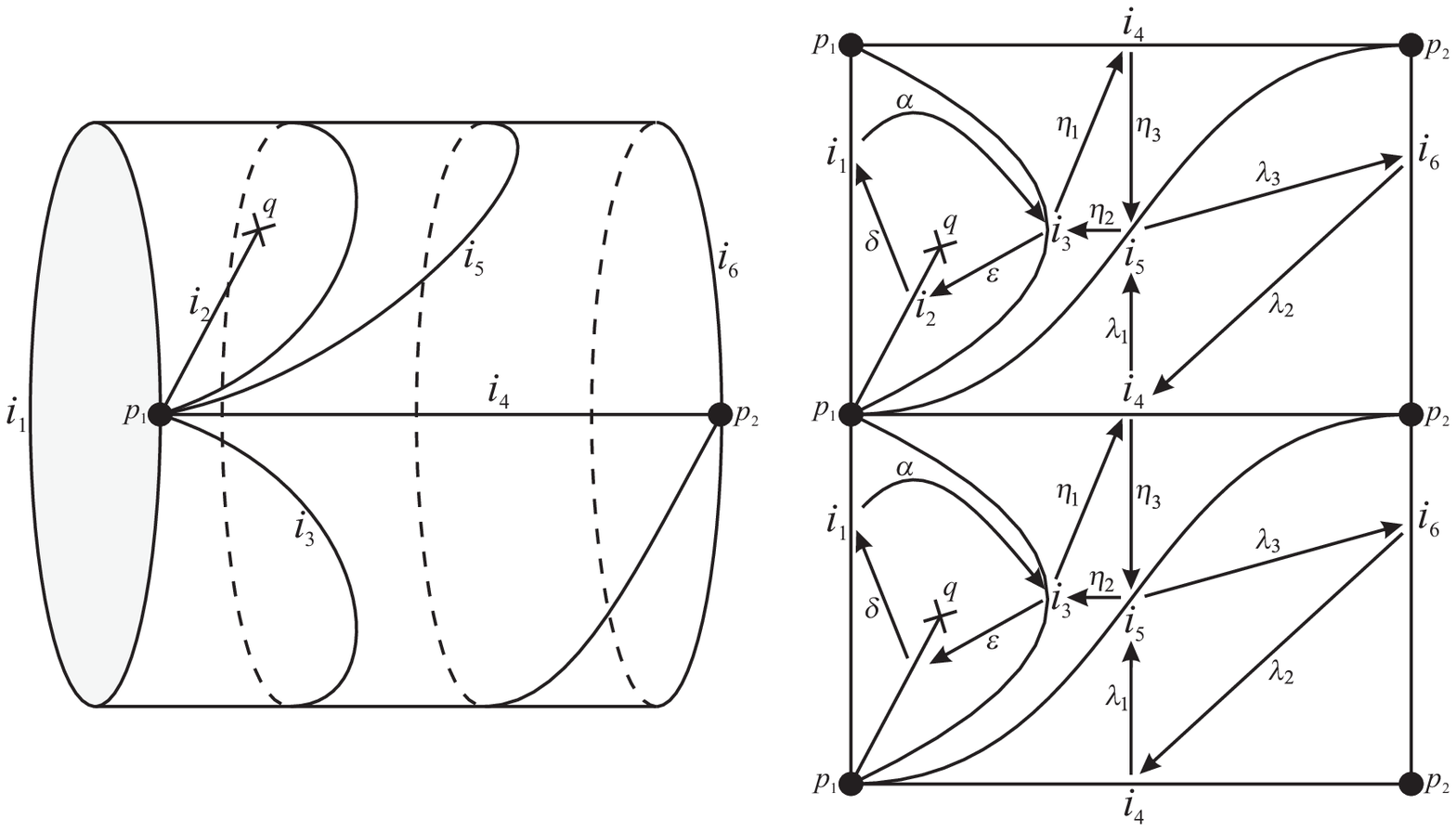}
                \caption{When $|\orb|\geq 1$ and $\surf$ satisfies $|\marked|\geq 7$ if $\Sigma$ is a sphere, the triangulation of the cylinder with one orbifold point depicted in this figure can be completed to a triangulation $\tau$ of $\surf$ satisfying the conclusion of Lemma \ref{lemma:orb-pop-existence-lemma}.}\label{Fig:orb_pop_basic_cylinder}
        \end{figure}
the rest of the construction sketched in the proof of Lemma \ref{lemma:existence-certain-triang} can be applied without modification to produce a triangulation $\tau$ satisfying the conclusion of Lemma \ref{lemma:orb-pop-existence-lemma}.
\end{proof}

\begin{prop}\label{prop:specific-orb-pop-arbitrary-genus} Let $\surf$ be a surface with empty boundary. Suppose that $\tau$ is a triangulation of $\surf$ satisfying the conclusion of Lemma \ref{lemma:orb-pop-existence-lemma}. With the notation used therein, the SPs $\AStau$ and $\AVtauq$ are right equivalent.
\end{prop}

Proposition \ref{prop:specific-orb-pop-arbitrary-genus} is a consequence of the next two lemmas, the proof of the first of which is a minor modification of the proof of \cite[Lemma 6.5]{Labardini-potsnoboundaryrevised}, modification that we choose not to omit here (because of the presence of the orbifold point $q$). Except for this minor modification, the whole proof of Proposition \ref{prop:specific-orb-pop-arbitrary-genus} is almost identical to the proof of \cite[Proposition 6.4]{Labardini-potsnoboundaryrevised}.

\begin{lemma}\label{lemma:S-is-W+trash-orb-pop} For the ideal triangulation $\tau$ above, the SP $\AStau$ is right-equivalent to a SP of the form $(\Atau,\Vtauq+L_1)$, where $L_1\in \RA{\Atau}$
is a potential satisfying the following condition:
\begin{eqnarray}
\label{eq:forced-factors-orb-pop}
&&\text{it can be written as $\eta_2\lambda_1\lambda_2\Theta_1+\eta_2\eta_3\lambda_2\Theta_2$ for some $\Theta_1,\Theta_2\in\maxid^4\subseteq\RA{\Atau}$.}
\end{eqnarray}
\end{lemma}

\begin{proof} The proof of \cite[Lemma 6.5]{Labardini-potsnoboundaryrevised} can almost be applied as is, but it requires a small modification. We choose not to skip details. Write
$\Pi=Y\lambda_1\eta_1\alpha\Lambda\Lambda\delta v\varepsilon+Z\lambda_1\eta_1\alpha\Lambda\lambda_3\eta_3\lambda_2\Omega\Lambda\lambda_3\eta_3\lambda_2\Omega\delta v\varepsilon$ and
define $R$-algebra automorphisms $\psi_\delta,\psi_{\alpha},\psi_{\eta_3}:\RA{\Atau}\rightarrow\RA{\Atau}$ by means of the rules
\begin{eqnarray*}
\psi_\delta &:& \delta\mapsto\delta v,\\
\psi_\alpha &:& \alpha\mapsto
\alpha
-Y(x_{p_1}\eta_2\lambda_1\eta_1\alpha\Lambda)
- Z(x_{p_1}\eta_2\lambda_1\eta_1\alpha\Lambda\lambda_3\eta_3\lambda_2\Omega),\\
\nonumber
\psi_{\eta_3} & : & \eta_3\mapsto\eta_3+x_{p_1}^2\Pi\eta_2\lambda_1.
\end{eqnarray*}
Direct computation shows that
\begin{eqnarray}\nonumber
\psi_{\eta_3}\psi_\alpha\psi_{\delta}(\Stau) &\sim_{\operatorname{cyc}}&
\Vtauq\\
\nonumber
&&
+Yx_{p_1}^2x_{p_2}\eta_2\lambda_1\lambda_2\Omega\lambda_3\Pi\\
\nonumber
&&
+Z\left(-x_{p_1}^3\eta_2\lambda_1\lambda_2\Omega\delta\varepsilon\eta_2\lambda_1\eta_1\alpha\Lambda\lambda_3\Pi\right.\\
\nonumber
&&
-x_{p_1}^4\eta_2\lambda_1\lambda_2\Omega\Lambda\lambda_3\eta_3\lambda_2\Omega\delta v\varepsilon\eta_2\lambda_1\eta_1\eta_2\lambda_1\eta_1\alpha\Lambda\lambda_3\Pi\\
\nonumber
&&
-x_{p_1}^4\eta_2\lambda_1\lambda_2\Omega\delta v\varepsilon\eta_2\lambda_1\eta_1\eta_2\lambda_1\eta_1\alpha\Lambda\lambda_3\eta_3\lambda_2\Omega\Lambda\lambda_3\Pi\\
\nonumber
&&
\left.-x_{p_1}^6\eta_2\lambda_1\lambda_2\Omega\delta v\varepsilon\eta_2\lambda_1\eta_1\eta_2\lambda_1\eta_1\alpha\Lambda\lambda_3\Pi\eta_2\lambda_1\lambda_2\Omega\Lambda\lambda_3\Pi\right).
\end{eqnarray}
This proves the Lemma.
\end{proof}

Recall that for a non-zero element $u$ of a complete path algebra $\RA{A}$, we denote by $\short(u)$ the largest integer $n$ with the property that $u\in\maxid^{n}$ but $u\notin\maxid^{n+1}$, and that we set $\short(0)=\infty$.

\begin{lemma}\label{lemma:making-tails-longer-arbitrary-genus-orb} Let $\tau$ be the ideal triangulation from the conclusion of Lemma \ref{lemma:orb-pop-existence-lemma}. Suppose $L\in \RA{\Atau}$ is a non-zero potential satisfying
\eqref{eq:forced-factors-orb-pop}.
Then there exist a potential $L'\in \RA{\Atau}$ and a right-equivalence
$\varphi:(\Atau,\Vtauq+L)\rightarrow(\Atau,\Vtauq+L')$, such that:
\begin{itemize}
\item $L'$ satisfies
\eqref{eq:forced-factors-orb-pop};
\item $\short(L')>\short(L)$;
\item $\varphi$ is a unitriangular automorphism of $\RA{\Atau}$, and $\depth(\varphi)=\short(L)-3$.
\end{itemize}
\end{lemma}

\begin{proof} The proof of \cite[Lemma 6.6]{Labardini-potsnoboundaryrevised} can be applied here as is, with no modification at all.
\end{proof}

\begin{proof}[Proof of Proposition \ref{prop:specific-orb-pop-arbitrary-genus}] A combination of Lemmas \ref{lemma:S-is-W+trash-orb-pop} and \ref{lemma:making-tails-longer-arbitrary-genus-orb} produces a limit process whose upshot is the desired right-equivalence. The reader interested in the details is kindly referred to the discussion surrounding Equations (6.7)--(6.10) in \cite{Labardini-potsnoboundaryrevised}, which can be applied to our current situation without any changes whatsoever.
\end{proof}


\begin{thebibliography}{50}







\bibitem{ACCERV1} M. Alim, S. Cecotti, C. Cordova, S. Espahbodi, A. Rastogi, C. Vafa. {\it BPS Quivers and Spectra of Complete N=2 Quantum Field Theories}. Communications in Mathematical Physics, Vol. 323 (2013), Issue 3, 1185--1227. arXiv:1109.4941

\bibitem{ACCERV2} M. Alim, S. Cecotti, C. Cordova, S. Espahbodi, A. Rastogi, C. Vafa. {\it N=2 Quantum Field Theories and Their BPS Quivers}. Adv. Theor. Math. Phys. 18 (2014) 27--127. arXiv:1112.3984

\bibitem{Amiot}  C. Amiot. {\it Cluster categories for algebras of global dimension 2 and quivers with potential}. Ann. Inst. Fourier \textbf{59} (2009), 2525-2590.  arXiv:0805.1035

\bibitem{Bautista-Lopez} R. Bautista, D. L\'opez-Aguayo. {\it Potentials for some tensor algebras}. arXiv:1506.05880

\bibitem{Bridgeland-Smith} T. Bridgeland, I. Smith. {\it Quadratic differentials as stability conditions}. Publications math\'ematiques de l'IH\'ES
June 2015, Volume 121, Issue 1, 155--278. arXiv:1302.7030

\bibitem{Cecotti} S. Cecotti. {\it Categorical Tinkertoys for N=2 Gauge Theories}. Int. J. of Modern Physics A (Particles and Fields; Gravitation; Cosmology) Vol. 28 (2013), Issue 05, No. 06. arXiv:1203.6734

\bibitem{CI-LF} G. Cerulli Irelli, D. Labardini-Fragoso. {\it Quivers with potentials associated to triangulated surfaces, part III: Tagged triangulations and cluster monomials}. Compositio Mathematica 148 (2012), No. 06, 1833--1866. arXiv:1108.1774



\bibitem{Chekhov-Shapiro} L. Chekhov, M. Shapiro. {\it Teichm\"uller spaces of Riemann surfaces with orbifold points of arbitrary order and cluster variables}. Int Math Res Notices (2014) 2014 (10): 2746--2772. doi:10.1093/imrn/rnt016.  arXiv:1111.3963

\bibitem{DMSS} B. Davison, D. Maulik, J. Schuermann, B. Szendroi. {\it Purity for graded potentials and quantum cluster positivity}. Compositio Mathematica, 2015. doi:10.1112/S0010437X15007332 arXiv:1307.3379

\bibitem{Demonet} L. Demonet. {\it Mutations of group species with potentials and their representations. Applications to cluster algebras}. arXiv:1003.5078

\bibitem{DWZ1} H. Derksen, J. Weyman, A. Zelevinsky. {\it Quivers with potentials and their representations I: Mutations}. Selecta Math. 14 (2008), no. 1, 59-"1¤7119. arXiv:0704.0649

\bibitem{DWZ2} H. Derksen, J. Weyman, A. Zelevinsky. {\it Quivers with potentials and their representations II: Applications to cluster algebras}. J. Amer. Math. Soc. 23 (2010), No. 3, 749-790. 	arXiv:0904.0676

\bibitem{DR} V. Dlab, C. M. Ringel. {\it Indecomposable representations of graphs and algebras}. Memoirs of the AMS, \textbf{6}, No. 173. 1976.

\bibitem{FeShTu-unfoldings} A. Felikson, M. Shapiro, P. Tumarkin. {\it Cluster algebras of finite mutation type via unfoldings}.
Int. Math. Res. Not. 2012, no. 8, 1768-"1¤71804. arXiv:1006.4276

\bibitem{FeShTu-orbifolds} A. Felikson, M. Shapiro, P. Tumarkin.  {\it Cluster algebras and triangulated orbifolds}. Advances in Mathematics 231(5): 2953-3002. arXiv:1111.3449

\bibitem{FST} S. Fomin, M. Shapiro, D. Thurston. {\it Cluster algebras and triangulated surfaces, part I: Cluster complexes}. Acta Mathematica 201 (2008), 83-146. arXiv:math.RA/0608367

\bibitem{FT} S. Fomin, D. Thurston. {\it Cluster algebras and triangulated surfaces. Part II: Lambda lengths}. arXiv:1210.5569.

\bibitem{FZ1} S. Fomin, A. Zelevinsky. {\it Cluster algebras I: Foundations}. Journal Amer. Math. Soc. 15 (2002), no. 2, 497--529. arXiv:math/0104151


\bibitem{FZ2} S. Fomin, A. Zelevinsky. {\it Cluster algebras II: Finite type classification}. Invent. math., 154 (2003), no.1, 63-121. math.RA/0208229



\bibitem{Geuenich-Labardini-2} J. Geuenich, D. Labardini Fragoso. Preprint in preparation.

\bibitem{Keller-Yang} B. Keller, D. Yang. {\it Derived equivalences from mutations of quivers with potentials}. Advances in Mathematics
Vol. 226 (2011), Issue 3, 2118--2168. arXiv:0906.0761

\bibitem{Labardini1} D. Labardini-Fragoso. {\it Quivers with potentials associated to triangulated surfaces}. Proc. London Mathematical Society (2009) 98 (3): 797-839. arXiv:0803.1328



\bibitem{Labardini-potsnoboundaryrevised} D. Labardini-Fragoso. {\it Quivers with potentials associated to triangulated surfaces, part IV: Removing boundary assumptions}. Selecta Mathematica (New series), Vol. 22 (2016), Issue 1, 145--189. DOI: 10.1007/s00029-015-0188-8  arXiv:1206.1798

\bibitem{LZ} D. Labardini-Fragoso, A. Zelevinsky. {\it Strongly primitive species with potentials I: Mutations}. Bolet\'in de la Sociedad Matem\'atica Mexicana (Third series),  Vol. 22 (2016), Issue 1, 47--115. arXiv:1306.3495.



\bibitem{Mosher} L. Mosher. {\it Tiling the projective foliation space of a punctured surface}. Trans. Amer. Math. Soc. 306 (1988), 1--70.

\bibitem{Musiker-classical} G. Musiker. {\it A graph theoretic expansion formula for cluster algebras of classical type}. Annals of Combinatorics,
 Volume 15, Issue 1 (2011), 147--184.

\bibitem{Nagao1} K. Nagao. {\it Mapping class group, Donaldson-Thomas theory and S-duality}. \\ \url{http://www.math.nagoya-u.ac.jp/~kentaron/MCG_DT.pdf}

\bibitem{Nagao2} K. Nagao. {\it Triangulated surface, mapping class group and Donaldson-Thomas theory}. Proceedings of the Kinosaki algebraic geometry symposium 2010. \\
    \url{http://www.math.sci.osaka-u.ac.jp/~kazushi/kinosaki2010/nagao.pdf}

\bibitem{Nguefack} B. Nguefack. {\it Modulated quivers with potentials and their Jacobian algebras}. arXiv:1004.2213

\bibitem{Penner} R. Penner. {\it Decorated Teichm\"uller Theory}. European Mathematical Society, the QGM Master Class Series. 2012. DOI 10.4171/075

\bibitem{Smith} I. Smith. {\it Quiver algebras as Fukaya categories}. Geom. Topol. 19 (2015), no. 5, 2557--2617. arXiv:1309.0452

\bibitem{Zelevinsky-unfolding} A. Zelevinsky. Unpublished definition of \emph{global unfoldings of skew-symmetrizable matrices}.


























\end{thebibliography}
\end{document}